\documentclass[a4paper,11pt]{article}
\usepackage[
    a4paper,
    top=1.05in,
    bottom=1.05in,
    left=1.00in,
    right=1.00in
]{geometry}
\usepackage[T1]{fontenc}

\usepackage{amsmath,amsthm,mathtools}
\usepackage{amssymb,mathrsfs}
\usepackage{txfonts}
\InputIfFileExists{ursfs.fd}{}{}
\DeclareFontShape{U}{rsfs}{m}{n}{<-6.5>rsfs5<6.5-8>rsfs7<8->rsfs10}{}
\DeclareMathAlphabet{\mathsf}{OT1}{phv}{m}{n}
\SetMathAlphabet{\mathsf}{bold}{OT1}{phv}{b}{n}
\usepackage{microtype}
\microtypesetup{protrusion=true,expansion=true}
\allowdisplaybreaks[4]

\usepackage{enumitem}
\setlist[enumerate]{leftmargin=2.5em,itemsep=0.25em,topsep=0.35em}
\setlist[itemize]{leftmargin=2.2em,itemsep=0.25em,topsep=0.35em}
\usepackage{booktabs,array,longtable}

\usepackage{titlesec}
\usepackage[titletoc]{appendix}
\titleformat{\section}{\normalfont\Large\bfseries}{\thesection}{0.75em}{}
\titleformat{\subsection}{\normalfont\large\bfseries}{\thesubsection}{0.75em}{}
\titlespacing*{\section}{0pt}{2.6ex plus 0.8ex minus 0.2ex}{1.2ex plus 0.2ex}
\titlespacing*{\subsection}{0pt}{2.0ex plus 0.6ex minus 0.2ex}{0.9ex plus 0.2ex}

\usepackage{fancyhdr}
\fancypagestyle{plain}{%
  \fancyhf{}%
  \fancyfoot[C]{\thepage}%
}

\usepackage{float}
\usepackage[dvipsnames]{xcolor}
\definecolor{linkblue}{RGB}{20,67,128}
\definecolor{citegreen}{RGB}{25,105,80}
\definecolor{urlpurple}{RGB}{100,55,125}
\usepackage[colorlinks=true,
            linkcolor=linkblue,
            citecolor=citegreen,
            urlcolor=urlpurple,
            plainpages=false]{hyperref}
\usepackage{bookmark}
\usepackage{aliascnt}
\usepackage[nameinlink,noabbrev,capitalise]{cleveref}
\usepackage{tikz-cd}
\usetikzlibrary{arrows.meta,calc}
\usepackage{multicol}
\let\templatebibliography\thebibliography
\let\endtemplatebibliography\endthebibliography
\renewenvironment{thebibliography}[1]{%
  \templatebibliography{#1}%
  \setlength{\itemsep}{2pt}%
  \setlength{\parsep}{0pt}%
}{\endtemplatebibliography}

\numberwithin{equation}{section}
\makeatletter
\renewcommand{\theequation}{\thesection.\ifnum\value{equation}<10 0\fi\arabic{equation}}
\makeatother

\theoremstyle{plain}
\newtheorem{Theorem}{Theorem}[section]

\theoremstyle{plain}
\newaliascnt{Proposition}{Theorem}
\newtheorem{Proposition}[Proposition]{Proposition}
\aliascntresetthe{Proposition}
\newaliascnt{Lemma}{Theorem}
\newtheorem{Lemma}[Lemma]{Lemma}
\aliascntresetthe{Lemma}
\newaliascnt{Corollary}{Theorem}
\newtheorem{Corollary}[Corollary]{Corollary}
\aliascntresetthe{Corollary}
\newaliascnt{Claim}{Theorem}

\aliascntresetthe{Claim}
\newaliascnt{Question}{Theorem}

\aliascntresetthe{Question}
\newaliascnt{Conjecture}{Theorem}

\aliascntresetthe{Conjecture}

\theoremstyle{definition}
\newaliascnt{Definition}{Theorem}
\newtheorem{Definition}[Definition]{Definition}
\aliascntresetthe{Definition}
\newaliascnt{Example}{Theorem}

\aliascntresetthe{Example}

\theoremstyle{remark}
\newaliascnt{Remark}{Theorem}

\aliascntresetthe{Remark}

\crefname{Theorem}{Theorem}{Theorems}
\Crefname{Theorem}{Theorem}{Theorems}
\crefname{Proposition}{Proposition}{Propositions}
\Crefname{Proposition}{Proposition}{Propositions}
\crefname{Lemma}{Lemma}{Lemmas}
\Crefname{Lemma}{Lemma}{Lemmas}
\crefname{Corollary}{Corollary}{Corollaries}
\Crefname{Corollary}{Corollary}{Corollaries}
\crefname{Definition}{Definition}{Definitions}
\Crefname{Definition}{Definition}{Definitions}
\crefname{Remark}{Remark}{Remarks}
\Crefname{Remark}{Remark}{Remarks}

\crefformat{equation}{(#2#1#3)}
\Crefformat{equation}{(#2#1#3)}
\crefrangeformat{equation}{(#3#1#4)--(#5#2#6)}
\Crefrangeformat{equation}{(#3#1#4)--(#5#2#6)}
\crefmultiformat{equation}{(#2#1#3)}{ and~(#2#1#3)}{, (#2#1#3)}{ and~(#2#1#3)}
\Crefmultiformat{equation}{(#2#1#3)}{ and~(#2#1#3)}{, (#2#1#3)}{ and~(#2#1#3)}

\crefformat{enumi}{(#2#1#3)}
\Crefformat{enumi}{(#2#1#3)}
\crefformat{enumii}{(#2#1#3)}
\Crefformat{enumii}{(#2#1#3)}
\crefformat{enumiii}{(#2#1#3)}
\Crefformat{enumiii}{(#2#1#3)}
\crefformat{enumiv}{(#2#1#3)}
\Crefformat{enumiv}{(#2#1#3)}

\DeclareMathOperator{\tr}{tr}
\DeclareMathOperator{\Vol}{Vol}

\DeclareMathOperator{\inj}{inj}

\DeclareMathOperator{\End}{End}

\DeclareMathOperator{\Ker}{Ker}

\DeclareMathOperator{\Dom}{Dom}

\newcommand{\R}{{\mathbb R}}

\renewcommand{\geq}{\geqslant}
\renewcommand{\leq}{\leqslant}

\newcommand{\Rm}{\mathrm{Rm}}
\newcommand{\Ric}{\mathrm{Ric}}

\providecommand{\hM}{\hat M}
\providecommand{\hg}{\hat g}
\providecommand{\hf}{\hat f}
\providecommand{\hb}{\hat b}
\providecommand{\hnu}{\hat\nu}
\providecommand{\hnd}{\hat\nabla}
\providecommand{\gz}{g^z}
\providecommand{\fz}{f^z}
\providecommand{\Id}{\operatorname{Id}}
\providecommand{\osc}{\operatorname{osc}}

\providecommand{\eps}{\epsilon}

\DeclareMathOperator{\ISD}{ISD}

\hypersetup{
  pdftitle={Paper II, arxiv version},
  pdfauthor={Yu Li, Longhuan Yan, and Wenjia Zhang},
  pdfsubject={Strong uniqueness of AC tangent flows},
  pdfkeywords={Ricci shrinkers, AC geometry, weighted analysis, Ricci flow}
}
\title{On the geometry of asymptotically conical \\Ricci shrinkers (II): strong uniqueness}
\author{Yu Li \and Longhuan Yan \and Wenjia Zhang}
\date{}

\begin{document}
\maketitle
\begin{abstract}
We prove strong uniqueness of tangent flows at the first singular time of a
closed Ricci flow when one tangent flow is generated by a complete smooth
asymptotically conical Ricci shrinker. The modified flow converges to this
shrinker in a fixed gauge, with quantitative estimates on expanding domains.
The main ingredients are a \L ojasiewicz--Simon inequality for Perelman's
entropy, based on the global weighted divergence-free gauge constructed in
the preceding paper, and estimates controlling the geometric radius under
Ricci flow.
\end{abstract}

\tableofcontents
\section{Introduction}\label{sec:introduction}

Tangent flows describe the geometry of a geometric flow near a singularity.
Compactness gives limits along sequences of parabolic rescalings, but the
resulting model may a priori depend on the sequence. Uniqueness identifies
one model at all sufficiently small scales and is a basic step toward
understanding the local geometry and structure of the singular set. A
stronger question asks whether the rescaled flow converges in one fixed
gauge, with a quantitative rate. Such convergence compares the geometry at
different scales through a single system of coordinates. We study these
questions for Ricci flow when one tangent flow is generated by a smooth
asymptotically conical (AC) Ricci shrinker.

A classical approach to uniqueness uses a \L ojasiewicz--Simon inequality
for a monotone functional. Simon \cite{simon1983asymptotics} extended the
finite-dimensional \L ojasiewicz inequality to analytic variational
problems by a Lyapunov--Schmidt reduction. One solves the equation in the
complement of the kernel of the linearized operator and obtains an
analytic function on the finite-dimensional kernel. The classical
inequality for this reduced function then controls the original
functional by its gradient. Together with the flow's energy identity,
this yields convergence and quantitative decay. Simon also proved
uniqueness of tangent cones of minimal varieties when one tangent cone
has multiplicity one and a smooth link. This method is particularly effective when the
analytic problem is posed on a compact manifold, where elliptic theory
provides the required Fredholm estimates after the geometric symmetries
have been removed.

Applications include the asymptotic analysis of harmonic maps into
analytic targets \cite{simon1983asymptotics}, Yang--Mills flow in
dimensions two and three \cite{rade1992yangmills}, and convergence rates
for the Yamabe flow \cite{carlotto2015slowly}.
Colding--Minicozzi \cite{colding2014tangent} used this approach to prove
uniqueness of tangent cones with smooth cross-sections for Ricci-flat
manifolds with Euclidean volume growth and for noncollapsed Einstein
limits. Their reduction takes place on the compact cross-section of the
cone. For Ricci flow, Sun--Wang \cite[Lemmas~2.2 and~3.1]{sun2015on}
established analytic dependence of the entropy minimizer and a
\L ojasiewicz inequality for Perelman's $\boldsymbol\mu$-functional near a
compact Ricci shrinker. Chan--Ma--Zhang \cite{chan2024closed} proved
uniqueness when one tangent flow is closed and smooth. In mean curvature
flow, Schulze \cite{schulze2014compact} proved uniqueness of compact,
smooth, multiplicity-one tangent flows by applying Simon's inequality to
Gaussian area.

Noncompact tangent flows require additional control at infinity.
Convergence on compact subsets does not place the whole rescaled flow in
a neighborhood where a global analytic inequality applies. Smallness in
a Gaussian integral norm also gives little pointwise control far from
the base point. Thus the choice of function spaces and the extension of
local geometric information to larger domains become part of the
uniqueness argument. The compact theory does not apply directly to a
general noncompact shrinker.

For cylindrical mean curvature tangent flows, Colding--Minicozzi
\cite{colding2015uniqueness} proved \L ojasiewicz inequalities directly,
using the explicit structure of the cylinder. Fang--Li
\cite{fangli2026strong} established strong uniqueness for standard round
cylindrical Ricci tangent flows and their quotients. They extended this
analysis to generalized cylinders $\mathbb R^k\times N$, where $N$ is a
closed Einstein manifold satisfying an obstruction-of-order-three
assumption and a spectral condition, and proved rectifiability of the
corresponding singular sets \cite{fangli2026generalized}. These proofs
derive pointed $\mathcal W$-entropy inequalities from the product
geometry and deformation estimates for the compact factor, and then
control the geometric radius to obtain convergence in a fixed gauge.
The comparison maps are
improved by successive gauge corrections, with the remaining weighted
divergence error retained in the estimates. This construction does not
provide a global exact weighted divergence free gauge. Our aim is
to recover a finite-dimensional analytic reduction for AC Ricci
shrinkers without an explicit description of their deformation spaces.

The work of Chodosh--Schulze \cite{chodosh2021conical} on
multiplicity-one AC mean curvature tangent flows is especially relevant.
They introduce weighted H\"older and Sobolev spaces that permit a
Lyapunov--Schmidt reduction, and develop an extension-of-closeness
argument to apply the resulting inequality along the flow. We use a
similar distinction between a global analytic inequality and the
geometric estimates needed for its application. There are two additional
issues for Ricci flow. Nearby hypersurfaces in their argument are
represented by normal graphs, whereas metric perturbations here require
a global weighted divergence free gauge. Moreover, their analytic
functional is Gaussian area at a fixed center and scale. Perelman's
functional $\boldsymbol\mu(g)$ is defined by minimizing over normalized
potentials, so its analyticity requires existence, uniqueness, and
analytic dependence of the minimizer on the metric.

The AC hypothesis supplies the structure at infinity needed for these
constructions. The canonical ends admit controlled asymptotic
expansions, and the weighted estimates can track the coefficients and
remainders separately. They also give the control needed to extend a
local comparison along the flow and estimate the errors from cutting
off the perturbation. The role of the AC condition is therefore both
analytic and geometric.

The preceding paper \cite[Theorem~6.6]{LYZ-I} constructs a global
weighted divergence free gauge using a fixed Banach graph. At coordinate
order $10$, a small input $h\in E_{14}$ has a
$C^1$ coordinate parameter and a global diffeomorphism $F$ with
\[
 \tilde h=F^*(\hat g+h)-\hat g\in E_9,
 \qquad \hat D\tilde h=0.
\]
The original metric spaces retain tangential degree-zero cone variation,
the degree-minus-one terms, and ordinary weight-minus-two remainders.
In \cref{sec:finite-linear-estimates}, we derive the finite scalar
and tensor estimates from the weighted analysis of \cite{LYZ-I}.
These estimates supply the linear inputs for the graphs used here.
The potential pullback estimate is proved in
\cref{lem:finite-potential-pullback}. Simultaneous pullback preserves
the entropy functionals and the physical defect integral by change of variables.

We now state our first main result. Fix a complete connected nontrivial
AC Ricci shrinker $(\hM^n,\hg,\hf)$, $n\ge3$, with
\[
 \Ric_{\hg}+\hnd^2\hf=\frac12\hg,\qquad
 (4\pi)^{-n/2}\int_{\hM}e^{-\hf}\,\mathrm dV_{\hg}=1.
\]
Here nontrivial means non-Gaussian. A Gaussian tangent flow corresponds
to a regular point, where tangent-flow uniqueness is automatic; see
\cite[Definition~7.14 and Theorem~7.15]{fangli2026structure}.
For a metric--potential pair, write
\[
 \mathrm dV_f=(4\pi)^{-n/2}e^{-f}\,\mathrm dV_g,\qquad
 \mathcal W(g,f)=\int_{\hM}
 \bigl(\mathrm R_g+|\nabla^gf|_g^2+f-n\bigr)\,\mathrm dV_f.
\]
The scale is fixed to be $1$. The metric expansion space $E_{14}$
and the ordinary weighted potential space $C^{3,1/4}_1$ are defined
in Section~\ref{sec:preliminaries}.

\begin{Theorem}\label{thm:intro-pair-LS}
There exist $\epsilon,C>0$ and $\vartheta\in(0,\frac12]$ such
that every Riemannian metric $g$ and potential $f$ satisfying
\[
 g-\hg\in E_{14},\qquad f-\hf\in C^{3,1/4}_1,\qquad
 \int_{\hM}\mathrm dV_f=1,\qquad
 \|g-\hg\|_{E_{14}}+\|f-\hf\|_{C^{3,1/4}_1}<\epsilon
\]
obey
\begin{equation}\label{eq:intro-pair-LS}
 \bigl|\mathcal W(g,f)-\mathcal W(\hg,\hf)\bigr|^{1-\vartheta}
 \le C\left(\int_{\hM}
 \left|\Ric_g+\nabla_g^2f-\frac12g\right|_g^2
 \,\mathrm dV_f\right)^{1/2}.
\end{equation}
\end{Theorem}

This is proved in Theorem~\ref{thm:global-pair-LS}. The estimate is global
over $\hM$, with smallness measured in the finite
$E_{14}$ and $C^{3,1/4}_1$ norms. No integrability or nondegeneracy assumption is imposed on the
infinitesimal solitonic deformation space. The potential $f$ is any
normalized potential in this neighborhood, not necessarily the
minimizer $f_g$ of $\mathcal W(g,\cdot)$. This distinction is needed in
the application to Ricci flow, where $f$ is a conjugate heat-kernel
potential. The potential space allows linear growth, as a bounded radial
displacement in the gauge can produce a linear potential difference.

The analytic argument uses two fixed inverse-image Banach graphs.
On the sliced metric graph, the radial identities express contractions
with $\hat\nabla\hat f$ in terms of divergence and derivatives of the
metric perturbation. These identities close the nonlinear scalar
Euler--Lagrange equation, including its leading coefficient, on the
scalar graph. The analytic Banach implicit function theorem constructs
the normalized branch. Variational compactness, exterior estimates,
and a one-form spectral gap identify it with the unique entropy
minimizer. The scalar branch and its entropy are analytic on the
sliced graph; the finite coordinate slice transfers existence and
uniqueness to the original metric classes.

The tensor graph is the Gaussian inverse image of a negative fractional
weighted H\"older target, with the infinitesimal solitonic deformation
space retained as a finite-dimensional summand. Whole-field estimates
and the same radial identities make the projected soliton defect
analytic on this graph. The ordinary analytic implicit function theorem
solves its normal equation and gives an analytic finite-dimensional
center. This construction permits a nonzero, nonintegrable deformation
space.

Since potential differences may grow linearly, the moving Gaussian
measures need not be uniformly equivalent. Section~\ref{sec:appendix-weakened-gaussian}
retains fixed weakened Gaussian weights, the weighted divergence
projection, and an $L^2$ reconstruction estimate. These compare nearby
metrics to the analytic center and yield a quadratic entropy error.
The finite-dimensional \L ojasiewicz inequality and physical-measure
interpolation give the metric inequality, which is then extended to
arbitrary normalized potentials. Only the original metric norm $E_{14}$
and a finite weighted potential norm need to be small.

We next describe the application to tangent flows. After translating
and rescaling time, let $(M^n,g(t))_{-2\le t<0}$ be a smooth closed
connected Ricci flow whose first singular time is $0$, with entropy
bounded below by $-Y$. Let $(Z,d_Z,\mathfrak t)$ be its spacetime
completion in the sense of Fang--Li \cite{fangli2026structure}, and
write $Z_0=\mathfrak t^{-1}(0)$. For $z\in Z_0$, tangent flows are
pointed limits of
\[
 (Z,r_i^{-1}d_Z,z,r_i^{-2}\mathfrak t),\qquad r_i\searrow0.
\]
The setup and convergence are recalled in
Section~\ref{s10:setup}. Write the conjugate heat-kernel measure based
at $z$ as
\[
 \mathrm d\nu_{z;t}=(4\pi(-t))^{-n/2}e^{-f(t)}\,\mathrm dV_{g(t)}.
\]
Let $\varphi_t$ solve
$\partial_t\varphi_t=-\nabla^{g(t)}f(t)\circ\varphi_t$, with
$\varphi_{-1}=\Id$, and define the modified pair by
\begin{equation}\label{eq:intro-modified-flow}
 \gz(s)=e^s\varphi_{-e^{-s}}^*g(-e^{-s}),\qquad
 \fz(s)=\varphi_{-e^{-s}}^*f(-e^{-s}).
\end{equation}

\begin{Theorem}[Main result]\label{thm:intro-strong-uniqueness}
Suppose one tangent flow at $z$ is generated by $(\hM,\hg,\hf)$.
There are an open subset $U\subset M$ and a smooth diffeomorphism
$\psi:\hM\to U$, independent of $s$, such that
\begin{equation}\label{eq:intro-fixed-convergence}
 \psi^*(\gz(s),\fz(s))\longrightarrow(\hg,\hf)
 \quad\text{in }C^\infty_{\mathrm{loc}}(\hM),\qquad
 \psi^*\mathrm dV_{\fz(s)}=\mathrm dV_{\hat f}.
\end{equation}
Let $\vartheta\in(0,\frac12)$ be an admissible exponent in
Theorem~\ref{thm:intro-pair-LS}. For every integer $m\ge0$ and
\[
 0<\gamma<\frac{\vartheta}{1-2\vartheta},\qquad 0<c<2\gamma,
\]
there are constants $C<\infty$ and $\bar s<\infty$ such that, for every
integer $j\ge\bar s$ and all $s\ge j$,
\begin{equation}\label{eq:intro-convergence-rate}
 [\psi^*\gz(s)-\hg]_m+[\psi^*\fz(s)-\hf]_m
 \le Ce^{\hf/2}(1+s)^{-\gamma}
 \quad\text{on }\{\hf\le c\log j\}.
\end{equation}
Here $[u]_m=\sum_{a=0}^m|\hnd^au|_{\hg}$. The constants may depend on
the reference shrinker, $Y$, and the displayed parameters; $\bar s$ may
also depend on the flow and $z$.
\end{Theorem}

The theorem is proved in Theorem~\ref{s10:main}. The comparison map is
fixed on the whole reference manifold. Thus the conclusion gives
convergence in a single gauge, with polynomial decay in modified time
$s=-\log(-t)$. The condition $c<2\gamma$ makes the right-hand side of
\eqref{eq:intro-convergence-rate} tend to zero uniformly on the displayed
expanding domains. For round cylindrical tangent flows,
Fang--Li \cite[Theorem~1.2]{fangli2026strong} obtain, for every
sufficiently small $\epsilon>0$, the weighted rate
$Ce^{\bar f/2}s^{-1+\epsilon}$, which gives almost inverse-linear decay
in modified time on each fixed compact set. In the AC setting, the
admissible decay exponent $\gamma$ depends on the \L ojasiewicz--Simon
exponent $\vartheta$ of the reference shrinker through
$\gamma<\vartheta/(1-2\vartheta)$.
Identifying the completed flow gives the following consequences,
proved in Corollaries~\ref{s10:uniqueness}
and~\ref{s10:iso-main}.

\begin{Corollary}\label{cor:intro-uniqueness-isolation}
Under the assumptions of Theorem~\ref{thm:intro-strong-uniqueness}, every
tangent flow at $z$ is isometric to the completed canonical flow of
$(\hM,\hg,\hf)$ as a pointed noncollapsed Ricci flow limit space.
Moreover, let $\mathcal S\subset Z_0$ denote the singular set of $Z$.
There is $r_0>0$ such that
\[
 B_{d_Z}(z,r_0)\cap\mathcal S=\{z\}.
\]
\end{Corollary}

The geometric part of the proof follows the radius strategy of
\cite{fangli2026strong}, adapted to conical ends. We define a geometric
radius by closeness on one time slice, and compare it with an entropy
radius determined by the soliton defect. Potential estimates and
matched product coordinates allow the local pair to be extended to a
complete normalized pair in the strong AC topology. The extension
retains a small variation of the asymptotic cone, and Gaussian tail
estimates control the cutoff errors. Applying
Theorem~\ref{thm:intro-pair-LS} then gives a pointed entropy inequality
for the original flow. Integrating the velocity estimates prevents exit
from the comparison neighborhood and yields convergence in a fixed gauge. The
weighted estimates give the rate on expanding potential sublevels.
Finally, the completed tangent flow has no terminal singular point
away from its vertex; uniqueness and regularity of limits then imply
the isolation assertion.

The paper is organized as follows. Section~\ref{sec:preliminaries}
collects the notation and results from \cite{LYZ-I}.
Section~\ref{sec:metric-LS-auxiliary} establishes analytic dependence of
the entropy minimizer. Section~\ref{sec:appendix-weakened-gaussian}
proves the estimates with weakened Gaussian weights used in the
Lyapunov--Schmidt reduction. Section~\ref{sec:finite-Banach-metric-LS}
proves the entropy inequalities, and
Section~\ref{sec:tangent-flow-uniqueness} establishes strong uniqueness
and isolation. Appendix~\ref{app:euclidean-model} contains the Euclidean
model estimates used in Section~\ref{sec:appendix-weakened-gaussian}.

\medskip
\noindent\textbf{Acknowledgements.}
Yu Li is supported by the National Key Research and Development Program
of China (Grant No.~2025YFA1018200), NSFC-12522105, and research funds
from the University of Science and Technology of China and the Chinese
Academy of Sciences.

\medskip
\noindent\textbf{AI declaration.}
This paper originated in a program begun by the authors in early 2025.
The authors developed the overall strategy and main ideas, together with
the corresponding arguments in this paper. Beginning in August 2026,
the authors used ChatGPT to assist with implementing some of the computations,
mathematical exposition, revision, and drafting, in particular in Sections 3
and 4. The authors take full responsibility for the mathematical content
and correctness.

\section{Preliminaries}\label{sec:preliminaries}
We collect the conventions, function spaces, and background geometry
from \cite{LYZ-I}, together with the intermediate-weight estimate
used below. All background spaces and operators are defined using a fixed
complete nontrivial AC Ricci shrinker $(M^n,\hat g,\hat f)$, unless another
metric or potential is displayed. Manifolds are connected, $n\geq3$, and
nontrivial means non-Gaussian. The weighted definitions allow $0<\alpha<1$; the nonlinear construction fixes $\alpha=1/4$.

\subsection{Notation and conventions}\label{sec:prelim-conventions}\label{sec:prelim-AC-geometry}
We follow \cite[Definitions~2.1--2.2]{LYZ-I}. The background normalization is
\begin{equation}\label{eq:prelim-shrinker-normalization}
 \Ric_{\hat g}+\hat\nabla^2\hat f=\frac12\hat g,
 \qquad (4\pi)^{-n/2}\int_Me^{-\hat f}\,\mathrm dV_{\hat g}=1.
\end{equation}
\phantomsection\label{idx:operator-conventions}
Write $T^{p,q}M=(TM)^{\otimes p}\otimes(T^*M)^{\otimes q}$.
The metric induces the tensor norms and connections; $\sharp$ and $\flat$
raise and lower indices with that metric. For a metric--potential pair
$(g,f)$, use the weighted inner product
$\langle u,v\rangle_f=\int_M\langle u,v\rangle_g\,\mathrm dV_f$,
where the measure is defined in \eqref{eq:prelim-entropy-measure} below.
For $u\in C^\infty(M,T^{p,q}M)$ with $q\geq1$, the weighted divergence
contracts the covariant derivative with the first covariant tensor slot:
\begin{equation}\label{eq:prelim-operator-conventions}
 (\delta_fu)^{i_1\cdots i_p}_{j_2\cdots j_q}
 =\nabla^a u^{i_1\cdots i_p}_{a j_2\cdots j_q}
  -(\nabla^af)u^{i_1\cdots i_p}_{a j_2\cdots j_q}.
\end{equation}
Thus $\delta_fu$ has type $(p,q-1)$, and $\nabla^a=g^{ab}\nabla_b$.
When $f$ is constant, we write $\delta$. Set
\[
 \Delta_f=\Delta-\nabla_{\nabla f},\qquad
 \Delta=\operatorname{tr}_g\nabla^2,\qquad
 \delta^*:=-\mathrm{Sym}\circ\nabla.
\]
With respect to $\langle\cdot,\cdot\rangle_f$, the formal adjoints of
$\delta_f$ on covariant tensors, differential forms, and symmetric tensors
are respectively $-\nabla$, $-\mathrm d$, and $\delta^*$.
Here $\mathrm{Sym}$ is normalized symmetrization; in particular,
$a\odot b=\frac12(a\otimes b+b\otimes a)$. For a one-form $\omega$,
\[
 (\delta^*\omega)_{ij}=-\frac12(\nabla_i\omega_j+\nabla_j\omega_i),
 \qquad \mathcal L_{\omega^\sharp}g=-2\delta^*\omega.
\]
The Fr\'echet differential is denoted by $\mathrm D$, and $\mathrm d$ denotes the exterior differential.
Also, $\delta_f\mathrm du=\Delta_fu$ for functions. The notation $S*T$
denotes a linear combination of metric contractions of $S\otimes T$;
$\tr_g h=g^{ij}h_{ij}$ denotes the metric trace of a covariant two-tensor,
and $\mathrm R_g$ denotes scalar curvature. Hats distinguish background
objects. Following \cite[(2.05) and (3.13)]{LYZ-I}, we set
\begin{equation}\label{eq:prelim-background-drift}
 \mathbf{X}:=\hat\nabla\hat f,\qquad
 \mathrm R_{\hat g}+|\mathbf{X}|_{\hat g}^2=\hat f-\boldsymbol{\mu}(\hat g),\qquad
 \boldsymbol{\rho}:=2\sqrt{1+\hat f-\boldsymbol{\mu}(\hat g)}.
\end{equation}

\phantomsection\label{idx:entropy-conventions}
The scale in Perelman's functionals is always $1$:
$\mathcal W(g,f)=\mathcal W(g,f,1)$ and $\boldsymbol{\mu}(g)=\boldsymbol{\mu}(g,1)$.
For a metric--potential pair, our weighted measures are
\begin{equation}\label{eq:prelim-entropy-measure}
 \mathrm dV_f:=(4\pi)^{-n/2}e^{-f}\mathrm dV_g,\qquad
 \mathrm dV_{\hat f}:=(4\pi)^{-n/2}e^{-\hat f}\mathrm dV_{\hat g}.
\end{equation}
A pair is normalized when $\int_M\mathrm dV_f=1$, and
\[
 \mathcal W(g,f)=\int_M
   \bigl(\mathrm R_g+|\nabla^gf|_g^2+f-n\bigr)\,\mathrm dV_f.
\]
The metric associated with $f$ is understood from the specified pair;
for a mixed pair whose metric is not clear from context, we write the
density explicitly.
The entropy $\boldsymbol{\mu}(g)$ is the infimum over normalized admissible potentials.
Whenever this infimum has a unique minimizer, it is denoted by $f_g$.
\phantomsection\label{idx:pair-defect}
The soliton defect is
\begin{equation}\label{eq:prelim-soliton-defect}
 \Phi(g,f):=\frac12g-\Ric_g-\nabla_g^2f.
\end{equation}
When $f_g$ exists, we write $\Phi(g):=\Phi(g,f_g)$ and use
$\mathrm dV_{f_g}$ for its normalized measure.
For $g=\hat g+h$, set
\begin{equation}\label{eq:prelim-metric-abbreviations}
 D_g:=\delta^g_{f_g},\qquad \hat D:=\hat\delta_{\hat f}.
\end{equation}
\phantomsection\label{idx:gaussian-measures}
For $0\leq\kappa<1$, define
\begin{equation}\label{eq:prelim-weakened-measures}
 Z_\kappa:=\int_Me^{\kappa\hat f}\,\mathrm dV_{\hat f},
 \qquad \mathrm dV_{\hat f,\kappa}:=Z_\kappa^{-1}e^{\kappa\hat f}\mathrm dV_{\hat f},
 \qquad L^2_\kappa:=L^2(\mathrm dV_{\hat f,\kappa}).
\end{equation}
For a tensor bundle $E$ and an integer $m\geq0$, let $H^m_\kappa(M,E)$
consist of sections with weak background covariant derivatives through
order $m$ and finite norm
\begin{equation*}
 \|u\|_{H^m_\kappa}^2
 :=\sum_{j=0}^m\int_M|\hat\nabla^ju|_{\hat g}^2\,
                  \mathrm dV_{\hat f,\kappa}.
\end{equation*}
In particular, $H^0_\kappa=L^2_\kappa$.
The tensor bundle is understood from context when omitted.

\begin{Definition}[Gaussian Sobolev spaces; {\cite[Definition~2.1]{LYZ-I}}]
\label{def:prelim-gaussian-spaces}
For a metric--potential pair $(g,f)$ and a tensor bundle $E$,
let $H^k_f(M,E)$ consist of sections with weak covariant derivatives
through order $k$ and finite norm
\[
 \|u\|_{H^k_f}^2
 :=\sum_{j=0}^k\int_M|\nabla^{g,j}u|_g^2\,\mathrm dV_f.
\]
We write $L^2_f=H^0_f$. On a domain $\Omega$, the notation
$H^k_{f,\mathrm{loc}}(\Omega,E)$ denotes local membership, and
$H^1_{f,0}(\Omega,E)$ is the closure of $C_c^\infty(\Omega,E)$
in $H^1_f(\Omega,E)$.
Background spaces use $(\hat g,\hat f)$.
Unless specified otherwise, orthogonal complements of background
spaces use the $L^2_{\hat f}$ inner product.
\end{Definition}

\phantomsection\label{idx:ISD}
Following \cite[(1.03)]{LYZ-I}, define
\begin{equation}\label{eq:prelim-ISD}
 \begin{aligned}
 \hat L&:=\frac12\hat\Delta_{\hat f}+\Rm_{\hat g},\\
 \ISD=\ISD(\hat g,\hat f)
 &:=\{k\in H^2_{\hat f}(M,S^2T^*M):\hat Lk=0,\ \hat Dk=0\}.
 \end{aligned}
\end{equation}
Here $\Rm_{\hat g}$ acts on symmetric two-tensors by curvature
contraction, with the convention of \cite{LYZ-I}.

\phantomsection\label{idx:AC-geometry}
Let $(\mathscr C(\Sigma),g_C)$ be the cone over a compact, possibly disconnected
link $(\Sigma^{n-1},g_\Sigma)$, with
$\mathscr C(\Sigma)=(0,\infty)\times\Sigma$ and
$g_C=\mathrm dr^2+r^2g_\Sigma$. We abbreviate the cone to $\mathscr C$.
Write $\nabla^C$ and $\Delta^C=\operatorname{tr}_{g_C}(\nabla^C)^2$
for its connection and rough Laplacian, and
$D_\lambda(r,\theta)=(\lambda r,\theta)$ for cone dilation.

\begin{Definition}[AC ends and degree; {\cite[Definitions~2.3--2.4]{LYZ-I}}]
\label{def:prelim-AC-degree}
A Riemannian manifold is asymptotic to $(\mathscr C,g_C)$ if an end
identification $\iota_0:(\rho_0,\infty)\times\Sigma\to M\setminus K_0$,
with $\rho_0>0$ and compact $K_0$, satisfies
\[
 \lambda^{-2}D_\lambda^*\iota_0^*g\longrightarrow g_C
 \quad\text{in }C^2_{\mathrm{loc}}(\mathscr C)\quad\text{as }\lambda\to\infty.
\]
For a tensor bundle $E=T^{p,q}M$, use the end chart to identify it with
$T^{p,q}\mathscr C$. Put $\Sigma_r=\{r\}\times\Sigma$ and
$E_\Sigma=T^{p,q}\mathscr C|_{\Sigma_1}$, viewed as a bundle over $\Sigma$.
Parallel transport along $s\mapsto(s,\theta)$ for $g_C$ gives the
fiberwise isometry
\[
 \boldsymbol{\tau}_r:E_\Sigma\longrightarrow T^{p,q}\mathscr C|_{\Sigma_r}.
\]
Both restrictions are regarded as bundles over $\Sigma$ using the product
identification. The radial extension is characterized by
$\boldsymbol{\tau}(a)|_{\Sigma_r}=\boldsymbol{\tau}_ra$ for a section $a$ of $E_\Sigma$.
A tensor $r^\gamma\boldsymbol{\tau}(a)$ is homogeneous of degree $\gamma$.
The extension includes radial components; for example,
$\boldsymbol{\tau}(a\,\mathrm dr+b)=a\,\mathrm dr+r\pi^*b$ for a link
function $a$ and a tangential link one-form $b$, where
$\pi:\mathscr C\to\Sigma$ is the projection $(r,\theta)\mapsto\theta$.
\end{Definition}

By \cite[Theorem~2.5 and Lemma~2.6]{LYZ-I}, there is a smooth end embedding
$\iota:(\underline{r},\infty)\times\Sigma\to M$, with compact complement after
restriction to every farther end, such that
\begin{equation}\label{eq:prelim-canonical-end}
 \iota^*\mathbf{X}=\frac r2\partial_r,\qquad
 |\nabla^{C,j}(\iota^*\hat g-g_C-2\Ric(g_C))|_{g_C}
 \leq C_jr^{-4-j}\quad(j\geq0).
\end{equation}
In this chart, the connection difference and the potential satisfy
\begin{equation}\label{eq:prelim-end-errors}
 |\nabla^{C,j}(\hat\nabla-\nabla^C)|_{g_C}=O(r^{-3-j}),\qquad
 \left|\nabla^{C,j}\left(\hat f-\boldsymbol{\mu}(\hat g)-\frac{r^2}{4}\right)\right|_{g_C}
 =O(r^{-2-j}).
\end{equation}
The canonical shrinker flow is $\hat g(t)=(-t)\Psi_t^*\hat g$, where
$\partial_t\Psi_t=(-t)^{-1}\mathbf{X}\circ\Psi_t$ and $\Psi_{-1}=\Id$.
On the end,
$\Psi_t(r,\theta)=(r/\sqrt{-t},\theta)$ whenever both points belong
to the chart. The metrics $\hat g(t)$ converge smoothly to $g_C$ on
compact subsets of the end as $t\uparrow0$.

We fix the standard AC data of \cite[Definition~3.8]{LYZ-I}.
Enlarge $\underline{r}$ so that the canonical chart extends across its initial
slice, and set $Q_\rho:=\iota((\rho,\infty)\times\Sigma)$ for $\rho\geq \underline{r}$.
Choose a smooth cutoff $\chi:M\to[0,1]$ with
$\operatorname{Supp}\chi\subset Q_{\underline{r}}$ and $\chi=1$ on $Q_{2\underline{r}}$,
and a compact core $K$ containing $M\setminus Q_{2\underline{r}}$. The cutoff has
$|\hat\nabla^j\chi|\leq C_j\underline{r}^{-j}$. Constants in weighted estimates
may depend on these fixed data, the tensor type, and the displayed
regularity and weight parameters. We identify the end with its image
under $\iota$ when writing cone tensors on $M$.
As in \cite[Lemma~4.11]{LYZ-I}, bars denote intrinsic link operators:
$\bar\nabla=\nabla^{g_\Sigma}$,
$\bar\Delta=\operatorname{tr}_{g_\Sigma}\bar\nabla^2$, and $\bar\delta$
and $\bar\delta^*$ use the conventions in
\cref{eq:prelim-operator-conventions} with the link metric and constant
potential. For a scalar link function $a$,
\begin{equation}\label{eq:prelim-cone-scalar-Laplacian}
 \Delta^C(r^\gamma a)
 =r^{\gamma-2}\bigl(\bar\Delta a+\gamma(\gamma+n-2)a\bigr).
\end{equation}

\subsection{Weighted H\"older spaces and estimates}\label{sec:prelim-weighted-spaces}\label{sec:prelim-elliptic}
On the canonical end put $t=\log(r/\underline{r})$, so $r=\underline{r}e^t$, and let
$\pi_\Sigma:[0,\infty)\times\Sigma\to\Sigma$ be the cylinder projection.
We use $E_\Sigma$ also for the pullback bundle $\pi_\Sigma^*E_\Sigma$.
Its fiber metric is induced by $g_C$ at $\Sigma_1$.
The fixed link-bundle connection is the cone-induced connection
\[
 \mathcal D_Ya
 :=\left.\nabla^C_{\widetilde Y}\boldsymbol{\tau}(a)\right|_{\Sigma_1},
\]
where $Y$ is tangent to $\Sigma$ and $\widetilde Y$ is its
$r$-independent lift to $\mathscr C$. On the cylinder we use its
pullback, ordinary differentiation in $t$, and the product metric
$\mathrm dt^2+g_\Sigma$. For $\beta\in\mathbb R$ and a section $u$
of $E$ on the end, define the section $U=\mathscr T_\beta u$ of this fixed bundle by
\begin{equation}\label{eq:prelim-cylindrical-representative}
 \begin{aligned}
 U(t,\theta)&:=(\mathscr T_\beta u)(t,\theta)
      =r^{-\beta}\boldsymbol{\tau}_r^{-1}\bigl(u|_{\Sigma_r}\bigr)(\theta),\\
 \|U\|_{C^{k,\alpha}_{\mathrm{cyl}}}
 &:=\sup_{T\geq0}\|U(T+\cdot,\cdot)\|_{C^{k,\alpha}([0,1]\times\Sigma,E_\Sigma)}.
 \end{aligned}
\end{equation}
Here $k\in\mathbb Z_{\geq0}$, and $U(T+\cdot,\cdot)$ means the section
$(s,\theta)\mapsto U(T+s,\theta)$ for $0\leq s\leq1$.
The same supremum defines $C^k_{\mathrm{cyl}}$ using ordinary $C^k$ norms.

\begin{Definition}[Weighted H\"older spaces; {\cite[Definition~3.9]{LYZ-I}}]
\label{def:prelim-weighted-Holder}
For $\beta\in\mathbb R$, $k\geq0$, and $\alpha\in[0,1]$, the space $C^{k,\alpha}_\beta(M,E)$
consists of locally $C^{k,\alpha}$ sections with
\begin{equation}\label{eq:prelim-weighted-norm}
 \|u\|_{C^{k,\alpha}_\beta}
 :=\|u\|_{C^{k,\alpha}(K)}+\|\mathscr T_\beta u\|_{C^{k,\alpha}_{\mathrm{cyl}}}<\infty.
\end{equation}
At $\alpha=0$ these are ordinary $C^k$ norms, and we write
$C^k_\beta=C^{k,0}_\beta$. At $\alpha=1$ the derivatives of order $k$
are Lipschitz. For fixed $\alpha\in(0,1)$,
set $C^\infty_\beta=\bigcap_{k\geq0}C^{k,\alpha}_\beta$.
\end{Definition}
\phantomsection\label{idx:weighted-O}
We write $O_{m+\alpha}(r^\beta)$ for a bound in $C^{m,\alpha}_\beta$.
In coordinate-component formulas this denotes the corresponding
fixed-link family norm, with mixed radial and angular derivatives
on normalized annuli.

For comparison with rescaled annuli, set $R_T=\underline{r}e^T$,
$\kappa_T(s,\theta)=\iota(R_Te^s,\theta)$, and
$g_T=R_T^{-2}\kappa_T^*\hat g$. Then, for $E=T^{p,q}M$,
\begin{equation}\label{eq:prelim-annular-norm}
 \|u\|_{C^{k,\alpha}_\beta}
 \asymp\|u\|_{C^{k,\alpha}(K)}
 +\sup_{T\geq0}\|R_T^{p-q-\beta}\kappa_T^*u\|_{C^{k,\alpha}([0,1]\times\Sigma),g_T}.
\end{equation}
This is \cite[(3.18)]{LYZ-I}. At integer order,
$\|u\|_{C^k_\beta}\asymp\sum_{j=0}^k\sup_M\boldsymbol{\rho}^{j-\beta}|\hat\nabla^ju|$.

By \cite[Lemmas~3.10--3.11]{LYZ-I}, covariant differentiation lowers the
weight by one, products and contractions add weights, and
$C^{k,\alpha}_\beta\hookrightarrow H^k_{\hat f}$ continuously.
Multiplication by $\boldsymbol{\rho}^{\beta'-\beta}$ is a bounded
isomorphism $C^{k,\alpha}_\beta\to C^{k,\alpha}_{\beta'}$.

\begin{Definition}[Critical graph spaces; {\cite[Definition~3.14 and Corollary~3.16]{LYZ-I}}]
\label{def:prelim-graph-spaces}
For $k\geq2$, the space $C^{k,\alpha}_{\beta,\mathbf{X}}(M,E)$ consists of locally
$C^{k,\alpha}$ tensors for which
\[
 u,\hat\nabla_{\mathbf{X}}u\in C^{k-2,\alpha}_\beta,\qquad
 \hat\nabla^{k-1}u,\hat\nabla^ku\in C^{0,\alpha}_{\beta+2-k}.
\]
The appropriate tensor bundles are understood. As in \cite{LYZ-I}, its norm is
\[
 \|u\|_{C^{k,\alpha}_{\beta,\mathbf{X}}}
 :=\|u\|_{C^{k-2,\alpha}_\beta}
  +\|\hat\nabla_{\mathbf{X}}u\|_{C^{k-2,\alpha}_\beta}
  +\|\hat\nabla^{k-1}u\|_{C^{0,\alpha}_{\beta+2-k}}
  +\|\hat\nabla^ku\|_{C^{0,\alpha}_{\beta+2-k}}.
\]
Equivalently, $u,\hat\nabla^2u,\hat\nabla_{\mathbf{X}}u\in C^{k-2,\alpha}_\beta$, and
\begin{equation}\label{eq:prelim-graph-norm}
 \|u\|_{C^{k,\alpha}_{\beta,\mathbf{X}}}
 \asymp\|u\|_{C^{k-2,\alpha}_\beta}
       +\|\hat\nabla^2u\|_{C^{k-2,\alpha}_\beta}
       +\|\hat\nabla_{\mathbf{X}}u\|_{C^{k-2,\alpha}_\beta}.
\end{equation}
The subscript $\mathbf{X}$ is the background version of the subscript $\nabla f$
in \cite{LYZ-I}.
\end{Definition}

\begin{Definition}[Finite expansions; {\cite[Definition~3.17]{LYZ-I}}]
\label{def:prelim-finite-expansions}
Let $\Lambda\subset(\beta,\infty)$ be finite. A tensor belongs to
$C^{k,\alpha}_{\Lambda;\beta}$ if it has the unique decomposition
\begin{equation}\label{eq:prelim-finite-expansion}
 u=\sum_{\gamma\in\Lambda}[u]_\gamma+u^\star,
 \qquad [u]_\gamma:=\chi r^\gamma\boldsymbol{\tau}(u_\gamma),\qquad
 u_\gamma\in C^{k,\alpha}(\Sigma,E_\Sigma),\quad
 u^\star\in C^{k,\alpha}_\beta.
\end{equation}
Its norm is the sum of the link-coefficient norms and the remainder norm.
For $k\geq2$, replacing the remainder space by $C^{k,\alpha}_{\beta,\mathbf{X}}$
defines $C^{k,\alpha}_{\Lambda;\beta,\mathbf{X}}$, with the analogous sum norm.
We omit braces for explicitly listed degrees and write
$C^{k,\alpha}_{\Lambda;-\infty}$ for the subspace with $u^\star=0$.
Degrees at or below $\beta$ are omitted from $\Lambda$ by convention.
\end{Definition}

We use the compact Gaussian Sobolev embeddings, elliptic regularity, and
Fredholm theory of \cite[Lemma~3.2 and Propositions~3.3--3.4]{LYZ-I}.
The following estimate from \cite[Proposition~4.18]{LYZ-I} applies to each
fixed tensor bundle over the background shrinker; solutions are understood
distributionally.

\begin{Proposition}
\label{prop:prelim-intermediate-Schauder}
Let $k\geq0$ and $\lambda,\beta\in\mathbb R$. Suppose $2\lambda+\beta\in(-2,0]$ and
$(\hat\Delta_{\hat f}-\lambda)u=v\in C^{k,\alpha}_\beta$, with
$u\in L^2_{\hat f}$. Then
\[
 \|\hat\nabla^2u\|_{C^{k,\alpha}_\beta}
 +\|(\hat\nabla_{\mathbf{X}}+\lambda)u\|_{C^{k,\alpha}_\beta}
 \leq C\bigl(\|u\|_{L^2_{\hat f}}+\|v\|_{C^{k,\alpha}_\beta}\bigr).
\]
If $2\lambda+\beta<0$, there is a unique decomposition
$u=[u]_{-2\lambda}+u^\star$ with
\begin{equation}\label{eq:prelim-intermediate-expansion}
 \|[u]_{-2\lambda}\|_{C^{k,\alpha}_{-2\lambda;-\infty}}
 +\|u^\star\|_{C^{k,\alpha}_\beta}
 \leq C\bigl(\|u\|_{L^2_{\hat f}}+\|v\|_{C^{k,\alpha}_\beta}\bigr).
\end{equation}
\end{Proposition}

\subsection{Metric and slice spaces}\label{sec:prelim-ordinary-spaces}\label{sec:prelim-slice}\label{sec:finite-coordinate-slice}
\begin{Definition}[Smooth weighted classes]
\label{def:prelim-ordinary-spaces}
The smooth weighted and finite-expansion spaces are
\begin{equation}\label{eq:prelim-ordinary-grading}
 C^\infty_\beta:=\bigcap_{k\geq0}C^{k,\alpha}_\beta,\qquad
 C^\infty_{\Lambda;\beta}:=\bigcap_{k\geq0}C^{k,\alpha}_{\Lambda;\beta}.
\end{equation}
Thus the coefficients in \cref{def:prelim-finite-expansions} are smooth
link tensors, and the final remainder belongs to $C^\infty_\beta$.
These remainder bounds control all conically weighted spatial derivatives.
\end{Definition}

At a fixed finite order, weight nesting and coefficient extraction are
bounded linear maps. The smooth intersections are used only as classes
of fields satisfying all finite bounds. Ordinary weighted norms and
critical graph norms retain their separate meanings.

The metric spaces below are those of \cite[Definition~5.1]{LYZ-I}.

\begin{Definition}[Metric expansion spaces]\label{def:prelim-radial-filtration}
For each integer $m\geq0$, define $E_m(M,S^2T^*M)$ by the unique expansion
\begin{equation}\label{eq:prelim-E-expansion}
 h=\chi\bigl(\boldsymbol\tau(c_0)
       +r^{-1}\boldsymbol\tau(c_{-1}+2\,\mathrm dr\odot b_{-1})\bigr)+h^\star,
 \qquad h^\star\in C^{m,\alpha}_{-2}.
\end{equation}
Here $c_0,c_{-1}\in C^{m,\alpha}(\Sigma,S^2T^*\Sigma)$ and
$b_{-1}\in C^{m,\alpha}(\Sigma,T^*\Sigma)$ are tangential coefficients.
The norm is
\begin{equation}\label{eq:prelim-radial-gradings}
 \|h\|_{E_m}:=\|(c_0,c_{-1},b_{-1})\|_{C^{m,\alpha}(\Sigma)}
                    +\|h^\star\|_{C^{m,\alpha}_{-2}},
\end{equation}
where the tuple norm is the sum of the coefficient norms.
These finite spaces are Banach. Define
\begin{equation}\label{eq:prelim-radial-spaces}
 E_\infty:=\bigcap_{m\geq0}E_m
 =\bigl\{h\in C^\infty_{0,-1;-2}:
 h(\mathbf X,\cdot)\in C^\infty_{0;-1},\quad
 h(\mathbf X,\mathbf X)\in C^\infty_0\bigr\}.
\end{equation}
Its elements have smooth link coefficients and a remainder in $C^\infty_{-2}$.
\end{Definition}

\phantomsection\label{idx:coordinate-graphs}
By \cite[Definition~5.5 and Propositions~5.6--5.7]{LYZ-I}, the auxiliary
coordinate target $\mathscr Y_k$ consists of
$v=\chi\boldsymbol\tau(c_0)+v^\star$, with tangential
$c_0\in C^{k+1,\alpha}$, $v^\star\in C^{k,\alpha}_{-1}$, and
\[
 \hat Dv\in\mathscr S_k
 :=C^{k+1,\alpha}_0+C^{k-1,\alpha}_{-2}.
\]
This sum carries its quotient norm; the target norm adds the coefficient,
remainder, and sum-space norms. Its inverse-image domain is
\[
 \mathscr X_k=\{Z\in H^2_{\hat f}(T^*M)\cap\mathcal K^\perp:
       \hat PZ\in\mathscr Y_k\},\qquad
 \|Z\|_{\mathscr X_k}=\|Z\|_{H^2_{\hat f}}+\|\hat PZ\|_{\mathscr Y_k},
 \quad \hat P=\hat D\hat\delta^*.
\]
Here $\mathcal K=\ker\hat P$. The sum is a quotient by the closed
kernel of addition into distributions; the target graph is closed.
The Gaussian inverse therefore proves completeness and identifies
$\hat P:\mathscr X_k\to\mathscr Y_k\cap\mathcal K^\perp$ as an
isomorphism.
Every $Z\in\mathscr X_k$ has an end expansion with coefficients $b_1,a_0\in C^{k+2,\alpha}$,
$b_0\in C^{k+1,\alpha}$ and remainder $Z^\star\in C^{k,\alpha}_{-1}$;
the whole-field estimates are those of
\cite[Proposition~5.7]{LYZ-I}.

\begin{Theorem}[Finite coordinate slice; {\cite[Theorem~6.6]{LYZ-I}}]
\label{thm:finite-coordinate-slice}
Fix $k\ge6$. For sufficiently small $h\in E_{k+4}$ there is a $C^1$
coordinate parameter $Z(h)\in\mathscr X_k$ such that the associated
global $C^{k+2,\alpha}_{\mathrm{loc}}$ diffeomorphism satisfies
\[
 \tilde h=F_{Z(h)}^*(\hat g+h)-\hat g\in E_{k-1}\cap\ker\hat D,
 \qquad
 \|Z(h)\|_{\mathscr X_k}+\|\tilde h\|_{E_{k-1}}
                         \le C_k\|h\|_{E_{k+4}}.
\]
\end{Theorem}

\begin{Lemma}
\label{lem:finite-potential-pullback}
For the diffeomorphisms $F_Z$ of
\cref{thm:finite-coordinate-slice}, with $k\ge6$, and
$u\in C^{3,\alpha}_1(M)$, one has
\[
 \|F_Z^*(\hat f+u)-\hat f\|_{C^{3,\alpha}_1}
 +\|(F_Z^{-1})^*(\hat f+u)-\hat f\|_{C^{3,\alpha}_1}
 \le C\bigl(\|Z\|_{\mathscr X_k}+\|u\|_{C^{3,\alpha}_1}\bigr).
\]
Both pullbacks depend continuously on $(Z,u)$ as maps from a small
$\mathscr X_k$ ball times $C^{4,\alpha}_1$ into $C^{3,\alpha}_1$.
\end{Lemma}
\begin{proof}
Write $F_Z=(R,\Theta)$ on the end and put $r_F=R-r$.
By \cite[Lemma~6.3]{LYZ-I}, the radii $R$ and $r$ are comparable,
the normalized maps have uniformly bounded $C^{4,\alpha}$ norms, and
\[
 \|r_F\|_{C^{4,\alpha}_0}\le C\|Z\|_{\mathscr X_k}.
\]
The chain rule and the H\"older product estimate on normalized annuli
therefore give
\[
 \|u\circ F_Z\|_{C^{3,\alpha}_1}\le C\|u\|_{C^{3,\alpha}_1}.
\]
By \eqref{eq:prelim-end-errors}, write
$\hat f=\boldsymbol\mu(\hat g)+r^2/4+e$ on the end, with
$e\in C^\infty_{-2}$. Then
\[
 \hat f\circ F_Z-\hat f
   =\tfrac14r_F(2r+r_F)+e\circ F_Z-e.
\]
The first term has $C^{3,\alpha}_1$ norm at most
$C\|Z\|_{\mathscr X_k}$. On nested normalized annuli the mean value
estimate for a function $E$ and a coordinate map $\Phi$ is
\[
 \|E\circ\Phi-E\|_{C^{3,\alpha}}
 \leq C\|E\|_{C^{4,\alpha}}\|\Phi-\Id\|_{C^{3,\alpha}},
\]
with a uniform constant on the bounded set of maps under consideration.
By \cite[Lemma~6.3]{LYZ-I}, the normalized displacement has
$C^{3,\alpha}$ norm at most $C\|Z\|_{\mathscr X_k}$.
Apply this estimate to the normalized representative of $e$ and scale back:
\[
 \|e\circ F_Z-e\|_{C^{3,\alpha}_{-2}}
 \leq C\|e\|_{C^{4,\alpha}_{-2}}\|Z\|_{\mathscr X_k}
 \leq C\|Z\|_{\mathscr X_k}.
\]
The same mean value estimate for the smooth function $\hat f$ on the
compact core proves the remaining bound. This proves the forward estimate.

Smallness of the normalized differential and the inverse-function
formula give uniformly bounded $C^{4,\alpha}$ normalized inverse maps;
their radial coordinates are comparable to $r$. Thus composition with
$F_Z^{-1}$ is also bounded on $C^{3,\alpha}_1$. The identity
\[
 \hat f\circ F_Z^{-1}-\hat f
   =-(\hat f\circ F_Z-\hat f)\circ F_Z^{-1}
\]
proves the inverse estimate. Finally, the map estimates in
\cite[Lemma~6.3]{LYZ-I} and the inverse-function formula give continuous
dependence of the normalized maps and their inverses at order four.
For $u\in C^{4,\alpha}_1$, the mean value estimate for the difference
of two compositions, with $\mathrm du\in C^{3,\alpha}_0$, gives
continuity in $C^{3,\alpha}_1$. The smooth background term is controlled
one derivative higher by the same radial formula above, so its inverse
pullback is continuous as well. These estimates are uniform on
normalized annuli and on the fixed compact core.
\end{proof}

\subsection{Finite scalar and tensor estimates}\label{sec:finite-linear-estimates}\label{sec:finite-graphs}
We record the finite estimates for scalar entropy equations and for the
tensor stability operator. All norms in this subsection use the fixed
background metric and potential. Fix
\[
 0<\alpha<\frac12,\qquad \alpha-1<\beta<-\alpha.
\]
For a tensor bundle $E$,
$E_\Sigma$ denotes its restriction to the unit link with the radial
parallel identification $\boldsymbol\tau$.
The following trace estimate uses the interpolation argument in the
proof of \cite[Lemma~5.3]{LYZ-I} at the fractional weight $\beta$.

\begin{Lemma}\label{lem:fractional-terminal}
Let $k\geq0$ be an integer. Suppose
\[
 (\hat\Delta_{\hat f}+\tfrac12)u=v\in C^{k,\alpha}_{\beta}(M,E),
 \qquad u\in L^2_{\hat f}(M,E).
\]
There is a unique $a\in C^{k+1,\alpha-\beta}(\Sigma,E_\Sigma)$
such that, with $u_1=\chi r\boldsymbol\tau(a)$,
\begin{align*}
 &\|a\|_{C^{k+1,\alpha-\beta}}
 +\|u-u_1\|_{C^{k,\alpha}_{\beta}}
 +\|\hat\nabla^2u\|_{C^{k,\alpha}_{\beta}}\\
 &\quad+\|\hat\nabla u-\nabla^Cu_1\|_{C^{k,\alpha}_{\beta}}
 +\|(\hat\nabla_{\mathbf X}-\tfrac12)u\|_{C^{k,\alpha}_{\beta}}
 \leq C\bigl(\|v\|_{C^{k,\alpha}_{\beta}}+\|u\|_{L^2_{\hat f}}\bigr).
\end{align*}
\end{Lemma}
\begin{proof}
Write $N=\|v\|_{C^{k,\alpha}_\beta}+\|u\|_{L^2_{\hat f}}$.
Apply \cref{prop:prelim-intermediate-Schauder} with $\lambda=-1/2$.
Its strict condition is $-2<\beta-1<0$. It gives the expansion with
$a\in C^{k,\alpha}$, the remainder, and the whole-field Hessian and
corrected drift estimates, with bound $CN$. On the end put
\[
 A(r,\theta)=r^{-1}\boldsymbol\tau_r^{-1}u(r,\theta),\qquad
 a_R=A(R,\cdot),\qquad U_R(s,\theta)=A(Rs,\theta).
\]
We work on fixed nested annuli in $1/2<s<4$. The remainder and
Hessian estimates give
\[
 \|U_R-a\|_{C^{k,\alpha}}\leq CNR^{\beta-1},\qquad
 \|U_R\|_{C^{k+2,\alpha}}\leq CNR^{\beta+1}.
\]
For the second bound, express ordinary second derivatives using the
rescaled covariant Hessian and absorb the first derivatives by local
interpolation on nested annuli. The rescaled connection coefficients
and their required derivatives are uniformly bounded. The zeroth-order
term is also controlled, since $1+\beta>0$. Restriction to $s=1$ gives
\[
 K(R^{-2},a;C^{k,\alpha},C^{k+2,\alpha})
 \leq\|a-a_R\|_{C^{k,\alpha}}
       +R^{-2}\|a_R\|_{C^{k+2,\alpha}}
 \leq CNR^{\beta-1}.
\]
These scales and the initial trace bound control the full $K$-functional.
By \cite[Lemma~3.12(iii)]{LYZ-I}, interpolation at parameter
$(1-\beta)/2$ yields $a\in C^{k+1,\alpha-\beta}$ with norm at most $CN$.

For the gradient estimate, let $d_R=a_{2R}-a_R$. The preceding bounds imply
\[
 \|d_R\|_{C^{k,\alpha}}\leq CNR^{\beta-1},\qquad
 \|d_R\|_{C^{k+2,\alpha}}\leq CNR^{\beta+1}.
\]
Interpolation at parameter $1/2$ gives
$\|d_R\|_{C^{k+1,\alpha}}\leq CNR^\beta$.
Since $\beta<0$, the series of these differences converges in
$C^{k+1,\alpha}$, and its limit agrees with the known $C^{k,\alpha}$ limit.
Consequently
\[
 \|a_R-a\|_{C^{k+1,\alpha}}
 \leq CN\sum_{\ell=0}^\infty(2^\ell R)^\beta
 \leq CNR^\beta.
\]
On a fixed rescaled annulus, extend $a_R$ constantly in $s$.
The difference $U_R-a_R$ has $C^{k,\alpha}$ norm at most
$CNR^{\beta-1}$ and $C^{k+2,\alpha}$ norm at most $CNR^{\beta+1}$.
The same interpolation bounds it by $CNR^\beta$ in $C^{k+1,\alpha}$.
Adding the last display gives this bound for $U_R-a$.
A cone derivative of $r\boldsymbol\tau(A-a)$ consists of
$(r\partial_r)(A-a)$, link derivatives, and zeroth-order terms;
on the rescaled annulus $r\partial_r=s\partial_s$.
Thus its $C^{k,\alpha}$ norm is at most $CNR^\beta$.
The connection difference acting on $u$ contributes
$O_{k+\alpha}(r^{-2})N$, which has the required weight.
This proves the gradient estimate. Local elliptic estimates handle
the compact core, and division of the remainder by $r$ proves uniqueness.
\end{proof}

\begin{Lemma}\label{lem:fractional-degree-zero}
Let $k\geq0$. If
\[
 (\hat\Delta_{\hat f}+\tfrac12)u
   =\chi\boldsymbol\tau(v_0)+v^\star,\qquad
 v_0\in C^{k,\alpha-\beta}(\Sigma,E_\Sigma),\quad
 v^\star\in C^{k,\alpha}_{\beta}(M,E),\quad u\in L^2_{\hat f},
\]
then $u=\chi r\boldsymbol\tau(a)+2\chi\boldsymbol\tau(v_0)+u^\star$
for a unique $a\in C^{k+1,\alpha-\beta}$, and
\begin{align*}
 &\|a\|_{C^{k+1,\alpha-\beta}}+\|u^\star\|_{C^{k,\alpha}_{\beta}}
 +\|\hat\nabla^2u\|_{C^{k,\alpha}_{\beta}}\\
 &\quad+\|\hat\nabla u-\nabla^C(\chi r\boldsymbol\tau(a))
       \|_{C^{k,\alpha}_{\beta}}
 +\|(\hat\nabla_{\mathbf X}-\tfrac12)u+\chi\boldsymbol\tau(v_0)
       \|_{C^{k,\alpha}_{\beta}}\\
 &\hspace{12mm}\leq C\bigl(\|v_0\|_{C^{k,\alpha-\beta}}
                +\|v^\star\|_{C^{k,\alpha}_{\beta}}
                +\|u\|_{L^2_{\hat f}}\bigr).
\end{align*}
\end{Lemma}
\begin{proof}
Use the lift from the proof of \cite[Lemma~5.4]{LYZ-I}:
\[
 F=\chi\boldsymbol\tau(v_0),\qquad
 U_0=-4(\hat\Delta_{\hat f}-2)^{-1}F,\qquad W=U_0-2F.
\]
The endpoint estimates there, whose proofs apply at every $k\geq0$,
bound $W$ and $\hat\nabla^2U_0$ in $C^{k,\alpha}_0$ for
$v_0\in C^{k,\alpha}$, and in $C^{k,\alpha}_{-2}$ for
$v_0\in C^{k+2,\alpha}$. The gradient is bounded in
$C^{k,\alpha}_0$ at the first endpoint and in $C^{k,\alpha}_{-1}$
for $v_0\in C^{k+1,\alpha}$, by the commuted equation used there.
Interpolate the first two maps at parameter $-\beta/2$ and the
gradient map at parameter $-\beta$, using
\cite[Lemma~3.12(ii)--(iii)]{LYZ-I}. In each case the input is
$C^{k,\alpha-\beta}$ and the output is $C^{k,\alpha}_\beta$. Hence
\[
 \|W\|_{C^{k,\alpha}_\beta}
 +\|\hat\nabla U_0\|_{C^{k,\alpha}_\beta}
 +\|\hat\nabla^2U_0\|_{C^{k,\alpha}_\beta}
 \leq C\|v_0\|_{C^{k,\alpha-\beta}}.
\]
The exact identities are
\[
 (\hat\Delta_{\hat f}+\tfrac12)U_0-F=\tfrac52W,
 \qquad
 (\hat\nabla_{\mathbf X}-\tfrac12)U_0+F
       =\hat\Delta U_0-\tfrac52W.
\]
Apply \cref{lem:fractional-terminal} to $u-U_0$, whose source is
$v^\star-\tfrac52W$, and add these lift estimates.
The Gaussian inverse bounds the low norm of $U_0$ by the source norm.
The field $U_0=2F+W$ has zero degree-one trace, so the resulting
coefficient is the unique degree-one coefficient of $u$.
\end{proof}

\begin{Lemma}\label{lem:tensor-terminal}
Let $k\geq0$, and suppose $u\in L^2_{\hat f}(M,E)$ satisfies
\[
 \hat\Delta_{\hat f}u=v\in C^{k,\alpha}_{\beta}(M,E).
\]
There is a unique $a\in C^{k,\alpha-\beta}(\Sigma,E_\Sigma)$
such that
\begin{align*}
 &\|a\|_{C^{k,\alpha-\beta}}
 +\|u-\chi\boldsymbol\tau(a)\|_{C^{k,\alpha}_{\beta}}
 +\|\hat\nabla u\|_{C^{k,\alpha}_{\beta}}\\
 &\quad+\|\hat\nabla^2u\|_{C^{k,\alpha}_{\beta}}
 +\|\hat\nabla_{\mathbf X}u\|_{C^{k,\alpha}_{\beta}}
 \leq C\bigl(\|v\|_{C^{k,\alpha}_{\beta}}+\|u\|_{L^2_{\hat f}}\bigr).
\end{align*}
The same conclusion holds for every Gaussian solution of
$\hat Lu=v\in C^{k,\alpha}_{\beta}(S^2T^*M)$, where
$\hat L=\frac12\hat\Delta_{\hat f}+\mathrm{Rm}_{\hat g}$.
\end{Lemma}
\begin{proof}
Apply \cref{prop:prelim-intermediate-Schauder} with $\lambda=0$,
and write the resulting bound as $CN$.
The rescaled-annulus argument in \cref{lem:fractional-terminal}, now with
$U_R(s,\theta)=\boldsymbol\tau_{Rs}^{-1}u(Rs,\theta)$ and
$a_R=U_R(1,\cdot)$, gives
\[
 \|U_R-a\|_{C^{k,\alpha}}\leq CNR^\beta,\qquad
 \|U_R\|_{C^{k+2,\alpha}}\leq CNR^{2+\beta}.
\]
The $K$-functional at scale $R^{-2}$ and parameter $-\beta/2$
give $a\in C^{k,\alpha-\beta}$.
Interpolation bounds $a_{2R}-a_R$ in $C^{k+1,\alpha}$ by
$CNR^{1+\beta}$. Summing from a fixed initial radius up to $R$
gives $\|a_R\|_{C^{k+1,\alpha}}\leq CNR^{1+\beta}$, since
$1+\beta>0$. The same interpolation applied to $U_R-a_R$
gives this bound on the annulus. A physical derivative contributes
$R^{-1}$, proving the whole-field gradient estimate.
The compact-core bound and uniqueness follow as before.

For the stability operator, choose $0<\beta_0<2+\beta$.
Apply \cite[Proposition~4.15]{LYZ-I} to
$2\hat L=\hat\Delta_{\hat f}+2\mathrm{Rm}_{\hat g}$ at weight $\beta_0$.
It gives $u\in C^{k+2,\alpha}_{\beta_0,\mathbf X}$, and hence
$\mathrm{Rm}_{\hat g}u\in C^{k,\alpha}_\beta$, with bound by
$C(\|v\|_{C^{k,\alpha}_\beta}+\|u\|_{L^2_{\hat f}})$.
Apply the rough-tensor assertion to
$\hat\Delta_{\hat f}u=2v-2\mathrm{Rm}_{\hat g}u$.
\end{proof}

\begin{Corollary}\label{cor:weight-minus-two}
Let $k\geq0$. Every Gaussian solution of
$\hat\Delta_{\hat f}u=v\in C^{k,\alpha}_{-2}(M,E)$ has
\[
 u=\chi\boldsymbol\tau(a)+u^\star,\qquad
 a\in C^{k+2,\alpha}(\Sigma,E_\Sigma),\qquad
 u^\star\in C^{k+2,\alpha}_{-2,\mathbf X},
\]
with norm bounded by $C(\|v\|_{C^{k,\alpha}_{-2}}+
\|u\|_{L^2_{\hat f}})$. The same holds for
$\hat Lu=v\in C^{k,\alpha}_{-2}(S^2T^*M)$.
\end{Corollary}
\begin{proof}
The first assertion is \cite[Proposition~4.16]{LYZ-I} with
$\lambda=0$ and $m=0$. For the second, first apply
\cref{lem:tensor-terminal} using the inclusion of the source in
$C^{k,\alpha}_{\beta}$. This gives $u\in C^{k,\alpha}_0$,
so $\mathrm{Rm}_{\hat g}u\in C^{k,\alpha}_{-2}$.
Apply the first assertion to
$\hat\Delta_{\hat f}u=2v-2\mathrm{Rm}_{\hat g}u$.
\end{proof}

For the nonlinear construction set
\begin{equation}\label{eq:finite-orders}
 \mathcal V=C^{6,1/4}_{-1/2}(S^2T^*M).
\end{equation}
Define
\begin{equation}\label{eq:finite-tensor-graphs}
 \begin{aligned}
 \mathscr F&=\mathcal V\cap\ker\hat D,\qquad
 \mathscr F'=\mathscr F\cap\ISD^{\perp_{L^2_{\hat f}}},\\
 \mathscr H'&=\{h\in H^2_{\hat f}:\hat Dh=0,\ h\perp_{L^2_{\hat f}}\ISD,
                      \ \hat Lh\in\mathscr F'\},\\
 \mathscr H&=\ISD\oplus\mathscr H'.
 \end{aligned}
\end{equation}
The target has its inherited weighted norm, and
\[
 \|h\|_{\mathscr H'}=\|h\|_{H^2_{\hat f}}+\|\hat Lh\|_{\mathscr F'}.
\]
Use a fixed norm on the finite-dimensional summand of $\mathscr H$.
All equations and kernel conditions in these definitions are distributional.

For functions put
\begin{equation}\label{eq:finite-scalar-graphs}
 \begin{aligned}
 L_{1/2}&=\hat\Delta_{\hat f}+\tfrac12,
 &G_{1/2}&=L_{1/2}^{-1},\\
 \mathcal T&=\chi C^{6,3/4}(\Sigma)
                     \oplus C^{6,1/4}_{-1/2}(M),
 &\mathscr Z&=G_{1/2}\mathcal T,\qquad
 \|u\|_{\mathscr Z}=\|L_{1/2}u\|_{\mathcal T}.
 \end{aligned}
\end{equation}
The target norm is the sum of its coefficient and remainder norms.
The scalar inverse exists on a nontrivial AC shrinker by
\cite[Section~5.2]{LYZ-I}.

\begin{Proposition}\label{prop:finite-graph-realization}
The spaces in \eqref{eq:finite-tensor-graphs} and
\eqref{eq:finite-scalar-graphs} are Banach spaces continuously embedded
in distributions. The map $\hat L:\mathscr H'\to\mathscr F'$ is a
bounded isomorphism. For $h\in\mathscr H$ there is a unique tangential
link tensor $a_h\in C^{6,3/4}$ such that
\begin{equation}\label{eq:finite-tensor-bounds}
 \begin{aligned}
 &\|a_h\|_{C^{6,3/4}}
 +\|h-\chi\boldsymbol\tau(a_h)\|_{C^{6,1/4}_{-1/2}}
 +\|h\|_{C^{6,1/4}_0}\\
 &\quad+\|\hat\nabla h\|_{C^{6,1/4}_{-1/2}}
 +\|\hat\nabla^2h\|_{C^{6,1/4}_{-1/2}}
 +\|\hat\nabla_{\mathbf X}h\|_{C^{6,1/4}_{-1/2}}
 \le C\|h\|_{\mathscr H}.
 \end{aligned}
\end{equation}
For $v=h(\mathbf X,\cdot)$, one also has
\begin{equation}\label{eq:graph-radial-bounds}
 \|v\|_{C^{6,1/4}_{-1/2}}
 +\|\hat\nabla v\|_{C^{6,1/4}_{-1/2}}
 +\|h(\mathbf X,\mathbf X)+\tfrac12\chi\tr_{g_\Sigma}a_h
       \|_{C^{6,1/4}_{-1/2}}
 \leq C\|h\|_{\mathscr H}.
\end{equation}
For $u\in\mathscr Z$ write $L_{1/2}u=\chi b+v^\star$.
There is a unique $a_u\in C^{7,3/4}(\Sigma)$ with
\begin{equation}\label{eq:finite-scalar-bounds}
 \begin{aligned}
 &\|a_u\|_{C^{7,3/4}}
 +\|u-\chi ra_u-2\chi b\|_{C^{6,1/4}_{-1/2}}
 +\|\hat\nabla^2u\|_{C^{6,1/4}_{-1/2}}\\
 &\quad+\|\hat\nabla u-\nabla^C(\chi ra_u)\|_{C^{6,1/4}_{-1/2}}
 +\|(\hat\nabla_{\mathbf X}-\tfrac12)u+\chi b\|_{C^{6,1/4}_{-1/2}}
 \le C\|u\|_{\mathscr Z}.
 \end{aligned}
\end{equation}
Both graphs embed continuously in $C^{8,1/4}_{\mathrm{loc}}$,
and $\mathscr Z\hookrightarrow C^{6,1/4}_1$. Moreover,
\begin{equation}\label{eq:original-to-graph}
 E_8\cap\ker\hat D\hookrightarrow\mathscr H.
\end{equation}
Every sufficiently small $h\in\mathscr H$ defines a complete positive
metric $\hat g+h$.
\end{Proposition}
\begin{proof}
The weighted target embeds continuously in $L^2_{\hat f}$ and in local
distributions. Distributional differentiation and Gaussian pairing
against any fixed kernel tensor are continuous in this norm. Thus
$\mathscr F$ and $\mathscr F'$ are closed. The Gaussian self-adjoint
realization of $\hat L$ has compact resolvent by
\cite[Propositions~3.3--3.4]{LYZ-I}. Its closed divergence-free
subspace is reducing: the identities
\[
 \hat D\hat L=(\tfrac12\hat\Delta_{\hat f}+\tfrac14)\hat D,
 \qquad
 \hat L\hat D^*=\hat D^*(\tfrac12\hat\Delta_{\hat f}+\tfrac14)
\]
first hold on smooth fields and extend to the closed realizations.
For a sufficiently large positive resolvent parameter the one-form
operator on the right is invertible; applying $\hat D$ shows that
the tensor resolvent preserves $\ker\hat D$. Self-adjointness gives
the reducing assertion. The restricted Fredholm alternative gives,
for each $q\in\mathscr F'$, a unique normal solution $h$ with
\[
 \|h\|_{H^2_{\hat f}}\le C\|q\|_{L^2_{\hat f}}\le C\|q\|_{\mathscr F'}.
\]
This proves surjectivity and the inverse bound, as well as completeness
of the inverse-image graph. The finite-dimensional kernel consists of
smooth fields. Repeated use of \cref{cor:weight-minus-two} with
zero source shows that each kernel tensor has a smooth degree-zero
trace and weight-minus-two remainder. The slice equation forces its
trace to be tangential, so it also belongs to every original $E_m$.

Apply \cref{lem:tensor-terminal} to the normal Gaussian solution and
to the kernel summand. This gives \eqref{eq:finite-tensor-bounds},
apart from tangentiality. Since
$h(\mathbf X,\cdot)=\hat\delta h\in C^{6,1/4}_{-1/2}$,
the degree-one radial row of $\chi\boldsymbol\tau(a_h)(\mathbf X,\cdot)$
vanishes. Thus $a_h$ is tangential. The exact slice identities are
\begin{equation}\label{eq:graph-radial-identities}
 h(\mathbf X,\cdot)=\hat\delta h,\qquad
 h(\mathbf X,\mathbf X)=\hat\delta(\hat\delta h)
                              -\langle h,\hat\nabla\mathbf X\rangle.
\end{equation}
The second follows by differentiating the first and contracting.
The whole-field gradient and Hessian estimates control $v=\hat\delta h$
and $\hat\nabla v$ in $C^{6,1/4}_{-1/2}$. Since
$\hat\nabla\mathbf X=\tfrac12\hat g-\Ric_{\hat g}$, the tensor expansion
and the second identity give \eqref{eq:graph-radial-bounds}.
The scalar statements are precisely
\cref{lem:fractional-terminal,lem:fractional-degree-zero}, with the bounded Gaussian
inverse absorbing the low norm. Local elliptic regularity proves the
local inclusions. The scalar target decomposition is unique, and its
inverse-image norm proves completeness of $\mathscr Z$.

For \eqref{eq:original-to-graph}, write the original expansion
$h=h_0+h_{-1}+h^\star$. Direct differentiation gives
\[
 \hat Lh_0\in C^{6,1/4}_{-2},\qquad
 \hat Lh_{-1}-\tfrac14h_{-1}\in C^{6,1/4}_{-3},\qquad
 \hat Lh^\star\in C^{6,1/4}_{-2},
 \qquad \|\hat Lh\|_{C^{6,1/4}_{-1}}\le C\|h\|_{E_8}.
\]
The drift on the remainder uses one derivative, and its Laplacian
uses two. The commutation identity gives $\hat D\hat Lh=0$;
self-adjointness gives $\hat Lh\perp\ISD$. Gaussian projection onto
$\ISD$ is bounded on $E_8$, proving the inclusion. Finally,
\eqref{eq:finite-tensor-bounds} controls $\|h\|_{C^0}$, so small
graph balls give metrics uniformly equivalent to the complete $\hat g$.
\end{proof}

\section{Analyticity of Perelman's functional}\label{sec:metric-LS-auxiliary}
We construct the scalar branch on the fixed tensor graph $\mathscr H$.
The same domain will be used for the analytic center in
Section~\ref{sec:finite-Banach-metric-LS}.

\subsection{Deformations and the analytic scalar branch}\label{sec:computation-of-deformation}\label{sec:analyticity-of-mu-IFT}
All derivatives in the following computations use $\hat g$.
For $g=\hat g+h$ let $H=\hat g^{-1}h$, $B=(I+H)^{-1}$,
$v=h(\mathbf X,\cdot)$.
Write $v^\sharp$ for raising with $\hat g$. Put
\[
 T_{ij\ell}=\hat\nabla_i h_{j\ell}+\hat\nabla_j h_{i\ell}
                           -\hat\nabla_\ell h_{ij},\qquad
 A^k{}_{ij}=\tfrac12g^{k\ell}T_{ij\ell}.
\]
Thus $A=\nabla^g-\hat\nabla$ and
$\nabla^g\hat f=\mathbf X-Bv^\sharp$.

\begin{Lemma}\label{lem:graph-nonlinear-closure}
On a sufficiently small ball in $\mathscr H\times\mathscr Z$, the maps
\begin{align}
 W(h,u)&=2\Delta_g(\hat f+u)-|\mathrm d(\hat f+u)|_g^2
       +\mathrm R_g+\hat f+u-n-\boldsymbol\mu(\hat g),
       \label{eq:function-for-minimizer-IFT}\\
 (h,u)&\longmapsto\Phi(g,\hat f+u)
       \label{eq:graph-defect-map}
\end{align}
are real analytic into $\mathcal T$ and $\mathcal V$, respectively.
At the origin,
\begin{equation}\label{eq:scalar-Banach-differential}
 \mathrm D_uW(0,0)=2L_{1/2},\qquad
 \mathrm D_hW(0,0)[k]=-L_{1/2}\tr_{\hat g}k
 \quad(k\in\mathscr H).
\end{equation}
For $\epsilon=\|h\|_{\mathscr H}+\|u\|_{\mathscr Z}$, one also has
\begin{equation}\label{eq:finite-low-coefficients}
 \begin{gathered}
 \|h\|_{C^{3,1/4}_0}+\|u\|_{C^{3,1/4}_1}\le C\epsilon,
 \quad\|\nabla^g(\hat f+u)-\mathbf X\|_{C^0}\le C\epsilon,\\
 \|\Phi(g,\hat f+u)\|_{C^0_{-1/2}}\le C\epsilon.
 \end{gathered}
\end{equation}
\end{Lemma}
\begin{proof}
The space of degree-zero coefficient expansions with
$C^{6,3/4}$ coefficient and $C^{6,1/4}_{-1/2}$
remainder is a Banach algebra: subtract the product of the leading
coefficients, and each remaining product has a decaying factor.
Multiplication by a bounded $C^{6,1/4}_0$ field is bounded on the
decaying target. These assertions follow on each normalized annulus
from the H\"older product inequality, with a uniform constant.
Consequently $B$, $g^{-1}$, their traces and determinants are analytic
in the indicated expansion algebra, by their convergent power series.

We verify the contractions with the growing background potential.
The slice identities \eqref{eq:graph-radial-identities} give
\[
 \tfrac12 T_{ij\ell}X^\ell
 =\operatorname{Sym}(\hat\nabla v)_{ij}
      -\operatorname{Sym}(h\circ\hat\nabla^2\hat f)_{ij}
      -\tfrac12\hat\nabla_{\mathbf X}h_{ij}.
\]
Here compositions use $\hat g$. It follows exactly that
\begin{equation}\label{eq:graph-Hessian-cancellation}
 \begin{split}
 K_h&:=\nabla_g^2\hat f-\hat\nabla^2\hat f-\tfrac12h\\
 &=-\operatorname{Sym}(\hat\nabla v)
   -\operatorname{Sym}(h\circ\Ric_{\hat g})
   +\tfrac12\hat\nabla_{\mathbf X}h
   +\tfrac12T(Bv^\sharp).
 \end{split}
\end{equation}
Every displayed term is analytic into $\mathcal V$, using the
whole-field bounds of \cref{prop:finite-graph-realization}.
The curvature difference is a sum of contractions of
$g^{-1}\hat\nabla^2h$, $g^{-1}g^{-1}\hat\nabla h\hat\nabla h$,
and the background curvature times $h$, so it has the same property.
The exact identity
\begin{equation}\label{eq:graph-defect-cancellation}
 \Phi(g,\hat f+u)=-(\Ric_g-\Ric_{\hat g})-K_h
                       -\hat\nabla^2u+A\,\mathrm du
\end{equation}
proves the second assertion: $\mathrm du$ is bounded in
$C^{6,1/4}_0$, $A\in C^{6,1/4}_{-1/2}$, and the
whole Hessian of $u$ is in the decaying target.

For the scalar expression, the metric-only part is exactly
\begin{equation}\label{eq:graph-scalar-background}
 \begin{split}
 W(h,0)={}&\mathrm R_g-\mathrm R_{\hat g}
  +2(\mathrm R_{\hat g}-\tr_g\Ric_{\hat g})
  +2\tr_g K_h\,+h(\mathbf X,\mathbf X)-\langle v^\sharp,Bv^\sharp\rangle.
 \end{split}
\end{equation}
All terms except the double radial contraction are in the decaying
target. By \eqref{eq:graph-radial-bounds}, its leading coefficient is
$-\tfrac12\tr_{g_\Sigma}a_h$. Moreover,
\begin{equation}\label{eq:graph-scalar-full}
 \begin{split}
 W(h,u)={}&W(h,0)+2L_{1/2}u
 +2(g^{-1}-\hat g^{-1}):\hat\nabla^2u
 -2(\tr_g A)\,\mathrm du\\
 &+2\langle Bv^\sharp,\hat\nabla u\rangle
 -|\mathrm du|_g^2.
 \end{split}
\end{equation}
The middle three terms have weight at most $-1/2$.
If $L_{1/2}u=\chi b+v^\star$, the gradient estimate in
\eqref{eq:finite-scalar-bounds} shows that the last term has leading
coefficient
$a_u^2+|\mathrm d_\Sigma a_u|_{g_\Sigma+a_h}^2$.
Thus the full leading scalar coefficient is
\begin{equation}\label{eq:scalar-leading-coefficient}
 -\tfrac12\tr_{g_\Sigma}a_h+2b
                       -a_u^2-|\mathrm d_\Sigma a_u|_{g_\Sigma+a_h}^2.
\end{equation}
It belongs to $C^{6,3/4}$, while the remainder belongs
to $C^{6,1/4}_{-1/2}$. The cutoff transition is compact and
is included in the remainder. These formulas prove, on fixed small balls,
\[
 \|W(h,u)-2L_{1/2}u\|_{\mathcal T}
 \le C\bigl(\|h\|_{\mathscr H}
              +\|h\|_{\mathscr H}\|u\|_{\mathscr Z}
              +\|u\|_{\mathscr Z}^2\bigr).
\]
The same bounded products and convergent inverse-matrix series prove
real analyticity. The differential
\eqref{eq:scalar-Banach-differential} follows from the standard variation
formulas \cite[(2.9)--(2.11)]{cao2012second}, using $\hat Dk=0$.

The first bound in \eqref{eq:finite-low-coefficients} follows from
\cref{prop:finite-graph-realization}. The identity
\[
 \nabla^g(\hat f+u)-\mathbf X
       =B(\hat\nabla u-v^\sharp)
\]
and the same graph bounds give the drift estimate. Analyticity of the
defect map and $\Phi(\hat g,\hat f)=0$ give the remaining bound.
Differentiating the inverse-metric, connection, and drift formulas also
controls the inverse metric through order two, the connection difference
through order one, and the drift difference through order one, with
perturbation bounds $C\epsilon$ in the corresponding polynomial norms.
These are the coefficient estimates used in
Section~\ref{sec:appendix-weakened-gaussian}.
\end{proof}

\begin{Theorem}\label{thm:uniqueness-of-minimizer-IFT}
There are fixed open neighborhoods $U_1\subset\mathscr H$ and
$U_2\subset\mathscr Z$ and a unique real-analytic map $h\mapsto u_h\in U_2$
such that $W(h,u_h)=0$. It satisfies $u_0=0$,
$\|u_h\|_{\mathscr Z}\le C\|h\|_{\mathscr H}$, and
$\mathrm Du_0[k]=\tfrac12\tr_{\hat g}k$.
\end{Theorem}
\begin{proof}
The differential $2L_{1/2}:\mathscr Z\to\mathcal T$ is
a bounded isomorphism by definition of the inverse-image domain.
Apply the analytic Banach implicit function theorem to
\cref{lem:graph-nonlinear-closure}. Operator-norm continuity gives
nearby invertibility by the Neumann series on these same two spaces.
The derivative formula follows from
\eqref{eq:scalar-Banach-differential}; the trace belongs to $\mathscr Z$
because $L_{1/2}\tr k\in\mathcal T$ and it is Gaussian integrable.
\end{proof}

\subsection{Asymptotic estimates and existence of the minimizer}\label{sec:asymptotic-estimates-on-minimizer}
The main purpose of this subsection is to show that every minimizer,
if it exists, lies, up to a constant, in the potential neighborhood
obtained in \cref{thm:uniqueness-of-minimizer-IFT}.
We then establish existence and identify the normalized solution.
The metrics have local $C^{8,1/4}$ regularity and uniform curvature and
injectivity-radius bounds by \cref{prop:finite-graph-realization}.
The elliptic estimates and Dirichlet exhaustion arguments below apply by
smooth approximation and elliptic compactness, with constants depending
only on the controlled finite norms.
\begin{Theorem}\label{thm:asymptotic-estimates-on-minimizer}
    Given an open neighborhood $U_1\subset \mathscr Z$ of the origin, there exists an open neighborhood $U_2\subset \mathscr H$ of the origin such that the following holds.
    Let $h\in U_2$ and $g:=\hat{g}+h$.
    For any normalized entropy minimizer $f$ satisfying $\int_M\mathrm dV_f=1$
    and the Euler--Lagrange equation
    \begin{equation}\label{eq:Euler-Lagrange-equation}
        2\Delta^g f-|\nabla^gf|^2+\mathrm{R}_g+f-n-\boldsymbol{\mu}(g)=0
    \end{equation}
    of $\mathcal W(g,\cdot)$, we have $(f-\boldsymbol{\mu}(g))-(\hat{f}-\boldsymbol{\mu}(\hat{g}))\in U_1$.
\end{Theorem}

We denote $u:=f-\hat f$ in the remainder of the asymptotic argument.
\begin{Lemma}\label{lem:quadratic-growth-minimizer}
Let $f$ be a normalized solution as in
\cref{thm:asymptotic-estimates-on-minimizer}. Then
$f\leq C(1+\boldsymbol{\rho}^2)$ on $M$ for some $C>0$.
\end{Lemma}
\begin{proof}
Put $v=(4\pi)^{-n/4}e^{-f/2}$. It satisfies
\[
 4\Delta^gv-\mathrm{R}_gv+2v\log v+\lambda v=0,
 \qquad \lambda=\boldsymbol{\mu}(g)+n+\frac n2\log(4\pi).
\]
By \cite[Lemma~2.1]{zhang2012extremal}, $v\to0$ at infinity.
Thus, on $E_R=\{\boldsymbol{\rho}>R\}$ for sufficiently large $R$,
$v\leq e^{-2}$ and
\begin{equation}\label{ineq:v-as-supersolution}
 4\Delta^gv+3v\log v=(\log v+\mathrm{R}_g-\lambda)v\leq0.
\end{equation}
The definition of $\boldsymbol{\rho}$ and the deformation estimates give
\[
 |\mathrm d(\boldsymbol{\rho}^2)|_g^2\geq c\boldsymbol{\rho}^2,
 \qquad |\Delta^g(\boldsymbol{\rho}^2)|\leq C
 \qquad\text{on }E_R.
\]
Choose $0<a\leq\min\{e^{-2},\min_{\partial E_R}v\}$ and set
$z=a e^{-b\boldsymbol{\rho}^2}$. For sufficiently large $b$,
\[
 \frac{4\Delta^gz+3z\log z}{z}
 =4b^2|\mathrm d(\boldsymbol{\rho}^2)|_g^2
   -4b\Delta^g(\boldsymbol{\rho}^2)+3\log a-3b\boldsymbol{\rho}^2
 \geq0.
\]
Both $v,z$ are at most $e^{-2}$, so the mean value theorem gives
\[
 4\Delta^g(v-z)+3(\log\xi+1)(v-z)\leq0,\qquad
 3(\log\xi+1)\leq-3,
\]
where $\xi$ lies between $v$ and $z$. Since $v-z\geq0$ on
$\partial E_R$ and $v-z\to0$ at infinity, the maximum principle gives
$v\geq z$. Taking logarithms and enlarging $C$ on the compact core
proves the assertion.
\end{proof}

\begin{Lemma}\label{lem:higher-order-estimates-on-minimizer}
For $f$ as in \cref{lem:quadratic-growth-minimizer}, we have
$|\hat\nabla f|_{\hat g}\leq C\boldsymbol{\rho}$ and
$|\hat\nabla^2f|_{\hat g}\leq C\boldsymbol{\rho}^2$
on a sufficiently far end.
\end{Lemma}
\begin{proof}
By the gradient estimate \cite[Theorem~1.1]{wang2023gradient},
\[
 |\nabla^gf|_g^2\leq\frac32f+C,
\]
where $C$ depends on $n$, a lower bound on $\Ric_g$, and bounds on
$|\mathrm R_g|$, $|\nabla^g\mathrm R_g|_g$, and $\boldsymbol\mu(g)$.
The first assertion follows from \cref{lem:quadratic-growth-minimizer}
and uniform equivalence of $g$ and $\hat g$.

For the Hessian estimate, use a ball of radius $c/\boldsymbol{\rho}(p)$
about a far point $p$. The gradient bound makes the oscillation of $f$
on this ball uniformly bounded. After scaling to unit size, the equation
for $f-f(p)$ has uniformly bounded coefficients, right-hand side, and
first derivatives of the unknown. Interior elliptic estimates followed
by the local Schauder estimate bound its second derivatives.
Scaling back gives $|\nabla^{g,2}f|(p)\leq C\boldsymbol{\rho}(p)^2$.
The connection difference is $O(\boldsymbol{\rho}^{-1})$, so the same
estimate holds for $\hat\nabla^2f$.
\end{proof}

The next lemma uses the weighted integration-by-parts arguments of
\cite[Lemmas~3.1 and~4.1]{LYZ-I}. Here the displayed bounds on $f$
replace the shrinker identities, since $(M,g,f)$ need not be a Ricci shrinker.
\begin{Lemma}\label{lem:weighted-exterior-maximum-principle}
Let $E_R=\{\boldsymbol{\rho}>R\}$. Suppose that, on $E_R$,
\begin{equation}\label{eq:auxiliary-exterior-potential-bounds}
 f\geq1,\qquad
 |\nabla^gf|_g^2\leq\Lambda f,\qquad
 \Delta_f^gf\leq-\lambda f
\end{equation}
for constants $\Lambda,\lambda>0$. Then
\begin{equation}\label{eq:weighted-exterior-Poincare}
 \int_{E_R}f\phi^2\,\mathrm dV_f
 \leq \frac{4\Lambda}{\lambda^2}
       \int_{E_R}|\nabla^g\phi|_g^2\,\mathrm dV_f
\end{equation}
for every $\phi\in C_c^\infty(E_R)$.

Let $V\in L^\infty_{\mathrm{loc}}(E_R)$ satisfy
$V_+\leq\kappa f$, where $\kappa<\lambda^2/(4\Lambda)$.
If $w\in H^1_{f,\mathrm{loc}}(E_R)$,
$w_+\in H^1_{f,0}(E_R)$, and
\[
 \Delta_f^gw+Vw\geq0
\]
weakly on $\{w>0\}$, then $w\leq0$ almost everywhere on $E_R$.
\end{Lemma}
\begin{proof}
Weighted integration by parts and
\cref{eq:auxiliary-exterior-potential-bounds} give
\[
\lambda\int f\phi^2\,\mathrm dV_f
 \leq-\int\phi^2\Delta_f^gf\,\mathrm dV_f
 =2\int\phi\langle\nabla^g\phi,\nabla^gf\rangle_g\,\mathrm dV_f
 \leq\frac{\lambda}{2}\int f\phi^2\,\mathrm dV_f
      +\frac{2\Lambda}{\lambda}\int|\nabla^g\phi|_g^2\,\mathrm dV_f.
\]
This proves \cref{eq:weighted-exterior-Poincare}. Testing the weak
inequality with cutoff approximations to $w_+$ and using that estimate yields
\[
 0\leq-\int|\nabla^gw_+|_g^2\,\mathrm dV_f+\int Vw_+^2\,\mathrm dV_f
 \leq-\left(1-\frac{4\Lambda\kappa}{\lambda^2}\right)
       \int|\nabla^gw_+|_g^2\,\mathrm dV_f.
\]
Applying \cref{eq:weighted-exterior-Poincare}, extended by density, gives
$\int_{E_R}fw_+^2\,\mathrm dV_f=0$. Since $f\geq1$, we obtain $w_+=0$.
\end{proof}

\begin{Lemma}
Define
\begin{align*}
 (\phi_h)_i^{\phantom{i}j}&:=\Phi(g,\hat f)_{ik}g^{kj},\\
 \psi_h&:=\boldsymbol\mu(g)-\boldsymbol\mu(\hat g)
       -(\mathrm R_g-\mathrm R_{\hat g})
       -2(\Delta^g-\hat\Delta)\hat f
       +|\mathrm d\hat f|_g^2-|\mathrm d\hat f|_{\hat g}^2.
\end{align*}
Then $\phi_h\in C^{6,1/4}_{-1/2}(M,\End(T^*M))$ and $\psi_h\in\mathcal T$.
Moreover, $\beta:=\mathrm du$ solves
\begin{equation}\label{eq:beta-equation}
 \Delta_f^g\beta+\phi_h\beta=\frac12\mathrm d\psi_h.
\end{equation}
\end{Lemma}
\begin{proof}
The coefficient estimates follow from
\cref{lem:graph-nonlinear-closure},
applied to the pair $(g,\hat f)$.
Substitute $f=\hat f+u$ into \cref{eq:Euler-Lagrange-equation}, use
\[
 |\mathrm d(\hat f+u)|_g^2
 =|\mathrm d\hat f|_g^2
   +2\langle\mathrm d\hat f,\mathrm du\rangle_g+|\mathrm du|_g^2,
\]
and subtract $W(0,0)=0$. The terms involving only $\hat f$ give
$-\psi_h$, leaving
\begin{equation}\label{eq:potential-defect-difference-equation}
 (2\Delta_{\hat f}^g+1)u-|\mathrm du|_g^2=\psi_h,
\end{equation}
where $\Delta_{\hat f}^g=\Delta^g-\nabla^g_{\nabla^g\hat f}$.
The commutation identities
\[
 \mathrm d\Delta_{\hat f}^gu
 =\Delta_{\hat f}^g\mathrm du
  -(\Ric_g+\nabla_g^2\hat f)^\sharp(\mathrm du),
 \qquad
 \mathrm d|\mathrm du|_g^2=2\nabla^g_{\nabla^gu}\mathrm du
\]
therefore give
\[
 2\bigl(\Delta_{\hat f}^g-\nabla^g_{\nabla^gu}\bigr)\beta
 +\bigl(I-2(\Ric_g+\nabla_g^2\hat f)^\sharp\bigr)\beta
 =\mathrm d\psi_h.
\]
Since $f=\hat f+u$, the first operator is $2\Delta_f^g$, and the
zeroth-order coefficient is $2\Phi(g,\hat f)^\sharp=2\phi_h$.
Dividing by two proves \cref{eq:beta-equation}.
\end{proof}

\begin{proof}[Proof of \cref{thm:asymptotic-estimates-on-minimizer}]
Suppose the assertion fails for a fixed neighborhood $U_1$.
Since $\mathscr H$ is metrizable, there are $h_i\to0$ in $\mathscr H$ and
normalized solutions $f_i$ for $g_i=\hat g+h_i$ such that
\[
 (f_i-\boldsymbol\mu(g_i))-(\hat f-\boldsymbol\mu(\hat g))
 \notin U_1.
\]
We first prove compactness and a small global gradient bound, and then
identify these solutions with those of
\cref{thm:uniqueness-of-minimizer-IFT}.

Write
\[
 v_i=(4\pi)^{-n/4}e^{-f_i/2},\qquad
 \lambda_i=\boldsymbol\mu(g_i)+n+\frac n2\log(4\pi).
\]
Then
\[
 4\Delta^{g_i}v_i-\mathrm R_{g_i}v_i+2v_i\log v_i+\lambda_i v_i=0,
 \qquad \int_Mv_i^2\,\mathrm dV_{g_i}=1.
\]
The AC geometry and $h_i\to0$ in $\mathscr H$ give uniformly bounded
curvature and a uniform positive lower bound for the injectivity radius
of $g_i$. Hence the Sobolev inequality
\cite[Theorem~2.21]{aubin1998some} gives
\[
 \|\varphi\|_{L^{2n/(n-2)}(g_i)}^2
 \leq C\bigl(\|\nabla^{g_i}\varphi\|_{L^2(g_i)}^2
                  +\|\varphi\|_{L^2(g_i)}^2\bigr).
\]
For $\varphi\in C_c^\infty(M)$ with
$\int_M\varphi^2\,\mathrm dV_{g_i}=1$,
Jensen's inequality for the probability measure
$\varphi^2\,\mathrm dV_{g_i}$ yields, as in
\cite[Section~3, p.~6]{sesum2004notes},
\begin{align*}
 \int_M\varphi^2\log\varphi^2\,\mathrm dV_{g_i}
 &\leq\frac{n-2}{2}\log\int_M|\varphi|^{2n/(n-2)}\,\mathrm dV_{g_i}\\
 &\leq\frac n2\log\bigl[C(1+\|\nabla^{g_i}\varphi\|_{L^2(g_i)}^2)\bigr]
 \leq2\|\nabla^{g_i}\varphi\|_{L^2(g_i)}^2+C.
\end{align*}
Thus $\boldsymbol\mu(g_i)\geq-C$.
Testing with normalized compactly supported approximations of
$(4\pi)^{-n/4}e^{-\hat f/2}$ also gives
\[
 \limsup_i\boldsymbol\mu(g_i)\leq\boldsymbol\mu(\hat g)<0.
\]
The strict inequality follows because the background AC shrinker is
nontrivial, by the entropy-gap theorem of Li--Wang \cite{li2020heat}.

The local estimates in \cite[Lemma~2.1]{zhang2012extremal} give
\[
 \sup_M\bigl(v_i+|\nabla^{g_i}v_i|_{g_i}\bigr)\leq C
\]
with $C$ independent of $i$. They also give $v_i\to0$ at infinity
for each fixed $i$, and \cite[Lemma~2.3]{zhang2012extremal} gives
Gaussian upper decay for that solution.
Together with \cref{lem:quadratic-growth-minimizer,lem:higher-order-estimates-on-minimizer},
these individual bounds justify integration of the equation against
$v_i$ by cutoffs. Hence
$\mathcal W(g_i,f_i)=\boldsymbol\mu(g_i)$, so the normalized solutions
under consideration are minimizers.

After passing to a subsequence, let $\boldsymbol\mu(g_i)\to\ell<0$
and choose a maximum point $p_i$ of $v_i$.
At this point the equation implies
\[
 v_i(p_i)\geq
 \exp\!\left(\frac{\inf\mathrm R_{g_i}-\lambda_i}{2}\right)\geq c>0.
\]
The gradient estimate used in \cref{lem:higher-order-estimates-on-minimizer}
gives $|\mathrm d\sqrt{C+f_i}|_{g_i}\leq C$ after one uniform shift.
Since $f_i(p_i)\leq C$, integration along geodesics yields
\[
 f_i(x)\leq C\bigl(1+d_{g_i}(x,p_i)^2\bigr).
\]
Thus $v_i$ has a uniform positive lower bound on each fixed-radius ball
about $p_i$, and local elliptic estimates give compactness in $C^4$ on compact subsets there.

If $p_i$ escaped to infinity, uniform AC geometry would give pointed
$C^3$ convergence to Euclidean space and a nonzero positive limit
$v_\infty$ with squared $L^2$ norm $a\in(0,1]$.
The decay and cutoff arguments apply to this limit.
Normalizing it by $a^{-1/2}$ changes the Euler--Lagrange constant by
$\log a$, so its entropy is $\ell+\log a$. This contradicts
\[
 0=\boldsymbol\mu(\mathbb R^n)\leq\ell+\log a<0.
\]
Hence $p_i$ remain in a fixed compact set, and we obtain a nonzero local
limit on $(M,\hat g)$. Writing $a\in(0,1]$ for its squared $L^2$ norm,
the same normalization calculation gives
\[
 \boldsymbol\mu(\hat g)\leq\ell+\log a\leq\ell
 \leq\boldsymbol\mu(\hat g).
\]
Therefore $v_\infty$ has unit $L^2(M,\hat g)$ mass and
$\ell=\boldsymbol\mu(\hat g)$.
For its normalized potential $f_\infty$,
\cref{eq:beta-equation} at $h=0$ and weighted integration by parts
show that $\mathrm d(f_\infty-\hat f)$ is parallel.
Since a nontrivial AC shrinker does not split off a line, normalization
gives $f_\infty=\hat f$.
Every convergent subsequence has the same limit, and hence
\[
 f_i\longrightarrow\hat f\quad\text{in }C^4_{\mathrm{loc}},
 \qquad \boldsymbol\mu(g_i)\longrightarrow\boldsymbol\mu(\hat g).
\]

The unit mass of $v_\infty$ and local convergence give tightness of
$v_i^2\,\mathrm dV_{g_i}$.
Together with the uniform gradient estimate for $v_i$, this implies
$v_i\to0$ at infinity uniformly in $i$: a value at least $\epsilon>0$
at a point tending to infinity would give $v_i\geq\epsilon/2$ on a
ball of fixed radius. Uniform volume lower bounds would then place
a fixed positive mass outside every compact set, a contradiction.
The upper barriers in \cite[Lemma~2.3]{zhang2012extremal} can therefore
be chosen on one fixed exterior, giving
\[
 v_i\leq Ce^{-c\boldsymbol{\rho}^2}
\]
with constants independent of $i$.
Local convergence also gives uniform positive boundary values for the
lower barriers in \cref{lem:quadratic-growth-minimizer}.
Consequently, on a fixed exterior,
\[
 c\boldsymbol{\rho}^2-C\leq f_i\leq C\boldsymbol{\rho}^2+C,\qquad
 |\nabla^{g_i}f_i|^2\leq\tfrac32 f_i+C,\qquad
 |\hat\nabla^2f_i|\leq C\boldsymbol{\rho}^2,
\]
with constants independent of $i$.

Set $u_i=f_i-\hat f$, $\beta_i=\mathrm du_i$, and
$q_i=|\beta_i|_{g_i}^2$. The coefficient formulas give
\[
 |\phi_{h_i}|+|\mathrm d\psi_{h_i}|
 \leq\eta_i\boldsymbol{\rho}^{-1/2},\qquad \eta_i\longrightarrow0.
\]
The entropy difference disappears upon differentiating $\psi_{h_i}$,
so this bound does not require a global estimate for $u_i$.
Taking the inner product of \cref{eq:beta-equation} with $\beta_i$ gives
\[
 \frac12\Delta_{f_i}^{g_i}q_i
 =|\nabla^{g_i}\beta_i|^2-\langle\phi_{h_i}\beta_i,\beta_i\rangle
   +\frac12\langle\mathrm d\psi_{h_i},\beta_i\rangle.
\]
The Euler--Lagrange equation also gives
\[
 \Delta_{f_i}^{g_i}f_i
 =-\frac12|\mathrm df_i|_{g_i}^2-\frac12 f_i-\frac12\mathrm R_{g_i}
   +\frac n2+\frac12\boldsymbol\mu(g_i).
\]
On a sufficiently far fixed exterior $E_R=\{\boldsymbol{\rho}>R\}$,
we thus have $f_i^{-1/4}\leq1/2$, $f_i\asymp\boldsymbol{\rho}^2$, and
\[
 \Delta_{f_i}^{g_i}f_i\leq-c_0f_i,\qquad
 |\mathrm df_i|_{g_i}^2\leq C_1f_i,\qquad
 \Delta_{f_i}^{g_i}q_i\geq-A_i(q_i+1),
 \quad A_i=C\eta_i f_i^{-1/4}.
\]
The hypotheses of \cref{lem:weighted-exterior-maximum-principle}
hold uniformly. Polynomial growth of the functions and their first
derivatives, together with Gaussian decay, puts the positive parts
used below in $H^1_{f_i,0}(E_R)$ whenever their inner boundary trace vanishes.

Let $m_i=\sup_{\partial E_R}q_i\to0$ and set
\[
 K_i=2m_i+\frac{16C}{c_0}\eta_i,\qquad
 B_i=K_i(1-f_i^{-1/4}).
\]
Since
\[
 \Delta_{f_i}^{g_i}(f_i^{-1/4})
 =-\frac14f_i^{-5/4}\Delta_{f_i}^{g_i}f_i
 +\frac5{16} f_i^{-9/4}|\mathrm df_i|_{g_i}^2
 \ge\frac{c_0}{4}f_i^{-1/4},
\]
for all sufficiently large $i$ we have
\[
 (\Delta_{f_i}^{g_i}+A_i)B_i
 \leq-\frac{c_0K_i}{8}f_i^{-1/4}\leq-A_i,
 \qquad B_i\geq K_i/2\geq q_i\quad\text{on }\partial E_R.
\]
On the other hand, $(\Delta_{f_i}^{g_i}+A_i)q_i\geq-A_i$.
Apply \cref{lem:weighted-exterior-maximum-principle} to $q_i-B_i$.
It gives $q_i\leq B_i\leq K_i$ on $E_R$.
Since $K_i\to0$, local convergence yields
\[
 \|\mathrm d(f_i-\hat f)\|_{C^0(M)}\longrightarrow0.
\]

By \cite[Proposition~3.3(i)]{LYZ-I}, the operator
$-\Delta_{\hat f}^{\hat g}$ on one-forms has discrete spectrum and
kernel consisting of parallel forms.
Nonsplitting excludes nonzero parallel one-forms, so its first
eigenvalue is positive. Thus, for some $c_*>0$,
\[
 c_*\|\xi\|_{L^2_{\hat f}}^2
 \leq\|\hat\nabla\xi\|_{L^2_{\hat f}}^2
 \qquad\text{for every }\xi\in H^1_{\hat f}(M,T^*M).
\]
This gap persists for the connection energy with metric $g_i$ and
measure $\mathrm dV_{f_i}$.
Let $J_i:(T^*M,g_i)\to(T^*M,\hat g)$ be the positive fiberwise isometry and put
\[
 T_i\xi=e^{-(f_i-\hat f)/2}
       \left(\frac{\mathrm dV_{g_i}}{\mathrm dV_{\hat g}}\right)^{1/2}J_i\xi.
\]
This is an isometry between the two $L^2$ spaces.
The connection comparison, $h_i\to0$, and the preceding global gradient estimate give
\[
 \|\hat\nabla(T_i\xi)\|_{L^2_{\hat f}}
 \leq(1+o(1))\|\nabla^{g_i}\xi\|_{L^2_{f_i}}
       +o(1)\|\xi\|_{L^2_{f_i}}.
\]
Indeed, differentiation of the density factor introduces only
$\mathrm d(f_i-\hat f)$ and the derivative of the volume ratio;
differentiation of $J_i$ introduces only first derivatives of $h_i$.
Absorbing the last term into the background gap yields, for large $i$,
\[
 \frac{c_*}{4}\|\xi\|_{L^2_{f_i}}^2
 \leq\|\nabla^{g_i}\xi\|_{L^2_{f_i}}^2.
\]
The calculation is first made for compactly supported forms and then
extended by density.

Let $a_i\in \mathscr Z$ be the solution supplied by
\cref{thm:uniqueness-of-minimizer-IFT}. Then $a_i\to0$ in that space and
\[
 f_i^0=\hat f+a_i,\qquad
 2\Delta^{g_i}f_i^0-|\mathrm df_i^0|_{g_i}^2+\mathrm R_{g_i}+f_i^0
 =n+\boldsymbol\mu(\hat g).
\]
Its defect $\Phi_i^0=\Phi(g_i,f_i^0)$ tends to zero in $C^0$, by
\cref{lem:graph-nonlinear-closure}.
Set $w_i=f_i-f_i^0$. Subtracting the scalar equations and differentiating,
as in \cref{eq:beta-equation}, gives
\[
 \Delta_{f_i}^{g_i}\mathrm dw_i+(\Phi_i^0)^\sharp(\mathrm dw_i)=0.
\]
The preceding bounds and $a_i\in \mathscr Z$ imply
$\mathrm dw_i\in H^1_{f_i}(M,T^*M)$.
Testing by $\mathrm dw_i$ and removing the Gaussian cutoffs gives
\[
 \|\nabla^{g_i}\mathrm dw_i\|_{L^2_{f_i}}^2
 \leq\|\Phi_i^0\|_{C^0}\|\mathrm dw_i\|_{L^2_{f_i}}^2.
\]
The connection-energy gap and $\|\Phi_i^0\|_{C^0}<c_*/4$ imply $\mathrm dw_i=0$.
The undifferentiated equation fixes this constant as
$w_i=\boldsymbol\mu(g_i)-\boldsymbol\mu(\hat g)$. Consequently,
\[
 (f_i-\boldsymbol\mu(g_i))-(\hat f-\boldsymbol\mu(\hat g))
 =a_i\longrightarrow0\quad\text{in }\mathscr Z.
\]
The left-hand side therefore belongs to $U_1$ for large $i$,
contradicting the choice of the sequence.
\end{proof}

\begin{Theorem}\label{thm:existence-and-uniqueness-of-minimizer}
On a sufficiently small neighborhood of zero in $\mathscr H$, every
$g=\hat g+h$ has a unique entropy minimizer $f_g$. The maps
$h\mapsto f_g-\hat f\in\mathscr Z$ and
$h\mapsto\boldsymbol\mu(g)\in\mathbb R$ are real analytic.
The same existence and uniqueness conclusions hold for original
metrics with $\|h\|_{E_{14}}$ sufficiently small. On this original
neighborhood the minimizer depends continuously in $C^{3,1/4}_1$
and the entropy depends $C^1$ on the metric.
\end{Theorem}
\begin{proof}
Testing with a compactly supported approximation to the background
minimizer gives $\boldsymbol\mu(g)<0$. Apply the Dirichlet exhaustion
construction of \cite[Theorem~1.1]{zhang2012extremal}.
Choose smooth exhaustion domains with uniformly controlled boundary
charts and extend the normalized Dirichlet minimizers by zero.
As in the cited proof, these extensions are weak subsolutions with
uniform $W^{1,2}$ and local H\"older bounds, and their maxima are bounded
below by a positive constant. If maximizing points escaped to infinity,
the asymptotically conical geometry and these bounds would give a nonzero
Euclidean limit $v_\infty$, with squared mass $a\in(0,1]$, satisfying
\[
 4\Delta v_\infty+2v_\infty\log v_\infty+\lambda v_\infty\geq0,
 \qquad \lambda=\boldsymbol\mu(g)+n+\frac n2\log(4\pi).
\]
Testing this weak inequality by cutoffs times $v_\infty$ and normalizing gives
\[
 0=\boldsymbol\mu(\mathbb R^n)\le\boldsymbol\mu(g)+\log a<0.
\]
Thus the maxima stay in a compact set, where interior compactness gives
a positive Euler--Lagrange solution. The same mass argument on $M$ gives
$\boldsymbol\mu(g)\le\boldsymbol\mu(g)+\log a$, hence $a=1$ and a
positive minimizer. The preceding asymptotic and spectral-gap argument
identifies every minimizer with $\hat f+u_h$ up to a constant.

Define
\begin{equation}\label{eq:normalized-analytic-branch}
 \begin{gathered}
 Z(g)=\int_M e^{-u_h}\sqrt{\det(I+\hat g^{-1}h)}\,\mathrm dV_{\hat f},\\
 f_g=\hat f+u_h+\log Z(g),\qquad
 \boldsymbol\mu(g)=\boldsymbol\mu(\hat g)+\log Z(g).
 \end{gathered}
\end{equation}
These relations follow from normalized mass and the scalar equation.
Constants belong to $\mathscr Z$, since their image under
$L_{1/2}$ is a constant in $\mathcal T$.
Treat $h$ and $u$ as independent real variables. Since
$\mathscr Z\hookrightarrow C^{6,1/4}_1$, one has
$|u|\le C\|u\|_{\mathscr Z}\boldsymbol\rho$. Hence, for every fixed
$\epsilon>0$,
\[
 \sum_{m=0}^{\infty}\frac1{m!}
 \sup_{\|u\|_{\mathscr Z}\le\epsilon}
       \int_M|u|^m\,\mathrm dV_{\hat f}
 \le\int_M e^{C\epsilon\boldsymbol\rho}\,\mathrm dV_{\hat f}<\infty.
\]
Thus $e^{-u}=\sum_{m\ge0}(-u)^m/m!$ is a series of continuous
homogeneous polynomials converging absolutely and uniformly on bounded
$\mathscr Z$ balls in $L^1_{\hat f}$. For small $\|h\|_{\mathscr H}$,
the determinant square root has a convergent real power series in
$L^\infty$. Multiplication from $L^1_{\hat f}\times L^\infty$ to
$L^1_{\hat f}$ and integration are bounded, so
\[
 (h,u)\longmapsto
 \int_M e^{-u}\sqrt{\det(I+\hat g^{-1}h)}\,\mathrm dV_{\hat f}
\]
is real analytic. Compose with $h\mapsto(h,u_h)$ and then with the
real logarithm near $Z(\hat g)=1$. Together with
\eqref{eq:normalized-analytic-branch}, this proves both analytic assertions.

For an original input, \cref{thm:finite-coordinate-slice} at $k=10$
gives $\tilde h\in E_9\cap\ker\hat D$. The inclusion
\eqref{eq:original-to-graph} uses only $E_8$, so this metric belongs
to the preceding analytic neighborhood. Pull the unique minimizer
back by the inverse diffeomorphism. By diffeomorphism invariance, change of variables preserves
the variational problem and proves existence and uniqueness on the
original metric.

For the dependence assertions, the pullback map
$(h,Z)\mapsto F_Z^*(\hat g+h)-\hat g$ is $C^1$ into $E_8$:
differentiate the finite coefficient and
remainder formulas in \cite[Proposition~6.4(ii)]{LYZ-I}, using
the parameter estimates of \cite[Lemma~6.3 and the proof of
Proposition~6.5]{LYZ-I}. Taylor's integral formula applies to each
composition with $F_Z$. The differentiated
composition uses one extra derivative of the field being composed;
$h\in E_{14}$ and the order-$10$ map estimates supply it, with the
output measured one level below $E_9$. The bounded product estimates
also control the derivative remainder in that norm. Thus the sliced
metric is $C^1$ into $E_8\cap\ker\hat D$, and hence into $\mathscr H$
by \eqref{eq:original-to-graph}. Composing with the analytic entropy proves
its $C^1$ dependence. The analytic scalar branch takes values in
$\mathscr Z\hookrightarrow C^{6,1/4}_1$, so
\cref{lem:finite-potential-pullback} gives continuous dependence of
the pulled-back minimizer in $C^{3,1/4}_1$.
\end{proof}

\begin{Lemma}\label{lem:variation-formulas}
    Let $h$ lie in the analytic neighborhood in $\mathscr H$ and let
    $k\in\mathscr H$. The gradient of $\boldsymbol{\mu}$ at $g$ is given by
    \begin{equation}
        \nabla\boldsymbol{\mu}(g)=\frac{1}{2}g-\Ric_g-\nabla^2_gf_g,
    \end{equation}
    where $f_g$ is the unique minimizer of $\mathcal W(g,\cdot)$.
    That is,
    \begin{equation*}
        \mathrm D\boldsymbol{\mu}_g[k]=\int_M\langle  \Phi(g),k  \rangle  \,\mathrm dV_{f_g},
    \end{equation*}
    where $k$ denotes a symmetric $(0,2)$-tensor.
    For simplicity, we write $f:=f_g$.
    Operators without specifying the metric are understood to be defined by $g$.
    Let $v$ be the variation of $f_g$ in $g$ along the direction $k$, i.e., $v:=\mathrm D f_g[k]$.
    Then $v$ is characterized by
    \begin{equation*}
        (2\Delta_f+1)\big(v-\frac12\tr_gk\big)
        +\delta_f^2k+\langle \Phi(g),k\rangle_g=\mathrm D\boldsymbol{\mu}_g[k],
    \end{equation*}
    together with the normalization condition
    \begin{equation*}
        \int_M  \big(\frac{1}{2}\tr_gk-v\big)  \mathrm dV_{f_g}=0.
    \end{equation*}
    The linearization of the shrinker defect is
    \begin{equation*}
        \mathrm D\Phi_g[k]=\frac{1}{2}\Delta_fk  +\Rm_g(k)  +\delta^*\delta_fk  +\nabla_g^2\big(\frac{1}{2}\tr_gk-v\big)  +\frac{1}{2}(\Phi(g)\circ k+k\circ \Phi(g)),
    \end{equation*}
    where $(\Phi(g)\circ k)_{ij}:=(\Phi(g))_{ip}k_{qj}g^{pq}$.
\end{Lemma}
\begin{proof}
The first variation is Perelman's formula
\cite[\S3]{perelman2002entropy}; see also
\cite[Lemma~2.1]{cao2012second}, with fixed scale $\tau=1$.
Differentiating
$2\Delta_gf-|\nabla_gf|^2+\mathrm R_g+f-n=\boldsymbol\mu(g)$
and $\int_M\mathrm dV_f=1$ gives the scalar equation and normalization.
The variations of $\Ric_g$ and $\nabla_g^2f$ are those used in
\cite[proof of Lemma~2.3]{cao2012second}. Substituting
$\Ric_g+\nabla_g^2f=\frac12g-\Phi(g)$ gives the displayed
linearization, including the two terms involving $\Phi(g)$.
All integrations by parts follow by exhaustion cutoffs, using
Gaussian decay and the polynomial bounds for the variations.
\end{proof}

\section{Analysis with weakened Gaussian weight}\label{sec:appendix-weakened-gaussian}
For $g=\hat g+h$ with $h$ small in $\mathscr H$,
\cref{thm:existence-and-uniqueness-of-minimizer} gives
$f_g-\hat f\in\mathscr Z$, with at most linear growth. For every fixed
$0<\kappa<1$, the measure $\mathrm dV_{f_g}$ is therefore bounded above
by a constant multiple of $\mathrm dV_{\hat f,\kappa}$, uniformly on
a small graph neighborhood. We use this fixed measure to establish the
graph estimates, inverse gauge operator, and nonlinear reconstruction
needed in Section~\ref{sec:finite-Banach-metric-LS}.

\subsection{Coercivity and weakened Gaussian estimates}
\phantomsection\label{idx:oscillator-weight}
Throughout this section, all unmarked derivatives and pointwise norms in
analytic estimates are taken with respect to $\hat g$.
We retain the weight $\boldsymbol{\rho}$ from
\eqref{eq:prelim-background-drift}. Thus $\boldsymbol{\rho}\asymp r$ on the end, $|\nabla\boldsymbol{\rho}|\leq C$, and
$|\nabla^j\boldsymbol{\rho}|\leq C_j\boldsymbol{\rho}^{1-j}$.  Bundle components on the end are
always taken in normalized AC frames.  The operators
$\partial_{\mathrm c,0}=r\partial_r$ and $\partial_{\mathrm c,i}=\partial_{\theta^i}$ act
on these components, where $(\theta^1,\ldots,\theta^{n-1})$ are local
coordinates on the link $\Sigma$. A finite atlas and partition of unity
on $\Sigma$ are understood.  On the compact core we use ordinary coordinate
 derivatives.  Different such choices give equivalent norms.

\phantomsection\label{idx:oscillator-domains}
We first establish a graph estimate using balls of radius comparable to
$\boldsymbol{\rho}^{-1}$.  Let $E$ be one
of the fixed tensor bundles, with its induced metric and connection, and put
\begin{equation*}
 \begin{gathered}
 H^1_{\boldsymbol{\rho}}=\{v\in H^1_{\mathrm{loc}}(M,E;\hat g):
                    \nabla v,\boldsymbol{\rho} v\in L^2(\mathrm dV_{\hat g})\},
 \qquad
 \|v\|_{H^1_{\boldsymbol{\rho}}}^2=\|\nabla v\|_2^2+\|\boldsymbol{\rho} v\|_2^2,\\
 H^2_{\boldsymbol{\rho}}=\{v\in H^2_{\mathrm{loc}}(M,E;\hat g):
          \nabla^2v,\boldsymbol{\rho}\nabla v,\boldsymbol{\rho}^2v\in L^2(\mathrm dV_{\hat g})\},
 \qquad
 \|v\|_{H^2_{\boldsymbol{\rho}}}=\|\nabla^2v\|_2+
                         \|\boldsymbol{\rho}\nabla v\|_2+\|\boldsymbol{\rho}^2v\|_2.
 \end{gathered}
\end{equation*}
Here and until a Gaussian measure is explicitly indicated, $\|\cdot\|_2$
and $\langle\cdot,\cdot\rangle$ denote the unweighted $L^2$ norm and
inner product.  Since $\boldsymbol{\rho}\geq1$, both displayed norms control the
$L^2$ norm.  Derivatives in these spaces are weak derivatives.
For a smooth differential operator $A$, recall its maximal domain
(cf.\ \cite[Section~2.1]{braverman2002essential}):
\begin{equation*}
 \Dom_{\max}(A)=\{v\in L^2:Av\in L^2\text{ distributionally}\}.
\end{equation*}
Its minimal domain consists of those $v\in L^2$ for which there are
$v_j\in C_c^\infty(M,E)$ such that $v_j\to v$ and $Av_j$ converges
in $L^2$.  The latter limit necessarily equals $Av$ distributionally.
Thus the minimal realization is the closure of $A|_{C_c^\infty}$ in
the graph norm $\|v\|_2+\|Av\|_2$.
For the related Gaussian realizations, see
\cite[Propositions~3.3(iii) and~3.4]{LYZ-I}.

\begin{Lemma}
\label{lem:appendix-elementary-oscillator}
Let $A$ be a smooth second-order operator on the tensor bundle $E$,
written using the fixed connection as
\[
 Av=-a^{ij}\nabla_i\nabla_jv+b^j\nabla_jv+qv.
\]
The coefficients are bundle endomorphisms. Assume $a^{ij}=a^{ji}$
and, for some $\lambda_0>0$,
\[
 \langle a^{ij}(x)\xi_i\xi_j e,e\rangle
 \geq\lambda_0|\xi|_{\hat g}^2|e|^2
 \qquad(x\in M,\ \xi\in T_x^*M,\ e\in E_x).
\]
Suppose that, for every integer $k\geq0$,
\[
 |\nabla^ka|\leq C_k\boldsymbol{\rho}^{-k},\qquad
 |\nabla^kb|\leq C_k\boldsymbol{\rho}^{1-k},\qquad
 |\nabla^kq|\leq C_k\boldsymbol{\rho}^{2-k}.
\]
All coefficient derivatives and norms use $\hat g$ and the induced
bundle connections. Assume that there are constants $C,C_0>0$ such that
\begin{equation}\label{ineq:appendix-basic-coercivity}
 \|\nabla v\|_2^2+\|\boldsymbol{\rho}v\|_2^2
 \leq C\bigl(\langle Av,v\rangle+C_0\|v\|_2^2\bigr)
\end{equation}
for every $v\in C_c^\infty(M,E)$. Then
\begin{equation}\label{ineq:appendix-elementary-graph}
 \|v\|_{H^2_{\boldsymbol{\rho}}}\leq C(\|Av\|_2+\|v\|_2)
\end{equation}
on the common domain
\[
 \Dom_{\min}(A)=\Dom_{\max}(A)=H^2_{\boldsymbol{\rho}}.
\]
The graph norm is equivalent to $\|\cdot\|_{H^2_{\boldsymbol{\rho}}}$ on this
domain. Smooth compactly supported sections form a core, and
$H^2_{\boldsymbol{\rho}}\hookrightarrow L^2$ is compact. The same conclusions
hold for the formal unweighted adjoint $A^*$.
The constants are uniform for families with common ellipticity,
coefficient, and basic coercivity bounds on the fixed background.
\end{Lemma}
\begin{proof}
We first derive the weighted estimates from
\cref{ineq:appendix-basic-coercivity}. Set
\[
 s=\boldsymbol{\rho}\quad\text{or}\quad
 s=s_T:=\frac{\boldsymbol{\rho}}{\sqrt{1+\boldsymbol{\rho}^2/T^2}},
 \qquad b_T=s_T^2,\qquad T\geq1.
\]
The bounds for $\boldsymbol{\rho}$ give, uniformly in $T$,
\[
 0<s\leq\boldsymbol{\rho},\qquad
 |\nabla s|\leq C s\boldsymbol{\rho}^{-1},\qquad
 |\nabla^2s|\leq C s\boldsymbol{\rho}^{-2}.
\]
Since $a^{ij}=a^{ji}$, the product rule gives
\[
 [A,s]v=-2a^{ij}(\nabla_i s)\nabla_jv
        -a^{ij}(\nabla_i\nabla_j s)v+b^j(\nabla_j s)v.
\]
Consequently, for $v\in C_c^\infty(M,E)$ and $\eta>0$,
\[
 |\langle[A,s]v,sv\rangle|
 \leq\eta\|s\nabla v\|_2^2
      +C_\eta\bigl(\|sv\|_2^2+\|v\|_2^2\bigr).
\]
Apply \cref{ineq:appendix-basic-coercivity} to $sv$, use
$\langle A(sv),sv\rangle=\langle Av,s^2v\rangle+
\langle[A,s]v,sv\rangle$ and
$\|s\nabla v\|_2^2\leq2\|\nabla(sv)\|_2^2+C\|v\|_2^2$,
and choose $\eta$ sufficiently small. After enlarging $C,C_0$, this gives
\begin{equation}\label{ineq:appendix-weighted-coercivity}
 \|s\nabla v\|_2^2+\|\boldsymbol{\rho}sv\|_2^2
 \leq C\bigl(\langle Av,s^2v\rangle
                  +C_0\|sv\|_2^2+C_0\|v\|_2^2\bigr).
\end{equation}
The choice $s=\boldsymbol{\rho}$ gives the weighted estimate, while
$s=s_T$ gives its truncated form with constants independent of $T$.

For $v\in C_c^\infty$, Cauchy's inequality gives
\begin{equation*}
 |\langle Av,\boldsymbol{\rho}^2v\rangle|
 \leq\eta\|\boldsymbol{\rho}^2v\|_2^2+C_\eta\|Av\|_2^2,
 \qquad
 \|\boldsymbol{\rho} v\|_2^2\leq\eta\|\boldsymbol{\rho}^2v\|_2^2+C_\eta\|v\|_2^2.
\end{equation*}
Insert these bounds into \cref{ineq:appendix-weighted-coercivity}
with $s=\boldsymbol{\rho}$ and choose $\eta$ sufficiently small to absorb the terms involving
$\|\boldsymbol{\rho}^2v\|_2^2$.  This proves
\begin{equation}\label{ineq:appendix-weighted-low-graph}
 \|\boldsymbol{\rho}\nabla v\|_2+\|\boldsymbol{\rho}^2v\|_2
 \leq C(\|Av\|_2+\|v\|_2).
\end{equation}

Fix $p\in M$ and put $R=\boldsymbol{\rho}(p)$.  Choose a fixed sufficiently small
$c_*>0$.  On $B_{2c_*/R}(p)$ the function $\boldsymbol{\rho}$ is comparable to $R$,
since $|\nabla\boldsymbol{\rho}|\leq C$.  Rescale the metric by $R^2$ and the
operator by $R^{-2}$.  In uniformly controlled coordinate and bundle
frames, the principal coefficients have size $O(1)$, and the
first-order and zeroth-order coefficients have sizes
$R^{-1}O(R)$ and $R^{-2}O(R^2)$, respectively.
A rescaled derivative contributes a factor $R^{-1}$.
The differentiated coefficient bounds therefore give uniform local
bounds, in particular uniform Lipschitz bounds for the principal
coefficients.  The AC geometry supplies uniformly controlled charts
on these balls; the compact core is covered by finitely many such
charts.  Uniform strong ellipticity is unchanged by this rescaling.
The ordinary interior $H^2$ estimate, rescaled back to $\hat g$, gives
\begin{equation}\label{ineq:appendix-shrinking-ball-H2}
 \|\nabla^2v\|_{L^2(B_{c_*/R}(p))}
 +R\|\nabla v\|_{L^2(B_{c_*/R}(p))}
 \leq C\bigl(\|Av\|_{L^2(B_{2c_*/R}(p))} 
             +R^2\|v\|_{L^2(B_{2c_*/R}(p))}\bigr).
\end{equation}
This local estimate also holds for every $H^2_{\mathrm{loc}}$ section.
A Vitali selection gives a cover by the smaller balls with uniformly
bounded overlap of the doubled balls.  Indeed, intersecting balls
have comparable radii, and the uniformly controlled rescaled geometry
bounds their packing number.  Squaring and summing gives
\begin{equation*}
 \|\nabla^2v\|_2^2+\|\boldsymbol{\rho}\nabla v\|_2^2
 \leq C\bigl(\|Av\|_2^2+\|\boldsymbol{\rho}^2v\|_2^2\bigr).
\end{equation*}
Together with \cref{ineq:appendix-weighted-low-graph}, this proves
\cref{ineq:appendix-elementary-graph} on $C_c^\infty$.

It is easy to see that $H^1_{\boldsymbol{\rho}}$ and $H^2_{\boldsymbol{\rho}}$ are complete.
Moreover, $C_c^\infty$ is dense in $H^1_{\boldsymbol{\rho}}$.
To see this, take $\chi_L=\chi(\boldsymbol{\rho}/L)$, where $\chi=1$ on $[0,1]$
and $\chi=0$ on $[2,\infty)$.  The tails of $\nabla v$ and $\boldsymbol{\rho} v$
tend to zero, while
$\|v\nabla\chi_L\|_2\leq CL^{-1}\|v\|_2$.
Thus $\chi_Lv\to v$ in $H^1_{\boldsymbol{\rho}}$, and local mollification finishes
the approximation.

Integrating the second-order term by parts once gives the form
\begin{equation*}
 \mathfrak a(u,v)
 ={}\int_M\langle a^{ij}\nabla_j u,\nabla_i v\rangle\,\mathrm dV_{\hat g} 
 +\int_M\langle(b^j+\nabla_i a^{ij})\nabla_j u,v\rangle\,\mathrm dV_{\hat g}
   +\int_M\langle qu,v\rangle\,\mathrm dV_{\hat g}.
\end{equation*}
For compactly supported smooth sections it equals $\langle Au,v\rangle$.
The assumptions give $|\nabla a|\leq C\boldsymbol{\rho}^{-1}$ and hence
$|b+\nabla a|\leq C\boldsymbol{\rho}$.  Consequently,
\begin{equation*}
 |\mathfrak a(u,v)|
 \leq C\bigl(\|\nabla u\|_2\|\nabla v\|_2
       +\|\nabla u\|_2\|\boldsymbol{\rho} v\|_2
       +\|\boldsymbol{\rho} u\|_2\|\boldsymbol{\rho} v\|_2\bigr) 
 \leq C\|u\|_{H^1_{\boldsymbol{\rho}}}\|v\|_{H^1_{\boldsymbol{\rho}}}.
\end{equation*}
Thus $\mathfrak a$ extends continuously to
$H^1_{\boldsymbol{\rho}}\times H^1_{\boldsymbol{\rho}}$.
By density, \cref{ineq:appendix-basic-coercivity} implies
\begin{equation*}
 \mathfrak a(v,v)\geq C^{-1}\|v\|_{H^1_{\boldsymbol{\rho}}}^2-C_0\|v\|_2^2
 \qquad(v\in H^1_{\boldsymbol{\rho}}).
\end{equation*}
Fix $\lambda>C_0$ large enough for both $A$ and $A^*$.
The shifted form
$\mathfrak a_\lambda(u,v)=\mathfrak a(u,v)+\lambda\langle u,v\rangle$
is bounded and coercive on $H^1_{\boldsymbol{\rho}}$.
Since $\boldsymbol{\rho}\geq1$, the functional $\phi\mapsto\langle F,\phi\rangle$
is continuous on $H^1_{\boldsymbol{\rho}}$ for every $F\in L^2$.
Lax--Milgram therefore gives a unique $w\in H^1_{\boldsymbol{\rho}}$ with
\begin{equation*}
 \mathfrak a(w,\phi)+\lambda\langle w,\phi\rangle
 =\langle F,\phi\rangle\qquad(\phi\in H^1_{\boldsymbol{\rho}}),
 \qquad \|w\|_{H^1_{\boldsymbol{\rho}}}\leq C\|F\|_2.
\end{equation*}
In particular, $(A+\lambda)w=F$ distributionally.
Local elliptic regularity gives $w\in H^2_{\mathrm{loc}}$.

The chosen truncations satisfy
\begin{equation*}
 0<b_T\leq\min\{\boldsymbol{\rho}^2,T^2\},\qquad
 b_T\uparrow\boldsymbol{\rho}^2,\qquad
 |\nabla b_T|\leq Cb_T\boldsymbol{\rho}^{-1},\qquad
 |\nabla^2b_T|\leq Cb_T\boldsymbol{\rho}^{-2}.
\end{equation*}
These bounds follow by differentiating the displayed formula for $b_T$;
their constants are independent of $T$.
For fixed $T$, multiplication by $b_T$ maps $H^1_{\boldsymbol{\rho}}$ continuously
to itself: $b_T\leq T^2$, and
$|\nabla b_T|\leq C T$ bounds the additional term
$w\nabla b_T$.  Hence $b_Tw$ is an admissible test in the weak equation.

The truncated inequality extends from $C_c^\infty$ to $H^1_{\boldsymbol{\rho}}$
for each fixed $T$.  In fact, if $w_j\to w$ in $H^1_{\boldsymbol{\rho}}$, then
$b_Tw_j\to b_Tw$ in $H^1_{\boldsymbol{\rho}}$, the form pairings converge, and
both weighted terms on the left converge since $b_T^{1/2}\leq T$.
Using the weak equation with test $b_Tw$ therefore gives
\begin{equation*}
 \|b_T^{1/2}\nabla w\|_2^2+\|\boldsymbol{\rho} b_T^{1/2}w\|_2^2
 \leq C\bigl(|\langle F,b_Tw\rangle| 
       +C_0\|b_T^{1/2}w\|_2^2+C_0\|w\|_2^2\bigr).
\end{equation*}
Here the term $-\lambda\|b_T^{1/2}w\|_2^2$ has a favorable sign
and was discarded.  Since $b_T\leq\boldsymbol{\rho}^2$,
\begin{equation*}
 |\langle F,b_Tw\rangle|
 \leq\|F\|_2\|\boldsymbol{\rho} b_T^{1/2}w\|_2
 \leq\eta\|\boldsymbol{\rho} b_T^{1/2}w\|_2^2+C_\eta\|F\|_2^2.
\end{equation*}
Also, uniformly in $T$,
\begin{equation*}
 \|b_T^{1/2}w\|_2^2
 \leq\eta\|\boldsymbol{\rho} b_T^{1/2}w\|_2^2+C_\eta\|w\|_2^2.
\end{equation*}
Indeed, on $\{\eta\boldsymbol{\rho}^2\geq1\}$ the first term suffices, and on
its complement $b_T\leq\boldsymbol{\rho}^2<\eta^{-1}$.
Choosing $\eta$ small and absorbing gives
\begin{equation*}
 \|b_T^{1/2}\nabla w\|_2^2+\|\boldsymbol{\rho} b_T^{1/2}w\|_2^2
 \leq C(\|F\|_2^2+\|w\|_2^2),
\end{equation*}
with $C$ independent of $T$.
Monotone convergence now gives $\boldsymbol{\rho}\nabla w,\boldsymbol{\rho}^2w\in L^2$.
Apply \cref{ineq:appendix-shrinking-ball-H2} locally to $w$ and sum
using the shrinking-ball cover.  Since $Aw=F-\lambda w\in L^2$, this yields
$\nabla^2w\in L^2$, so $w\in H^2_{\boldsymbol{\rho}}$ and
$\|w\|_{H^2_{\boldsymbol{\rho}}}\leq C_\lambda\|F\|_2$.

The coefficient bounds and strong ellipticity also hold for the formal
unweighted adjoint $A^*$.
Indeed, its first-order coefficients contain $b^*$ and $\nabla a^*$,
and its zeroth-order coefficients contain $q^*$, $\nabla b^*$,
and $\nabla^2a^*$; any curvature terms introduced by reordering
covariant derivatives satisfy the same bounds on the AC background.
For $v\in C_c^\infty(M,E)$,
$\langle A^*v,v\rangle=\langle Av,v\rangle$.
Hence \cref{ineq:appendix-basic-coercivity} holds for $A^*$, and the
multiplier argument gives \cref{ineq:appendix-weighted-coercivity}
for $A^*$ as well. The preceding form construction and weighted
regularity therefore apply to the adjoint.
We have proved that both $A+\lambda$ and $A^*+\lambda$ map
$H^2_{\boldsymbol{\rho}}$ onto $L^2$ by their form solutions.

The coefficient bounds directly imply
\begin{equation*}
 \|Av\|_2+\|v\|_2\leq C\|v\|_{H^2_{\boldsymbol{\rho}}}
 \qquad(v\in H^2_{\boldsymbol{\rho}}).
\end{equation*}
For the cutoffs $\chi_L=\chi(\boldsymbol{\rho}/L)$ used above, put
$\Omega_L=\{L\leq\boldsymbol{\rho}\leq2L\}$.
On $\Omega_L$ we have $|\nabla\chi_L|\leq CL^{-1}$ and
$|\nabla^2\chi_L|\leq CL^{-2}$.
In differentiating $\chi_Lv-v$, the terms where no derivative meets
the cutoff are tails of $\nabla^2v$, $\boldsymbol{\rho}\nabla v$, or $\boldsymbol{\rho}^2v$.
The remaining terms satisfy
\begin{equation*}
 \begin{split}
 \|\boldsymbol{\rho} v\nabla\chi_L\|_2
 &\leq C\|v\|_{L^2(\Omega_L)},\\
 \|\nabla\chi_L\otimes\nabla v\|_2
       +\|v\nabla^2\chi_L\|_2
 &\leq CL^{-1}\|\nabla v\|_{L^2(\Omega_L)}
          +CL^{-2}\|v\|_{L^2(\Omega_L)}.
 \end{split}
\end{equation*}
They tend to zero, so $\chi_Lv\to v$ in $H^2_{\boldsymbol{\rho}}$.
On a fixed compact set the weights are bounded, and ordinary local
mollification in bundle charts gives approximation in $H^2$, hence
in $H^2_{\boldsymbol{\rho}}$.  A diagonal choice gives $v_j\in C_c^\infty$
with $v_j\to v$ in $H^2_{\boldsymbol{\rho}}$, and the preceding coefficient
estimate gives convergence in the graph norm.  Thus
$H^2_{\boldsymbol{\rho}}\subset\Dom_{\min}(A)$.

Conversely, a sequence defining an element of $\Dom_{\min}(A)$
is Cauchy in $H^2_{\boldsymbol{\rho}}$ by
\cref{ineq:appendix-elementary-graph} applied to differences.
Completeness and the $L^2$ convergence identify its limit as an
element of $H^2_{\boldsymbol{\rho}}$.  Therefore
$\Dom_{\min}(A)=H^2_{\boldsymbol{\rho}}$, and
\cref{ineq:appendix-elementary-graph} extends to this domain.
The reverse coefficient estimate proves the asserted norm equivalence.
Exactly the same argument applies to $A^*$.

Let $v\in\Dom_{\max}(A)$ and set $F=(A+\lambda)v\in L^2$.
The form construction gives a solution $w\in H^2_{\boldsymbol{\rho}}$ of
$(A+\lambda)w=F$.  Then $h=v-w\in L^2$ and
$(A+\lambda)h=0$ distributionally.  At this point $h$ is not known
to lie in $H^1_{\boldsymbol{\rho}}$, so the basic coercivity inequality cannot
be applied directly to $h$.
Instead solve the adjoint equation
\begin{equation*}
 (A^*+\lambda)z=h,\qquad z\in H^2_{\boldsymbol{\rho}},
\end{equation*}
by the adjoint form construction and weighted regularity.
By the core property for $A^*$ there are $z_j\in C_c^\infty$ such that
$z_j\to z$ and $(A^*+\lambda)z_j\to h$ in $L^2$.
The distributional equation for $h$ gives
\begin{equation*}
 0=\langle h,(A^*+\lambda)z_j\rangle
       \longrightarrow\langle h,h\rangle.
\end{equation*}
Hence $h=0$ and $v=w\in H^2_{\boldsymbol{\rho}}$.
Together with the minimal-domain identity and the inclusion
$\Dom_{\min}(A)\subset\Dom_{\max}(A)$, this proves the domain
identity.  Interchanging $A$ and $A^*$ proves the adjoint identity.

Let $(v_j)$ be bounded in $H^2_{\boldsymbol{\rho}}$.
On every fixed compact set it is bounded in $H^2$, so a diagonal
application of Rellich compactness gives a subsequence convergent
locally in $L^2$.  Uniformly in $j$,
\begin{equation*}
 \|v_j\|_{L^2(\{\boldsymbol{\rho}>L\})}
 \leq L^{-2}\|\boldsymbol{\rho}^2v_j\|_2.
\end{equation*}
These tails tend to zero uniformly as $L\to\infty$.
The locally convergent subsequence is therefore Cauchy in global
$L^2$, proving compactness.
\end{proof}

We use the probability measures $\mathrm dV_{\hat f,\kappa}$ and spaces
$L^2_\kappa$ fixed before \cref{sec:metric-LS-auxiliary}, with
$0\leq\kappa<1$. The AC growth of $\hat f$ gives $Z_\kappa<\infty$,
and $Z_0=1$. We write $L^2_\kappa(E)$ when specifying the tensor
bundle; all fiber norms use $\hat g$. For $\kappa>0$, the slower decay of
the density gives a continuous inclusion
$L^2_\kappa(E)\subset L^2_0(E)$.

The Gaussian conjugation is the surjective isometry
\begin{equation*}
 \begin{split}
 U_\kappa:L^2_\kappa(E)&\longrightarrow L^2(M,E;\mathrm dV_{\hat g}),\\
 U_\kappa u
 &=\left(\frac{\mathrm dV_{\hat f,\kappa}}{\mathrm dV_{\hat g}}\right)^{1/2}u
   =(4\pi)^{-n/4}Z_\kappa^{-1/2}
           e^{-(1-\kappa)\hat f/2}u.
 \end{split}
\end{equation*}
Indeed, multiplication by the square root of the density preserves the
inner product and has inverse multiplication by its reciprocal. The normalizing constant cancels
in conjugated differential operators.

Recall $\hat D$ from \cref{eq:prelim-metric-abbreviations}, and $\hat L$
and the infinitesimal solitonic deformation space $\ISD$ from
\cref{eq:prelim-ISD}.
If a nearby normalized density obeys
\begin{equation*}
 \left|\log\frac{\mathrm dV_{f_g}}{\mathrm dV_{\hat f}}\right|
 \leq\epsilon(1+\hat f)+C,
 \qquad \epsilon<\frac\kappa4,
\end{equation*}
then
\begin{equation}\label{ineq:moving-density-dominated-by-weaker-gaussian}
 \mathrm dV_{f_g}\leq C_\kappa \mathrm dV_{\hat f,\kappa}.
\end{equation}

\begin{Lemma}
\label{lem:weakened-gaussian-oscillator-package}
\label{prop:weakened-gaussian-closed-oscillator-realization}
\label{lem:weakened-gaussian-normal-gap}
\label{lem:weakened-gaussian-oscillator-graph}
\label{lem:weakened-gaussian-scalar-response}
Let $(M,\hat g,\hat f)$ be the fixed complete nontrivial AC shrinker.
For $E=S^2T^*M$ or the scalar bundle, set
\begin{equation}\label{eq:weakened-Gaussian-closed-domain}
 \begin{gathered}
 \|u\|_{H^2_{\boldsymbol{\rho},\kappa}}
 :=\|\hat\nabla^2u\|_{L^2_\kappa}
   +\|\boldsymbol{\rho}\hat\nabla u\|_{L^2_\kappa}
   +\|\boldsymbol{\rho}^2 u\|_{L^2_\kappa},\\
 H^2_{\boldsymbol{\rho},\kappa}(E)
 :=\{u\in H^2_{\mathrm{loc}}(M,E;\hat g):
                  \|u\|_{H^2_{\boldsymbol{\rho},\kappa}}<\infty\}.
 \end{gathered}
\end{equation}
Let $\mathcal L$ denote $\hat L$ on symmetric tensors or
$2\hat\Delta_{\hat f}+1$ on functions.
For every fixed $0\leq\kappa<1$, its minimal and maximal
$L^2_\kappa$ realizations have domain $H^2_{\boldsymbol{\rho},\kappa}$, with
$C_c^\infty$ as a core, compact inclusion into $L^2_\kappa$, and
\begin{equation}\label{ineq:closed-weakened-Gaussian-oscillator-graph}
 \|u\|_{H^2_{\boldsymbol{\rho},\kappa}}
 \leq C_\kappa\bigl(\|\mathcal Lu\|_{L^2_\kappa}
                         +\|u\|_{L^2_\kappa}\bigr).
\end{equation}

For tensors satisfying $\hat Dk=0$ and
$k\perp_{L^2_{\hat f}}\ISD$, and for arbitrary scalars $u$, one has
\[
 \|k\|_{H^2_{\boldsymbol\rho,\kappa}}
       \le C_\kappa\|\hat Lk\|_{L^2_\kappa},\qquad
 \|u\|_{H^2_{\boldsymbol\rho,\kappa}}
       \le C_\kappa\|(2\hat\Delta_{\hat f}+1)u\|_{L^2_\kappa}.
\]
For the scalar operator, suppose
\[
 \mathcal Eu=a_2*\hat\nabla^2u+a_1*\hat\nabla u+a_0u,\qquad
 \sup_M\bigl(|a_2|+\boldsymbol\rho^{-1}|a_1|
                       +\boldsymbol\rho^{-2}|a_0|\bigr)\le\epsilon.
\]
If $\epsilon$ is sufficiently small in terms of $\kappa$, then
\begin{equation}\label{ineq:background-scalar-response}
 \|u\|_{H^2_{\boldsymbol\rho,\kappa}}
 \le 2C_\kappa\|(2\hat\Delta_{\hat f}+1+\mathcal E)u\|_{L^2_\kappa}.
\end{equation}
All fields lie in the displayed graph domains. The constants depend
only on $\kappa$ and the fixed background, uniformly over these scalar
perturbations.
\end{Lemma}

\begin{proof}
Write $\mathcal L=c\hat\Delta_{\hat f}+V$, with
$(c,V)=(1/2,\Rm_{\hat g})$
or $(2,1)$.
Conjugation by $U_\kappa$ gives
\begin{equation}\label{eq:weakened-Gaussian-common-oscillator}
 U_\kappa(-\mathcal L)U_\kappa^{-1}
 =-c\hat\Delta+c\kappa\hat\nabla_{\mathbf{X}}
   +\frac{c(1-\kappa^2)}4|\mathbf{X}|^2
   -\frac{c(1-\kappa)}2\hat\Delta\hat f-V.
\end{equation}
The coefficients in this expression satisfy the explicit assumptions
of \cref{lem:appendix-elementary-oscillator}.
Indeed, for every $k\geq0$ the AC identities give
\begin{equation*}
 |\nabla^k\mathbf{X}|\leq C_k\boldsymbol{\rho}^{1-k},\qquad
 |\nabla^k\hat\Delta\hat f|+|\nabla^kV|\leq C_k\boldsymbol{\rho}^{-k},
\end{equation*}
and $|\mathbf{X}|\asymp\boldsymbol{\rho}$ outside a compact set.
The principal coefficient is parallel, the drift coefficient is
$c\kappa \mathbf{X}$, and the zeroth-order coefficient is
\begin{equation*}
 Q_\kappa=\frac{c(1-\kappa^2)}4|\mathbf{X}|^2
           -\frac{c(1-\kappa)}2\hat\Delta\hat f-V.
\end{equation*}
Thus their $k$th derivatives have orders
$O(\boldsymbol{\rho}^{-k})$, $O(\boldsymbol{\rho}^{1-k})$, and $O(\boldsymbol{\rho}^{2-k})$,
respectively.  Write the conjugated operator as
$\widetilde L_\kappa=-c\hat\Delta+c\kappa\hat\nabla_{\mathbf{X}}+Q_\kappa$.
For compactly supported $v$, integration of the drift gives
\begin{equation*}
 \langle \widetilde L_\kappa v,v\rangle
 =c\|\nabla v\|_2^2+
 \int_M\left\langle\left(
    \frac{c(1-\kappa^2)}4|\mathbf{X}|^2
       -\frac c2\hat\Delta\hat f-V\right)v,v\right\rangle \mathrm dV_{\hat g}.
\end{equation*}
For each fixed $0\leq\kappa<1$, the potential on the right is bounded
below by $c_\kappa\boldsymbol{\rho}^2-C_\kappa$.
This proves \cref{ineq:appendix-basic-coercivity} with suitable
constants $C,C_0$, completing the hypotheses of
\cref{lem:appendix-elementary-oscillator}.
That lemma proves the graph estimate and identifies the
minimal and maximal domains with $H^2_{\boldsymbol{\rho}}$ after conjugation.

The product rule, $|\mathbf{X}|\leq C\boldsymbol{\rho}$, and
$|\nabla \mathbf{X}|\leq C$ show, in both directions, that
\begin{equation*}
 \|U_\kappa u\|_{H^2_{\boldsymbol{\rho}}}\asymp\|u\|_{H^2_{\boldsymbol{\rho},\kappa}}.
\end{equation*}
Thus the core and compactness assertions follow as well.  At $\kappa=0$
this is the self-adjoint realization, with conjugated form domain
$H^1_{\boldsymbol{\rho}}$.

By \cref{prop:finite-graph-realization} and its preceding weighted
estimates, every element of $\ISD$ has polynomial bounds at all
derivative orders. Thus $\ISD\subset H^2_{\boldsymbol\rho,\kappa}$ for
every $0\le\kappa<1$. Conversely, the continuous inclusion
$L^2_\kappa\hookrightarrow L^2_{\hat f}$ and maximal-domain regularity
identify the tensor kernel on $\ker\hat D$ with $\ISD$.
The scalar operator $2\hat\Delta_{\hat f}+1=2L_{1/2}$ has zero
$L^2_{\hat f}$ kernel by \cite[Section~5.2]{LYZ-I}, as used in
\eqref{eq:finite-scalar-graphs}. The same inclusion and domain
regularity exclude a weakened-Gaussian scalar kernel.

Now a failure of either low-norm gap gives elements of unit
$L^2_\kappa$ norm whose images tend to zero.  The graph estimate
and compact embedding give a strongly convergent subsequence in
$L^2_\kappa$, with a nonzero limit in the corresponding kernel.
The condition $\hat Dk=0$ passes distributionally, and background
$L^2_{\hat f}$ orthogonality to $\ISD$ is closed in $L^2_\kappa$.
Thus the tensor limit is orthogonal to $\ISD$, while the scalar
kernel is zero, a contradiction in both cases.  Substitution in
\cref{ineq:closed-weakened-Gaussian-oscillator-graph} gives the following
estimate for tensors satisfying the stated conditions and for arbitrary scalars:
\begin{equation}\label{ineq:weakened-gaussian-normal-gap}
 \|u\|_{H^2_{\boldsymbol{\rho},\kappa}}
 \leq C_\kappa\|\mathcal Lu\|_{L^2_\kappa}.
\end{equation}
Nontriviality is used here for the scalar conclusion.

For scalars, the coefficient bound and \cref{ineq:weakened-gaussian-normal-gap} give
\begin{equation}\label{ineq:oscillator-relative-smallness}
 \|\mathcal Eu\|_{L^2_\kappa}
 \leq C_*\epsilon\|u\|_{H^2_{\boldsymbol{\rho},\kappa}}
 \leq C_*C_\kappa\epsilon\|\mathcal Lu\|_{L^2_\kappa}.
\end{equation}
Consequently,
\begin{equation*}
 \|u\|_{H^2_{\boldsymbol{\rho},\kappa}}
 \leq C_\kappa\|(\mathcal L+\mathcal E)u\|_{L^2_\kappa}
       +C_\kappa C_*\epsilon\|u\|_{H^2_{\boldsymbol{\rho},\kappa}}.
\end{equation*}
Choosing $C_\kappa C_*\epsilon\leq1/2$ proves
\cref{ineq:background-scalar-response}. The scalar graph norms are
equivalent, so the perturbed operator is closed on the same domain.
\end{proof}

\subsection{Nonlinear coefficient estimates}

Recall that $\mathscr H$ is the fixed Banach graph contained in the
background weighted divergence free slice and that $\Phi(g)=\frac12g-\Ric_g-\nabla_g^2f_g$.
The derivative $\mathrm D\Phi_g[k]$ includes the variation of the minimizer $f_g$.

\begin{Proposition}
\label{prop:nonlinear-normal-coefficient-reduction}
Fix $0<\kappa<1$.
For any $\epsilon>0$, there exists an open neighborhood
$U_1\subset\mathscr H$ of the origin such that
\begin{equation}\label{ineq:unprojected-nonlinear-normal-graph-smallness}
 \|\mathrm D\Phi_g[k]-\hat Lk\|_{L^2_\kappa}
 \leq C_\kappa\epsilon\|\hat Lk\|_{L^2_\kappa}
\end{equation}
for all $g$ satisfying $g-\hat g\in U_1$ and all
$k\in\mathscr H'$.
Here orthogonality is taken in $L^2_0=L^2(\mathrm dV_{\hat f})$.
All Hilbert norms use the fixed background metric and measure.
\end{Proposition}

\begin{proof}
Put $f=f_g$ and recall that $D_g=\delta_f^g$.
We use $D_g$ on both symmetric tensors and one-forms; thus
$D_g^2k=D_g(D_gk)$ is a scalar.
The constants below are uniform after choosing one sufficiently small
neighborhood of the origin.

Since $k\in\mathscr H'$, its polynomial growth and that
of its derivatives give $k\in H^2_{\boldsymbol{\rho},\kappa}$.
Applying the tensor gap \cref{ineq:weakened-gaussian-normal-gap}, we obtain
\begin{equation*}
 \hat Dk=0,\qquad
 \|k\|_{H^2_{\boldsymbol{\rho},\kappa}}
 \leq C_\kappa\|\hat Lk\|_{L^2_\kappa}.
\end{equation*}
For any operator
$\mathcal E=E_2*\hat\nabla^2+E_1*\hat\nabla+E_0$ satisfying
\begin{equation*}
 \sup_M\bigl(|E_2|+\boldsymbol{\rho}^{-1}|E_1|+\boldsymbol{\rho}^{-2}|E_0|\bigr)
 \leq C\epsilon,
\end{equation*}
we consequently have
\begin{equation*}
 \|\mathcal Ek\|_{L^2_\kappa}
 \leq C\epsilon\|k\|_{H^2_{\boldsymbol{\rho},\kappa}}
 \leq C_\kappa\epsilon\|\hat Lk\|_{L^2_\kappa}.
\end{equation*}

We apply this observation to the tensor-valued operators
\begin{equation*}
 \mathcal T_g=\frac12\Delta_f^g+\Rm_g-\hat L,
 \qquad \delta_g^*(D_g-\hat D),
\end{equation*}
and the scalar-valued operator $D_g(D_g-\hat D)$.
Their coefficients vanish at $g=\hat g$.
The inverse-metric and connection estimates in \cref{sec:computation-of-deformation},
\cref{eq:finite-low-coefficients}, and the analytic
dependence of $f_g$ in
\cref{thm:existence-and-uniqueness-of-minimizer} imply the preceding
weighted coefficient bounds after shrinking $U_1$.
In particular, when $D_g$ or $\delta_g^*$ differentiates a coefficient
of $D_g-\hat D$, its derivative is controlled by the differentiated
coefficient estimates.
For $\mathcal T_g$, write the two connections in the fixed background
frame.  The principal coefficient is the difference of inverse
metrics; the remaining coefficients are formed from the connection
difference, its first derivative, the drift difference, and the
curvature difference.  Their orders are bounded by
$\epsilon$, $\epsilon\boldsymbol{\rho}$, and $\epsilon\boldsymbol{\rho}^2$ in differential
orders two, one, and zero.

We also shrink $U_1$ so that
\begin{equation*}
 \|\Phi(g)\|_{C^0}\leq C\epsilon,
\end{equation*}
which follows from $\Phi(\hat g)=0$ and the same coefficient estimates.
The scalar operator
\begin{equation*}
 2\Delta_f^g+1
   =(2\hat\Delta_{\hat f}+1)+2(\Delta_f^g-\hat\Delta_{\hat f}),
\end{equation*}
is a small perturbation in the coefficient norm of
\cref{lem:weakened-gaussian-oscillator-package}.
Here the neighborhood is chosen small enough in terms of $\kappa$
to meet the scalar absorption threshold, as well as the requested
$\epsilon$ bounds.

We have
\begin{equation*}
 D_gk=(D_g-\hat D)k,\qquad
 D_g^2k=D_g(D_g-\hat D)k,\qquad
 \delta_g^*D_gk=\delta_g^*(D_g-\hat D)k.
\end{equation*}
The preceding coefficient bounds imply
\begin{equation*}
 \|D_g^2k\|_{L^2_\kappa}
       +\|\delta_g^*D_gk\|_{L^2_\kappa}
 \leq C_\kappa\epsilon\|\hat Lk\|_{L^2_\kappa}.
\end{equation*}

Set
\begin{equation*}
 v=\mathrm D f_g[k],\qquad t=\tr_gk,\qquad
 \boldsymbol{\mu}'=\mathrm D\boldsymbol{\mu}_g[k],\qquad y=\frac12t-v+\boldsymbol{\mu}'.
\end{equation*}
The number $\boldsymbol{\mu}'$ is constant as a function on $M$.
Differentiate the Euler--Lagrange equation
\begin{equation*}
 2\Delta_gf-|\nabla_gf|^2+\mathrm{R}_g+f-n=\boldsymbol{\mu}(g)
\end{equation*}
along $g+sk$ and the corresponding minimizers.
The metric and potential variation formulas give
\begin{equation*}
 \boldsymbol{\mu}'
 =(2\Delta_f^g+1)v-\Delta_f^g t
       +D_g^2k-\langle\Ric_g+\nabla_g^2f,k\rangle_g 
 =(2\Delta_f^g+1)\left(v-\frac12t\right)
       +D_g^2k+\langle \Phi(g),k\rangle_g.
\end{equation*}
For the second equality, use
$\Ric_g+\nabla_g^2f=\frac12g-\Phi(g)$.
Since $(2\Delta_f^g+1)\boldsymbol{\mu}'=\boldsymbol{\mu}'$, the definition of $y$ gives
\begin{equation*}
 (2\Delta_f^g+1)y=D_g^2k+\langle \Phi(g),k\rangle_g,
 \qquad
 \nabla_g^2y=\nabla_g^2\left(\frac12\tr_gk-v\right).
\end{equation*}
The analytic dependence of $f_g$ gives
$v=\mathrm D f_g[k]\in\mathscr Z$.
Thus $v$, $t$, and their derivatives have polynomial growth.
Since $\boldsymbol{\mu}'$ is constant and $\kappa<1$, we have
$y\in H^2_{\boldsymbol{\rho},\kappa}$.

By the coefficient bounds and uniform equivalence of $g$ and $\hat g$,
\begin{equation*}
 \|\langle \Phi(g),k\rangle_g\|_{L^2_\kappa}
 \leq C\epsilon\|k\|_{L^2_\kappa}
 \leq C_\kappa\epsilon\|\hat Lk\|_{L^2_\kappa}.
\end{equation*}
Together with the estimate for $D_g^2k$, this bounds the right-hand
side of the scalar equation.  The scalar estimate
\cref{ineq:background-scalar-response}, applied to $(2\Delta_f^g+1)$,
now gives
\begin{equation*}
 \|y\|_{H^2_{\boldsymbol{\rho},\kappa}}
 \leq C_\kappa\|(2\Delta_f^g+1)y\|_{L^2_\kappa}
 \leq C_\kappa\epsilon\|\hat Lk\|_{L^2_\kappa}.
\end{equation*}
In particular, the connection-difference bound and the formula
$\nabla_g^2y=\hat\nabla^2y-(\nabla^g-\hat\nabla)*\hat\nabla y$ give
\begin{equation*}
 \|\nabla_g^2y\|_{L^2_\kappa}
 \leq C\|y\|_{H^2_{\boldsymbol{\rho},\kappa}}
 \leq C_\kappa\epsilon\|\hat Lk\|_{L^2_\kappa}.
\end{equation*}

Differentiating $\Phi(g)=\frac12g-\Ric_g-\nabla_g^2f_g$ gives
\begin{equation*}
 \mathrm D\Phi_g[k]={}\frac12\Delta_f^gk+\Rm_g(k)
       +\delta_g^*D_gk
       +\nabla_g^2\left(\frac12\tr_gk-v\right) 
 +\frac12\bigl(\Phi(g)\circ_g k+k\circ_g \Phi(g)\bigr).
\end{equation*}
Here $\circ_g$ denotes contraction of one index using $g$.
Using the identity for $\nabla_g^2y$ and subtracting $\hat Lk$ yields
\begin{equation*}
 \mathrm D\Phi_g[k]-\hat Lk
 ={}\mathcal T_gk+\delta_g^*D_gk+\nabla_g^2y 
    +\frac12\bigl(\Phi(g)\circ_g k+k\circ_g \Phi(g)\bigr).
\end{equation*}
The coefficient, divergence, and Hessian estimates above control
the first three terms.
For the last term, the $C^0$ bound for $\Phi(g)$ gives
\begin{equation*}
 \|\Phi(g)\circ_g k+k\circ_g \Phi(g)\|_{L^2_\kappa}
 \leq C\epsilon\|k\|_{L^2_\kappa}
 \leq C_\kappa\epsilon\|\hat Lk\|_{L^2_\kappa}.
\end{equation*}
Adding these four estimates proves
\cref{ineq:unprojected-nonlinear-normal-graph-smallness}.
\end{proof}

\subsection{The inverse gauge operator and orthogonal projection}\label{subsec:direct-AC-fixed-Bianchi-inverse}
The aim of this subsection is to study the $L^2(\mathrm dV_{\hat f,\kappa})$-orthogonal projection of a symmetric $(0,2)$-tensor to its weighted divergence free component, which will be used in \cref{sec:finite-Banach-metric-LS} for the Lyapunov--Schmidt reduction.

\phantomsection\label{idx:bianchi-operators}
\begin{Definition}
    Fix $0<\kappa<1$. Define $B_{\kappa}$ to be the $L^2(\mathrm dV_{\hat f,\kappa})$-adjoint of $\hat D$, i.e.,
    \begin{equation*}
        \langle  \hat Du,v  \rangle_{L^2(\mathrm dV_{\hat f,\kappa})}
        =\langle  u,B_{\kappa}v  \rangle_{L^2(\mathrm dV_{\hat f,\kappa})}
    \end{equation*}
    for all $u\in C^{\infty}_{c}(M,S^2T^*M)$ and $v\in C^{\infty}_{c}(M,T^*M)$.
    Set $P_{\kappa}:=\hat DB_{\kappa}$.
    Denote $\mathcal{K}_\kappa:=\Ker(B_{\kappa}\restriction L^2(\mathrm dV_{\hat f,\kappa})(M,T^*M))$ and $\mathcal{K}_\kappa^\perp$ its $L^2(\mathrm dV_{\hat f,\kappa})$-orthogonal complement.
\end{Definition}
Integration by parts with respect to
$\mathrm dV_{\hat f,\kappa}$ gives
\begin{equation}\label{eq:weakened-gaussian-bianchi-adjoint}
 B_\kappa\zeta=\hat\delta^*\zeta
       -\kappa\operatorname{Sym}(\mathrm d\hat f\otimes\zeta).
\end{equation}
Indeed, the adjoint of $\hat\delta$ contributes
$(1-\kappa)\operatorname{Sym}(\mathrm d\hat f\otimes\zeta)$,
whereas the contraction term in $\hat D$ contributes its negative
with coefficient one.

Recall the Gaussian isometry $U_\kappa$ from Subsection~4.1.  For a
tensor bundle $E$ and an integer $j\geq0$, put
\begin{equation*}
 \mathcal S^j_{\mathrm{osc}}(E)
 :=\{v:\boldsymbol\rho^a\hat\nabla^b v\in L^2(\mathrm dV_{\hat g})
                         \text{ for all }a+b\leq j\},
 \qquad
 \|v\|_{\mathcal S^j_{\mathrm{osc}}}
 :=\sum_{a+b\leq j}\|\boldsymbol\rho^a\hat\nabla^b v\|_2.
\end{equation*}
Here $a,b$ are nonnegative integers and derivatives are weak derivatives.
The norms and contractions use $\hat g$.

\begin{Lemma}
\label{lem:weakened-Gaussian-weighted-Sobolev}
\label{lem:weakened-Gaussian-Sobolev-conjugation}
Fix $0\leq\kappa<1$ and a tensor bundle $E$ with the connection induced
by $\hat g$. For every $u\in H^1_\kappa(M,E)$,
\begin{equation}\label{ineq:weakened-Gaussian-weighted-Sobolev}
 \|\boldsymbol\rho u\|_{L^2_\kappa}
 \leq C_\kappa\bigl(\|\hat\nabla u\|_{L^2_\kappa}
                         +\|u\|_{L^2_\kappa}\bigr).
\end{equation}
More generally, for every integer $j\geq0$ and $u\in H^j_\kappa(M,E)$,
\begin{equation}\label{ineq:weakened-Gaussian-iterated-Sobolev}
 \sum_{a+b\leq j}\|\boldsymbol\rho^a\hat\nabla^b u\|_{L^2_\kappa}
 \leq C_{j,\kappa}\|u\|_{H^j_\kappa},
\end{equation}
where $a,b$ are nonnegative integers. In particular,
$H^2_\kappa=H^2_{\boldsymbol\rho,\kappa}$, with equivalent norms.

For every integer $j\geq0$, multiplication by $U_\kappa$ restricts to
an isomorphism
\begin{equation*}
 U_\kappa:H^j_\kappa(M,E)\longrightarrow\mathcal S^j_{\mathrm{osc}}(E),
 \qquad
 \|U_\kappa u\|_{\mathcal S^j_{\mathrm{osc}}}
       \asymp_{j,\kappa}\|u\|_{H^j_\kappa}.
\end{equation*}
Moreover, $\mathcal S^1_{\mathrm{osc}}=H^1_{\boldsymbol\rho}$,
$\mathcal S^2_{\mathrm{osc}}=H^2_{\boldsymbol\rho}$, with equivalent norms.
\end{Lemma}
\begin{proof}
For $u\in C_c^\infty(M,E)$, integration of the divergence of
$e^{-(1-\kappa)\hat f}|u|^2\mathbf X$ gives
\begin{align*}
 (1-\kappa)\||\mathbf X|u\|_{L^2_\kappa}^2
 &=\int_M(\hat\Delta\hat f)|u|^2\,\mathrm dV_{\hat f,\kappa}
   +2\int_M\langle\hat\nabla_{\mathbf X}u,u\rangle\,
                                      \mathrm dV_{\hat f,\kappa}\\
 &\leq C\|u\|_{L^2_\kappa}^2
   +2\||\mathbf X|u\|_{L^2_\kappa}\|\hat\nabla u\|_{L^2_\kappa}.
\end{align*}
Here $\hat\Delta\hat f$ is bounded. Young's inequality and
$|\mathbf X|^2\geq c\boldsymbol\rho^2-C$ prove
\cref{ineq:weakened-Gaussian-weighted-Sobolev} on compact support.
Applying this estimate successively to
$\boldsymbol\rho^a\hat\nabla^b u$, and using
$|\hat\nabla\boldsymbol\rho^a|\leq C_a\boldsymbol\rho^{a-1}$,
proves \cref{ineq:weakened-Gaussian-iterated-Sobolev} by induction.
Only derivatives of total order at most $j$ occur.
Compactly supported smooth tensors are dense in each $H^j_\kappa$:
use exhaustion cutoffs with bounded covariant derivatives and then
smooth on compact subsets. Applying the estimates to differences of
such approximations extends both inequalities to the stated spaces.
At $j=2$, these estimates control the additional terms
$\boldsymbol\rho\hat\nabla u$ and $\boldsymbol\rho^2u$ in the
oscillator norm. The reverse norm bound follows from
$\boldsymbol\rho\geq1$, proving the asserted equality of spaces.

Set $a_\kappa=(1-\kappa)/2$.  The product rule gives
\begin{equation*}
 \hat\nabla(U_\kappa u)
 =U_\kappa(\hat\nabla u-a_\kappa\mathrm d\hat f\otimes u).
\end{equation*}
Its iterates and the background derivative bounds on $\hat f$ imply,
for $u\in C_c^\infty(M,E)$,
\begin{equation*}
 \|U_\kappa u\|_{\mathcal S^j_{\mathrm{osc}}}
 \leq C_{j,\kappa}\sum_{a+b\leq j}
                  \|\boldsymbol\rho^a\hat\nabla^b u\|_{L^2_\kappa}
 \leq C_{j,\kappa}\|u\|_{H^j_\kappa},
\end{equation*}
by the moment estimate just proved.
Applying the corresponding product identities to $U_\kappa^{-1}v$
gives the reverse bound
$\|U_\kappa^{-1}v\|_{H^j_\kappa}
\leq C_{j,\kappa}\|v\|_{\mathcal S^j_{\mathrm{osc}}}$.
Cutoffs and local smoothing give density of compactly supported smooth
sections in both spaces, so these mutually inverse maps extend to the
stated spaces. The two norm identifications follow from
$\boldsymbol\rho\geq1$ and the definitions.
\end{proof}

\begin{Proposition}\label{lem:weakened-Gaussian-Bianchi-Hilbert-inverse}
    Fix $0<\kappa<1$ and an integer $m\geq0$.  The operator $P_{\kappa}$ restricts to an isomorphism
    \begin{equation*}
        P_{\kappa}:H^{m+2}_\kappa(M,T^*M)\cap\mathcal{K}_{\kappa}^{\perp}
        \to H^{m}_\kappa(M,T^*M)\cap\mathcal{K}_{\kappa}^{\perp}.
    \end{equation*}
    Moreover, for $\zeta\in H^1_\kappa(M,T^*M)\cap\mathcal K_\kappa^\perp$,
    \begin{equation}\label{ineq:appendix-bianchi-form-gap}
     \|\zeta\|_{H^1_\kappa}
       +\|\boldsymbol{\rho}\zeta\|_{L^2_\kappa}
     \leq C_\kappa\|B_\kappa\zeta\|_{L^2_\kappa}.
    \end{equation}
\end{Proposition}

\begin{proof}
Put $a_\kappa=(1-\kappa)/2$ and $b_\kappa=(1+\kappa)/2$.
Use $U_\kappa$ on both one-forms and symmetric two-tensors; on each
bundle it is multiplication by the same normalized Gaussian factor.
Since
\begin{equation*}
 \hat\delta^*(hv)=h\hat\delta^*v-\operatorname{Sym}(\mathrm dh\otimes v),
\end{equation*}
the formula for $B_\kappa$ gives
\[
U_\kappa B_\kappa U_\kappa^{-1}v
 =e^{-a_\kappa\hat f}\left[
    \hat\delta^*(e^{a_\kappa\hat f}v)
     -\kappa\operatorname{Sym}(\mathrm{d}\hat f\otimes e^{a_\kappa\hat f}v)
    \right] 
 =\hat\delta^*v-b_\kappa\operatorname{Sym}(\mathrm{d}\hat f\otimes v)
 =:\widetilde B_\kappa v.
\]
Here $a_\kappa+\kappa=b_\kappa$.  The weighted adjoint of $B_\kappa$
is $\hat D$.  An isometry onto $L^2$ preserves adjoints, so
\begin{equation*}
 U_\kappa \hat DU_\kappa^{-1}=\widetilde B_\kappa^*,
 \qquad
 U_\kappa P_\kappa U_\kappa^{-1}
 =\widetilde B_\kappa^*\widetilde B_\kappa=:\widetilde P_\kappa.
\end{equation*}
The star on $\widetilde B_\kappa$ is the adjoint for the unweighted
measure $\mathrm dV_{\hat g}$.  Directly, for a symmetric tensor $h$,
\begin{equation*}
 U_\kappa \hat DU_\kappa^{-1}h
 =\hat\delta h-(1-a_\kappa)h(\mathbf{X},\cdot)
 =\hat\delta h-b_\kappa h(\mathbf{X},\cdot).
\end{equation*}
This is the unweighted adjoint of $\widetilde B_\kappa$, since the
pointwise adjoint of $v\mapsto\operatorname{Sym}(\mathrm{d}\hat f\otimes v)$
is $h\mapsto h(\mathbf{X},\cdot)$.

For $v\in C_c^\infty(M,T^*M)$, all integrals below use
$\mathrm dV_{\hat g}$, and all contractions use $\hat g$.
In a background orthonormal frame,
\begin{equation*}
 |\hat\delta^*v|^2
 =\frac12|\hat\nabla v|^2
   +\frac12\hat\nabla_iv_j\hat\nabla_jv_i.
\end{equation*}
The commutation identity
$\hat\nabla_i\hat\nabla_jv_i
 =\hat\nabla_j(\hat\delta v)+(\Ric_{\hat g})_{jk}v_k$
and integration by parts give
\begin{align*}
 \int_M\hat\nabla_iv_j\hat\nabla_jv_i
 &=-\int_Mv_j\hat\nabla_i\hat\nabla_jv_i
   =\int_M(\hat\delta v)^2-\int_M\Ric_{\hat g}(v,v),\\
 \|\hat\delta^*v\|_2^2
 &=\frac12\|\hat\nabla v\|_2^2
   +\frac12\|\hat\delta v\|_2^2
   -\frac12\int_M\Ric_{\hat g}(v,v).
\end{align*}
Also,
\begin{equation*}
 |\operatorname{Sym}(\mathrm d\hat f\otimes v)|^2
 =\frac12|\mathbf X|^2|v|^2+\frac12v(\mathbf X)^2.
\end{equation*}
The mixed term satisfies
\begin{align*}
 &-2b_\kappa\int_M
   \langle\hat\delta^*v,\operatorname{Sym}(\mathrm d\hat f\otimes v)\rangle\\
 &\qquad=b_\kappa\int_M
   \left(\mathbf X_iv_j\hat\nabla_iv_j
                +\mathbf X_jv_i\hat\nabla_iv_j\right)\\
 &\qquad=-\frac{b_\kappa}{2}\int_M(\hat\Delta\hat f)|v|^2
          -b_\kappa\int_M\hat\nabla^2\hat f(v,v)
          -b_\kappa\int_Mv(\mathbf X)\hat\delta v.
\end{align*}
Combining these formulas and completing the square gives
\begin{align}
 \|\widetilde B_\kappa v\|_2^2
 &=\frac12\|\hat\nabla v\|_2^2
   +\frac12\|\hat\delta v-b_\kappa v(\mathbf X)\|_2^2
   +\frac{b_\kappa^2}{2}\||\mathbf X|v\|_2^2\notag\\
 &\quad-\int_M\left\{
   \frac12\Ric_{\hat g}(v,v)+b_\kappa\hat\nabla^2\hat f(v,v)
        +\frac{b_\kappa}{2}(\hat\Delta\hat f)|v|^2\right\}
                   \mathrm dV_{\hat g}.
 \label{eq:appendix-bianchi-coercive-identity}
\end{align}
On the fixed AC shrinker, $\Ric_{\hat g}$, $\hat\nabla^2\hat f$,
and $\hat\Delta\hat f$ are bounded, and
$|\mathbf X|^2\geq c\boldsymbol{\rho}^2-C$.
Since $b_\kappa\geq1/2$, dropping the nonnegative divergence square
proves \cref{ineq:appendix-basic-coercivity} for $\widetilde P_\kappa$.

For completeness, the principal symbol of $\widetilde P_\kappa$ is
\begin{equation*}
 \sigma_2(\widetilde P_\kappa)(\xi)v
 =\frac12\left(|\xi|^2v+\langle\xi,v\rangle\xi\right),
 \qquad
 \langle\sigma_2(\widetilde P_\kappa)(\xi)v,v\rangle
 \geq\frac12|\xi|^2|v|^2.
\end{equation*}
In the fixed covariant expression
$\widetilde P_\kappa=-a^{ij}\hat\nabla_i\hat\nabla_j+b^j\hat\nabla_j+q$, the coefficient $a$
is parallel, $b$ is linear in $\mathbf{X}$, and $q$ consists of
contractions of $\mathbf{X}\otimes \mathbf{X}$, $\hat\nabla \mathbf{X}$ and
$\Rm_{\hat g}$.  Consequently, for every $\ell\geq0$,
\begin{equation*}
 |\hat\nabla^\ell a|\leq C_\ell\boldsymbol{\rho}^{-\ell},\qquad
 |\hat\nabla^\ell b|\leq C_\ell\boldsymbol{\rho}^{1-\ell},\qquad
 |\hat\nabla^\ell q|\leq C_\ell\boldsymbol{\rho}^{2-\ell}.
\end{equation*}
The operator is formally self-adjoint, so the adjoint hypotheses are
the same.  All hypotheses of \cref{lem:appendix-elementary-oscillator}
have now been checked.  That lemma gives
\begin{equation*}
 D_{\min}(\widetilde P_\kappa)=D_{\max}(\widetilde P_\kappa)=H^2_{\boldsymbol{\rho}},
 \qquad
 \|v\|_{H^2_{\boldsymbol{\rho}}}\leq C_\kappa(\|\widetilde P_\kappa v\|_2+\|v\|_2),
\end{equation*}
with $C_c^\infty$ a core and $H^2_{\boldsymbol{\rho}}\hookrightarrow L^2$ compact.
Since the minimal and maximal realizations agree, the realization on
$H^2_{\boldsymbol{\rho}}$ is self-adjoint.  It is nonnegative and is the Friedrichs
realization: on compact support its quadratic form is
$\|\widetilde B_\kappa v\|_2^2$, and its closure has form domain
$H^1_{\boldsymbol{\rho}}$ by the basic coercivity estimate and
$|\mathbf{X}|\leq C\boldsymbol{\rho}$.  In particular,
$(\widetilde P_\kappa+1)^{-1}:L^2\to H^2_{\boldsymbol{\rho}}$ is bounded, so its composition
with the compact inclusion into $L^2$ is compact.

By \cref{lem:weakened-Gaussian-Sobolev-conjugation},
$U_\kappa H^j_\kappa=\mathcal S^j_{\mathrm{osc}}$ with equivalent norms.
Thus the operator domain of $P_\kappa$ is $H^2_\kappa$ and its form
domain is $H^1_\kappa$.

The closed quadratic form gives
\begin{equation*}
 \ker \widetilde P_\kappa=\ker\widetilde B_\kappa=U_\kappa\mathcal K_\kappa.
\end{equation*}
For the converse inclusion implicit here, an $L^2$ distributional
solution of $\widetilde B_\kappa v=0$ satisfies $\widetilde P_\kappa v=0$ and hence
belongs to $H^2_{\boldsymbol{\rho}}$ by the maximal-domain statement.
Nonnegativity and compact resolvent imply that the kernel is
finite-dimensional and that the positive spectrum is bounded below by
some $\lambda_\kappa>0$.  On its orthogonal complement,
\begin{equation*}
 \|v\|_2\leq\lambda_\kappa^{-1}\|\widetilde P_\kappa v\|_2,
 \qquad
 \|v\|_{H^2_{\boldsymbol{\rho}}}\leq C_\kappa\|\widetilde P_\kappa v\|_2.
\end{equation*}
The second estimate makes the range closed.  Self-adjointness gives
$(\operatorname{Ran}\widetilde P_\kappa)^\perp=\ker \widetilde P_\kappa$, so the range is
exactly the kernel complement.  Conjugating proves the asserted
isomorphism for $m=0$, with
\begin{equation*}
 \|\zeta \|_{H^2_\kappa}
 \leq C_\kappa\|P_\kappa \zeta \|_{L^2_\kappa}
 \qquad(\zeta \perp\mathcal K_\kappa).
\end{equation*}

We spell out the higher-order estimate separately from the spectral
argument.  On a ball of radius $c/R$ about a point with $R=\boldsymbol{\rho}(p)$,
rescale as in \cref{lem:appendix-elementary-oscillator}.
Uniform interior elliptic estimates at order $m+2$, followed by summation
over the same bounded-overlap covering, give
\begin{equation*}
 \|v\|_{\mathcal S^{m+2}_{\mathrm{osc}}}
 \leq C_{m,\kappa}\left(
       \|\widetilde P_\kappa v\|_{\mathcal S^{m}_{\mathrm{osc}}}+\|\boldsymbol{\rho}^{m+2}v\|_2\right).
\end{equation*}
More explicitly, the local estimate controls
$R^{m+2-j}\|\hat\nabla^jv\|_{L^2(B_{c/R})}$ for $0\leq j\leq m+2$
by the sum of
$R^{m-j}\|\hat\nabla^j(\widetilde P_\kappa v)\|_{L^2(B_{2c/R})}$ for $0\leq j\leq m$
and $R^{m+2}\|v\|_{L^2(B_{2c/R})}$.  The coefficient bounds established above
make all these local constants uniform.  This argument also proves the
regularity assertion whenever the right-hand side is finite, by local
elliptic regularity and summation.

To estimate $\|\boldsymbol{\rho}^{m+2}v\|_2$, suppose inductively that $v\in\mathcal S^{m}_{\mathrm{osc}}$
and that $\widetilde P_\kappa v\in\mathcal S^{m}_{\mathrm{osc}}$.  For $m\geq1$, the product rule
and the coefficient bounds give
\begin{equation*}
 \|[\widetilde P_\kappa,\boldsymbol{\rho}^m]v\|_2
 \leq C_{m,\kappa}\left(
        \|\boldsymbol{\rho}^{m-1}\hat\nabla v\|_2+\|\boldsymbol{\rho}^mv\|_2\right)
 \leq C_{m,\kappa}\|v\|_{\mathcal S^{m}_{\mathrm{osc}}}.
\end{equation*}
Indeed, the derivative term contains $\hat\nabla(\boldsymbol{\rho}^m)=O(\boldsymbol{\rho}^{m-1})$;
the zeroth-order terms contain $\hat\nabla^2(\boldsymbol{\rho}^m)=O(\boldsymbol{\rho}^{m-2})$ or
$b\hat\nabla(\boldsymbol{\rho}^m)=O(\boldsymbol{\rho}^m)$.  Thus $w=\boldsymbol{\rho}^mv$ and
\begin{equation*}
 \widetilde P_\kappa w=\boldsymbol{\rho}^m\widetilde P_\kappa v+[\widetilde P_\kappa,\boldsymbol{\rho}^m]v
\end{equation*}
belong to $L^2$.  The maximal-domain result places $w$ in $H^2_{\boldsymbol{\rho}}$
and gives
\begin{equation*}
 \|\boldsymbol{\rho}^{m+2}v\|_2
 \leq C_{m,\kappa}\left(
        \|\widetilde P_\kappa v\|_{\mathcal S^{m}_{\mathrm{osc}}}+\|v\|_{\mathcal S^{m}_{\mathrm{osc}}}\right).
\end{equation*}
Starting with the order-two graph estimate, induction now proves both
the regularity implication and the estimate
\begin{gather*}
 v\in L^2,\quad \widetilde P_\kappa v\in\mathcal S^{m}_{\mathrm{osc}}
 \quad\Longrightarrow\quad v\in\mathcal S^{m+2}_{\mathrm{osc}},\\
 \|v\|_{\mathcal S^{m+2}_{\mathrm{osc}}}
 \leq C_{m,\kappa}\left(\|\widetilde P_\kappa v\|_{\mathcal S^{m}_{\mathrm{osc}}}+\|v\|_2\right).
\end{gather*}
At $m=1$ the induction uses the already established $\mathcal S^{2}_{\mathrm{osc}}$
bound, and at subsequent orders it uses the preceding estimate to
control $\|v\|_{\mathcal S^{m}_{\mathrm{osc}}}$.
Conversely the coefficient bounds give a bounded map
$\widetilde P_\kappa:\mathcal S^{m+2}_{\mathrm{osc}}\to\mathcal S^{m}_{\mathrm{osc}}$.
For data in $H^m_\kappa\cap\mathcal K_\kappa^\perp$, first solve in
$L^2_\kappa$ and then apply this regularity implication and the
identification of the Sobolev spaces under $U_\kappa$.
The spectral gap removes the $L^2$ term, yielding
\begin{equation*}
 \|\zeta \|_{H^{m+2}_\kappa}
 \leq C_{m,\kappa}\|P_\kappa \zeta \|_{H^m_\kappa}
 \qquad(\zeta \perp\mathcal K_\kappa).
\end{equation*}
This proves the isomorphism at every Sobolev level.
The same regularity implication gives
$\mathcal K_\kappa\subset H^j_\kappa$ for every $j\geq0$.

Finally, on the form domain orthogonal to the kernel, the spectral gap
gives $\lambda_\kappa\|v\|_2^2\leq\|\widetilde B_\kappa v\|_2^2$.
Combining this with the basic coercivity estimate gives
$\|v\|_{\mathcal S^{1}_{\mathrm{osc}}}\leq C_\kappa\|\widetilde B_\kappa v\|_2$.
\Cref{lem:weakened-Gaussian-weighted-Sobolev} and the isometry give
\cref{ineq:appendix-bianchi-form-gap}.
\end{proof}

Below, $P_\kappa^{-1}$ denotes the inverse on
$\mathcal K_\kappa^\perp$, extended by zero on $\mathcal K_\kappa$.

\begin{Proposition}\label{prop:weakened-projection-mapping-property}
    The operator
    \begin{equation}
        \Pi_{\kappa}:=I-B_{\kappa}P_{\kappa}^{-1}\hat D,
    \end{equation}
    initially defined on $H^1_\kappa(M,S^2T^*M)$, extends to a continuous linear operator
    \begin{equation*}
        \Pi_{\kappa}:L^2(\mathrm dV_{\hat f,\kappa})\to L^2(\mathrm dV_{\hat f,\kappa}).
    \end{equation*}
    It is the $L^2(\mathrm dV_{\hat f,\kappa})$-orthogonal projection to the closed subspace $\Ker\big(\hat D\restriction L^2(\mathrm dV_{\hat f,\kappa})\big)$.
\end{Proposition}

\begin{proof}
All inner products in this proof use $\mathrm dV_{\hat f,\kappa}$.
The unknown and test fields are one-forms in the Hilbert space
\begin{equation*}
 \mathcal E_\kappa
 :=H^1_\kappa(M,T^*M)\cap\mathcal K_\kappa^\perp.
\end{equation*}
\Cref{lem:weakened-Gaussian-weighted-Sobolev} and the formula for
$B_\kappa$ give a bounded map
\begin{equation*}
 B_\kappa:H^1_\kappa(M,T^*M)
       \longrightarrow L^2_\kappa(S^2T^*M).
\end{equation*}
Given $W\in L^2_\kappa(S^2T^*M)$, solve
\begin{equation*}
 \langle B_\kappa\zeta,B_\kappa Y\rangle
 =\langle W,B_\kappa Y\rangle
 \qquad\text{for all }Y\in\mathcal E_\kappa.
\end{equation*}
The bilinear form is continuous and, by
\cref{ineq:appendix-bianchi-form-gap}, coercive in the $H^1_\kappa$ norm on
$\mathcal E_\kappa$. Its right-hand side is continuous there.
Lax--Milgram gives a unique $\zeta\in\mathcal E_\kappa$.

Let $\pi_{\mathcal K_\kappa}$ denote the orthogonal projection onto
$\mathcal K_\kappa$. By
\cref{lem:weakened-Gaussian-Bianchi-Hilbert-inverse}, this kernel is
finite-dimensional, smooth, and contained in every $H^j_\kappa$.
Thus, for any $\phi\in C_c^\infty(T^*M)$, the field
\begin{equation*}
 Y=\phi-\pi_{\mathcal K_\kappa}\phi
\end{equation*}
belongs to $\mathcal E_\kappa$ and satisfies $B_\kappa Y=B_\kappa\phi$.
The form equation gives
\begin{equation*}
 \langle W-B_\kappa\zeta,B_\kappa\phi\rangle=0
 \quad\text{for all }\phi\in C_c^\infty(T^*M),
 \qquad \hat D(W-B_\kappa\zeta)=0
\end{equation*}
distributionally.

If $Z\in L^2_\kappa(S^2T^*M)$ satisfies $\hat DZ=0$ distributionally,
choose $\zeta_j\in C_c^\infty(T^*M)$ converging to $\zeta$ in $H^1_\kappa$.
The boundedness of $B_\kappa$ gives $B_\kappa\zeta_j\to B_\kappa\zeta$
in $L^2_\kappa$, and the distributional adjoint relation gives
\begin{equation*}
 \langle Z,B_\kappa\zeta\rangle
 =\lim_{j\to\infty}\langle Z,B_\kappa\zeta_j\rangle=0.
\end{equation*}
Consequently $W=(W-B_\kappa\zeta)+B_\kappa\zeta$ is the orthogonal
decomposition into $\ker \hat D$ and $(\ker \hat D)^\perp$.
The distributional kernel is closed in $L^2_\kappa$, because its
defining pairings against $B_\kappa\phi$ are continuous.
Thus $QW:=W-B_\kappa\zeta$ is its orthogonal projection and
$\|QW\|_{L^2_\kappa}\leq\|W\|_{L^2_\kappa}$.

If $W\in H^1_\kappa(M,S^2T^*M)$, then
$\hat DW=\hat\delta W-W(\mathbf X,\cdot)\in L^2_\kappa$ by
\cref{lem:weakened-Gaussian-weighted-Sobolev}.
The extended adjoint relation shows
\begin{equation*}
 \langle \hat DW,K\rangle=\langle W,B_\kappa K\rangle=0
 \qquad(K\in\mathcal K_\kappa).
\end{equation*}
By \cref{lem:weakened-Gaussian-Bianchi-Hilbert-inverse}, the field
$\zeta_0=P_\kappa^{-1}\hat DW$ belongs to
$H^2_\kappa\cap\mathcal K_\kappa^\perp$ and satisfies the same form
equation. Uniqueness gives $\zeta=\zeta_0$, so
\begin{equation*}
 QW=W-B_\kappa P_\kappa^{-1}\hat DW=\Pi_\kappa W.
\end{equation*}
This proves the claimed extension and its projection property.
\end{proof}

\subsection{Weighted H\"older estimates for the projection}

\begin{Proposition}
\label{lem:fixed-level-Bianchi-correction}
Fix $0<\kappa<1$, $m\in\mathbb Z_{\geq0}$, $\alpha\in(0,1)$, and
$\beta\in\mathbb R$.

For
$W\in C^{m,\alpha}_\beta(M,S^2T^*M)$, define $\zeta =P_\kappa^{-1} \hat DW$
as the unique element of
$H^1_\kappa(T^*M)\cap\mathcal K_\kappa^\perp$ satisfying
\begin{equation*}
 \langle B_\kappa \zeta ,B_\kappa Y\rangle_{L^2_\kappa}
 =\langle W,B_\kappa Y\rangle_{L^2_\kappa}
 \qquad\bigl(Y\in\operatorname{Dom}(B_\kappa)\bigr).
\end{equation*}
Here $\operatorname{Dom}(B_\kappa)=H^1_\kappa$ is the Gaussian form domain of
$P_\kappa=B_\kappa^*B_\kappa$.
    Then
    \begin{equation}\label{ineq:fixed-level-Bianchi-correction}
        \|B_\kappa \zeta \|_{C^{m,\alpha}_\beta}
        \leq C_{m,\beta,\kappa}
        \|W\|_{C^{m,\alpha}_\beta}.
    \end{equation}
The constants also depend on $\alpha$ and the fixed background shrinker.
\end{Proposition}

\begin{proof}
The weak solution is the one constructed in the proof of
\cref{prop:weakened-projection-mapping-property}.  Testing with $Y=\zeta $ gives
\begin{equation}\label{ineq:appendix-projection-contraction}
 \|B_\kappa \zeta \|_{L^2_\kappa}
       \leq\|W\|_{L^2_\kappa},\qquad
 \|\zeta \|_{L^2_\kappa}
       \leq C_\kappa\|W\|_{L^2_\kappa}.
\end{equation}
Also $\|W\|_{L^2_\kappa}\leq C_{\beta,\kappa}\|W\|_{C^0_\beta}$.
We first take $W\in C_c^\infty$. Adjointness gives
$\hat DW\perp\mathcal K_\kappa$, and
\cref{lem:weakened-Gaussian-Bianchi-Hilbert-inverse} gives
$\zeta\in H^j_\kappa$ for every finite $j$.

\emph{Step 1.}
For large $L$, choose a smooth cap $\Theta_L$ with
\[
 \begin{gathered}
 \Theta_L(s)=s\quad(s\le L),\qquad
 \Theta_L(s)=\Theta_L(2L)\quad(s\ge2L),\\
 0\le\Theta_L'\le1,\qquad
 |\Theta_L^{(j)}|\le C_jL^{1-j}\quad(j\ge1).
 \end{gathered}
\]
Such a cap is obtained by rescaling a fixed smooth function. Put
\begin{equation}\label{eq:projection-capped-Gaussian}
 \begin{aligned}
 w_L&=\exp\!\left[-\frac{1-\kappa}{2}
                  \bigl(\hat f-\Theta_L(\hat f)\bigr)\right],
 &\zeta_L&=w_L\zeta,\\
 a_L&=\frac{1-\kappa}{2}\bigl(1-\Theta_L'(\hat f)\bigr),
 &-\mathrm d\log w_L&=a_L\,\mathrm d\hat f.
 \end{aligned}
\end{equation}
Then $0\le a_L\le(1-\kappa)/2$. For $j\ge1$, derivatives of $a_L$
are supported where $L\le\hat f\le2L$. There
$\hat f\asymp L\asymp\boldsymbol\rho^2$; the product rule and the AC
bounds for $\hat f$ give
\[
 |\hat\nabla^j a_L|\le C_{j,\kappa}\boldsymbol\rho^{-j}.
\]
Thus all conic coefficient bounds are uniform in $L$. We take $L$
large enough that $w_L=1$ on $\operatorname{supp}W$.

By the Gaussian Sobolev identification,
$v:=U_\kappa\zeta\in\mathcal S^j_{\mathrm{osc}}$ for every $j$.
To see the resulting weighted H\"older bounds, fix $m,\alpha,\gamma$
and an integer $q>m+\alpha+n/2$. Sobolev embedding on normalized
annuli bounds the $C^{m,\alpha}$ norm of the component representative
of $v$ on $A_R=\{R<r<2R\}$, with $A_R^*=\{R/2<r<4R\}$, by
\[
 C\sum_{j=0}^q R^{j-n/2}
             \|\hat\nabla^jv\|_{L^2(A_R^*)}.
\]
Multiplication by $R^{-\gamma}$ is controlled by sufficiently high
polynomial moments of these derivatives. They are finite because
$v$ belongs to every oscillator Sobolev space. Hence
$v\in C^{m,\alpha}_\gamma$ for every finite $m$ and every $\gamma$.
Outside $\{\hat f\le2L\}$, $\zeta_L$ is a constant multiple of $v$.
Smoothness on the remaining compact set therefore makes every weighted
H\"older norm of $\zeta_L$ finite for each fixed $L$.

The product rules for divergence and its adjoint give
\begin{align*}
 D_L:=w_L\hat D w_L^{-1}
     &=\hat\delta-(1-a_L)\iota_{\mathbf X},\\
 B_{\kappa,L}:=w_LB_\kappa w_L^{-1}
     &=\hat\delta^*-(\kappa+a_L)
                 \operatorname{Sym}(\mathrm d\hat f\otimes\cdot).
\end{align*}
Thus the exact equation is
\begin{equation}\label{eq:projection-capped-equation}
 D_LB_{\kappa,L}\zeta_L=D_L(w_LW)=D_LW.
\end{equation}
The source remains in divergence form.

\emph{Step 2.}
Pull $A_R^*=\{R/2<r<4R\}$ back to a fixed annulus, use normalized
frames, and set
\begin{equation*}
 \epsilon=R^{-2},\qquad U=R^{1-\beta}\zeta_L,\qquad W_R=R^{-\beta}W.
\end{equation*}
With pullback and component identifications understood, define
\begin{equation*}
 \mathscr B_{\epsilon,R}=R^{-1}B_{\kappa,L},\qquad
 \mathscr D_{\epsilon,R}=R^{-1}D_L,\qquad
 \mathscr P_{\epsilon,R}=\mathscr D_{\epsilon,R}\mathscr B_{\epsilon,R}.
\end{equation*}
Then
\begin{equation*}
 \mathscr P_{\epsilon,R}U=\mathscr D_{\epsilon,R}W_R,
 \qquad R^{-\beta}B_{\kappa,L}\zeta_L=\mathscr B_{\epsilon,R}U.
\end{equation*}
Each factor is of the form
$\epsilon A^i(y)\partial_i+A^0(y)+\epsilon A^1(y)$, with uniformly
bounded smooth coefficients, uniformly in $L$. At a frozen point put
$\omega=R^{-1}\mathrm d\hat f$ in the normalized identification and
$a=a_L$. Up to uniformly controlled linear coordinate and frame changes,
the leading factors are
\[
 U\longmapsto-\operatorname{Sym}
       ((\epsilon\partial+(\kappa+a)\omega)\otimes U),\qquad
 V\longmapsto\epsilon\operatorname{div}V-(1-a)\iota_{\omega^\sharp}V.
\]
These are the model operators \cref{eq:euclidean-Holder-model} with
\[
 \widetilde\omega=(1-a)\omega,\qquad
 \widetilde\kappa=\frac{\kappa+a}{1-a}\in[\kappa,1].
\]
The norm of $\omega$ lies between fixed positive constants, and
$1-a\in[(1+\kappa)/2,1]$. Hence the appendix estimate applies uniformly
in $R,L$, including the endpoint $\widetilde\kappa=1$. A ball of radius
$\epsilon$ in the rescaled metric corresponds to a $\hat g$-ball of
radius $R\epsilon=R^{-1}$.

Next, we control the ordinary H\"older norm on the whole fixed annulus.
We use $\hat g$-geodesic balls of radius $\eta R$, together with their
doubled balls, where $\eta>0$ is fixed independently of $R$.  After
annular rescaling these have radius $\eta$ in the rescaled metric;
we use uniformly controlled coordinate charts on these balls and apply
\cref{lem:euclidean-divergence-Holder}.

On $B_{2\eta}(y_0)$ suppress the subscript $R$ and write
\begin{align*}
 \mathscr B_\epsilon U
 &=\epsilon b^j(y)\partial_jU+b^0(y)U+\epsilon b^1(y)U,\\
 \mathscr D_\epsilon V
 &=\epsilon\partial_i(d^i(y)V)+d^0(y)V+\epsilon c(y)V.
\end{align*}
The second formula is obtained by moving the derivative coefficients
inside the divergence; their derivatives are absorbed into $c$.
The leading coefficients have uniform $C^1$ bounds, and all displayed
coefficients have uniform $C^{0,\alpha}$ bounds.  For any leading
coefficient $A$,
\begin{equation*}
 \|A-A(y_0)\|_\infty\leq C\eta,\qquad
 [A-A(y_0)]_\alpha\leq C(4\eta)^{1-\alpha}.
\end{equation*}
Thus $\|A-A(y_0)\|_{C^{0,\alpha}}\leq C\eta^{1-\alpha}$
for $0<\eta\leq1$.

Let a star denote evaluation of a leading coefficient at $y_0$, and put
\begin{equation*}
 \mathscr B_*U=\epsilon b_*^j\partial_jU+b_*^0U,\qquad
 \mathscr D_*V=\epsilon\partial_i(d_*^iV)+d_*^0V,
 \qquad \mathscr P_*=\mathscr D_*\mathscr B_*.
\end{equation*}
For $N_\epsilon(U)=\|U\|_{C^{0,\alpha}}
+\epsilon\|\partial U\|_{C^{0,\alpha}}$, set
\begin{equation*}
 V=\mathscr B_\epsilon U,\qquad
 E_B=(\mathscr B_\epsilon-\mathscr B_*)U
 =\epsilon(b^j-b_*^j)\partial_jU+(b^0-b_*^0)U+\epsilon b^1U.
\end{equation*}
The H\"older product inequality gives
\begin{equation*}
 \|V\|_{C^{0,\alpha}}\leq C N_\epsilon(U),\qquad
 \|E_B\|_{C^{0,\alpha}}
 \leq C(\eta^{1-\alpha}+\epsilon)N_\epsilon(U).
\end{equation*}
The operator difference has the exact decomposition
\begin{align*}
 (\mathscr P_\epsilon-\mathscr P_*)U
 &=(\mathscr D_\epsilon-\mathscr D_*)V+\mathscr D_*E_B
   =F_0^{\rm err}+\epsilon\partial_i(F_1^{\rm err})^i,\\
 (F_1^{\rm err})^i&=(d^i-d_*^i)V+d_*^iE_B,\\
 F_0^{\rm err}&=(d^0-d_*^0)V+\epsilon cV+d_*^0E_B.
\end{align*}
Consequently
\begin{equation*}
 \|F_0^{\rm err}\|_{C^{0,\alpha}}
       +\|F_1^{\rm err}\|_{C^{0,\alpha}}
 \leq C(\eta^{1-\alpha}+\epsilon)
       \bigl(\|U\|_{C^{0,\alpha}}
                    +\epsilon\|\partial U\|_{C^{0,\alpha}}\bigr).
\end{equation*}
The estimate is on $F_1^{\rm err}$, not its divergence; no second
derivative of $U$ is estimated separately.  Likewise,
$\mathscr D_\epsilon W_R$ has divergence-form source coefficients
$d^iW_R$ and $(d^0+\epsilon c)W_R$, bounded by
$C\|W_R\|_{C^{0,\alpha}}$ without differentiating $W_R$.

\emph{Step 3.}
Choose $\chi=1$ on $B_\eta(y_0)$, supported in $B_{2\eta}(y_0)$,
with derivatives bounded in terms of $\eta$.  Apply
\cref{ineq:appendix-parameter-divergence-Schauder} to $\chi U$, extended by zero.  In the coefficient error, extend
$A-A(y_0)$ using a cutoff on a slightly larger patch; its global
$C^{0,\alpha}$ norm is still $O(\eta^{1-\alpha})$.
Choose $\eta$ small enough, and then $R$ large enough, to absorb
$C(\eta^{1-\alpha}+\epsilon)N_\epsilon(\chi U)$.

For clarity, the product rules give
\[
\mathscr P_\epsilon(\chi U)
 ={}\mathscr D_\epsilon(\chi W_R)
    +\epsilon d^i(\partial_i\chi)(\mathscr B_\epsilon U-W_R) 
 +\mathscr D_\epsilon\bigl(
                    \epsilon b^i(\partial_i\chi)U\bigr).
\]
The last two terms have divergence-form source norms bounded by
$C_\eta\epsilon(N_\epsilon(U)+\|W_R\|_{C^{0,\alpha}})$
on the enlarged patch.  Hence the local estimate is
\begin{equation*}
 N_\epsilon(U;B_\eta)
 \leq C_\eta\|W_R\|_{C^{0,\alpha}(B_{2\eta})}
       +C_\eta\epsilon N_\epsilon(U;B_{2\eta}).
\end{equation*}
Take a finite covering of the fixed annulus with a Lebesgue number
comparable to $\eta$.  Pairs of sufficiently close points lie in a
common enlarged patch.  For the remaining pairs,
\begin{equation*}
 \frac{|V(y)-V(z)|}{|y-z|^\alpha}
       \leq C_\eta\|V\|_\infty.
\end{equation*}
Thus patching controls the full annular H\"older seminorm, with no
factor $\epsilon^{-\alpha}$.  Changes of smooth bundle frame have
uniformly bounded coefficients and obey the same product estimates.

For $m>0$, differentiate the equation on each fixed annular chart and
put
\[
 N_{\epsilon,m}(U)=\|U\|_{C^{m,\alpha}}
                        +\epsilon\|\partial U\|_{C^{m,\alpha}}.
\]
Every commutator differentiates at least one coefficient. Keeping its
second-order terms in divergence form bounds the source coefficients by
\[
 C_m\bigl(N_{\epsilon,m-1}(U)+\|W_R\|_{C^{m,\alpha}}\bigr).
\]
For example, the highest remaining derivative in the divergence source
is $\epsilon(\partial a)\partial^mU$, controlled by
$\epsilon\|\partial U\|_{C^{m-1,\alpha}}$. All coefficient derivatives
have bounds independent of $L$. The cutoff terms at order $m$ are
bounded by $C_{m,\eta}\epsilon N_{\epsilon,m}(U)$ on the enlarged patch.

Take suprema over overlapping dyadic annuli. Neighboring weights are
uniformly comparable. At order zero, choosing the inner radius large
absorbs the neighboring-annulus errors, which carry $\epsilon=R^{-2}$.
At order $m$, the same choice, enlarged in terms of $m$, absorbs the
highest-order errors; the already established order-$(m-1)$ estimate
controls the commutators. These absorptions are legitimate because all
weighted norms of $\zeta_L$ are finite for each fixed $L$.

The annular constants and the inner radius are independent of $L$ and
of $\operatorname{supp}W$. Choose a fixed enlarged compact core $K$
containing the remaining annuli. On this core take $L$ large enough
that $w_L=1$ on a neighborhood of $K$, as well as on
$\operatorname{supp}W$. The interior divergence-form Schauder estimate
then controls $\zeta_L=\zeta$ in $C^{m+1,\alpha}$ by the
$C^{m,\alpha}$ norm of $W$ and $\|\zeta\|_{L^2(K)}$.
Induction in $m$ consequently gives
\[
 \|\zeta_L\|_{C^{m,\alpha}_{\beta-1}}
       +\|\hat\nabla\zeta_L\|_{C^{m,\alpha}_\beta}
 \le C_{m,\alpha,\beta,\kappa}
       \bigl(\|W\|_{C^{m,\alpha}_\beta}+\|\zeta\|_{L^2(K)}\bigr),
\]
uniformly in $L$ and the support of $W$. Indeed,
$U=R^{1-\beta}\zeta_L$, and $\epsilon\partial U$ corresponds to
$R^{-\beta}\hat\nabla\zeta_L$ up to bounded frame terms controlled by
$U$. The annular norms are equivalent to the weighted norms by
\cref{eq:prelim-weighted-norm,eq:prelim-annular-norm}.

\emph{Step 4.}
On every fixed compact set $\zeta_L=\zeta$ for all sufficiently large
$L$. Taking the limit on each annulus and then the supremum yields
\begin{equation}\label{ineq:appendix-global-parameter-Schauder}
 \|\zeta\|_{C^{m,\alpha}_{\beta-1}}
       +\|\hat\nabla\zeta\|_{C^{m,\alpha}_\beta}
 \le C_{m,\alpha,\beta,\kappa}
       \bigl(\|W\|_{C^{m,\alpha}_\beta}+\|\zeta\|_{L^2(K)}\bigr).
\end{equation}
On the fixed core the unweighted and Gaussian norms are equivalent, so
\cref{ineq:appendix-projection-contraction} controls the last term.
By \cref{eq:weakened-gaussian-bianchi-adjoint}, $B_\kappa\zeta$ is a sum
of $\hat\nabla\zeta$ and $O(\boldsymbol\rho)\zeta$, with the corresponding
conic derivative bounds. This proves the asserted estimate for
compactly supported smooth $W$.

For general $W\in C^{m,\alpha}_\beta$, take compactly supported smooth
$W_j$ obtained by cutoffs and mollification in normalized charts, with
uniformly bounded $C^{m,\alpha}_\beta$ norms and $W_j\to W$ in
$L^2_\kappa$. The form estimate and the weak equation for differences
give $\zeta_j\to\zeta$ in $H^1_\kappa$ and
$B_\kappa\zeta_j\to B_\kappa\zeta$ in $L^2_\kappa$.
The uniform graph bounds give local compactness in $C^{m+1,\alpha'}$
for every $\alpha'<\alpha$. These local limits agree with the weak
solution. Lower semicontinuity of the weighted $C^{m,\alpha}$ norms
proves the graph and correction estimates at the stated exponent.

\end{proof}

\subsection{Gaussian reconstruction}
\phantomsection\label{idx:bianchi-correction}
Set $\mathcal C_\kappa=B_\kappa P_\kappa^{-1}$, with its weak extension
on the form-dual space and the fixed kernel complement. For
$h\in\mathscr H$ near zero let $\epsilon(h)=C\|h\|_{\mathscr H}$;
this controls the finite coefficient norms in
\eqref{eq:finite-low-coefficients} and tends to zero at the background.
\begin{Theorem}[Gaussian reconstruction and normal estimate]
\label{prop:corrected-orthogonal-fixed-target-normal-estimate}
Fix $0<\kappa<1$. For $h$ in a sufficiently small fixed neighborhood in $\mathscr H$,
\begin{equation}\label{eq:corrected-orthogonal-fixed-target-reconstruction}
 \Phi(g)=\mathscr R_{\kappa,h}\mathscr Q_\kappa(h),\qquad
 \mathscr Q_\kappa(h)=\Pi_\kappa\Phi(g),\qquad
 \mathscr R_{\kappa,h}=(I-\mathcal C_\kappa(\hat D-D_g))^{-1},
\end{equation}
with $\|\mathscr R_{\kappa,h}\|_{L^2_\kappa\to L^2_\kappa}\le2$.
The inverse is continuous in $h$ in this operator norm. Moreover,
\begin{equation}\label{ineq:corrected-orthogonal-normal-graph-smallness}
 \|\mathrm D\mathscr Q_\kappa(h)[k]-\hat Lk\|_{L^2_\kappa}
 \le C_\kappa\epsilon(h)\|\hat Lk\|_{L^2_\kappa},
 \qquad k\in\mathscr H'.
\end{equation}
\end{Theorem}
\begin{proof}
We first prove the operator estimate
\begin{equation}\label{ineq:orthogonal-projection-low-reconstruction-smallness}
 \|\mathcal C_\kappa(\hat D-D_g)q\|_{L^2_\kappa}
 \le C_\kappa\epsilon(h)\|q\|_{L^2_\kappa},\qquad q\in L^2_\kappa.
\end{equation}

Write
$E_h=\hat D-D_g=a^i\hat\nabla_i+b$ in the fixed background frame.
The inverse-metric and connection estimates, \cref{eq:finite-low-coefficients}, and the dependence of
the minimizer on $h$ give
\begin{equation*}
 \|a\|_{C^1}+\|b\|_{C^0}
       \leq C\epsilon(h),\qquad \epsilon(h)\longrightarrow0.
\end{equation*}
The maps $a^i,b$ send symmetric tensors to one-forms; their stars
below are pointwise adjoints for the background metric.  All pairings
use $\mathrm dV_{\hat f,\kappa}=C_\kappa e^{-(1-\kappa)\hat f}\mathrm dV_{\hat g}$.
For smooth compactly supported fields, integration by parts gives
\[
\langle a^i\hat\nabla_iq,\zeta \rangle_{L^2_\kappa}
 ={}-\int_M\left\langle q,(a^i)^*\hat\nabla_i\zeta
                         +(\hat\nabla_i a^i)^*\zeta \right\rangle
                                                    \mathrm dV_{\hat f,\kappa} 
 +(1-\kappa)\int_M\left\langle q,
                         \hat f_i(a^i)^*\zeta \right\rangle \mathrm dV_{\hat f,\kappa}.
\]
Thus the weighted formal adjoint is
\begin{equation*}
 E_h^{*_{\kappa}}\zeta 
 =-(a^i)^*\hat\nabla_i\zeta-(\hat\nabla_i a^i)^*\zeta 
       +(1-\kappa)\hat f_i(a^i)^*\zeta +b^*\zeta .
\end{equation*}
Since $|\mathrm{d}\hat f|\leq C\boldsymbol{\rho}$, Cauchy--Schwarz implies
\[
|\langle E_hq,\zeta \rangle_{L^2_\kappa}|
 \leq \|q\|_{L^2_\kappa}
                       \|E_h^{*_{\kappa}}\zeta \|_{L^2_\kappa} 
 \leq C_\kappa\epsilon(h)\|q\|_{L^2_\kappa}
       \bigl(\|\hat\nabla \zeta \|_{L^2_\kappa}
              +\|\zeta \|_{L^2_\kappa}
              +\|\boldsymbol{\rho} \zeta \|_{L^2_\kappa}\bigr).
\]
For $q\in L^2_\kappa$ and a form-domain vector $\zeta $, define
the left side by $\langle q,E_h^{*_{\kappa}}\zeta \rangle$.
\Cref{lem:weakened-Gaussian-weighted-Sobolev} and density extend the preceding calculation
to these fields.  If moreover
$\zeta \perp\mathcal K_\kappa$, the full form estimate
\cref{ineq:appendix-bianchi-form-gap} gives
\begin{equation*}
 \|\hat\nabla \zeta \|_{L^2_\kappa}
       +\|\zeta \|_{L^2_\kappa}+\|\boldsymbol{\rho} \zeta \|_{L^2_\kappa}
 \leq C_\kappa\|B_\kappa \zeta \|_{L^2_\kappa}.
\end{equation*}
Consequently
\begin{equation}\label{ineq:weakened-Gaussian-moving-divergence-dual-bound}
 |\langle(\hat D-D_g)q,\zeta \rangle_{L^2_\kappa}|
 \leq C_\kappa\epsilon(h)
 \|q\|_{L^2_\kappa}\|B_\kappa \zeta \|_{L^2_\kappa}.
\end{equation}

Let $\zeta _q=P_\kappa^{-1} E_hq$ be the weak solution on
$\mathcal K_\kappa^\perp$.  Its equation is
\begin{equation*}
 \langle B_\kappa \zeta _q,B_\kappa Y\rangle_{L^2_\kappa}
       =\langle E_hq,Y\rangle_{L^2_\kappa}
 \qquad(Y\in\operatorname{Dom}(B_\kappa)\cap\mathcal K_\kappa^\perp).
\end{equation*}
Testing with $Y=\zeta _q$ gives
\begin{equation*}
 \|B_\kappa \zeta _q\|_{L^2_\kappa}^2
 \leq C_\kappa\epsilon(h)\|q\|_{L^2_\kappa}
                         \|B_\kappa \zeta _q\|_{L^2_\kappa}.
\end{equation*}
Since $B_\kappa \zeta _q=\mathcal C_\kappa E_hq$, this proves
\cref{ineq:orthogonal-projection-low-reconstruction-smallness}.

Choose $C_\kappa\epsilon(h)<1/2$ in
\cref{ineq:orthogonal-projection-low-reconstruction-smallness}. The Neumann
series gives the inverse and its bound. The weighted Bianchi identity
$D_g\Phi(g)=0$, obtained from diffeomorphism invariance and the first
variation \cite[Lemma~2.1]{cao2012second}, gives
$\Phi(g)=\Pi_\kappa\Phi(g)+\mathcal C_\kappa(\hat D-D_g)\Phi(g)$,
which proves reconstruction. The coefficient bounds used in the
form-dual estimate are continuous in $h$; the resolvent identity
therefore proves operator-norm continuity. Finally
$\hat D\hat Lk=0$ and the orthogonal projection is a contraction, so
\[
 \mathrm D\mathscr Q_\kappa(h)[k]-\hat Lk
 =\Pi_\kappa(\mathrm D\Phi_g[k]-\hat Lk).
\]
Applying \cref{prop:nonlinear-normal-coefficient-reduction} concludes the proof.
\end{proof}

\section{\texorpdfstring{\L ojasiewicz--Simon}{Lojasiewicz--Simon} inequality}
\label{sec:finite-Banach-metric-LS}
We first perform the reduction on the fixed Banach graph and then
transfer the resulting inequalities to the original metric and pair
classes.

\subsection{Lyapunov--Schmidt reduction}
The Gaussian normal inverse of \cref{prop:finite-graph-realization} is
\begin{equation}\label{eq:finite-LS-strong-background-Hessian}
 \hat L:\mathscr H'\longrightarrow\mathscr F'.
\end{equation}
Choose an $L^2_0$-orthonormal basis $\eta_1,\ldots,\eta_d$ of $\ISD$.
Fix $0<\kappa<1$, and take $R$ sufficiently large that
$\zeta_i=\Pi_\kappa(\chi_R\eta_i)$ have invertible pairing matrix
$M_{ij}=\langle\zeta_i,\eta_j\rangle_{L^2_0}$.
Indeed $\chi_R\eta_i\to\eta_i$ in $L^2_\kappa$ and the projection is
bounded there, while $L^2_\kappa\hookrightarrow L^2_0$.
Since $\chi_R\eta_i$ is smooth and compactly supported,
\cref{lem:fixed-level-Bianchi-correction}, applied at every finite
order and weight, shows that $\zeta_i$ is smooth and decays to
arbitrary polynomial order. In particular $\zeta_i\in\mathscr F$. Set
$\mathcal C=\operatorname{span}\{\zeta_i\}_{i=1}^d$ and
\begin{equation}\label{eq:finite-LS-fixed-splittings}
 \mathscr F=\mathscr F'\oplus\mathcal C,
 \qquad
 \Pi_{\mathcal C}q=\sum_{i,j}\langle q,\eta_i\rangle_{L^2_0}
                             (M^{-1})_{ij}\zeta_j,
 \qquad \Pi_{\mathscr F'}=I-\Pi_{\mathcal C}.
\end{equation}
These projections are bounded on $\mathscr F$ and on each fixed
$L^2_{\kappa'}$: Cauchy--Schwarz applies to the Gaussian pairing
functionals, and the range is finite-dimensional. The same convention
applies when $d=0$.

\begin{Proposition}\label{thm:finite-LS-normal-inverse-family}
On a fixed neighborhood in $\mathscr H$, the map
\begin{equation}\label{eq:finite-LS-quasigradient}
 \mathscr Q_\kappa(h)=\Pi_\kappa\Phi(\hat g+h,f_{\hat g+h})
                 :\mathscr H\longrightarrow\mathscr F
\end{equation}
is real analytic and
\begin{equation}\label{eq:finite-LS-quasigradient-linearization}
 \mathrm D\mathscr Q_\kappa(0)=\hat L.
\end{equation}
The normal differential
$\Pi_{\mathscr F'}\mathrm D\mathscr Q_\kappa(h)|_{\mathscr H'}$
is a bounded isomorphism $\mathscr H'\to\mathscr F'$ for small $h$.
\end{Proposition}
\begin{proof}
\Cref{lem:graph-nonlinear-closure,thm:existence-and-uniqueness-of-minimizer}
give analyticity of the unprojected defect into
$C^{6,1/4}_{-1/2}$. The same-order estimate
\cref{ineq:fixed-level-Bianchi-correction} at
$(m,\alpha,\beta)=(6,1/4,-1/2)$ gives boundedness of $\Pi_\kappa$ into that
space. Its distributional divergence vanishes, proving the full target
assertion. Differentiating the scalar equation gives
$\mathrm Df_{\hat g}[k]=\tfrac12\tr k$ for sliced $k$; the
normalization contributes zero at the background because the first
variation of the entropy vanishes. The defect variation then gives
$\mathrm D\Phi_{\hat g}[k]=\hat Lk$ on $\ker\hat D$.
The projection fixes $\hat Lk$, proving
\eqref{eq:finite-LS-quasigradient-linearization}.
At zero the normal operator is \eqref{eq:finite-LS-strong-background-Hessian}.
Analyticity implies operator-norm continuity from $\mathscr H'$ to
$\mathscr F'$; its inverse persists by a Neumann series.
\end{proof}

\begin{Corollary}\label{cor:finite-LS-analytic-center-reduction}
There is a unique real-analytic map on a neighborhood of zero in
$\ISD$,
\begin{equation}\label{eq:finite-LS-center-graph}
 \eta_\kappa:U_1\subset\ISD\longrightarrow\mathscr H',\qquad
 \Pi_{\mathscr F'}\mathscr Q_\kappa(\xi+\eta_\kappa(\xi))=0,
\end{equation}
with $\eta_\kappa(0)=0$ and $\mathrm D\eta_\kappa(0)=0$.
The center and its entropy are
\begin{equation}\label{eq:finite-LS-center-and-reduced-energy}
 c_\kappa(\xi)=\xi+\eta_\kappa(\xi),\qquad
 e_\kappa(\xi)=\boldsymbol\mu(\hat g+c_\kappa(\xi)),
\end{equation}
and $e_\kappa$ is real analytic.
\end{Corollary}
\begin{proof}
Apply the analytic Banach implicit function theorem to
$\mathcal G_\kappa(\xi,w)=\Pi_{\mathscr F'}\mathscr Q_\kappa(\xi+w)$
on $\ISD\times\mathscr H'$. Its normal differential at zero is the
bounded isomorphism above. Differentiation in $\xi$ gives
$\mathrm D\eta_\kappa(0)=0$ since $\hat L\xi=0$.
Composition with the scalar entropy branch on this same domain gives
analyticity of $e_\kappa$.
\end{proof}

The first variation and the Gaussian reconstruction give
\begin{equation}\label{eq:finite-LS-variable-pairing}
 \mathrm D\boldsymbol\mu_{\hat g+h}[k]
 =\mathscr B_{\kappa,h}(\mathscr Q_\kappa(h),k),\qquad
 \mathscr B_{\kappa,h}(q,k)
 =\int_M\langle\mathscr R_{\kappa,h}q,k\rangle_g\,\mathrm dV_{f_g}.
\end{equation}
For two fixed weights define, on $\mathcal C$,
\begin{equation}\label{eq:finite-LS-two-delta-target-isomorphism}
 J_{\kappa',\kappa,h}q=\Pi_{\kappa'}\mathscr R_{\kappa,h}q,
 \qquad
 \mathscr Q_{\kappa'}(h)=\Pi_{\kappa'}\mathscr R_{\kappa,h}
                                      \mathscr Q_\kappa(h).
\end{equation}
When $0<\kappa'<\kappa$, reconstruction in $L^2_\kappa$ and the
embedding $L^2_\kappa\hookrightarrow L^2_{\kappa'}$ make the first
map continuous in $h$ into $L^2_{\kappa'}$, and it equals the identity
at zero on $\mathcal C$.

\subsection{\texorpdfstring{\L ojasiewicz--Simon}{Lojasiewicz--Simon}
inequalities for Perelman's functionals}

\begin{Theorem}[Global \L ojasiewicz--Simon inequality for $\boldsymbol{\mu}$]\label{thm:finite-global-metric-LS}
    There exist an open neighborhood $U_1\subset E_{14}(M,S^2T^*M)$ of the origin and constants $C>0$, $\vartheta\in(0,1/2)$ such that
    \begin{equation}\label{ineq:finite-global-metric-LS}
        |\boldsymbol{\mu}(g)-\boldsymbol{\mu}(\hat g)|^{1-\vartheta}
        \leq C\left(  \int_M\Big|  \Ric_g+\nabla_g^2f_g-\frac{1}{2}g  \Big|_g^2\mathrm dV_{f_g}  \right)^{1/2}
    \end{equation}
    for all $h\in U_1$.
    Here $g:=\hat{g}+h$ and $f_g$ denotes the unique minimizer of $\mathcal W(g,\cdot)$ given by \cref{thm:existence-and-uniqueness-of-minimizer}.
\end{Theorem}
\begin{proof}
\emph{Step 1.}
By \cref{thm:finite-coordinate-slice}, a small $E_{14}$ input has
sliced output in $E_9$. By \eqref{eq:original-to-graph}, the sliced
output is small in $\mathscr H$. Diffeomorphism invariance therefore
reduces the proof to $h=g-\hat g\in\mathscr H$; the same proof gives
the metric inequality on a small graph neighborhood itself.  Fix $0<\kappa<1$ and use
\cref{cor:finite-LS-analytic-center-reduction} to construct $\eta_\kappa$ and
$e_\kappa$.  Write $h=\xi+\eta$ with $\xi\in\ISD$ and
$\eta\in\mathscr H'$, and set
\begin{equation}\label{eq:finite-LS-center-normal-decomposition}
 c_\kappa=c_\kappa(\xi):=\xi+\eta_\kappa(\xi),\qquad
 w:=\eta-\eta_\kappa(\xi)\in\mathscr H',\qquad h=c_\kappa+w.
\end{equation}
The finite-dimensional \L ojasiewicz inequality for the real-analytic
function $e_\kappa$ gives
\begin{equation}\label{ineq:finite-LS-center-Lojasiewicz}
 |\boldsymbol{\mu}(\hat g+c_\kappa)-\boldsymbol{\mu}(\hat g)|^{1-\theta}\leq C|\mathrm D e_\kappa(\xi)|
\end{equation}
for some $\theta=\theta_\kappa\in(0,\frac12]$.

Recall from \cref{eq:finite-LS-fixed-splittings} that
$\mathcal C=\operatorname{span}\{\zeta_1,\ldots,\zeta_d\}$ is the fixed
finite-dimensional complement in $\mathscr F=\mathscr F'\oplus\mathcal C$.
The center equation gives $\mathscr Q_\kappa(c_\kappa)\in\mathcal C$.
Recall the pairing from \cref{eq:finite-LS-variable-pairing}: with
$g_{c_\kappa}=\hat g+c_\kappa(\xi)$,
\[
 \mathscr B_{\kappa,c_\kappa(\xi)}(q,\ell)
 =\int_M\langle\mathscr R_{\kappa,c_\kappa(\xi)}q,\ell\rangle_{g_{c_\kappa}}
                   \,\mathrm dV_{f_{g_{c_\kappa}}}.
\]
Define
\begin{equation*}
 F_\xi:\mathcal C\longrightarrow\ISD^*,\qquad
 (F_\xi q)[a]:=\mathscr B_{\kappa,c_\kappa(\xi)}(q,\mathrm D c_\kappa(\xi)[a]).
\end{equation*}
The first-variation pairing \cref{eq:finite-LS-variable-pairing} implies
$\mathrm D e_\kappa(\xi)=F_\xi\mathscr Q_\kappa(c_\kappa)$.
Since $\mathrm D c_\kappa(0)[a]=a$ and the pairing at zero is the
$L^2_0$ pairing, the matrix of $F_0$ is
$M_{ij}=\langle\zeta_i,\eta_j\rangle_{L^2_0}$.
It is invertible by the choice of $\mathcal C$.
Continuity in these finite-dimensional spaces gives uniformly bounded
$F_\xi$ and $F_\xi^{-1}$ near zero, and hence
\begin{equation}\label{ineq:finite-LS-center-pairing}
 C^{-1}|\mathscr Q_\kappa(c_\kappa)|\leq |\mathrm D e_\kappa(\xi)|
                          \leq C|\mathscr Q_\kappa(c_\kappa)|.
\end{equation}
Here and below $|\cdot|$ denotes fixed norms on the indicated
finite-dimensional spaces.

\emph{Step 2.}
The exponent $\theta=\theta_\kappa$ is obtained after fixing $\kappa$
and may depend on it. Choose
\begin{equation}\label{eq:finite-LS-physical-delta-choice}
 0<\kappa'<\min\{\kappa,\theta\},
\end{equation}
keeping the center graph $c_\kappa$ fixed. This allows
$1-\theta<a<1-\kappa'$ in Step 5. The transition identity
\cref{eq:finite-LS-two-delta-target-isomorphism} gives
$\mathscr Q_{\kappa'}(c_\kappa)
 =J_{\kappa',\kappa,c_\kappa}\mathscr Q_\kappa(c_\kappa)$.
All constants below may depend on the two fixed parameters.

\emph{Step 3.}
Shrink the neighborhood so that $h_t:=c_\kappa+tw$ stays in the neighborhoods
of the preceding estimates for $0\leq t\leq1$, and set $g_t:=\hat g+h_t$.  The fundamental
theorem of calculus gives
\begin{equation}\label{eq:difference-of-quasi-gradient}
 \mathscr Q_{\kappa'}(h)
 =\mathscr Q_{\kappa'}(c_\kappa)
  +\int_0^1 \mathrm D\mathscr Q_{\kappa'}(c_\kappa+tw)[w]\,\mathrm{d}t.
\end{equation}
Set
\begin{equation*}
 B_{c_\kappa,w}:=\int_0^1\mathrm D\mathscr Q_{\kappa'}(c_\kappa+tw)|_{\mathscr H'}\,\mathrm{d}t,
 \qquad
 \mathbb L_{c_\kappa,w}(q,v):=J_{\kappa',\kappa,c_\kappa}q+B_{c_\kappa,w}v.
\end{equation*}
Thus $\mathbb L_{c_\kappa,w}:\mathcal C\times\mathscr H'\to L^2_{\kappa'}\cap\ker\hat D$ and
$\mathscr Q_{\kappa'}(h)=\mathbb L_{c_\kappa,w}(\mathscr Q_\kappa(c_\kappa),w)$.

Boundedness of the
fixed projections on $L^2_{\kappa'}$, as established in
\eqref{eq:finite-LS-fixed-splittings}, and the background normal gap
\cref{ineq:weakened-gaussian-normal-gap} imply
\begin{equation*}
 |q|+\|v\|_{L^2_{\kappa'}}+\|\hat Lv\|_{L^2_{\kappa'}}
 \leq C_0\|q+\hat Lv\|_{L^2_{\kappa'}}
 \quad(q\in\mathcal C,\ v\in\mathscr H').
\end{equation*}
Indeed, $\Pi_{\mathcal C}(q+\hat Lv)=q$ and
$\Pi_{\mathscr F'}(q+\hat Lv)=\hat Lv$, and the normal gap controls $v$.
At $c_\kappa=0$, $J_{\kappa',\kappa,0}q=\Pi_{\kappa'}q=q$ for
$q\in\mathcal C\subset\mathscr F$.
Operator-norm continuity of the Gaussian reconstruction in $c_\kappa$ and finite-dimensionality of $\mathcal C$
therefore give, after shrinking the neighborhood,
\begin{equation*}
 \|(J_{\kappa',\kappa,c_\kappa}-I)q\|_{L^2_{\kappa'}}\leq\epsilon|q|.
\end{equation*}
Also, \cref{ineq:corrected-orthogonal-normal-graph-smallness} gives
\begin{equation*}
 \|(B_{c_\kappa,w}-\hat L)v\|_{L^2_{\kappa'}}
 \leq\epsilon\|\hat Lv\|_{L^2_{\kappa'}}.
\end{equation*}
Choose $C_0\epsilon\leq\frac12$ and absorb these errors in the
background estimate.  Applying the result to
$(q,v)=(\mathscr Q_\kappa(c_\kappa),w)$ yields
\begin{equation}\label{ineq:finite-LS-block-inverse}
 |\mathscr Q_\kappa(c_\kappa)|+\|w\|_{L^2_{\kappa'}}
       +\|\hat Lw\|_{L^2_{\kappa'}}
 \leq C\|\mathscr Q_{\kappa'}(h)\|_{L^2_{\kappa'}}.
\end{equation}

\emph{Step 4.}
Put $E:=\|\mathscr Q_{\kappa'}(h)\|_{L^2_{\kappa'}}$.
The transition map restricted to $\mathcal C$ and
\cref{ineq:finite-LS-block-inverse} give
\begin{equation*}
 \|\mathscr Q_{\kappa'}(c_\kappa)\|_{L^2_{\kappa'}}
 \leq C|\mathscr Q_\kappa(c_\kappa)|\leq CE.
\end{equation*}
The forward derivative estimate
\cref{ineq:corrected-orthogonal-normal-graph-smallness} also gives
\begin{equation*}
 \|\mathrm D\mathscr Q_{\kappa'}(h_t)[w]\|_{L^2_{\kappa'}}
 \leq C\|\hat Lw\|_{L^2_{\kappa'}}\leq CE.
\end{equation*}
Integrating from $0$ to $t$, and then using the uniformly bounded
reconstruction in \cref{prop:corrected-orthogonal-fixed-target-normal-estimate},
we obtain
\begin{equation*}
 \sup_{0\leq t\leq1}
 \left(\|\mathscr Q_{\kappa'}(h_t)\|_{L^2_{\kappa'}}
                     +\|\Phi(g_t)\|_{L^2_{\kappa'}}\right)\leq CE.
\end{equation*}

The metrics $g_t$ are uniformly equivalent to
$\hat g$, and the normalized-potential estimates give
$|f_{g_t}-\hat f|\leq C\boldsymbol{\rho}$ uniformly in $t$.
Since $\hat f=\frac14\boldsymbol{\rho}^2+\boldsymbol{\mu}(\hat g)-1$, the exact density ratio satisfies
\begin{equation*}
 \frac{\mathrm dV_{f_{g_t}}}{\mathrm dV_{\hat f,\kappa'}}
 =Z_{\kappa'}e^{-\kappa'\hat f-(f_{g_t}-\hat f)}
                  \frac{\mathrm dV_{g_t}}{\mathrm dV_{\hat g}}
 \leq C.
\end{equation*}
Thus the first variation and Cauchy's inequality imply
\[
\left|\frac{\mathrm{d}}{\mathrm{d}t}\boldsymbol{\mu}(g_t)\right|
 =\left|\int_M\langle \Phi(g_t),w\rangle_{g_t}\,\mathrm dV_{f_{g_t}}\right| 
 \leq C\|\Phi(g_t)\|_{L^2_{\kappa'}}
           \|w\|_{L^2_{\kappa'}}\leq CE^2.
\]
Integrating over $0\leq t\leq1$ proves
\begin{equation}\label{ineq:finite-LS-normal-energy-quadratic}
 |\boldsymbol{\mu}(\hat g+h)-\boldsymbol{\mu}(\hat g+c_\kappa)|
 \leq C\|\mathscr Q_{\kappa'}(h)\|_{L^2_{\kappa'}}^2.
\end{equation}
Combining \cref{ineq:finite-LS-center-Lojasiewicz},
\cref{ineq:finite-LS-center-pairing}, and
\cref{ineq:finite-LS-block-inverse} gives
$|\boldsymbol{\mu}(\hat g+c_\kappa)-\boldsymbol{\mu}(\hat g)|^{1-\theta}\leq CE$.
Since $0<1-\theta<1$, subadditivity of this power yields
\begin{equation*}
 |\boldsymbol{\mu}(\hat g+h)-\boldsymbol{\mu}(\hat g)|^{1-\theta}
 \leq C\bigl(E+E^{2(1-\theta)}\bigr).
\end{equation*}
Here $2(1-\theta)\geq1$, so the second term is at most $CE$ after
shrinking the neighborhood so that $E\leq1$.  Therefore
\begin{equation}\label{ineq:finite-LS-weakened-Gaussian}
 |\boldsymbol{\mu}(\hat g+h)-\boldsymbol{\mu}(\hat g)|^{1-\theta}
 \leq C\|\mathscr Q_{\kappa'}(h)\|_{L^2_{\kappa'}}.
\end{equation}

\emph{Step 5.}
Choose $a\in(1-\theta,1-\kappa')$, which is possible by
\cref{eq:finite-LS-physical-delta-choice}. Put $U=|\Phi(g)|_{\hat g}$ and
$\omega_h=\mathrm dV_{\hat f,\kappa'}/\mathrm dV_{f_g}$.  Then
\begin{equation*}
 \int_M U^2\,\mathrm dV_{\hat f,\kappa'}
 \leq\left(\int_M U^2\,\mathrm dV_{f_g}\right)^a
      \left(\int_M U^2\omega_h^{1/(1-a)}\,\mathrm dV_{f_g}\right)^{1-a}.
\end{equation*}
Writing $j_h=\mathrm dV_g/\mathrm dV_{\hat g}$, the measure in the second factor is
\begin{equation*}
 \omega_h^{1/(1-a)}\mathrm dV_{f_g}
 =C_{a,\kappa'}j_h^{-a/(1-a)}
   \exp\left(-\frac{1-\kappa'-a}{1-a}\hat f
             +\frac{a}{1-a}(f_g-\hat f)\right)\mathrm dV_{\hat g}.
\end{equation*}
Here $1-\kappa'-a>0$, $f_g-\hat f=O(\boldsymbol{\rho})$, and $j_h$ is uniformly
bounded above and below.  The uniform polynomial bound on $\Phi(g)$
therefore makes the second factor uniformly finite.
The metric equivalence and the orthogonal contraction of $\Pi_{\kappa'}$
now give
\begin{equation}\label{ineq:finite-LS-physical-power-comparison}
 \|\mathscr Q_{\kappa'}(h)\|_{L^2_{\kappa'}}
 \leq\|\Phi(g)\|_{L^2_{\kappa'}}
 \leq C_a\|\Phi(g)\|_{L^2(\mathrm dV_{f_g})}^a.
\end{equation}
Combining this with \cref{ineq:finite-LS-weakened-Gaussian} and taking
the power $1/a$ proves \cref{ineq:finite-global-metric-LS} with
\begin{equation*}
 \vartheta:=1-\frac{1-\theta}{a}.
\end{equation*}
The inequalities $1-\theta<a<1$ give
$0<\vartheta<\theta\leq\frac12$, as asserted.
\end{proof}

The same inequality and the following pair estimates hold for finite
regularity at the stated orders; all derivatives and integrations below
are classical at these orders and are justified by Gaussian cutoffs.
Here $f$ denotes an arbitrary normalized potential for $g$, whereas
$f_g$ is the entropy minimizer for $g$.

\begin{Theorem}[Global \L ojasiewicz--Simon inequality for $\mathcal{W}$]\label{thm:global-pair-LS}
    Let $(M^n,\hat{g},\hat{f})$ be a complete nontrivial AC shrinker.
    There exist $\epsilon,C>0$ and $\vartheta\in(0,\frac12]$ such
    that every Riemannian metric $g$ and potential $f$ satisfying
    \begin{equation}\label{eq:finite-pair-hypotheses}
     \begin{gathered}
      g-\hat g\in E_{14},\qquad f-\hat f\in C^{3,1/4}_1,\qquad
      \int_M\mathrm dV_f=1,\\
      \|g-\hat g\|_{E_{14}}+\|f-\hat f\|_{C^{3,1/4}_1}<\epsilon
     \end{gathered}
    \end{equation}
    obey
    \begin{equation}\label{ineq:global-pair-LS}
        |\mathcal{W}(g,f)-\mathcal{W}(\hat{g},\hat{f})|^{1-\vartheta}
        \leq C\left(\int_M  \Big|\Ric_g+\nabla_g^2f-\frac{1}{2}g\Big|_g^2  \,\mathrm dV_f  \right)^{1/2}.
    \end{equation}
\end{Theorem}

The following two lemmas also hold for a small $\mathscr H$ metric
and a normalized potential with $\|f-\hat f\|_{C^{3,1/4}_1}$ small.
Together with the metric inequality they prove the theorem.

\begin{Lemma}\label{lem:nonlinear-Hessian-Poincare}
    There exist $\epsilon,C>0$ such that every pair satisfying
    \eqref{eq:finite-pair-hypotheses} obeys
    \begin{equation}\label{ineq:nonlinear-Hessian-Poincare}
        \|\nabla_g(f-f_g)\|_{L^2(\mathrm dV_{f_g})}
        \leq C\|\nabla_g^2(f-f_g)\|_{L^2(\mathrm dV_{f_g})}.
    \end{equation}
\end{Lemma}
\begin{proof}
By \cite[Proposition~3.3(i)]{LYZ-I}, the weighted rough Laplacian
$-\hat\Delta_{\hat f}$ on one-forms has discrete spectrum and its
kernel consists of parallel forms. The nontrivial AC shrinker has no
nonzero parallel one-forms, as used in the proof of
\cref{thm:asymptotic-estimates-on-minimizer}. Hence, for some $c_*>0$,
\[
 c_*\|\omega\|_{L^2_{\hat f}}^2
 \leq\|\hat\nabla\omega\|_{L^2_{\hat f}}^2,
 \qquad \omega\in H^1_{\hat f}(M,T^*M).
\]
The same proof gives persistence of this gap for nearby metrics and
weighted measures. Indeed, let $J_g:(T^*M,g)\to(T^*M,\hat g)$ be the
positive fiberwise isometry and set
\[
 T_g\omega=e^{-(f_g-\hat f)/2}
       \left(\frac{\mathrm dV_g}{\mathrm dV_{\hat g}}\right)^{1/2}J_g\omega.
\]
This is an isometry of the two $L^2$ spaces. By
\cref{thm:existence-and-uniqueness-of-minimizer}, both
$\|h\|_{C^1}$ and $\|\mathrm d(f_g-\hat f)\|_{C^0}$ are small.
Differentiating $T_g$ and comparing the connections therefore gives
\[
 \|\hat\nabla(T_g\omega)\|_{L^2_{\hat f}}
 \leq(1+\delta)\|\nabla^g\omega\|_{L^2(\mathrm dV_{f_g})}
       +\delta\|\omega\|_{L^2(\mathrm dV_{f_g})},
 \qquad \delta\longrightarrow0\quad\text{as }h\longrightarrow0.
\]
Apply the background gap and absorb the last term. After shrinking
the neighborhood, this yields
\begin{equation}\label{ineq:nearby-one-form-gap}
 \|\omega\|_{L^2(\mathrm dV_{f_g})}
 \leq C\|\nabla^g\omega\|_{L^2(\mathrm dV_{f_g})}.
\end{equation}
The calculation is first made for compactly supported forms and
extends to the Gaussian Sobolev space by density. Finally take
$\omega=\mathrm d(f-f_g)$; its polynomial growth and that of its
covariant derivative ensure membership in this space. This proves
\cref{ineq:nonlinear-Hessian-Poincare}.
\end{proof}

\begin{Lemma}\label{lem:quadratic-pair-energy}
    For every $a\in(0,1)$, there exist $\epsilon,C_a>0$ such that
    every pair satisfying \eqref{eq:finite-pair-hypotheses} obeys
    \begin{equation}\label{ineq:pair-relative-entropy-defect}
        0\leq\mathcal{W}(g,f)-\boldsymbol{\mu}(g)
        \leq\int_M|\nabla(f-f_g)|^2\,\mathrm dV_f
        \leq C_a
        \|\nabla_g^2(f-f_g)\|_{L^2(\mathrm dV_{f_g})}^{2a}.
    \end{equation}
\end{Lemma}

\begin{proof}
    The first inequality $\mathcal{W}(g,f)\geq\boldsymbol{\mu}(g)$ follows from the very definition of the $\boldsymbol{\mu}$-entropy.
    Put $u=f-f_g$ and $p=e^{-u}$, so $\mathrm dV_f=p\,\mathrm dV_{f_g}$ and $\int_Mp\,\mathrm dV_{f_g}=1$.
    Expanding $f=f_g+u$ in the definition of $\mathcal{W}(g,f)$ and using the Euler--Lagrange equation
    \begin{equation*}
        2\Delta f_g-|\nabla f_g|^2+\mathrm{R}_g+f_g-n=\boldsymbol{\mu}(g),
    \end{equation*}
    we obtain
    \begin{equation}\label{eq:difference-of-W-and-mu}
        \mathcal{W}(g,f)-\boldsymbol{\mu}(g)
        =\int_M\bigl(  |\nabla u|^2+u-2\Delta_f f_g  \bigr)  \,\mathrm dV_f.
    \end{equation}
    Weighted integration by parts gives $\int_M\Delta_f f_g\,\mathrm dV_f=0$.
    Moreover, by Jensen's inequality,
    \begin{equation*}
        \int_M  u  \,\mathrm dV_f
        =-\int_M  p\log p  \,\mathrm dV_{f_g}  \leq 0.
    \end{equation*}
    Combining these with \cref{eq:difference-of-W-and-mu}, we obtain
    \begin{equation*}
        \mathcal{W}(g,f)-\boldsymbol{\mu}(g)
        \leq\int_M  |\nabla u|^2  \,\mathrm dV_f.
    \end{equation*}
    Finally, H\"older's inequality gives
    \begin{equation*}
        \int_M|\nabla u|^2e^{-u}\,\mathrm dV_{f_g}
        \leq\left(\int_M|\nabla u|^2\,\mathrm dV_{f_g}\right)^a
        \left(\int_M|\nabla u|^2e^{-u/(1-a)}\,\mathrm dV_{f_g}\right)^{1-a}
        \leq C_a\|\nabla_g u\|_{L^2(\mathrm dV_{f_g})}^{2a}.
    \end{equation*}
    Applying \cref{lem:nonlinear-Hessian-Poincare} proves the last inequality in \cref{ineq:pair-relative-entropy-defect}.
\end{proof}

Now we present the proof of \cref{thm:global-pair-LS}.
\begin{proof}[Proof of \cref{thm:global-pair-LS}]
    Apply \cref{thm:finite-coordinate-slice} to $g$ and pull back the pair
    $(g,f)$ simultaneously. The sliced metric is small in $\mathscr H$,
    and \cref{lem:finite-potential-pullback} makes the transformed
    arbitrary potential small in $C^{3,1/4}_1$. Work with this transformed
    pair and its unique
    minimizer, and at the end use diffeomorphism invariance
    to return to the original pair.
    Let $\vartheta_0\in(0,1/2)$ be the exponent in \cref{thm:finite-global-metric-LS}.
    Then
    \begin{equation*}
        |\boldsymbol{\mu}(g)-\boldsymbol{\mu}(\hat{g})|^{1-\vartheta_0}
        \leq C\|\Phi(g)\|_{L^2(\mathrm dV_{f_g})}.
    \end{equation*}
    Choose numbers
    \begin{equation*}
        1-\vartheta_0<b<1,
        \qquad \frac{1}{2b}<a<1.
    \end{equation*}
    These choices are possible because $1-\vartheta_0\geq1/2$, and they give $2ab>1$.

    Set $u=f-f_g$.
    The weighted Bianchi identity $\delta_{f_g}^g\Phi(g)=0$, following
    from diffeomorphism invariance and the first variation
    \cite[Lemma~2.1]{cao2012second}, and integration by parts give
    \begin{equation*}
        \int_M\langle \Phi(g),\nabla_g^2u\rangle\,\mathrm dV_{f_g}=0.
    \end{equation*}
    Since $\Phi(g,f)=\Phi(g)-\nabla_g^2u$, the above $L^2(\mathrm dV_{f_g})$-orthogonality implies
    \begin{equation}\label{eq:pair-Pythagorean-identity}
        \|\Phi(g,f)\|_{L^2(\mathrm dV_{f_g})}^2
        =\|\Phi(g)\|_{L^2(\mathrm dV_{f_g})}^2
        +\|\nabla_g^2u\|_{L^2(\mathrm dV_{f_g})}^2.
    \end{equation}
    Write $E:=\|\Phi(g,f)\|_{L^2(\mathrm dV_f)}$.
    By decreasing $\epsilon$ if necessary, we may assume $E\leq1$.
    Apply H\"older's inequality to $U=\Phi(g,f)$:
    \begin{equation*}
        \int_M|U|_g^2\,\mathrm dV_{f_g}
        \leq\left(\int_M|U|_g^2\,\mathrm dV_f\right)^b
        \left(\int_M|U|_g^2e^{bu/(1-b)}\,\mathrm dV_{f_g}\right)^{1-b}.
    \end{equation*}
    The second integral is uniformly bounded on a sufficiently small
    neighborhood: $u=O(\boldsymbol\rho)$, $U$ has uniformly controlled
    polynomial growth, and $\mathrm dV_{f_g}$ has Gaussian decay.
    Hence $\|\Phi(g,f)\|_{L^2(\mathrm dV_{f_g})}\leq C_bE^b$.
    Together with \cref{eq:pair-Pythagorean-identity}, this gives
    \begin{equation*}
        \|\Phi(g)\|_{L^2(\mathrm dV_{f_g})}\leq C_bE^b,
        \qquad
        \|\nabla_g^2u\|_{L^2(\mathrm dV_{f_g})}\leq C_bE^b.
    \end{equation*}
    The \L ojasiewicz--Simon inequality for $\boldsymbol{\mu}$ therefore gives
    \begin{equation}\label{ineq:pair-metric-energy-physical-power}
        |\boldsymbol{\mu}(g)-\boldsymbol{\mu}(\hat{g})|
        \leq C E^{b/(1-\vartheta_0)}.
    \end{equation}
    Moreover, \cref{lem:quadratic-pair-energy} gives
    \begin{equation}\label{ineq:pair-potential-energy-physical-power}
        0\leq\mathcal W(g,f)-\boldsymbol{\mu}(g)
        \leq C_a\|\nabla_g^2u\|_{L^2(\mathrm dV_{f_g})}^{2a}
        \leq CE^{2ab}.
    \end{equation}
    Taking
    \begin{equation*}
        \gamma:=\min\left\{  \frac{b}{1-\vartheta_0},2ab  \right\}>1,  \qquad
        1-\vartheta:=\frac1\gamma,
    \end{equation*}
    the inequalities \cref{ineq:pair-metric-energy-physical-power} and \cref{ineq:pair-potential-energy-physical-power} then imply
    \begin{equation*}
        |\mathcal{W}(g,f)-\boldsymbol{\mu}(\hat{g})|
        \leq CE^{\gamma}.
    \end{equation*}
    Clearly, $\gamma<2$ and hence $\vartheta\in(0,\frac{1}{2}]$.
\end{proof}

\section{Strong uniqueness of AC tangent flows}
\label{s10:section}\label{sec:tangent-flow-uniqueness}

In this section, we prove strong uniqueness when one tangent flow is generated by a complete nontrivial AC shrinker. We first establish the geometric and potential estimates on a closed Ricci flow $M\times[-2,0)$ with a uniform entropy lower bound. Both radius functions in this part of the argument refer to the single slice $t=-1$. The geometric radius is defined by metric closeness and one bound on the kernel potential at a base point. Spacetime estimates are consequences of this definition.

The organization follows the radius argument of Fang--Li \cite[Section 6]{fangli2026strong}, with two changes appropriate to the AC setting. The entropy radius uses the instantaneous soliton defect, and the complete extension retains a small change of the asymptotic cone.

In the applications below, we choose a fixed integer $q\ge100$
sufficiently large for the finite-order comparison and extension estimates.
This choice is made after fixing the exponents and relative losses, and
before choosing the geometric tolerance $\sigma$ and the subsequent radius
and time thresholds. Higher-order convergence follows from the parabolic
estimates established below.

\subsection{Basic setups}
\label{s10:setup}

Throughout Subsections~\ref{s10:setup}--\ref{s10:cutoff}, let
\[
 (M^n,g(t))_{-2\le t<0}
\]
be a smooth closed connected Ricci flow with entropy bounded below by $-Y$, in the sense of \cite[Definition~2.20]{fangli2026structure}. We use the Ricci flow spacetime, heat kernels, and spacetime completion of \cite{fangli2026structure}. Write $x^*=(x,t)$ for a spacetime point and $\mathfrak t(x^*)=t$ for its time component. For $x^*,y^*\in M\times[-1,0)$, put
\[
 t_+=\max\{\mathfrak t(x^*),\mathfrak t(y^*)\},\qquad
 t_-=\min\{\mathfrak t(x^*),\mathfrak t(y^*)\}.
\]
The spacetime distance is
\[
 d^*(x^*,y^*)=\inf_{-1.98\le u\le t_-}
 \max\left\{\sqrt{t_+-u},\,
 W_1^{g(u)}(\nu_{x^*;u},\nu_{y^*;u})\right\},
\]
where $\nu_{x^*;u}$ is the conjugate heat-kernel measure based at $x^*$, and $W_1^{g(u)}$ is the Wasserstein distance with cost $d_{g(u)}$; see \cite[Definition~3.2 and Section~9]{fangli2026structure}. Denote the completion of $M\times[-1,0)$ by $(Z,d_Z)$, so $d_Z=d^*$ on the original negative-time spacetime.\phantomsection\label{idx:completion-metric} The time function $\mathfrak t$ extends continuously to $Z$, and the added points lie in $Z_0=\mathfrak t^{-1}(0)$.

Let $z\in Z_0$. The heat-kernel construction in \cite[Section~5]{fangli2026structure}, extended to earlier slices by the reproduction formula as in its Section~9, gives, for $\tau=-t$,
\begin{equation}\label{s10:physicalkernel}
 \mathrm{d}\nu_t=K(z;\cdot,t)\,\mathrm dV_{g(t)}
       =(4\pi\tau)^{-n/2}e^{-f(t)}\,\mathrm dV_{g(t)}.
\end{equation}
The kernel is smooth and positive on every negative-time slice and is the limit there of kernels based at regular terminal points.

Fix a connected complete nontrivial AC shrinker $(\hM,\hg,\hf)$, normalized by
\[
 \Ric_{\hg}+\hnd^2\hf=\frac12\hg,\qquad
 (4\pi)^{-n/2}\int_{\hM}e^{-\hf}\,\mathrm dV_{\hg}=1.
\]
Fix a minimum point $\hat p$ of $\hf$. Denote its canonical Ricci flow by
\begin{equation}\label{s10:canonical}
 \hg(t)=(-t)\Psi_t^*\hg,\qquad \hf(t)=\Psi_t^*\hf,\qquad
 \partial_t\Psi_t=\frac1{-t}\hnd\hf\circ\Psi_t,\quad \Psi_{-1}=\Id.
\end{equation}
On each canonical end in \eqref{eq:prelim-canonical-end}, write $g_C=\mathrm{d}r^2+r^2g_\Sigma$ and $\mathbf{X}=\hnd\hf=\frac12r\partial_r$. In these coordinates $\Psi_{-\tau}(r,\theta)=(r/\sqrt{\tau},\theta)$ whenever the points in question remain on the end. Consequently,
\[
 \sum_{j=0}^{m}r^j|\nabla^{C,j}(\hg(t)-g_C)|_{g_C}
 \le C_m\tau r^{-2},\qquad r\ge A_0,\quad -1\le t<0.
\]

The heat-kernel upper bound in \cite[Theorem~5.30]{fangli2026structure} gives $f(t)\ge-C_0(n,Y)$ on $[-1,0)$. Since
\[
 \hf-\boldsymbol{\mu}(\hg)=\mathrm R_{\hg}+|\hnd\hf|_{\hg}^2\ge0,
\]
we may set
\[
 C_*=1+\max\{C_0(n,Y),-\boldsymbol{\mu}(\hg)\}.
\]
Then $f(t)+C_*\ge1$ on $[-1,0)$ and $\hf+C_*\ge1$. Put
\begin{equation}\label{s10:radial}
 b=2\sqrt{f(-1)+C_*},\qquad
 \hb=2\sqrt{\hf+C_*},\qquad \Omega_L=\{\hb<L\}\subset\hM.
\end{equation}
In particular, $b,\hb\ge2$; we use these positive radii directly in all radial weights. On the end $\hb=r+O(r^{-1})$, with differentiated expansions. Unless a time is displayed, the symbols
\begin{equation}\label{s10:slicenotation}
 g=g(-1),\quad f=f(-1),\quad \mathrm dV_f=\mathrm{d}\nu_{-1},\quad
 \Phi=\Phi(g,f),\quad \mathcal D=\|\Phi\|_{L^2(\mathrm dV_f)}
\end{equation}
always refer to this Ricci-flow time slice, with $\Phi(g,f)$ as in \eqref{eq:prelim-soliton-defect}.
We also write the entropy and, within this section, the pointwise derivative sum as
\[
 \mathcal W(g,f)=\int_M(\mathrm{R}_g+|\nabla f|^2+f-n)\,\mathrm dV_f,\qquad
 [u]_m=\sum_{j=0}^{m}|\hnd^ju|_{\hg}.
\]

We use the weighted integer norms $C^\ell_\beta$ from Section~\ref{sec:prelim-weighted-spaces}, also on restricted domains. Throughout this section, the weights are zero for metric perturbations, two for potentials, and one for vector fields of linear growth. Differentiation orders are fixed before the small tolerances and radius thresholds. The H\"older-based strong topology is used when applying \cref{thm:global-pair-LS}; the comparison is given in Lemma~\ref{s10:completion}.

\begin{Lemma}\label{s10:kernel}
For $-1\le t<0$, every $H_n$-center $z_H(t)$ of $z$ at time $t$, and every $\eta>0$,
\begin{equation}\label{s10:kernelpoint}
 f(x,t)\ge\frac{d_{g(t)}(z_H(t),x)^2}{(4+\eta)\tau}-C_\eta,
\end{equation}
\begin{equation}\label{s10:kerneltail}
 \int_{M\setminus B_{g(t)}(z_H(t),L\sqrt{\tau})}
 \bigl(1+\tau|\mathrm{R}|+\tau|\nabla f|^2+\tau|\Delta f|+|f|\bigr)\,\mathrm{d}\nu_t
 \le C_\eta e^{-L^2/(4+\eta)}.
\end{equation}
The constants depend on $n,Y,\eta$, and not on the flow. Moreover $\mathrm{R}_{g(t)}\ge-n/2$ on this time interval and $-Y-n\le\mathcal W(g,f)\le0$.
\end{Lemma}
\begin{proof}
The scalar maximum principle gives $\mathrm R_{g(t)}\ge-n/[2(t+2)]$. Equation~\eqref{s10:kernelpoint} follows directly from \cite[Theorem~5.30]{fangli2026structure}. For \eqref{s10:kerneltail}, apply \cite[Proposition~3.1]{fangli2026strong}. The entropy lower bound, $\tau\le1$, and $\tau\mathrm R_{\min}\ge-n/2$ make the constants uniform. For a completion point, approximate the terminal point by regular ones. Their kernels converge smoothly and positively on each fixed earlier slice; centers have a convergent subsequence, and their second-moment bounds pass to the limit. The same tail estimate follows by adjusting $L$ and $\eta$.

Finally, \cite[Proposition~2.8]{fangli2026strong} gives
\[
 \mathcal W_z(1)=\mathcal N_z(1)+\mathcal N'_z(1)\le0,
 \qquad \mathcal N_z(1)\ge-Y,\qquad \mathcal N'_z(1)\ge-n.
\]
These inequalities also pass to a completion point by the preceding kernel convergence and prove the asserted entropy bound.
\end{proof}

\subsection{Single-time radius functions}
\label{s10:radii}

\begin{Definition}[Geometric radius]\label{s10:defC}
Fix an integer $q\ge4$ and $0<\sigma<1$. The number $r_C^{(\sigma,q)}$ is the supremum of the radii $L$ for which there is a smooth embedding $\iota$ of a neighborhood of $\overline{\Omega_L}$ into $M$ satisfying
\begin{equation}\label{s10:Cdefinition}
 \|\iota^*g(-1)-\hg\|_{C^{q}_0(\Omega_L)}<\sigma,
 \qquad f(\iota(\hat p),-1)\le n.
\end{equation}
Only radii for which $\hat p\in \Omega_L$ are considered; the supremum is zero if there are none. Assertions involving $r_C^{(\sigma,q)}$ are first proved for $L<r_C^{(\sigma,q)}$ and an embedding.
\end{Definition}

When the parameters $\sigma$ and $q$ are understood, we write $r_C$ for $r_C^{(\sigma,q)}$.
The constants in the following estimates are fixed independently of $L$. A reserved collar always means that, for fixed numbers $0<\lambda<\lambda'<1$, the estimate on $\iota(\Omega_{\lambda L})$ is proved using the larger domain $\iota(\Omega_{\lambda'L})$. Both numbers can be chosen arbitrarily close to one, before the thresholds for $\sigma$ and $L$ are chosen.

\begin{Definition}[Entropy radius]\label{s10:defEP}
Define the single-time entropy radius by
\begin{equation}\label{s10:rE}
 r_E=\sqrt{8\log(1/\mathcal D)}\quad(0<\mathcal D<1),\qquad
 r_E=+\infty\quad(\mathcal D=0),\qquad r_E=0\quad(\mathcal D\ge1).
\end{equation}
Thus $\mathcal D=e^{-r_E^2/8}$ when $0<\mathcal D<1$.
\end{Definition}

We will construct matched product coordinates on regions whose radius is an arbitrarily large fixed fraction of $\min\{r_C,r_E\}$, once this minimum is sufficiently large. Figure~\ref{fig:s10:dependencies} records the order of the principal parameter choices.

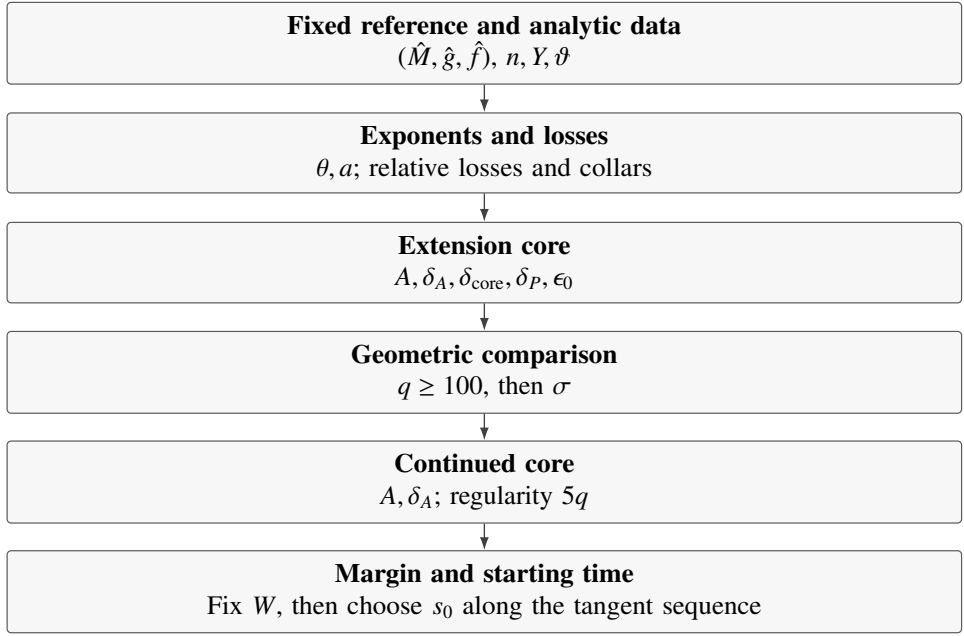
\begin{figure}[H]
\centering
\begin{tikzpicture}[
  x=1cm,y=1cm,
  font=\fontsize{10}{11.5}\selectfont,
  box/.style={draw=black!65,line width=.45pt,rounded corners=1.5pt,
    fill=black!3,align=center,inner xsep=6pt,inner ysep=5pt,
    text width=12.2cm,minimum height=.85cm},
  arrow/.style={-{Latex[length=1.7mm,width=1.2mm]},line width=.5pt,draw=black!75}
]
\node[box] (fixed) at (0,0)
  {\textbf{Fixed reference and analytic data}\\
   \mbox{$(\hM,\hg,\hf)$}, \mbox{$n,Y,\vartheta$}};
\node[box] (free) at (0,-1.45)
  {\textbf{Exponents and losses}\\
   \mbox{$\theta,a$}; relative losses and collars};
\node[box] (extension) at (0,-2.90)
  {\textbf{Extension core}\\
   \mbox{$A,\delta_A,\delta_{\mathrm{core}},\delta_P,\epsilon_0$}};
\node[box] (thresholds) at (0,-4.35)
  {\textbf{Geometric comparison}\\
   \mbox{$q\ge100$, then $\sigma$}};
\node[box] (core) at (0,-5.80)
  {\textbf{Continued core}\\
   \mbox{$A,\delta_A$; regularity $5q$}};
\node[box] (start) at (0,-7.25)
  {\textbf{Margin and starting time}\\
   Fix \mbox{$W$}, then choose \mbox{$s_0$} along the tangent sequence};
\draw[arrow] (fixed)--(free);
\draw[arrow] (free)--(extension);
\draw[arrow] (extension)--(thresholds);
\draw[arrow] (thresholds)--(core);
\draw[arrow] (core)--(start);
\end{tikzpicture}
\caption{Principal parameter choices in Section~\ref{s10:section}. The extension core and the continued core involve independent choices of $A,\delta_A$, made before and after $q,\sigma$, respectively. The margin $W$ is fixed before choosing $s_0$ along the tangent sequence.}
\label{fig:s10:dependencies}
\end{figure}

\subsection{Forward geometry and potential estimates}
\label{s10:pseudo}

\begin{Lemma}\label{s10:shell}
Fix $0<\lambda<\lambda'<1$ and an integer $m\ge0$. There are $A_0,L_0,\sigma_0$ such that, if $q\ge m+4$, $L\ge L_0$, and $\iota$ satisfies \eqref{s10:Cdefinition} with $\sigma\le\sigma_0$, the following holds. For $A\ge A_0$, define
\[
 V_{A,\lambda L}=\{x\text{ on the model end}:2A<r(x)<\lambda L\},
 \qquad U_{A,\lambda L}=\iota(V_{A,\lambda L}).
\]
We use balls whose larger initial balls lie in
$\iota(\{A<r<\lambda'L\})\subset\iota(\Omega_L)$; enlarge $L_0$ to allow for $\hb-r=O(r^{-1})$. With $\boldsymbol r=r\circ\iota^{-1}$ on this fixed material domain,
\begin{equation}\label{s10:physicalcurv}
 |\nabla_{g(t)}^j\Rm_{g(t)}|(x)\le C_j\boldsymbol r(x)^{-2-j},
 \quad x\in U_{A,\lambda L},\quad -1\le t<0,\quad 0\le j\le m .
\end{equation}
The time-independent pullback $h(t)=\iota^*g(t)$ satisfies
\begin{equation}\label{s10:material}
 \sum_{j=0}^{m}r^j|\hnd^j(h(t)-h(-1))|_{\hg}
 \le C_m r^{-2},
 \quad x\in V_{A,\lambda L},\quad -1\le t<0.
\end{equation}
The smaller balls are noncollapsed, uniformly in $L$ and $t$. This lemma also holds for a single collared annulus on which the metric hypothesis holds.
\end{Lemma}
\begin{proof}
At radius $r_0\ge2A$, choose a ball of radius $c r_0$, where $c>0$ depends on the link and the two collars. In cone coordinates this ball is as close to a Euclidean ball as required after taking $c$ small, $A$ large, and $\sigma$ small. Its initial curvature is at most $(c r_0)^{-2}$ and its volume is at least $(1-\delta)\omega_n(c r_0)^n$, with $\delta$ as in the bounded-initial-curvature version of forward pseudolocality, \cite[Theorem 10.3]{perelman2002entropy}. That theorem gives curvature at most $C r_0^{-2}$ on a smaller ball for a time $c_1r_0^2$. Take $A$ so that this time is at least one. The same conclusion is used up to the available terminal time if the smooth flow ends sooner.

Choose the smaller covering balls strictly inside the balls supplied by the theorem. Integration of the Ricci bound and a first-exit argument keep their fixed material subsets inside those balls for $-1\le t<0$. To obtain the derivative estimates, write $\rho_0=cr_0$ and rescale the flow by
\[
 \widetilde g(s)=\rho_0^{-2}g(-1+\rho_0^2s),
 \qquad 0\le s<\rho_0^{-2}.
\]
The retained initial ball now has fixed radius, and pseudolocality bounds its curvature uniformly throughout this interval. The initial $C^q_0$ comparison gives
\[
 |\widetilde\nabla^{\,j}\mathrm{Rm}_{\widetilde g(0)}|
 \le C_j,\qquad 0\le j\le q-2,
\]
on a slightly larger fixed ball. Apply the local derivative estimate with initial derivative bounds, \cite[Theorem~14.16]{chow2008ricci}, on nested balls. If the initial curvature derivatives are bounded through order $\ell$, that theorem bounds the derivative of order $j$ by $C_js^{-(j-\ell)_+/2}$, where $(x)_+=\max\{x,0\}$. Choose $\ell\ge m$; the assumption $q\ge m+4$ supplies these initial bounds and the derivatives needed below. For $j\le m$ the time exponent is zero, so the estimates are uniform down to $s=0$ and independent of the length $\rho_0^{-2}$ of the rescaled interval. Scaling back gives
\[
 |\nabla_{g(t)}^j\mathrm{Rm}_{g(t)}|
 \le C_j\rho_0^{-2-j}\le C_j'r_0^{-2-j},
 \qquad 0\le j\le m,
\]
which proves \eqref{s10:physicalcurv}.

Integrate $\partial_th=-2\Ric_h$ and
$\partial_t\Gamma(h)=h^{-1}*\nabla_h\Ric_h$ on the fixed chart. Inductively converting to the initial connection proves \eqref{s10:material}; the initial metric comparison converts that connection to $\hnd$. The volume-form equation and metric comparison preserve the volume lower bounds. These arguments use only times at or after $-1$.
\end{proof}

The next proposition is analogous to \cite[Proposition~6.2]{fangli2026strong}; in the AC setting we treat the conical exterior and the compact core separately.

\begin{Proposition}\label{s10:forwardcomparison}
For every $0<\lambda<1$, $\epsilon>0$, $0<a<1$, and finite order $m$, there are $q_0,\sigma_0,L_0$ such that an embedding satisfying \eqref{s10:Cdefinition} with $q\ge q_0$, $\sigma\le\sigma_0$, $L\ge L_0$ satisfies
\begin{equation}\label{s10:forwardclose}
 \sup_{-1\le t\le-a}
 \|\iota^*g(t)-\hg(t)\|_{C^{m}_0(\Omega_{\lambda L};\,\hg)}
 <\epsilon .
\end{equation}
On a fixed exterior subdomain as in Lemma~\ref{s10:shell}, the metric comparison with $g_C$ and the curvature estimates hold for the full interval $[-1,0)$. The core comparison \eqref{s10:forwardclose} is asserted only through $-a$.
\end{Proposition}
\begin{proof}
Choose $\lambda<\lambda'<1$, and put $h(t)=\iota^*g(t)$ on the fixed model domain. On its exterior,
\[
 h(t)-\hg(t)=[h(t)-h(-1)]+[h(-1)-\hg]+[\hg-\hg(t)].
\]
Lemma~\ref{s10:shell} bounds the first term in scaled derivatives through order $m$ by $C_mr^{-2}$. The initial comparison bounds the second by $C_m\sigma$. For the last term, the canonical end expansion gives
\[
 \sum_{j=0}^m r^j|\hnd^j(\hg(t)-g_C)|_{\hg}
 \le C_m(-t)r^{-2},\qquad -1\le t<0;
\]
the same estimate at $t=-1$ bounds $\hg-g_C$. Consequently,
\[
 \|h(t)-\hg(t)\|_{C^m_0(\{2A<r<\lambda L\})}
 \le C_m\sigma+C_mA^{-2},\qquad -1\le t<0.
\]
Take $A$ large and then $\sigma$ small so that this is less than $\epsilon/2$. The larger collar determined by $\lambda'$ supplies all the balls used here; the difference $\hb-r=O(r^{-1})$ is absorbed by increasing the radius threshold.

Suppose now that the asserted comparison on the remaining fixed core fails. The negation of the proposition gives flows $g_i(t)$ and embeddings $\iota_i$ with
\[
 q_i\ge i+m+4,\qquad \sigma_i\longrightarrow0,\qquad L_i\longrightarrow\infty,
\]
for which \eqref{s10:forwardclose} fails. Write $h_i(t)=\iota_i^*g_i(t)$. The exterior estimate confines this failure to a fixed compact core $K$. After reducing $\epsilon$ by a fixed norm-equivalence constant if necessary, there are $t_i\in[-1,-a]$, $x_i\in K$, and $j_i\le m$ such that
\[
 |\hnd^{j_i}(h_i(t_i)-\hg(t_i))|_{\hg}(x_i)\ge c_m\epsilon.
\]
Pass to a subsequence with $t_i\to t_\infty$, $x_i\to x_\infty$, and $j_i=j$.

We show that $h_i\to\hg(t)$ smoothly on compact subsets throughout $[-1,-a]$. At time $-1$ the convergence is smooth on every compact subset of $\hM$. Cover an enlarged core by finitely many small balls on which the reference metric has controlled curvature and almost Euclidean volume ratio. Initial convergence and forward pseudolocality give a uniform curvature bound on smaller balls for a common time interval $[-1,-1+\delta]$. The derivative estimate with initial data used in Lemma~\ref{s10:shell} gives every fixed curvature derivative bound there, including at $-1$. The exterior estimates hold independently on the rest of the exhausting model domains.

These bounds give convergence in the original maps $\iota_i$. Indeed, integrating $\partial_th_i=-2\Ric_{h_i}$ first gives uniform metric equivalence. The equation
\[
 \partial_t\Gamma(h_i)=h_i^{-1}*\nabla_{h_i}\Ric_{h_i}
\]
and its differentiated equations then bound the connection difference from the fixed initial connection and its derivatives. Induction converts the intrinsic curvature bounds to bounds for all fixed-coordinate derivatives of $h_i$ on compact sets. The flow equation also controls the time derivatives. A diagonal compactness argument therefore produces a smooth limit $h_\infty(t)$ on $\hM\times[-1,-1+\delta]$, with $h_\infty(-1)=\hg$.

The core and exterior curvature bounds are uniform globally. Integrating the Ricci bound makes $h_\infty(t)$ uniformly equivalent to the complete metric $\hg$, hence complete. Uniqueness for complete bounded-curvature Ricci flows, \cite[Theorem~1.1]{chen2006uniqueness}, identifies $h_\infty(t)$ with the canonical flow $\hg(t)$.

The same construction can be restarted at any time already reached. On $[-1,-a]$ the canonical flow has uniform bounded geometry on the enlarged core, so its finite covering, curvature bounds, and pseudolocality constants can be chosen uniformly in the restart time. The independently controlled exterior completes each limit flow. Thus there is a common positive continuation step, and finitely many steps give convergence in the same fixed maps on every compact subset of $\hM\times[-1,-a]$. This contradicts the displayed violation at $(x_i,t_i)$. Combining the core convergence with the exterior estimate proves \eqref{s10:forwardclose}.
\end{proof}

\begin{Lemma}\label{s10:annulardistance}
Fix $0<\lambda<\lambda'<1$, $0<a<1$, and $\eta>0$. After the choices in Proposition~\ref{s10:forwardcomparison}, put $p=\iota(\hat p)$. For $x$ on the end with $2A\le r(x)\le\lambda L$ and $-1\le t\le-a$,
\begin{equation}\label{s10:distance}
 \bigl|d_{g(t)}(p,\iota(x))-r(x)\bigr|\le\eta r(x)+C_{A,a,\eta},
 \qquad d_{g(t)}(p,z_H(t))\le C_{a,\eta}.
\end{equation}
For every $R\le\lambda L$ above a fixed threshold,
\begin{equation}\label{s10:Couter}
 M\setminus\iota(\Omega_R)\subset
 M\setminus B_{g(t)}(z_H(t),(1-\eta)R-C_{A,a,\eta}).
\end{equation}
The constants and thresholds may depend on the reference shrinker, $n,Y,\lambda,\lambda',a,\eta$, and on $A$ where indicated, but are independent of $L$ and the particular flow.

At time $-1$, assuming only \eqref{s10:Cdefinition} with $q\ge4$ and sufficiently small $\sigma$, one also has, for $B\ge1$,
\[
 \{f\le B\}\subset\iota(\Omega_{\lambda L})
 \quad\text{if}\quad L>C_\lambda(1+\sqrt{B+C_*}),
\]
with $\iota(\Omega_{\lambda'L})$ providing the larger collar. Given $0<\zeta<1$, the collars can be chosen sufficiently close to one and $\sigma$ sufficiently small so that
\[
 \{b\le(1-\zeta)L-C_\zeta\}\subset\iota(\Omega_{\lambda L}).
\]
These are entire sublevels in $M$; the constants and geometric thresholds in these time-$-1$ assertions are independent of higher derivative orders.
\end{Lemma}
\begin{proof}
On the exterior the metric differs from $g_C$ by as little as desired in $C^0$. A radial path from a fixed inner level to $x$, preceded by a path in the controlled compact core, has length at most $(1+\eta)r(x)+C$. For the lower bound, any curve from $p$ to $\iota(x)$ which remains in $\iota(\Omega_{\lambda'L})$ has length at least $(1-\eta)r(x)-C$, by integrating the radial function once the curve leaves the core. If it leaves that chart first, its initial segment crosses the outer radial level and has length at least $(1-\eta)\lambda'L-C$. This is at least the required lower bound for $r(x)\le\lambda L$, after adjusting the thresholds. The same first-exit argument at $\partial\iota(\Omega_R)$ proves the lower distance to the complement.

At $t=-1$, $f(p,-1)\le n$ and the Gaussian kernel upper bound imply that $p$ is a bounded distance from an $H_n$-center. For $-1<t\le-a$, reproduction gives
\[
 K(z;p,-1)=\int_M K(w,t;p,-1)\,\mathrm d\nu_t(w)
 \ge(4\pi)^{-n/2}e^{-n}.
\]
Forward heat-kernel localization about $p$ on the controlled core, together with Gaussian concentration of $\nu_t$ about $z_H(t)$, would make this integral arbitrarily small if $d_{g(t)}(p,z_H(t))$ were large. The forward-time argument of \cite[Claim~6.4]{fangli2026strong} therefore gives $d_{g(t)}(p,z_H(t))\le C_a$, which also transfers the distance-to-complement bound to $z_H(t)$.

At time $-1$, the initial metric comparison alone gives the same first-exit estimate. Combining it with \eqref{s10:kernelpoint} yields, for any prescribed $\delta>0$ and sufficiently small geometric tolerance,
$b\ge(1-\delta)R-C_\delta$ on $M\setminus\iota(\Omega_R)$. Taking $R=\lambda L$ proves the bounded-sublevel assertion. Choosing $\lambda<\lambda'<1$ close to one and then $\delta$ small proves the assertion with loss $\zeta$.
\end{proof}

The next proposition is analogous to \cite[Proposition~6.3 and Corollary~6.6]{fangli2026strong}.

\begin{Proposition}\label{s10:rough}
Fix $0<\lambda<\lambda'<1$.
\begin{enumerate}[label=(\roman*)]
\item For every integer $m\ge0$ and $\eta>0$, there exist a finite order $q_0\ge4$ and constants $\sigma_0>0$ and $L_0<\infty$ such that, whenever $q\ge q_0$, $0<\sigma\le\sigma_0$, and $L\ge L_0$, an embedding $\iota$ of \eqref{s10:Cdefinition} has the following properties. These thresholds are chosen after $m,\eta$, and the collars are fixed, and may depend on them. Set
$\boldsymbol r=r\circ\iota^{-1}$ on the end. For some fixed $A$ and every $x\in\iota(\{2A\le r\le\lambda'L\})$,
\begin{equation}\label{s10:roughpotential}
 \left|f(x,t)-\frac{\boldsymbol r(x)^2}{4\tau}\right|
 \le\eta\frac{\boldsymbol r(x)^2}{\tau}+C_\eta,
 \qquad -1\le t\le-\frac12 .
\end{equation}
On the fixed core $\iota(\Omega_{4A})$ the potential is uniformly bounded on this same time interval. At time $-1$, for $x\in\iota(\Omega_{\lambda L})$, put $\rho_x=b(x)^{-1}$. Then
\begin{align}
 &\begin{aligned}
 &\sum_{j=0}^{m}b^{2+j}|\nabla^j\Rm_g|\le C_m,\qquad
 \inj_g(x)\ge c\rho_x,\\
 &B_g(x,4\rho_x)\subset\iota(\Omega_{\lambda'L}),\qquad
 \Vol_g B_g(x,\rho_x)\ge c\rho_x^n,
 \end{aligned}\label{s10:roughcurv}\\
 &\sum_{j=0}^{m+2}|\nabla^jf|+\sum_{j=0}^{m}|\nabla^j\Phi|
 \le C_{m,\eta}b^{N_m}e^{\eta b^2},
 \label{s10:roughderivatives}\\
 &\osc_{B_g(x,2\rho_x)}f\le C_\eta+\eta b(x)^2 .
 \label{s10:roughosc}
\end{align}
The nonnegative integer exponent $N_m=N_m(n,m)$ in \eqref{s10:roughderivatives} can be chosen independently of $\eta$; the constant and the geometric thresholds may depend on $\eta$. In particular, \eqref{s10:roughderivatives} holds for each prescribed finite $m$, with the preceding choices allowed to depend on $m$. All derivatives and norms in \eqref{s10:roughcurv}--\eqref{s10:roughosc} are intrinsic for $g=g(-1)$. The conically weighted curvature estimate in \eqref{s10:roughcurv} uses the prescribed finite initial order.

\item Fix any one order $q\ge4$ in Definition~\ref{s10:defC}. There is a small geometric tolerance $\sigma_*>0$ such that, for every $B\ge1$, there is $L_B<\infty$ with the following property. If $0<\sigma\le\sigma_*$ and an embedding has $L\ge L_B$, then
\[
 E_B:=\{x\in M:f(x,-1)\le B\}\subset\iota(\Omega_{\lambda L}),
\]
and, for every integer $m\ge0$,
\[
 \sup_{E_B}\left(
 \sum_{j=0}^{m}|\nabla^j\Rm_g|
 +\sum_{j=0}^{m+2}|\nabla^jf|
 +\sum_{j=0}^{m}|\nabla^j\Phi|
 \right)\le C_{B,m}.
\]
The fixed order $q$, the tolerance $\sigma_*$, and the threshold $L_B$ are independent of $m$; only the bound $C_{B,m}$ changes with the derivative order. The set $E_B$ is the entire sublevel in $M$. These are intrinsic, unweighted bounds at time $-1$.
\end{enumerate}
Constants in both parts are independent of $L$ and of the particular flow, with dependence on the fixed model, $n,Y$, and the stated parameters understood.
\end{Proposition}
\begin{proof}
We retain two further relative collars between $\lambda'$ and $1$ whenever applying the preceding estimates. They are fixed once, so this costs an arbitrarily small relative loss.

\smallskip\noindent\emph{(i):}
Choose $0<\epsilon<1/4$, to be made small relative to the desired loss, and apply the forward comparison through $-\epsilon$. Let $z_1$ be an $H_n$-center of $z$ at time $-\epsilon$. The set
\[
 E_1=B_{g(-\epsilon)}(z_1,\sqrt{2H_n\epsilon})
\]
has $\nu_{-\epsilon}$-mass at least $1/2$ and is a bounded distance from $p$. Fix $t=-\tau\in[-1,-1/2]$ and $x$ in the retained chart. Proposition~\ref{s10:forwardcomparison} bounds the curvature along $\{x\}\times[t,-\epsilon]$, so \cite[Corollary~9.5]{perelman2002entropy} gives
\[
 K(x,-\epsilon;x,t)\ge c_\epsilon(\tau-\epsilon)^{-n/2}.
\]
The forward heat-kernel spatial comparison used in \cite[proof of Proposition 6.3]{fangli2026strong} then gives, for any $a_1>1$,
\[
 K(w,-\epsilon;x,t)\ge c_{a_1,\epsilon}(\tau-\epsilon)^{-n/2}
 \exp\!\left(-\frac{1+a_1}{4a_1(\tau-\epsilon)}
                 d_{g(-\epsilon)}(w,x)^2\right).
\]
Its hypotheses are the entropy and scalar lower bounds and the displayed diagonal bound. The distance lemma bounds $d_{g(-\epsilon)}(w,x)$ by $(1+\delta)\boldsymbol r(x)+C$ for $w\in E_1$. Reproduction, restricted to $E_1$, therefore gives
\[
 K(z;x,t)\ge c\tau^{-n/2}
 \exp\!\left(-\frac{(1+a_1)(1+\delta)^2(1+\xi)}
                    {4a_1(1-\epsilon/\tau)}
                   \frac{\boldsymbol r(x)^2}{\tau}\right).
\]
The constant absorbs the additive distance term by Young's inequality with parameter $\xi$. Take $a_1$ large and then $\delta,\xi,\epsilon$ small. Since $\tau\ge1/2$, the coefficient is as close to $1/4$ as prescribed. This proves the upper potential bound. The lower bound follows from \eqref{s10:kernelpoint}, the center estimate, and the lower distance bound in \eqref{s10:distance}, again absorbing the additive constant with a smaller quadratic loss. The same calculation with bounded distances gives the core potential bound.

A coarse instance of \eqref{s10:roughpotential} gives
$\boldsymbol r\asymp b$ sufficiently far out on the end; on a fixed core, $b$ is uniformly bounded. The initial metric comparison and the model bound for $|\mathrm{d}\hb|_{\hg}$ give
\[
 |\mathrm{d}(\hb\circ\iota^{-1})|_g\le C
 \quad\text{on }\iota(\Omega_{\lambda'L}).
\]
Every curve from $x\in\iota(\Omega_{\lambda L})$ to the complement of this larger chart must cross its boundary. Integrating the displayed bound up to the first exit gives
\[
 d_g\bigl(x,M\setminus\iota(\Omega_{\lambda'L})\bigr)
 \ge c_0(\lambda'-\lambda)L.
\]
Since $b\ge2$, we have $4\rho_x=4/b(x)\le2$. Enlarge $L_0$ so that the last distance is greater than $2$. This proves the ambient ball inclusion in \eqref{s10:roughcurv} for the specified radius $\rho_x=b(x)^{-1}$. The initial geometry of the fixed AC model, the metric comparison, and the compact core bounds give the curvature estimate, $\inj_g(x)\ge c\rho_x$, and $\Vol_g B_g(x,\rho_x)\ge c\rho_x^n$, after adjusting the uniform constants.

The conjugate kernel satisfies $-\partial_tK=\Delta_{g(t)}K-\mathrm R_{g(t)}K$. Choose a small uniform $r_0>0$ so that the balls of radius $2r_0$ lie in the larger chart and $4r_0^2<1/2$. By standard interior parabolic regularity, applied in the forward variable $\theta=-t$, the finite-order geometry bounds give
\[
 \frac{|\nabla^\ell K|(z;x,-1)}{K(z;x,-1)}
 \le C_\ell\sup_{y,t}\frac{K(z;y,t)}{K(z;x,-1)},
\]
where the supremum is over $y\in B_g(x,2r_0)$ and $-1\le t\le-1+4r_0^2$. On the end, \eqref{s10:roughpotential} and $|\boldsymbol r(y)-\boldsymbol r(x)|\le C$ bound this ratio, for every small auxiliary loss $\xi>0$, by
\[
 C_\xi\exp\!\left((1/4+\xi)\boldsymbol r(x)^2
 -(1/4-\xi)(-t)^{-1}\boldsymbol r(y)^2\right)
 \le C_\xi e^{2\xi\boldsymbol r(x)^2+C\boldsymbol r(x)}.
\]
Differentiating $f=-\log K-\frac n2\log(4\pi(-t))$ and choosing $\xi$ small in terms of $m,\eta$ proves the potential derivative bounds in \eqref{s10:roughderivatives}. The definition of $\Phi$ gives the defect bounds. Polynomial factors can be absorbed into a smaller auxiliary Gaussian loss, so $N_m$ can be chosen independently of $\eta$. On the fixed core the kernel ratio is uniformly bounded.

For $x$ sufficiently far out on the end, the ball $B_g(x,2\rho_x)$ remains on the end, and the radial Lipschitz bound gives
\[
 |\boldsymbol r(y)-\boldsymbol r(x)|\le Cb(x)^{-1},\qquad
 \osc \boldsymbol r^2\le C,\qquad
 \sup \boldsymbol r^2\le Cb(x)^2+C.
\]
Write $f(y)=\boldsymbol r(y)^2/4+e(y)$ using \eqref{s10:roughpotential} with the smaller loss $\zeta_0$. For any two points in that ball,
\[
 |f(y)-f(y')|\le\frac14|\boldsymbol r(y)^2-\boldsymbol r(y')^2|
       +\zeta_0\bigl(\boldsymbol r(y)^2+\boldsymbol r(y')^2\bigr)
       +2C_{\zeta_0}
 \le C_{\zeta_0}+C\zeta_0 b(x)^2.
\]
Choose $\zeta_0\le\eta/C$. This proves \eqref{s10:roughosc}. For the remaining points, the ball lies in a fixed enlarged core, where the uniform potential bound applies.

\smallskip\noindent\emph{(ii):}
For the bounded-sublevel assertion, Lemma~\ref{s10:annulardistance} places the entire set $E_B$ in a retained chart with a fixed collar once $L\ge L_B$, where $L_B$ is independent of $m$. The $C^2$ geometric comparison supplies uniformly noncollapsed balls of a fixed radius with bounded curvature at every point of $E_B$. Since $t=-1$ is an interior time of $[-2,0)$, two-sided pseudolocality \cite[Theorem~2.14]{fangli2026strong} and local Shi estimates give all intrinsic curvature derivatives on smaller fixed spacetime cylinders. The geometric tolerance, radii, and threshold $L_B$ are independent of the derivative order.

Standard interior parabolic estimates for the conjugate kernel in the variable $\theta=-t$, together with its uniform upper bound, give all spatial kernel derivatives at $t=-1$. On $E_B$ one has $K(z;x,-1)\ge(4\pi)^{-n/2}e^{-B}$; the first derivative bound preserves a positive lower bound on a neighborhood of radius depending only on $B$ and the fixed data. Logarithmic differentiation then gives every intrinsic derivative of $f$, and $\Phi=\frac12g-\Ric_g-\nabla_g^2f$ gives those of the defect. Thus the asserted bounds hold on $E_B$ and on such neighborhoods, with constants $C_{B,m}$, while $q,\sigma_*,L_B$ remain independent of $m$.
\end{proof}

\subsection{Single-time estimates and product coordinates}
\label{s10:product}

All statements in this subsection concern $(g,f)=(g(-1),f(-1))$ from Subsection~\ref{s10:setup}. Constants are uniform over the Ricci flows under consideration, after the finite orders, collars, and losses are fixed.

The following lemma is analogous to \cite[Proposition~6.13]{fangli2026strong}, using the instantaneous defect norm.

\begin{Lemma}\label{s10:interpolation}
Fix an integer $k\ge0$, $0<a<1$, and $\eta>0$. Fix a relative collar. Choose $q$ large, $\sigma$ small, and the geometric-radius threshold large enough for Proposition~\ref{s10:rough}(i) at the required orders and losses. There is a nonnegative integer $N=N(n,k,a)$, independent of $\eta$ and of the radius parameters, such that on that retained part of an embedding for $r_C$ one has
\begin{equation}\label{s10:instant}
 \sum_{j=0}^kb^{2+j}|\nabla^j\Phi|(x)
 \le Cb^N
 \exp\!\left(-\frac a8\bigl(r_E^2-b(x)^2\bigr)+\eta b(x)^2\right)
\end{equation}
when $0<\mathcal D<1$. The constant $C$ may also depend on $\eta$, the fixed collars, the entropy bound, and the reference shrinker. The exponent $N$ is independent of $L$, $r_E$, the point $x$, and the particular flow. In particular, for every $\sigma_1>0$, after fixing sufficiently small rough losses,
\begin{equation}\label{s10:instantuniform}
 \sum_{j=0}^kb^{2+j}|\nabla^j\Phi|
 \le CR_*^N e^{-c_{\sigma_1}R_*^2}
 \quad\text{on }\{b\le(1-\sigma_1)R_*\},
 \qquad R_*=\min\{r_C,r_E\},
\end{equation}
with the fixed additive collar losses understood. The set in this assertion is the entire sublevel, by Lemma~\ref{s10:annulardistance}. More generally, without restricting $\mathcal D$ to be less than one, for every $\kappa_1>a/2$,
\begin{equation}\label{s10:weightedvelocityintrinsic}
 \sum_{j=0}^k|\nabla^j\Phi|(x)
 \le C_{k,a,\kappa_1}e^{\kappa_1(f(x)+C_*)}\mathcal D^a
\end{equation}
on $\iota(\Omega_{\lambda L})$, where $L<r_C$ and $\iota$ is an embedding satisfying \eqref{s10:Cdefinition}; the larger chart $\iota(\Omega_{\lambda'L})$ supplies the collar. The orders, tolerance, and rough losses are chosen for $k,a,\kappa_1$ and these fixed collars.
\end{Lemma}
\begin{proof}
Put $\rho_x=b(x)^{-1}$ and $\mathcal B_x=B_g(x,2\rho_x)$. Proposition~\ref{s10:rough}(i) supplies the larger ball $B_g(x,4\rho_x)$ and uniform local geometry after rescaling by $\rho_x^{-2}$. Potential oscillation gives
\begin{equation}\label{s10:unweightedL2}
 \|\Phi\|_{L^2(\mathcal B_x,\mathrm dV_g)}
 \le C_\eta e^{b(x)^2/8+\eta b(x)^2}\mathcal D.
\end{equation}
Choose $m=m(n,k,a)$ sufficiently large and put
\[
 M_x=\sum_{i=0}^m\rho_x^i\|\nabla^i\Phi\|_{L^\infty(\mathcal B_x)},
 \qquad L_x=\rho_x^{-n/2}\|\Phi\|_{L^2(\mathcal B_x)}.
\]
The local interpolation argument of \cite[Appendix~B, Lemma~B.1]{colding2015uniqueness}, applied to tensor components on a uniformly controlled coordinate ball and using H\"older's inequality to pass from $L^1$ to $L^2$, gives
\[
 \rho_x^j|\nabla^j\Phi|(x)\le C M_x^{1-a}L_x^a,
 \qquad 0\le j\le k.
\]
Indeed, for sufficiently large $m$ the interpolation exponent is greater than $a$; when $L_x\le M_x$ this gives the displayed bound, and when $L_x>M_x$ the independent bound by $M_x$ suffices.

For every small auxiliary loss $\xi>0$, \eqref{s10:roughderivatives} and \eqref{s10:unweightedL2} give
\[
 M_x\le C b(x)^{N_m}e^{\xi b(x)^2},\qquad
 L_x\le C b(x)^{n/2}e^{(1/8+\xi)b(x)^2}\mathcal D.
\]
Here the rough loss is first reduced so that replacing $b(y)$ by its bound in terms of $b(x)$ on $\mathcal B_x$ costs at most $\xi b(x)^2$. Since $N_m$ is independent of this loss, interpolation yields a fixed integer $N=N(n,k,a)$ such that
\begin{equation}\label{s10:intermediatevelocity}
 \sum_{j=0}^k b^{2+j}|\nabla^j\Phi|(x)
 \le Cb^N e^{(a/8+\eta)b^2}\mathcal D^a,
\end{equation}
after choosing $\xi<\eta$. Substituting $\mathcal D=e^{-r_E^2/8}$ proves \eqref{s10:instant}. On $b\le(1-\sigma_1)R_*$, the negative quadratic term dominates the chosen rough losses, proving \eqref{s10:instantuniform}.

For \eqref{s10:weightedvelocityintrinsic}, use $b^2=4(f+C_*)$ in \eqref{s10:intermediatevelocity} and absorb the polynomial and the arbitrarily small Gaussian loss into $\kappa_1-a/2$. For $\mathcal D\ge1$ the rough estimate gives the same conclusion. When $\mathcal D=0$, positivity of the density implies $\Phi=0$ on the whole slice.
\end{proof}

\begin{Lemma}\label{s10:corematching}
Fix a model compact set $\Omega_{2A}$, a finite order $\ell$, and $\delta_0>0$. There are a finite order $q_0$, $\sigma_0>0$, and $L_0<\infty$ such that, if an embedding $\iota$ in Definition~\ref{s10:defC} has $q\ge q_0$, $\sigma\le\sigma_0$, and
\[
 L\ge L_0,\qquad r_E\ge L_0,
\]
then
\begin{equation}\label{s10:corepair}
 \|\iota^*g-\hg\|_{C^{\ell}(\Omega_{2A})}
 +\|\iota^*f-\hf\|_{C^{\ell}(\Omega_{2A})}<\delta_0.
\end{equation}
Consequently a large regular inner model level can be matched to the actual potential level, with induced metric, lapse, and coordinate change as small as prescribed.
\end{Lemma}
\begin{proof}
If the assertion failed, take $\sigma_i\to0$, $L_i\to\infty$, and $r_{E,i}\to\infty$ while keeping a failing compact set and tolerance fixed. Thus $\mathcal D_i\to0$. The potential and derivative bounds of Proposition~\ref{s10:rough}(i) give, after a diagonal subsequence, a limit $f_\infty$ on all of $\hM$. The metric limits are $\hg$. Since the kernel densities have positive upper and lower bounds on each compact set, $\mathcal D_i\to0$ implies in distributions, and hence smoothly after regularity,
\[
 \Ric_{\hg}+\hnd^2 f_\infty=\frac12\hg.
\]
It follows that $\hnd^2(f_\infty-\hf)=0$. As in \cite[Section~2.3]{LYZ-I}, AC curvature decay excludes a line factor unless the shrinker is Gaussian. Hence $f_\infty=\hf+c$.

The mass normalization fixes the constant. The first-exit estimate \eqref{s10:Couter} at $t=-1$ and the Gaussian mass bound give
\[
 \int_{M_i\setminus\iota_i(\Omega_R)}\mathrm dV_{f_i}\le C e^{-cR^2}
\]
for each large fixed $R$ and all sufficiently large $i$, with constants independent of $i,R$. Pass to the limit on $\Omega_R$ and then let $R\to\infty$. The limiting density has total mass one, so
$1=(4\pi)^{-n/2}\int e^{-\hf-c}\,\mathrm dV_{\hg}=e^{-c}$.
Thus $c=0$. The finite-order compact convergence contradicts the failure of \eqref{s10:corepair}. This proves the quantitative assertion by compactness. On the regular level $\hb=A$, the implicit function theorem and the two radial flows now give the last statement.
\end{proof}

\begin{Proposition}\label{s10:potentialcoords}
Fix an integer $k\ge0$. Suppose an attached component $\Omega$ of $\{A\le b\le B\}$ has compact closure, boundary contained in its two potential levels, and a compact regular inner level $\Sigma_A$. Assume
\begin{equation}\label{s10:prodassump}
 cA\le|\nabla f|\le CA\quad\text{on }\Sigma_A,
 \qquad
 b^{2+j}|\nabla^j\Rm|\le C_j,\quad
 b^{2+j}|\nabla^j\Phi|\le\epsilon_0
\end{equation}
for $0\le j\le k+6$. After an identification with the reference inner level, assume also that the induced metric and lapse on $\Sigma_A$ are $\delta_A$-close to the reference data in scaled $C^{k+6}$ norms. The induced metric is divided by $A^2$ in this comparison, and the lapse is $|\nabla b|^{-1}$. For $A$ large and $\delta_A,\epsilon_0$ small, the following hold.
\begin{enumerate}[label=(\roman*)]
\item The vector field
\begin{equation}\label{s10:radialfield}
 V_b=\frac{\nabla b}{|\nabla b|^2}
     =\frac b{2|\nabla f|^2}\nabla f
\end{equation}
gives a diffeomorphism $\Sigma_A\times[A,B]\to\Omega$. Matching it with the reference potential coordinates gives
\begin{equation}\label{s10:productbound}
 \begin{gathered}
 f\circ\Psi=\hf\quad\text{on }\{A\le\hb\le B\},\\
 \|\Psi^*g-\hg\|_{C^{k+3}_0(\{A\le\hb\le B\})}
 \le C_k(\delta_A+A^{-2}+\epsilon_0).
 \end{gathered}
\end{equation}
\item\label{s10:relative} In these matched potential coordinates, suppress the pullback and write
\[
 g=a^2\mathrm db^2+b^2\gamma_b,\qquad
 \hg=\hat a^2\mathrm db^2+b^2\hat\gamma_b,\qquad
 f=\hf=b^2/4-C_*.
\]
Define
\begin{equation}\label{s10:Q}
 Q=\mathrm R_g+|\nabla f|^2-f,\qquad
 \hat Q=\mathrm R_{\hg}+|\hnd\hf|^2-\hf.
\end{equation}
Let $\delta$ be a bound for $\|g-\hg\|_{C^{k+3}_0(A\le b\le B)}$, and put
\[
 \epsilon_B=\sum_{j=0}^{k+2}\sup_{A\le b\le B}b^{2+j}|\nabla^j\Phi|.
\]
Set
\begin{equation}\label{s10:etaB}
 \kappa_A=\sup_{\Sigma_A}|Q-\hat Q|,\qquad
 h_B=\gamma_B-\hat\gamma_B,\qquad
 \eta_B=\delta+\kappa_A+\epsilon_B(1+\log(B/A)).
\end{equation}
Extend $h_B$ independently of $b$ on the fixed product and write
\[
 g-\hg=b^2h_B+w_B.
\]
Then
\begin{align}
 \|w_B\|_{C^{k}_{-2}(A\le b\le B)}
 &\le C_k\eta_B,\label{s10:relativebound}\\
 \|h_B-(\gamma_A-\hat\gamma_A)\|_{C^{k}(\Sigma_A)}
 &\le C_k\eta_B A^{-2}.\label{s10:conetrace}
\end{align}
\end{enumerate}
The constants depend on the stated curvature bounds and the reference geometry, and are independent of $B$.
\end{Proposition}
\begin{proof}
\noindent\emph{(i):}
\smallskip\noindent\emph{Step 1.}
The identity
\begin{equation}\label{s10:hessian}
 \nabla^2f=\frac12g-\Ric-\Phi
\end{equation}
gives $\frac13g\le\nabla^2f\le\frac23g$ on the exterior component. Along a trajectory of $V_b$, parametrized by $b$,
\[
 \frac b3\le\frac{\mathrm d}{\mathrm db}|\nabla f|^2
 =\frac b{|\nabla f|^2}\nabla^2f(\nabla f,\nabla f)\le\frac{2b}3.
\]
The inner gradient bounds imply $cb^2\le|\nabla f|^2\le Cb^2$, so the normal flow continues from $\Sigma_A$ to the outer level. Its image is open by the inverse function theorem and closed by compactness of $\Sigma_A\times[A,B]$. Connectedness gives coverage of the attached component, and ODE uniqueness gives injectivity. Thus these trajectories define the asserted product coordinates.

\smallskip\noindent\emph{Step 2.}
The contracted Bianchi identity gives
\begin{equation}\label{s10:dQ}
 \mathrm dQ=-2\mathrm d(\tr_g\Phi)+2\delta^g\Phi
       -2\Phi(\nabla f,\cdot).
\end{equation}
Hence $|V_bQ|\le C\epsilon_0/b$. Initial lapse matching gives $|Q|\le C(1+A^2\delta_A)$ on $\Sigma_A$. With
\[
 L_A(b)=1+A^2\delta_A+\epsilon_0\log(b/A),
\]
we obtain
\begin{equation}\label{s10:Hbounds}
 |Q|\le CL_A(b),\qquad |\nabla f|^2=\frac{b^2}4+O(L_A(b)),\qquad
 |\nabla b|^2=1+O(L_A(b)b^{-2}).
\end{equation}
Write $g=a^2\mathrm db^2+g_b$, $\gamma_b=b^{-2}g_b$, and $a^2=b^2/(4|\nabla f|^2)$. The exact level-metric equation is
\begin{equation}\label{s10:leveleq}
 \partial_b\gamma_b=\frac{2(a^2-1)}b\gamma_b
              +\frac{4a^2}{b^3}(-\Ric-\Phi)^\top,
\end{equation}
where ${}^\top$ denotes restriction to vectors tangent to the level $\Sigma_b$. We view these tensors on the fixed manifold $\Sigma_A$ through the potential flow and put
\[
 u=a^2-1,\qquad \mathcal S=(-\Ric-\Phi)^\top,\qquad
 a^2=\left(1+\frac{4(Q-\mathrm R_g-C_*)}{b^2}\right)^{-1}.
\]
Then $|u|\le CL_A(b)b^{-2}$ and $|\mathcal S|_{\gamma_b}\le C$, since a covariant two-tensor has norm $|U|_{\gamma_b}=b^2|U|_{g_b}$. Consequently,
\[
 |\partial_b\gamma_b|_{\gamma_b}\le CL_A(b)b^{-3},\qquad
 \int_A^\infty L_A(b)b^{-3}\,\mathrm db
 =\frac{1+A^2\delta_A}{2A^2}+\frac{\epsilon_0}{4A^2}.
\]
This proves uniform equivalence of $\gamma_b$ and $\gamma_A$ and a zeroth-order difference bounded by $C(\delta_A+A^{-2}+\epsilon_0A^{-2})$, independently of $B$.

\smallskip\noindent\emph{Step 3.}
\phantomsection\label{idx:level-connection}
Let $\nabla^b=\nabla^{g_b}=\nabla^{\gamma_b}$ and $\nabla^{b,j}=(\nabla^b)^j$. The normal $\mathbf n=\nabla f/|\nabla f|$ and tangential projection $\Pi$ satisfy
$|\nabla^\ell\mathbf n|+|\nabla^\ell\Pi|\le C_\ell b^{-\ell}$ through the required orders, by \eqref{s10:hessian} and the gradient bounds. In particular, the second fundamental form satisfies $|\mathrm{II}_b|\le Cb^{-1}$. Repeated tangential projection of the ambient connection gives, for an ambient covariant tensor $U$ and a scalar function $\phi$,
\[
 \begin{aligned}
 |\nabla^{b,j}(U^\top)|_{g_b}
 &\le C_j\sum_{\ell=0}^j b^{\ell-j}|\nabla^\ell U|_g,\\
 |\nabla^{b,j}\phi|_{g_b}
 &\le C_j\sum_{\ell=1}^j b^{\ell-j}|\nabla^\ell\phi|_g
 \qquad(j\ge1).
 \end{aligned}
\]
Differentiating \eqref{s10:dQ} gives $|\nabla^jQ|_g\le C_j\epsilon_0b^{-j}$ for positive $j$. Using the restriction estimates and rescaling every covariant index, we obtain
\[
 |\nabla^{b,j}\mathrm R_g|_{\gamma_b}\le C_jb^{-2},\qquad
 |\nabla^{b,j}\mathcal S|_{\gamma_b}\le C_j,\qquad
 |\nabla^{b,j}Q|_{\gamma_b}\le C_j\epsilon_0\quad(j\ge1).
\]
The first two bounds also hold for $j=0$. The reciprocal formula for $a^2$, whose denominator is uniformly positive, therefore gives
\[
 |\nabla^{b,j}u|_{\gamma_b}\le C_j(\epsilon_0b^{-2}+b^{-4})
 \qquad(1\le j\le k+5).
\]
The factor $A^2\delta_A$ affects the value of $Q$ but does not enter these positive tangential derivatives. Since $\nabla^b\gamma_b=0$, differentiating \eqref{s10:leveleq} now yields
$|\nabla^{b,j}(\partial_b\gamma_b)|_{\gamma_b}\le C_jb^{-3}$ for $1\le j\le k+5$.

Set $\nabla^A=\nabla^{\gamma_A}$ and $E_b=\nabla^b-\nabla^A$. Connection variation in the fixed product coordinates gives
\[
 \partial_b\Gamma(\gamma_b)^\ell_{ij}
 =\frac12\gamma_b^{\ell p}
 \bigl(\nabla^b_i\partial_b\gamma_{jp}
      +\nabla^b_j\partial_b\gamma_{ip}
      -\nabla^b_p\partial_b\gamma_{ij}\bigr).
\]
Thus $|\nabla^{b,j}(\partial_bE_b)|_{\gamma_b}\le C_jb^{-3}$ for $0\le j\le k+4$. Converting to $\nabla^A$ introduces only lower derivatives of $E_b$. Induction, integration from $E_A=0$, and the integrability of $b^{-3}$ give
\[
 \sup_{A\le b\le B}|(\nabla^A)^jE_b|_{\gamma_A}\le C_jA^{-2}
 \qquad(0\le j\le k+4).
\]
Metric compatibility, $\nabla^A\gamma_b=E_b*\gamma_b$, then controls the positive tangential derivatives. Together with Step 2 this proves
\[
 \sup_{A\le b\le B}\|\gamma_b-\gamma_A\|_{C^{k+4}(\Sigma_A,\gamma_A)}
 \le C_k(\delta_A+A^{-2}+\epsilon_0).
\]

\smallskip\noindent\emph{Step 4.}
\phantomsection\label{idx:scaled-radial}
Use $D_b=b\partial_b$ on tensors pulled back to $\Sigma_A$. The formula for $V_b$ and \eqref{s10:hessian} give $|\nabla^jV_b|_g\le C_jb^{-j}$. For tangent coordinate vectors $Z_i,Z_j$ transported by the flow, $[V_b,Z_i]=0$; with $\mathcal S_{\mathrm{amb}}=-\Ric-\Phi$, a representative calculation is
\[
 \begin{aligned}
 \partial_b\bigl(\mathcal S_{\mathrm{amb}}(Z_i,Z_j)\bigr)
 ={}&(\nabla_{V_b}\mathcal S_{\mathrm{amb}})(Z_i,Z_j)\\
 &+\mathcal S_{\mathrm{amb}}(\nabla_{Z_i}V_b,Z_j)
   +\mathcal S_{\mathrm{amb}}(Z_i,\nabla_{Z_j}V_b).
 \end{aligned}
\]
Each term has size $O(b^{-1})$ in the rescaled level metric. Radial differentiation of a level covariant derivative adds a contraction with $\partial_bE_b$, controlled by the connection-variation formula. Induction in the radial order, together with the tangential estimates, gives
\[
 |D_b^p(\nabla^A)^j\mathcal S|_{\gamma_A}\le C_{p,j},\qquad
 |D_b^p(\nabla^A)^j\mathrm R_g|_{\gamma_A}\le C_{p,j}b^{-2}.
\]
The differentiated identity \eqref{s10:dQ} similarly gives
$|D_b^p(\nabla^A)^jQ|_{\gamma_A}\le C_{p,j}\epsilon_0$ for $p+j\ge1$. All these estimates are used through total order $k+4$. Differentiating the reciprocal formula and then the equation
$D_b\gamma_b=2u\gamma_b+4a^2b^{-2}\mathcal S$, using $[D_b,\nabla^A]=0$, yields
\[
 \begin{aligned}
 |D_b^p(\nabla^A)^ju|_{\gamma_A}&\le C_{p,j}L_A(b)b^{-2}
 &&(p+j\le k+4),\\
 |D_b^p(\nabla^A)^j\gamma_b|_{\gamma_A}&\le C_{p,j}L_A(b)b^{-2}
 &&(p\ge1,\ p+j\le k+4).
 \end{aligned}
\]
These bounds are uniform in $B$, since $\sup_{b\ge A}L_A(b)b^{-2}\le C(\delta_A+A^{-2}+\epsilon_0A^{-2})$.

Apply the same identities to the reference pair, for which $\Phi=0$ and $Q$ is constant. Initial matching controls $\gamma_A-\hat\gamma_A$ and the change to $\hat\nabla^A=\nabla^{\hat\gamma_A}$, so
\[
 \sup_{A\le b\le B}\sum_{p+j\le k+4}
 \left(|D_b^p(\hat\nabla^A)^j(a^2-\hat a^2)|
       +|D_b^p(\hat\nabla^A)^j(\gamma_b-\hat\gamma_b)|\right)
 \le C_k(\delta_A+A^{-2}+\epsilon_0).
\]
The norms here use $\hat\gamma_A$. On $b\asymp R$, rescaling by $R^{-2}$ converts these to the weighted integer-order metric bounds through order $k+4$. Restriction to order $k+3$ proves \eqref{s10:productbound}. Both potentials equal $b^2/4-C_*$ in the matched coordinates.

\smallskip\noindent\emph{(ii):}
The reference $\hat Q$ is constant. Integrating \eqref{s10:dQ} in the product chart gives
\[
 |Q-\hat Q|\le\kappa_A+C\epsilon_B\log(b/A).
\]
Its positive conic derivatives are bounded by $C_j\epsilon_B$, by differentiating \eqref{s10:dQ} itself. Put $H=|\nabla f|^2$ and $\hat H=|\hnd\hf|^2$. Since the potentials agree,
\[
 H-\hat H=Q-\hat Q-(\mathrm{R}_g-\mathrm{R}_{\hg}),
 \qquad
 a^2-\hat a^2=-\frac{b^2(H-\hat H)}{4H\hat H}.
\]
The given conic metric estimate bounds $\mathrm{R}_g-\mathrm{R}_{\hg}$ in $C^{k+1}_{-2}$ by $C\delta$. Thus $a^2-\hat a^2$ has weight minus two, with norm at most $C\eta_B$. Subtracting \eqref{s10:leveleq} for the two metrics gives
\begin{equation}\label{s10:gammadiff}
 \|\partial_b(\gamma_b-\hat\gamma_b)\|_{C^{k}(\Sigma_A)}
 \le C\eta_Bb^{-3},
\end{equation}
with the corresponding conic derivative estimates. The Ricci difference is controlled by the given conic norm. Integration from $b$ to $B$, together with the lapse estimate, gives \eqref{s10:relativebound}; integration from $A$ to $B$ gives \eqref{s10:conetrace}. Link derivatives commute with the radial integral, and positive radial derivatives follow from \eqref{s10:gammadiff}.

\end{proof}

\begin{Proposition}[Matched product coordinates]\label{s10:radiuscomparison}
Fix an integer $k\ge0$, tolerances $\delta_{\mathrm{core}},\delta_P>0$, and $0<\zeta<1$. For every sufficiently large regular model level $A$, there are a finite order $q$, a geometric tolerance $\sigma>0$, and constants $R_0,C_\zeta$ with the following property. For a Ricci flow in Subsection~\ref{s10:setup}, set
$r_C=r_C^{(\sigma,q)}$ and $R_*=\min\{r_C,r_E\}$. If $R_*\ge R_0$, every
\[
 2A\le B\le(1-\zeta)R_*-C_\zeta
\]
admits a matched product embedding $P$ of a neighborhood of $\overline{\Omega_B}$ into $M$, with a larger reserved relative collar, such that
\[
 \|P^*g-\hg\|_{C^{k+3}(\Omega_{2A})}
 +\|f\circ P-\hf\|_{C^{k+3}(\Omega_{2A})}<\delta_{\mathrm{core}},
\]
and, on $\{A\le\hb\le B\}$,
\begin{equation}\label{s10:Pchart}
 f\circ P=\hf,\qquad
 P^*g=a^2\mathrm d\hb^2+\hb^2\gamma_{\hb},\qquad
 \|P^*g-\hg\|_{C^{k+3}_0}<\delta_P.
\end{equation}
The lapse is positive and uniformly bounded above and below, and the level metrics are uniformly equivalent to the reference level metrics. On a fixed inner collar, $P$ is matched to a core embedding by a controlled change of coordinates. Moreover,
\begin{equation}\label{s10:epsilonB}
 \epsilon_B
 \le C R_*^N e^{-c_\zeta R_*^2}.
\end{equation}
The same bounds hold on the reserved collar, chosen in terms of $\zeta$. The order $q$ is chosen to retain the defect estimate through $k+6$, as required by Proposition~\ref{s10:potentialcoords}(i), and through any further finite order prescribed before the parameters are fixed. All constants are independent of $B,R_*$ and the particular flow.
\end{Proposition}
\begin{proof}
Choose several nested radii strictly between $B$ and $R_*$, distributing the loss $\zeta$ among them and reserving a larger relative collar. Take an embedding from Definition~\ref{s10:defC} with all these retained radii below $L<r_C$. Proposition~\ref{s10:rough}(i) supplies ambient balls and their larger collars. Lemma~\ref{s10:interpolation} gives \eqref{s10:epsilonB}, with all the additional finite orders used below.

To apply Proposition~\ref{s10:potentialcoords}(i), we first use Lemma~\ref{s10:corematching} at a sufficiently high fixed order. It makes the metric and potential close to the reference pair on an enlarged core and supplies the required scaled $C^{k+6}$ inner matching. Hence $b=A$ has a compact regular component, or the prescribed union of components, close to the reference inner level; its lapse and induced metric have the required bounds. The lower bound in \eqref{s10:roughpotential} puts the outer boundary of a retained metric chart above a potential level strictly larger than $B$. Thus the portion of $\{A\le b\le B\}$ attached to the matched inner level cannot reach that boundary. Its closure is compact and its boundary is contained in its two potential levels.

Proposition~\ref{s10:potentialcoords}(i), applied to each attached component, now gives the product chart. Extend the inner level matching through a fixed collar by the two radial flows and then join to the original core embedding by a small isotopy supported in the enlarged core. The implicit-function estimate makes this change small in every prescribed finite norm. The pulled-back potential is exactly $\hf$ on the product end and remains close to $\hf$ on the core. First choose $A$ large for the term $A^{-2}$ in \eqref{s10:productbound}, then the core error small, then $\sigma$ small, and finally the threshold for $R_*$ large. These choices give the prescribed core and product tolerances and the asserted chart with its collar.
\end{proof}

\subsection{Complete extensions and the pointed \texorpdfstring{\L ojasiewicz--Simon}{Lojasiewicz--Simon} inequality}
\label{s10:cutoff}

We continue to work only with the time slice $t=-1$. The product chart is used to place a complete extension in the neighborhood of \cref{thm:global-pair-LS}. The analytic input is smallness in $E_{14}$ and the stated finite potential norm. Proposition~\ref{s10:potentialcoords}(ii) supplies the required decaying remainder.

\begin{Lemma}\label{s10:completion}
Assume that the matched product chart is constructed with sufficiently high fixed regularity. Suppose the matched product chart is available through $B=\Lambda R$, where $\Lambda>1$ is fixed and $R\ge2A$, and the matching extends through the core with error $\delta_{\mathrm{core}}$. There is a complete unnormalized pair $(g_R,\widetilde f_R)$ on $\hM$, equal to the actual pulled-back pair on $\Omega_R$, such that
\begin{equation}\label{s10:strongextension}
 \|g_R-\hg\|_{E_{14}}+\|\widetilde f_R-\hf\|_{C^{3,1/4}_1}
 \le C_A\delta_{\mathrm{core}}+C\eta_B.
\end{equation}
Outside $\Omega_B$ it equals $(\hg+H_0(h_B),\hf)$, where $H_0(h_B)=\hb^2h_B$ on the end. The perturbation belongs to $E_\infty\times C^\infty_{1;0}$. 
\end{Lemma}
\begin{proof}
Apply Proposition~\ref{s10:radiuscomparison} with its parameter equal to
$17$, and use Proposition~\ref{s10:potentialcoords}(ii) through order $15$.
The weighted $C^{15}$ estimates control the $C^{14,1/4}$ quantities
defining $E_{14}$.

Extend $\hb^2h_B$ by an inner cutoff independent of $R$ and denote the resulting tensor by $H_0(h_B)$. On the potential end it is tangential and annihilates $\hnd\hf$. The reference expansion $\hb=r+O(r^{-1})$ writes it as a degree-zero tangential term plus a weight-minus-two remainder, so it satisfies the radial contraction conditions in \cref{def:prelim-radial-filtration}.

Choose $\chi_R(\hb)=1$ on $\Omega_R$ and $\chi_R=0$ outside $\Omega_B$, with $|\hnd^j\chi_R|\le C_{j,\Lambda}R^{-j}$. For the actual metric difference $h$, set
\begin{equation}\label{s10:splice}
 g_R=\hg+H_0(h_B)+\chi_R\bigl(h-H_0(h_B)\bigr).
\end{equation}
Keep the actual potential on the core and set $\widetilde f_R=\hf$ on the product end; these definitions agree on the fixed matching collar. Proposition~\ref{s10:potentialcoords}(ii), at order $15$, controls the cutoff remainder with weight minus two. On the fixed core and uniformly rescaled annuli with their reserved collars,
\[
 \|\cdot\|_{C^{14,1/4}_\beta}
 \le C\|\cdot\|_{C^{15}_\beta}.
\]
Apply this comparison to the link coefficients and the weight-minus-two
remainder in the expansion defining $E_{14}$. The potential perturbation
vanishes on the product end and is controlled on the fixed core, so its
$C^{3,1/4}_1$ norm is at most $C_A\delta_{\mathrm{core}}$.
Since \eqref{s10:conetrace} also bounds $h_B$ on the fixed inner region by
$C_A\delta_{\mathrm{core}}+C\eta_BA^{-2}$, we obtain
\eqref{s10:strongextension}, with constants depending on the fixed $\Lambda$.

The perturbation is smooth and, outside a compact set, has the stated homogeneous metric term and reference potential. Hence it belongs to $E_\infty\times C^\infty_{1;0}$. Small conic norm makes $g_R$ uniformly equivalent to the complete metric $\hg$, so $g_R$ is complete.
\end{proof}

Before comparing the entropies, we normalize the unnormalized extension from Lemma~\ref{s10:completion}. Its total weighted mass need not be one, so put
\begin{equation}\label{s10:normalization}
 c_R=\log\left((4\pi)^{-n/2}\int_{\hM}e^{-\widetilde f_R}\mathrm dV_{g_R}\right),
 \qquad f_R=\widetilde f_R+c_R.
\end{equation}
Then $(4\pi)^{-n/2}\int_{\hM}e^{-f_R}\mathrm dV_{g_R}=1$. Gaussian integrability, uniform metric equivalence, and \eqref{s10:strongextension} give
\[
 |c_R|+\|g_R-\hg\|_{E_{14}}+\|f_R-\hf\|_{C^{3,1/4}_1}
 \le C_A\delta_{\mathrm{core}}+C\eta_B.
\]
Thus the normalized pair $(g_R,f_R)$ satisfies
\eqref{eq:finite-pair-hypotheses} when the displayed bound is sufficiently
small. Exact agreement on $\Omega_R$ holds for the unnormalized pair: if $P$ is the matched chart, then $\widetilde f_R=f\circ P$ and $f_R=f\circ P+c_R$ there. Outside $\Omega_B$, $f_R=\hf+c_R$. Since $c_R$ is constant,
\[
 \Phi(g_R,f_R)=\Phi(g_R,\widetilde f_R),\qquad
 \mathrm dV_{f_R}=e^{-c_R}\mathrm dV_{\widetilde f_R}.
\]
The next lemma gives the sharper Gaussian bound for $c_R$ needed in the entropy comparison.

\begin{Lemma}\label{s10:cutoffcomparison}
Let $(g_R,f_R)$ be the complete extension of Lemma~\ref{s10:completion}, normalized by \eqref{s10:normalization}. Assume its matched product chart $P$ is constructed as in Proposition~\ref{s10:radiuscomparison} inside a metric chart from Definition~\ref{s10:defC}: for some $L<r_C$ and fixed $0<\lambda<\lambda'<1$,
\[
 B=\Lambda R<\lambda L,\qquad
 P(\Omega_B)\subset\iota(\Omega_{\lambda L}),
\]
where $\iota$ satisfies \eqref{s10:Cdefinition}, and $\iota(\Omega_{\lambda'L})$ is the larger domain used for the estimates. Then, for every $\eta\in(0,1)$ and $\eps>0$, the following estimates hold for sufficiently large $R$:
\begin{align}
 |\mathcal W(g,f)-\mathcal W(g_R,f_R)|
 &\le C_{\eta,\eps}e^{-(1-\eta)^2R^2/(4+\eps)}
       +CR^Ne^{-R^2/4},\label{s10:Wcutoff}\\
 \|\Phi(g_R,f_R)\|_{L^2(\mathrm dV_{f_R})}
 &\le C\mathcal D+CR^Ne^{-R^2/8}.\label{s10:Tcutoff}
\end{align}
The constants and the lower threshold for $R$ depend only on the reference pair, $n,Y$, the fixed finite orders and bounds, the core and product matching parameters, the collars, $\Lambda$, and the displayed losses $\eta,\eps$. They are independent of the particular Ricci flow and of $R,B,L$. The tolerance $\sigma$ need only satisfy the chart-construction thresholds in Proposition~\ref{s10:radiuscomparison}; this lemma requires no further smallness depending on $\eta,\eps$.
\end{Lemma}
\begin{proof}
Let $\Omega_R^M$ be the image of the retained core and the attached product end $b<R$. Proposition~\ref{s10:potentialcoords}(i) gives
\[
 a=1+O\bigl(L_A(b)b^{-2}\bigr),\qquad
 \int_A^B\sup_{\Sigma_b}|a-1|\,\mathrm db
 \le \frac{C(1+A^2\delta_A+\epsilon_0)}A,
\]
uniformly in $B$, since $L_A(b)=1+A^2\delta_A+\epsilon_0\log(b/A)$. Choose $p=P(\hat p)$ in the retained core. Core matching bounds $f(p)$, and Lemma~\ref{s10:kernel} gives $d_g(p,z_H)\le C_A$. Any curve from $p$ to $M\setminus\Omega_R^M$ first exits through $b=R$, because $P$ embeds a neighborhood of the closed retained domain. Its segment after the last crossing of $b=A$ has length at least
\[
 \int_A^R\inf_{\Sigma_b}a\,\mathrm db\ge R-C_A.
\]
The triangle inequality therefore yields
\begin{equation}\label{s10:Omegadistance}
 d_g(z_H,M\setminus\Omega_R^M)\ge R-C_A.
\end{equation}
Applying Lemma~\ref{s10:kernel} at distance $R-C_A$ bounds the actual exterior entropy contribution by the first Gaussian in \eqref{s10:Wcutoff}, for $R$ sufficiently large in terms of the fixed losses $\eta,\eps$.

For the artificial exterior, all derivatives and norms in the following estimates are taken with respect to $\hg$, and $b=\hb\ge R$. Write
\[
 g_c=\hg+h_0,\qquad h_0=b^2h_B,\qquad
 g_R=g_c+e_R,\qquad e_R=\chi_Rw_B,
 \qquad \widetilde f_R=\hf=b^2/4-C_*.
\]
Here $h_B$ is extended independently of $b$ in the fixed product coordinates and is purely tangential. Fix $J\ge0$ and use the preceding metric and remainder estimates through order $J+2$. Homogeneity and the controlled link derivatives give
\[
 |\hnd^\ell h_0|\le C_\ell b^{-\ell},\qquad 0\le\ell\le J+2.
\]
In particular, $h_0$ has bounded size but need not decay. Since it annihilates $\partial_b$, the gradient of $\hf$ is unchanged:
\[
 \mathbf{X}=\nabla^{g_c}\hf=\hnd\hf
   =\frac{b}{2\hat a^2}\partial_b.
\]
For a covariant two-tensor $h$ one has
\[
 \mathcal L_{\phi Z_1}h=\phi\mathcal L_{Z_1}h
       +\mathrm{d}\phi\otimes h({Z_1},\cdot)+h({Z_1},\cdot)\otimes \mathrm{d}\phi.
\]
The last two terms vanish for $h_0$ and ${Z_1}=\partial_b$. Since $\mathcal L_{\partial_b}h_0=2bh_B$,
\[
 \mathcal L_{\mathbf{X}}h_0=\hat a^{-2}h_0.
\]
Using $\nabla_g^2\hf=\frac12\mathcal L_{\nabla^g\hf}g$ and the reference soliton equation therefore yields the exact cancellation
\[
 \Phi(g_c,\hf)=\Ric_{\hg}-\Ric_{g_c}
                  +\frac12(1-\hat a^{-2})h_0.
\]
The connection-difference formula gives
\[
 |\hnd^\ell(\nabla^{g_c}-\hnd)|\le C_\ell b^{-1-\ell},
 \qquad
 |\hnd^j(\Ric_{g_c}-\Ric_{\hg})|\le C_jb^{-2-j}.
\]
Also, the constant reference quantity $\hat Q=\mathrm{R}_{\hg}+|\hnd\hf|^2-\hf$ gives
\[
 \hat a^{-2}-1=4b^{-2}(\hat Q-\mathrm{R}_{\hg}-C_*),\qquad
 |\hnd^\ell(\hat a^{-2}-1)|\le C_\ell b^{-2-\ell}.
\]
The product rule now proves $|\hnd^j\Phi(g_c,\hf)|\le C_jb^{-2-j}$ for $0\le j\le J$. This proves the desired estimate outside $\Omega_B$, where $g_R=g_c$.

On $R\le b\le B=\Lambda R$, Proposition~\ref{s10:potentialcoords}(ii) at the required order and the cutoff bounds give
\[
 |\hnd^m e_R|
 \le C_m\sum_{\ell=0}^m b^{-\ell}\eta_Bb^{-2-(m-\ell)}
 \le C_m\eta_Bb^{-2-m},\qquad 0\le m\le J+2.
\]
Here $\Lambda$ is fixed, so $R^{-\ell}\le C_{\ell,\Lambda}b^{-\ell}$ on the transition annulus. Put $\nabla^c=\nabla^{g_c}$ and $\mathcal A=\nabla^{g_R}-\nabla^c$. The exact formula
\[
 \mathcal A^k{}_{ij}=\frac12(g_R)^{k\ell}
 \bigl(\nabla_i^c(e_R)_{j\ell}+\nabla_j^c(e_R)_{i\ell}
                         -\nabla_\ell^c(e_R)_{ij}\bigr)
\]
and uniform metric equivalence imply
\[
 |\hnd^\ell\mathcal A|\le C_\ell\eta_Bb^{-3-\ell},\qquad
 |\hnd^j(\Ric_{g_R}-\Ric_{g_c})|\le C_j\eta_Bb^{-4-j}.
\]
Conversion between $\nabla^c$ and $\hnd$ preserves these powers, since their difference has order $b^{-1}$ with the differentiated bounds above. The quadratic connection terms are absorbed using the uniform bound for $\eta_B$. For the Hessian, the exact identity is
\[
 (\nabla_{g_R}^2\hf-\nabla_{g_c}^2\hf)({Z_1},Y)
       =-\mathrm{d}\hf\bigl(\mathcal A({Z_1},Y)\bigr).
\]
Since $|\hnd^{m+1}\hf|\le C_m b^{1-m}$, this yields
\[
 |\hnd^j(\nabla_{g_R}^2\hf-\nabla_{g_c}^2\hf)|
 \le C_j\sum_{\ell=0}^j\eta_Bb^{-3-\ell}b^{1-(j-\ell)}
 \le C_j\eta_Bb^{-2-j}.
\]
Combining the negatives of the Ricci and Hessian differences with the term $e_R/2$ gives
\[
 |\hnd^j(\Phi(g_R,\hf)-\Phi(g_c,\hf))|
       \le C_j\eta_Bb^{-2-j}.
\]
Consequently,
\[
 |\hnd^j\Phi(g_R,\widetilde f_R)|\le C_jb^{-2-j}
 \quad\text{on }\{b\ge R\},\qquad 0\le j\le J.
\]
The constants are independent of $R$ for fixed $\Lambda$ and the prescribed finite-order bounds. The order-$15$ remainder estimate used in Lemma~\ref{s10:completion} supplies these bounds for $0\le J\le13$. The present $L^2$ comparison uses only $J=0$. The same estimate holds with $f_R$ in place of $\widetilde f_R$, since their difference is constant. Polynomial volume growth and $\widetilde f_R=b^2/4-C_*$ imply
\[
 \int_{b\ge R}(1+|\mathrm{R}_{g_R}|+|\nabla\widetilde f_R|^2+|\widetilde f_R|)
             \,\mathrm dV_{\widetilde f_R}\le CR^Ne^{-R^2/4},
\]
and its exterior $L^2$ defect norm is at most $CR^Ne^{-R^2/8}$.

The unnormalized pair agrees exactly with the actual pair on the retained inner region. Compare these unnormalized integrals first. The same mass tails give
\[
 |c_R|
 \le C_{\eta,\eps}e^{-(1-\eta)^2R^2/(4+\eps)}+CR^Ne^{-R^2/4}.
\]
Adding $c_R$ changes $\mathcal W$ by an error of this order. It does not change the defect tensor and only multiplies its squared norm integral by $e^{-c_R}$. This proves both estimates for the normalized pair.
\end{proof}

\begin{Proposition}[Pointed \L ojasiewicz--Simon inequality]\label{s10:pointedLS}
Let $\vartheta\in(0,\frac12)$ be as in \cref{thm:global-pair-LS}. For every $0<\theta<\vartheta$, choose a sufficiently large finite order $q$ and sufficiently small $\sigma$, and put $r_C=r_C^{(\sigma,q)}$. There is a uniform constant $C>0$ such that every Ricci flow in Subsection~\ref{s10:setup} satisfies
\begin{equation}\label{s10:LSerror}
 |\mathcal W(g(-1),f(-1))-\boldsymbol{\mu}(\hg)|^{1-\theta}
 \le C\mathcal D+Ce^{-r_C^2/8}.
\end{equation}
If $r_E\le r_C$ and $r_E$ is sufficiently large, then
\begin{equation}\label{s10:pureLS}
 |\mathcal W(g(-1),f(-1))-\boldsymbol{\mu}(\hg)|^{1-\theta}\le C\mathcal D.
\end{equation}
The constants depend only on the fixed reference pair, $n,Y$, and the displayed parameters.
\end{Proposition}
\begin{proof}
Write $\mathcal E=|\mathcal W(g,f)-\boldsymbol{\mu}(\hg)|$. Assume first $0<\mathcal D<1$ and put
\[
 p_0=\frac{1-\theta}{1-\vartheta}>1,\qquad
 R_*=\min\{r_C,r_E\}.
\]
For $R_*$ large, Proposition~\ref{s10:radiuscomparison} gives a product chart through $B$ slightly smaller than $R_*$. Choose a cutoff radius $R<B$, with $B/R>1$ fixed and close to one. Because all relative collar losses can be arbitrarily small, we can arrange
\begin{equation}\label{s10:lambdachoice}
 R\ge\lambda_0R_*,\qquad
 p_0\lambda_0^2>1,\qquad
 \frac{(1-\theta)(1-\eta)^2\lambda_0^2}{4+\eps}>\frac18
\end{equation}
with $\lambda_0<1$ sufficiently close to one.

The core matching gives $\kappa_A\le C_A\delta_{\mathrm{core}}$. Lemma~\ref{s10:corematching}, Proposition~\ref{s10:potentialcoords}(ii), and Lemma~\ref{s10:completion} therefore put the normalized extension in the fixed neighborhood of \cref{thm:global-pair-LS}. In particular, the exponentially small bound in \eqref{s10:epsilonB} absorbs the factor $\log(B/A)$ in $\eta_B$. Apply that theorem and Lemma~\ref{s10:cutoffcomparison}, and raise its gradient term to the power $p_0$. This yields
\[
\mathcal E^{1-\theta}\le C\bigl(\mathcal D^{p_0}
 +R^{Np_0}e^{-p_0R^2/8}
 +e^{-(1-\theta)(1-\eta)^2R^2/(4+\eps)} 
 +R^{N(1-\theta)}e^{-(1-\theta)R^2/4}\bigr).
\]
The strict inequalities in \eqref{s10:lambdachoice} absorb the polynomial factors and bound every Gaussian error by $Ce^{-R_*^2/8}$. If $R_*=r_E$, this equals $C\mathcal D=Ce^{-r_E^2/8}$. If $R_*=r_C$, they are bounded by $Ce^{-r_C^2/8}$. Since $\mathcal D^{p_0}\le \mathcal D$, these are the two conclusions.

If $R_*$ is bounded because $r_E$ is bounded, $\mathcal D$ has a positive lower bound. If the geometric radius is bounded, its Gaussian error has a positive lower bound. In either case Lemma~\ref{s10:kernel} bounds $\mathcal E$ uniformly, so \eqref{s10:LSerror} follows after enlarging $C$. The same uniform bound handles $\mathcal D\ge1$. For $\mathcal D=0$, use the extension with $R_*=r_C$ directly; the actual defect vanishes. Thus the assertion is uniform over the Ricci flows under consideration.
\end{proof}

\subsection{The modified flow and finite length}
\label{s10:length}

Throughout this subsection and the rest of the section, fix a closed connected Ricci flow
$(M,g(t))_{-2\le t<0}$ with entropy bounded below by $-Y$, and a completion point $z$ at time zero. Assume that one tangent flow at $z$ is generated by the reference shrinker $(\hM,\hg,\hf)$. Let $f(t)$ be the kernel potential based at $z$, as in \eqref{s10:physicalkernel}. Fix $\vartheta\in(0,\frac12)$ for which \cref{thm:global-pair-LS} holds, decreasing its exponent if necessary.

Let $\varphi_t$ solve
$\partial_t\varphi_t=-\nabla^{g(t)}f(t)\circ\varphi_t$ with
$\varphi_{-1}=\mathrm{Id}$. Define
\begin{equation}\label{s10:modified}
 \gz(s)=e^s\varphi_{-e^{-s}}^*g(-e^{-s}),\qquad
 \fz(s)=\varphi_{-e^{-s}}^*f(-e^{-s}).
\end{equation}
Set
\begin{equation}\label{s10:defect}
 \Phi(s):=\Phi(\gz(s),\fz(s)).
\end{equation}
The equations and the stationary probability density are
\begin{equation}\label{s10:evolution}
 \partial_s\gz=2\Phi,\qquad \partial_s\fz=\tr_{\gz}\Phi,\qquad
 \mathrm dV_{\fz(s)}=(4\pi)^{-n/2}e^{-\fz(s)}\,\mathrm dV_{\gz(s)},\qquad \partial_s \mathrm dV_{\fz(s)}=0.
\end{equation}
Indeed the logarithmic derivative of this density is
$-\partial_s\fz+\frac12\tr_{\gz}\partial_s\gz=0$.

For $s>0$, the rescaled Ricci flow and kernel potential
\begin{equation}\label{s10:renormalization}
 \bigl(e^s g(e^{-s}t),f(e^{-s}t)\bigr),\qquad -2\le t<0,
\end{equation}
satisfy the hypotheses of Subsections~\ref{s10:setup}--\ref{s10:cutoff}. Their time-minus-one pair pulls back by $\varphi_{-e^{-s}}$ to $(\gz(s),\fz(s))$, so the preceding intrinsic estimates apply to this pair. For fixed $(\sigma,q)$, let $r_C^{(\sigma,q)}(s)$ and $r_E(s)$ be its geometric and entropy radii from Definitions~\ref{s10:defC} and~\ref{s10:defEP}, and abbreviate the former to $r_C(s)$. Set
\begin{equation}\label{s10:energy}
 \mathcal D(s)=\|\Phi(s)\|_{L^2(\mathrm dV_{\fz(s)})},\qquad
 E(s)=\boldsymbol{\mu}(\hg)-\mathcal W(\gz(s),\fz(s)).
\end{equation}
The given tangent sequence, the Gaussian entropy remainder, and monotonicity give
\begin{equation}\label{s10:energyid}
 E(s)\ge0,\qquad E(s)\longrightarrow0,\qquad E'(s)=-2\mathcal D(s)^2.
\end{equation}
We write $b(s)=2\sqrt{\fz(s)+C_*}$. The original scalar lower bound also gives
$\mathrm{R}_{\gz(s)}\ge-Ce^{-s}$ at late times.

Choose a smooth exhaustion coordinate $r$ extending the canonical end coordinate, with its compact inner part below a fixed large level. For a fixed core embedding $\psi_0$, define
\begin{equation}\label{s10:Phi}
 (G(s),F(s))=\psi_0^*(\gz(s),\fz(s)).
\end{equation}

For a fixed large $A$, let $K=\{r\le8A\}$ and suppose the embedding is defined on a neighborhood of $K$. Fix the radius order $q\ge100$ chosen above. We make the following assumption on the core at order $5q$.

\noindent\textbf{Assumption:}
\begin{equation}\label{s10:coreassumption}
 \sup_{s_0\le u<S}
 \bigl(\|G(u)-\hg\|_{C^{5q}_0(K)}
       +\|F(u)-\hf\|_{C^{5q}_2(K)}\bigr)<\delta_A.
\end{equation}
The order $q$ is used in Definition~\ref{s10:defC}, while $5q$ supplies the mixed derivatives in the construction below.

Next, we obtain an exponential growth bound for $r_C$. The idea is to construct an expanding geometric chart using the controlled geometry of a fixed annulus at earlier modified times. At time $s$, each large radius $r$ is assigned an earlier time $u$ so that parabolic scaling carries that annulus to radius $r$. For the reference shrinker, the construction recovers the canonical chart; its conical leading term has exact scaling cancellation. The core hypothesis and forward curvature estimates control the error uniformly, while monotonicity of the time and radial parameters makes the transported level sets nested. The chart reaches radius comparable to $Ae^{(s-s_0)/2}$.

\begin{Proposition}[Growth of the geometric radius]\label{s10:Cpersistence}
For each fixed radius order $q\ge100$ and sufficiently small tolerance $\sigma$, write $r_C(s)=r_C^{(\sigma,q)}(s)$. One can choose $A$, $\delta_A$, $W_0$, and $c_0>0$ so that \eqref{s10:coreassumption} implies
\begin{equation}\label{s10:Cgrowth}
 r_C(s)\ge c_0 A e^{(s-s_0)/2},
 \qquad s_0+W_0\le s<S.
\end{equation}
The constants are independent of $s_0$, $S$, and $S-s_0$, provided $s_0>0$.
\end{Proposition}
\begin{proof}
\smallskip\noindent\emph{1.}
For this proof put $\Theta_u=\varphi_{-e^{-u}}\circ\psi_0$ and
\[
 \mathcal P(u,s)=e^u\Theta_u^*g(-e^{-s}),\qquad
 V(u)=\nabla^{G(u)}F(u).
\]
Thus $\mathcal P(u,u)=G(u)$ and $\partial_u\Theta_u=-\mathrm d\Theta_u(V(u))$.
For each $u\in[s_0,S)$, apply the annular version of Lemma~\ref{s10:shell} to $e^ug(e^{-u}t)$ and the fixed map $\Theta_u$ on $\{A/2<r<6A\}$. On $\{A\le r\le4A\}$ it gives
\[
 \|\mathcal P(u,s)-G(u)\|_{C^{2q+4}_0}\le CA^{-2},
 \qquad u\le s,\quad u<S.
\]
The core assumption supplies the initial derivatives, since $5q\ge4q+20$.
These forward estimates use only the core geometry at $u$ and remain valid
up to time zero. Together with the core hypothesis and the AC expansion,
they imply
\[
 \|\mathcal P(u,s)-g_C\|_{C^{2q+4}_0}
 \le C(\delta_A+A^{-2}),\qquad
 \|V(u)-\mathbf X\|_{C^{5q-1}_1}\le C\delta_A.
\]
Here and below the source-annulus norms may be computed with $g_C$; they are uniformly equivalent to the reference weighted norms.
In particular, $\mathrm dr(V)>0$ on the fixed annulus.

\smallskip\noindent\emph{2.}
Choose a smooth nondecreasing radial profile $a(r)$ with
\[
 a(r)=r\ \text{near }\{r\le A\},\qquad
 a(r)=a_\infty\ \text{for }r\ge4A,\qquad
 A<a_\infty<4A,\qquad 0\le\frac{ra'}a\le1.
\]
Choose the profile by scaling a fixed function, so all its scaled derivatives are uniformly bounded. For fixed $s$, set
\begin{equation}\label{s10:birthchart}
 u_s(r)=s-2\log\frac r{a(r)},\qquad
 \mathcal F_s(r,\theta)=\Theta_{u_s(r)}(a(r),\theta).
\end{equation}
On the inner core set $\mathcal F_s=\Theta_s$. The definitions agree near the matching level. Retain the core and $A\le r\le R_s$, where
\[
 R_s=\frac12a_\infty e^{(s-s_0)/2}.
\]
For a fixed sufficiently large $W_0$, whenever $s\ge s_0+W_0$,
one has $R_s\ge4A$ and $s_0\le u_s(r)\le s$ throughout this domain.
At $R_s$ the earlier time is $s_0+2\log2$, so a neighborhood of the closed retained domain is also available.

The boundary velocity of $\Theta_u(\{r<\rho\})$, for $A\le\rho\le4A$, is $-\mathrm d\Theta_u(V)$ and points strictly inward.
These domains therefore shrink strictly as $u$ increases.
As $r$ increases, $a(r)$ does not decrease and $u_s(r)$ does not increase; for distinct radii at least one change is strict.
Consequently the transported level sets are strictly nested, and the
inner core lies inside them. This proves global injectivity.
The differential estimate below proves local nonsingularity.

\smallskip\noindent\emph{3.}
The equations for the fixed pullbacks are
\[
 \partial_uG=G-2\Ric_G-2\nabla_G^2F,\qquad
 \partial_uF=n/2-\mathrm R_G-\Delta_GF.
\]
Identify the source annulus with a fixed annulus and rescale the metric
and potential by $A^{-2}$. The resulting equations have uniformly controlled
coefficients. Subtracting the stationary reference equations and differentiating gives
\[
 \|\partial_u^p(V-\mathbf X)\|_{C^\ell_1}\le C_{p,\ell}\delta_A,
 \qquad \ell+2p+1\le 5q.
\]
At fixed $s$, differentiation of the pullback defining $\mathcal P$ gives
$\partial_u\mathcal P=\mathcal P-\mathcal L_V\mathcal P$.
Since $\mathcal L_{\mathbf X}g_C=g_C$, this becomes
\[
 \partial_u(\mathcal P-g_C)
 =(1-\mathcal L_V)(\mathcal P-g_C)
       -\mathcal L_{V-\mathbf X}g_C.
\]
The spatial bound from Step~1 and the preceding estimates for $V$
now give, by repeated differentiation and the product rule,
\[
 \|\partial_u^p(\mathcal P-g_C)\|_{C^\ell_0}
 \le C_{p,\ell}(\delta_A+A^{-2}),\qquad p+\ell\le q+1.
\]
The order $2q+4$ in Step~1 and the core order $5q\ge4q+20$ supply all derivatives used here. The resulting mixed derivative bounds are uniform in $s-s_0$.

Abbreviate $a=a(r)$ and $u=u_s(r)$. Differentiating the chart and using
$a'-au_s'/2=a/r$ gives
\[
 \mathrm d\mathcal F_s(r\partial_r)
 =\mathrm d\Theta_u\bigl(a\partial_a-ru_s'(V-\mathbf X)\bigr),
 \qquad
 \mathrm d\mathcal F_s(W)=\mathrm d\Theta_u(W)
\]
for a link-tangent vector $W$. For link-tangent vectors $W,Z$ and
scalars $\xi,\zeta$, the exact metric identity is
\begin{equation}\label{s10:chartmetric}
 \begin{aligned}
 &r^{-2}\mathcal F_s^*(e^sg(-e^{-s}))
       (\xi r\partial_r+W,\zeta r\partial_r+Z)\\
 &\quad=a^{-2}\mathcal P(u,s)
 \bigl(\xi[a\partial_a-ru_s'(V-\mathbf X)]+W,\,
       \zeta[a\partial_a-ru_s'(V-\mathbf X)]+Z\bigr).
 \end{aligned}
\end{equation}
Here $e^{s-u}a^2=r^2$ cancels the scale factors.
Replacing $\mathcal P$ by $g_C$ and $V$ by $\mathbf X$ gives
$\xi\zeta+g_\Sigma(W,Z)$ exactly.
The error is $O(\delta_A+A^{-2})$, and the radial vector differs from
$a\partial_a$ by at most $C\delta_Aa$ in $g_C$ norm.
For small $\delta_A$ this also proves nonsingularity.

For every coefficient in fixed radial and angular frames, composition with $(u_s(r),a(r),\theta)$ is differentiated by
\[
 r\partial_r=(ru_s')\partial_u+\frac{ra'}a\,a\partial_a.
\]
All scaled derivatives of the two coefficients are uniformly bounded.
Apply this rule and the mixed estimates to \eqref{s10:chartmetric},
using radial and angular frames on a fixed finite cover of the link.
The product rule gives, on every rescaled annulus and uniformly in its
radius,
\[
 \|\mathcal F_s^*(e^sg(-e^{-s}))-g_C\|_{C^q_0(\{A\le r\le R_s\})}
 \le C_q(\delta_A+A^{-2}).
\]
Since $\hg-g_C$ has weight minus two and $\mathcal F_s=\Theta_s$
on the inner core, the same bound relative to $\hg$ holds on the
whole retained domain. Choose $A$ large and then $\delta_A$ small so
that this bound is less than $\sigma$.
Also $\hf(\hat p)\le n/2$, so core closeness gives
$f(-e^{-s})(\mathcal F_s(\hat p))=F(s)(\hat p)\le n$.
Finally $\hb=r+O(r^{-1})$ supplies a fixed $c_0>0$ such that the domain contains a neighborhood of $\overline{\Omega_{c_0Ae^{(s-s_0)/2}}}$.
Thus $\varphi_{-e^{-s}}^{-1}\circ\mathcal F_s$ is an embedding satisfying
\eqref{s10:Cdefinition} for $(\gz(s),\fz(s))$, proving \eqref{s10:Cgrowth}.
\end{proof}

The embeddings in this proposition may depend on $s$. They are used to obtain intrinsic estimates on the time slice. Time integration below is always in the original fixed modified gauge \eqref{s10:modified}, not in these varying embeddings.

\begin{Lemma}\label{s10:corevelocity}
Suppose \eqref{s10:coreassumption} holds on $[s_0,S)$.
There is a fixed $W\ge W_0$, independent of $s_0,S$, such that, for every $0<a<1$,
\begin{equation}\label{s10:corevel}
 \|\partial_sG\|_{C^{5q}(K)}
 +\|\partial_sF\|_{C^{5q}(K)}
 \le C_{A,q,a}\mathcal D(s)^a,
 \qquad s_0+W\le s<S.
\end{equation}
For every $B\ge1$ there is also $W_B\ge W_0$, independent of derivative order,
such that, for every integer $m\ge0$ and $0<a<1$,
\[
 \sum_{j=0}^m\left(
 |\nabla^j\partial_s\gz(s)|_{\gz(s)}
 +|\nabla^j\partial_s\fz(s)|_{\gz(s)}\right)
 \le C_{B,m,a}\mathcal D(s)^a
 \quad\text{on }\{\fz(s)\le B\},\qquad s_0+W_B\le s<S.
\]
Here $\nabla=\nabla^{\gz(s)}$; the constants and margins may depend on the fixed geometric data.
\end{Lemma}
\begin{proof}
Fix the geometric order and tolerance within the thresholds of Proposition~\ref{s10:rough}(ii).
The core hypothesis gives $F(s)\le B_A$ on $K$ for a fixed $B_A\ge1$,
so $\psi_0(K)\subset\{\fz(s)\le B_A\}$.
By Proposition~\ref{s10:Cpersistence}, choose $W_B\ge W_0$ such that
$c_0Ae^{W_B/2}>L_B$, and take $W=W_{B_A}$.
The threshold $L_B$ in Proposition~\ref{s10:rough}(ii) is independent of derivative order.
For $s_0+W_B\le s<S$, that proposition gives uniform intrinsic bounds of every
finite order on $\{\fz(s)\le B\}$ and uniformly controlled neighborhoods of its points.
The potential is uniformly bounded above on these neighborhoods, so
\[
 \mathrm dV_{\gz(s)}=(4\pi)^{n/2}e^{\fz(s)}\mathrm dV_{\fz(s)}
 \le C_B\mathrm dV_{\fz(s)}.
\]
Consequently, the ordinary local $L^2$ norm of $\Phi(s)$ is at most $C_B\mathcal D(s)$.
For $\mathcal D(s)\le1$, local interpolation with a sufficiently large finite
order $N=N(m,a)$ gives
\[
 \sum_{j=0}^m|\nabla^j\Phi(s)|_{\gz(s)}
 \le C_{B,m,a}\mathcal D(s)^a
 \quad\text{on }\{\fz(s)\le B\}.
\]
For $\mathcal D(s)\ge1$, the same estimate follows from the uniform rough derivative bounds.
Increasing $N$ changes only these bounds and the interpolation constants, not $q,\sigma,L_B$, or $W_B$.
The equations \eqref{s10:evolution} give the asserted intrinsic velocity estimates.

Apply the defect estimate with $B=B_A$ and $m=5q$, and pull back by $\psi_0$.
It gives the corresponding intrinsic bounds for $\Phi(G,F)=\psi_0^*\Phi(s)$ on $K$.
To pass to the fixed reference connection, put $\mathcal A_s=\nabla^{G(s)}-\hnd$.
The core hypothesis implies
\[
 \|\mathcal A_s\|_{C^{5q-1}(K,\hg)}\le C_{A,q}.
\]
Converting $\hnd^j\Phi(G,F)$ to intrinsic derivatives introduces derivatives of
$\mathcal A_s$ only through order $j-1$. Thus
\[
 \|\Phi(G,F)\|_{C^{5q}(K,\hg)}
 \le C_{A,q,a}\mathcal D(s)^a.
\]
Finally, pullback of \eqref{s10:evolution} gives
$\partial_sG=2\Phi(G,F)$ and $\partial_sF=\tr_G\Phi(G,F)$.
The uniform $C^{5q}$ bound for $G^{-1}$ proves \eqref{s10:corevel}.
\end{proof}

The length calculation is the power-of-gradient argument of Sun--Wang \cite[Lemma 3.2]{sun2015on}, retaining the Gaussian error and the original kernel potential.
\begin{Proposition}\label{s10:finitelength}
Fix $0<\theta<\vartheta$ and choose
\begin{equation}\label{s10:abeta}
 2-\frac1{1-\theta}<a<1,
 \qquad \beta=1-(2-a)(1-\theta)>0.
\end{equation}
Let $[s_0,S)$ be a sufficiently late core interval and $s_1=s_0+W<S$, with $W$ larger than the fixed initial margins. Then
\begin{equation}\label{s10:finiteinterval}
 \int_{s_1}^S\mathcal D(s)^a\,\mathrm{d}s
 \le CE(s_1)^\beta+C\exp(-c_aA^2e^W).
\end{equation}
The endpoint $S$ may be finite or infinite; the constants are independent of its value. This estimate is conditional on the core hypothesis on the stated interval.
\end{Proposition}
\begin{proof}
At every $s\in[s_1,S)$, Proposition~\ref{s10:Cpersistence} and the inequality \eqref{s10:LSerror}, applied to $(\gz(s),\fz(s))$, imply
\[
 E^{1-\theta}\le C(\mathcal D+\omega),\qquad
 \omega(s)=C\exp(-cA^2e^{s-s_0}).
\]
Where $\mathcal D\ge\omega$, this gives $E^{1-\theta}\le C'\mathcal D$, and therefore
\begin{equation}\label{s10:lengthderivative}
 -\frac{\mathrm{d}}{\mathrm{d}s}E^\beta
 =2\beta E^{-(2-a)(1-\theta)}\mathcal D^2\ge c\mathcal D^a.
\end{equation}
Integrate over this set, using monotonicity of $E^\beta$ on the whole interval. On the complementary set, $\mathcal D^a\le\omega^a$, whose integral is bounded by $C\exp(-c_aA^2e^W)$. If $E$ vanishes, the identity is justified by replacing it with $E+\delta$ and passing to the limit; alternatively \eqref{s10:energyid} makes the remaining trajectory stationary.

This proves \eqref{s10:finiteinterval} on the whole interval $[s_1,S)$.

\end{proof}

\begin{Corollary}\label{s10:noexit}
Under the standing tangent-flow assumption of this subsection, an initial embedding from the given tangent sequence and a sufficiently late time $s_0$ can be chosen so that \eqref{s10:coreassumption} holds with $S=\infty$. Consequently, \eqref{s10:Cgrowth} holds for all $s\ge s_0+W_0$, and, for the parameters in \eqref{s10:abeta} and a sufficiently large fixed $W$,
\[
 \int_{s_0+W}^{\infty}\mathcal D(s)^a\,\mathrm{d}s<\infty.
\]
\end{Corollary}
\begin{proof}
Fix the core and all tolerances needed above. Let
\[
 e_K(s)=\|G(s)-\hg\|_{C^{5q}_0(K)}
       +\|F(s)-\hf\|_{C^{5q}_2(K)}.
\]
On the fixed compact core, the weighted and unweighted norms are equivalent. Absorbing this equivalence into the constant in Lemma~\ref{s10:corevelocity}, let $C_v$ satisfy
\[
 e_K(s)\le e_K(s_1)+C_v\int_{s_1}^s\mathcal D(u)^a\,\mathrm{d}u
\]
whenever the core hypothesis holds between $s_0$ and $s$, with $s_1$ beyond the fixed initial margins. Choose $W$ larger than these margins and so large that
\[
 C_v C\exp(-c_aA^2e^W)<\delta_A/4,
\]
where $C,c_a$ are the constants in Proposition~\ref{s10:finitelength}.

Smooth spacetime convergence along the given tangent sequence, together with continuous dependence of the defining ODE for $\varphi_t$, gives a sufficiently late $s_0$ and an initial embedding for which $e_K(s)<\delta_A/4$ on $[s_0,s_0+W+1]$. Here $W$ is fixed before choosing the tangent time, so only convergence on a fixed compact spacetime domain is used. Since $E(s)\to0$, we can also arrange, with $s_1=s_0+W$, that
\[
 C_v C E(s_1)^\beta<\delta_A/4.
\]
Let $S$ be the first possible time at which $e_K$ reaches $\delta_A$; then $S>s_1$. On every compact subinterval $[s_0,S']$ with $S'<S$, the supremum of $e_K$ is strictly below $\delta_A$, so the preceding propositions apply. Their constants are independent of $S'$. Applying Proposition~\ref{s10:finitelength} on these intervals gives, for every $s_1\le s<S$,
\[
 e_K(s)
 \le e_K(s_1)+C_v\int_{s_1}^s\mathcal D(u)^a\,\mathrm{d}u 
 \le e_K(s_1)+C_v C E(s_1)^\beta
                +C_v C\exp(-c_aA^2e^W)
 <3\delta_A/4.
\]
If $S$ were finite, smoothness of the original Ricci flow and its globally defined diffeomorphisms at every finite modified time would continue the marked pair across $S$. Continuity of its fixed-core norm contradicts a first exit at level $\delta_A$. Thus $S=\infty$, and in fact the core error stays uniformly below $3\delta_A/4$. Proposition~\ref{s10:Cpersistence} gives the radius growth, and letting the upper endpoint tend to infinity in Proposition~\ref{s10:finitelength} gives the finite-length conclusion.
\end{proof}

\begin{Proposition}\label{s10:decay}
For the embedding supplied by Corollary~\ref{s10:noexit} and the parameters in \eqref{s10:abeta},
\begin{equation}\label{s10:energydecay}
 E(s)\le C(1+s)^{-1/(1-2\theta)},
\end{equation}
and
\begin{equation}\label{s10:taildecay}
 I_a(s):=\int_s^\infty \mathcal D(u)^a\,\mathrm{d}u
 \le C(1+s)^{-\beta/(1-2\theta)}
\end{equation}
at all sufficiently late times. Consequently, for every $\gamma<\vartheta/(1-2\vartheta)$ one can choose $0<\theta<\vartheta$ and $a<1$ as in \eqref{s10:abeta} so that
\begin{equation}\label{s10:Irate}
 I_a(s)\le C(1+s)^{-\gamma}
\end{equation}
with a strictly larger exponent than $\gamma$ available if needed.
\end{Proposition}
\begin{proof}
At sufficiently late times \eqref{s10:LSerror} and \eqref{s10:Cgrowth} give
\[
 E^{1-\theta}\le C(\mathcal D+\omega),\qquad
 \omega(s)\le C\exp(-c e^s).
\]
Squaring and using $-E'=2\mathcal D^2$ yields
\begin{equation}\label{s10:forcedode}
 E'\le-c_1E^{2(1-\theta)}+C_1\omega^2.
\end{equation}
The forcing decreases faster than every exponential. A direct supersolution proves \eqref{s10:energydecay}: with $\alpha=(1-2\theta)^{-1}$, take $H(s)=K(1+s-s_2)^{-\alpha}$, where $s_2$ is sufficiently late and $K$ is large. Since $2\alpha(1-\theta)=\alpha+1$, the inequality
\[
 H'\ge-c_1H^{2(1-\theta)}+C_1\omega^2
\]
follows by taking $c_1K^{2(1-\theta)}-\alpha K$ larger than the finite supremum of $C_1(1+s-s_2)^{\alpha+1}\omega(s)^2$. Scalar comparison gives $E\le H$. Restricting to later $s\ge2s_2+1$ gives \eqref{s10:energydecay}.

Apply the two-set integration in \eqref{s10:lengthderivative} on $[s,\infty)$. It gives
\[
 I_a(s)\le CE(s)^\beta+C\exp(-c_ae^s).
\]
Equation \eqref{s10:energydecay} proves \eqref{s10:taildecay}. Finally, as $a\uparrow1$, $\beta\uparrow\theta$; as $\theta\uparrow\vartheta$, $\theta/(1-2\theta)$ increases to $\vartheta/(1-2\vartheta)$. Choose both parameters with strict margins to obtain the last assertion.
\end{proof}

\subsection{Convergence in a fixed gauge}
\label{s10:fixedgauge}

In this subsection, we prove strong uniqueness. The general strategy is similar to \cite[Section~7]{fangli2026strong}: finite length gives convergence in fixed coordinates, and weighted estimates quantify this convergence on expanding potential sublevels. We then identify the completed tangent flow and deduce isolation of the singular point.

\begin{Theorem}[Strong uniqueness]\label{s10:main}
Suppose one tangent flow at $z$ is generated by $(\hM,\hg,\hf)$ as above. There are an open subset $U\subset M$ and a smooth diffeomorphism $\psi:\hM\to U$, independent of $s$, such that
\begin{equation}\label{s10:fixedlimit}
 \psi^*(\gz(s),\fz(s))\longrightarrow(\hg,\hf)
 \quad\text{in }C^\infty_{\mathrm{loc}}(\hM),\qquad
 \psi^*\mathrm dV_{\fz(s)}=\mathrm dV_{\hat f}.
\end{equation}
In particular, $\int_U\mathrm dV_{\fz(s)}=1$.

Let $\vartheta\in(0,\frac12)$ be an admissible exponent in \cref{thm:global-pair-LS}. For every integer $m\ge0$ and every
\begin{equation}\label{s10:ratechoices}
 0<\gamma<\frac{\vartheta}{1-2\vartheta},\qquad 0<c<2\gamma,
\end{equation}
there are $C<\infty$ and $\bar s<\infty$ such that, for every integer $j\ge\bar s$, the map $\psi$ satisfies, on
\begin{equation}\label{s10:omegaj}
 \Omega_j^{\log}=\{\hf\le c\log j\}
\end{equation}
for all $s\ge j$,
\begin{equation}\label{s10:mainrate}
 [\psi^*\gz(s)-\hg]_m+[\psi^*\fz(s)-\hf]_m
 \le C e^{\hf/2}(1+s)^{-\gamma}
 \quad\text{on }\Omega_j^{\log}.
\end{equation}
The constants may depend on the reference shrinker, $Y$, and the displayed parameters; the constant $\bar s$ may also depend on the flow and $z$.
\end{Theorem}

We now work on the fixed manifold carrying $(\gz(s),\fz(s))$. The fixed density in \eqref{s10:evolution} will identify the normalization of the limiting pair. The first step is to obtain a limit on an open subset of this manifold.

\begin{Lemma}\label{s10:persistlevels}
Fix $a$ for which $I_a(s)\to0$. For each sufficiently large $B$ there is a time $s_B$ and a constant $C_B$ such that, for $s\ge s_B$, all points with $\fz(s)\le B+2$ satisfy
\begin{equation}\label{s10:sublevelvelocity}
 |\partial_s\fz|+|\partial_s\gz|_{\gz}\le C_B\mathcal D(s)^a.
\end{equation}
If $s_1\ge s_B$ and $C_BI_a(s_1)<1/4$, then
\begin{align}
 \fz(x,s_1)\le B&\quad\Longrightarrow\quad
 |\fz(x,s)-\fz(x,s_1)|<1/4\quad(s\ge s_1),\label{s10:upperbarrier}\\
 \fz(x,s_1)\ge B&\quad\Longrightarrow\quad
 \fz(x,s)>B-1/4\quad(s\ge s_1).\label{s10:lowerbarrier}
\end{align}
For each fixed derivative order there are analogous intrinsic bounds by $C_{B,m,a}\mathcal D(s)^a$ on the same sublevels for the same $s_B$.
\end{Lemma}
\begin{proof}
Choose $s_B\ge s_0+W_{B+2}$. Lemma~\ref{s10:corevelocity} gives \eqref{s10:sublevelvelocity} and its differentiated versions on the entire displayed sublevel set.

For \eqref{s10:upperbarrier}, suppose a point initially at potential at most $B$ first reaches $B+1$. Up to that time it lies in the controlled sublevel set, and integration of \eqref{s10:sublevelvelocity} bounds its potential change by $C_BI_a(s_1)<1/4$, a contradiction. It therefore stays below $B+1$, and the same integral gives the asserted absolute change for all future times.

For \eqref{s10:lowerbarrier}, a trajectory starting above $B$ could reach $B-1/4$ only after crossing the level $B$. Between its last such crossing and its first arrival at $B-1/4$, it lies in the controlled sublevel set. Its potential change would be at least $1/4$, again contradicting the same integral.
\end{proof}

\begin{Proposition}\label{s10:limit}
There is an open set $U\subset M$ on which $(\gz(s),\fz(s))$ converges smoothly on compact subsets to a pair $(g_\infty,f_\infty)$. The function $f_\infty:U\to\R$ is proper and bounded below. Each connected component of $U$ is a complete Ricci shrinker. On $U$,
\begin{equation}\label{s10:limitdensity}
 \mathrm dV_{\fz(s)}=(4\pi)^{-n/2}e^{-f_\infty}\mathrm dV_{g_\infty}.
\end{equation}
Moreover, for every finite $B$, all sufficiently late sublevels $\{\fz(s)\le B\}$ lie in a fixed compact subset of $U$.
\end{Proposition}
\begin{proof}
Choose numbers $B_j\to\infty$ with $B_{j+1}\ge B_j+2$, and increasing times $s_j$ such that Lemma~\ref{s10:persistlevels} applies with threshold $B_j$ and $C_{B_j}I_a(s_j)<1/4$. Put
\begin{equation}\label{s10:Uj}
 U_j=\{x\in M:\fz(x,s_j)<B_j\},\qquad U=\bigcup_jU_j.
\end{equation}
The upper barrier gives $\overline{U_j}\subset U_{j+1}$ and, for $s\ge s_j$, $\fz(s)\le B_j+1/4$ on $\overline{U_j}$. The sets are nonempty for large $j$, since the no-exit core contains a fixed point of the modified manifold with bounded potential.

On a neighborhood of $\overline{U_j}$, \eqref{s10:sublevelvelocity} and $\int_{s_j}^\infty \mathcal D^a<\infty$ give uniform equivalence of $\gz(s)$ with $\gz(s_j)$. They also give a positive-definite $C^0$ limit and uniform convergence of the potential. We justify convergence of derivatives in these fixed coordinates. The connection variation is
\begin{equation}\label{s10:connectiontime}
 \partial_s\Gamma(\gz)^k_{ij}
 =(\gz)^{k\ell}
 \bigl(\nabla_i\Phi_{j\ell}+\nabla_j\Phi_{i\ell}-\nabla_\ell \Phi_{ij}\bigr).
\end{equation}
The intrinsic differentiated bounds in Lemma~\ref{s10:persistlevels} make its right-hand side integrable. Thus the connection difference relative to $\gz(s_j)$ is bounded and converges. Differentiate \eqref{s10:connectiontime} in that fixed connection. At each order, the conversion from current to fixed covariant derivatives involves already controlled lower connection derivatives, and the highest connection derivative occurs with an integrable coefficient bounded by $C\mathcal D^a$. Gronwall's inequality gives its uniform bound, and the equation then gives integrability of its time derivative. Induction proves convergence of all connection derivatives. Equations \eqref{s10:evolution} give convergence of all derivatives of the metric and potential. For each derivative order one may start at a later time at which its intrinsic estimates hold; the intervening finite time interval is smooth and does not affect convergence. A finite cover of $\overline{U_j}$ completes this argument. The limits agree on overlaps, since all tensors are in the same fixed coordinates.

There is a sequence $s_i\to\infty$ with $\mathcal D(s_i)\to0$. The local estimates imply $\Phi(s_i)\to0$ on every compact subset of $U$. Passing to the smooth limit gives
\begin{equation}\label{s10:limitsoliton}
 \Ric_{g_\infty}+\nabla^2_{g_\infty}f_\infty=\frac12g_\infty.
\end{equation}
The scalar lower bound $\mathrm{R}_{\gz(s)}\ge-Ce^{-s}$ from Subsection~\ref{s10:length} gives $\mathrm{R}_{g_\infty}\ge0$. Taking the local limit in the fixed-density identity proves \eqref{s10:limitdensity}.

If $x\in U\setminus U_j$, then $\fz(x,s_j)\ge B_j$, and the lower barrier gives $f_\infty(x)\ge B_j-1/4$. Hence every finite sublevel of $f_\infty$ is contained in some $U_j$, and its closure is a closed subset of the compact set $\overline{U_j}\subset U_{j+1}$. Thus $f_\infty$ is proper. It is bounded below by $1-C_*$.

On each connected component, the soliton equation gives a constant $Q_\infty$ such that
\[
 \mathrm{R}_{g_\infty}+|\nabla f_\infty|^2-f_\infty=Q_\infty.
\]
Choose a constant $C$ on that component so that $f_\infty+C\ge1$ and $C\ge Q_\infty$. Since $\mathrm{R}_{g_\infty}\ge0$,
\[
 |\nabla\sqrt{f_\infty+C}|\le\frac12.
\]
A finite-length curve cannot escape all compact sets: the displayed function remains bounded along it, whereas properness makes its sublevels compact. This proves completeness.

Finally, fix $B$ and choose $j$ with $B_j>B+1$. If $s\ge s_j$ and $\fz(x,s)\le B$, the lower barrier rules out $\fz(x,s_j)\ge B_j$. Thus $x\in U_j\subset\overline{U_j}\Subset U$. This proves the last assertion, which will prevent the images of tangent convergence maps from escaping in the fixed coordinates.
\end{proof}

\begin{Proposition}\label{s10:identify}
The open set $U$ in Proposition~\ref{s10:limit} is connected, and there is a diffeomorphism $\psi:\hM\to U$ satisfying \eqref{s10:fixedlimit}. In particular, $\int_U\mathrm dV_{\fz(s)}=1$.
\end{Proposition}
\begin{proof}
Use the originally given tangent sequence. Compose its time-minus-one embeddings with the inverse of $\varphi_{-e^{-s_i}}$ to obtain embeddings
\[
 \iota_i:\Omega_i\to M,
 \qquad \iota_i^*\gz(s_i)\to\hg,
 \qquad \iota_i^*\fz(s_i)\to\hf
\]
on exhausting compact model domains. For a fixed compact $K\subset\hM$, the potentials on $\iota_i(K)$ are uniformly bounded. The last assertion of Proposition~\ref{s10:limit} puts these images in a fixed compact subset of $U$ for all large $i$.

Smooth convergence of the target and pulled-back metrics gives compactness of the embeddings and their local inverses. A diagonal subsequence converges smoothly on compact subsets to a local isometry
\begin{equation}\label{s10:J}
 J:(\hM,\hg)\longrightarrow(U,g_\infty),
 \qquad J^*f_\infty=\hf.
\end{equation}
Local inverse convergence and injectivity of $\iota_i$ imply injectivity of $J$: two distinct preimages would give overlapping images of disjoint neighborhoods for all large $i$. Completeness of $\hM$ makes $J$ a covering of its target component $U_0$, hence an isometry onto $U_0$.

By \eqref{s10:limitdensity} and \eqref{s10:J},
\[
 \int_{U_0}\mathrm dV_{\fz(s)}=(4\pi)^{-n/2}\int_{\hM}e^{-\hf}\,\mathrm dV_{\hg}=1.
\]
Every nonempty additional component of $U$ would be open in $M$ and would have positive mass for the fixed smooth positive measure $\mathrm dV_{\fz(s)}$. This is impossible because $\int_M\mathrm dV_{\fz(s)}=1$. Therefore $U=U_0$. Set $\psi=J$. The density equality holds for every $s$ because $\mathrm dV_{\fz(s)}$ is stationary, and the smooth convergence follows from Proposition~\ref{s10:limit}.
\end{proof}

The fixed marking gives the identity
\begin{equation}\label{s10:fixedvolume}
 \psi^*\fz(s)-\hf
 =\log\frac{\mathrm dV_{\psi^*\gz(s)}}{\mathrm dV_{\hg}}
\end{equation}
at every time for which the modified pair is defined.

\begin{Lemma}\label{s10:weightedconvergence}
Let $\psi$ be given by Proposition~\ref{s10:identify}, and fix $m$ and $0<a<1$ with $I_a(s)<\infty$. Write
\[
 G(s)=\psi^*\gz(s),\qquad F(s)=\psi^*\fz(s),
 \qquad h(s)=G(s)-\hg,\quad v(s)=F(s)-\hf.
\]
Choose the geometric order $q$, tolerance $\sigma$, and fixed collars $0<\lambda<\lambda'<1$ so that \eqref{s10:weightedvelocityintrinsic} holds through order $m$ with $\kappa_1=1/2$. These parameters are chosen for the prescribed $m,a$; in this lemma write $r_C(s)=r_C^{(\sigma,q)}(s)$. There are constants $C_m$ and $L_{m,a}\ge1$ with the following property. Let $H$ be large and suppose that, on $\{\hf+C_*\le H\}$ and for $s\ge s_1$,
\begin{equation}\label{s10:ratebootstrap}
 [h(s)]_m+[v(s)]_m<\frac12.
\end{equation}
Assume also the explicit geometric-radius bound
\[
 r_C(s)>L_{m,a}\bigl(1+\sqrt{H+2}\bigr)
 \qquad\text{for every }s\ge s_1.
\]
Then
\begin{equation}\label{s10:weightedint}
 [h(s)]_m+[v(s)]_m
 \le C_m e^{(\hf+C_*)/2} I_a(s)
 \quad(s\ge s_1).
\end{equation}
These parameters and constants are fixed for $m,a$ and the reference data, independently of $H$ and $s_1$.
\end{Lemma}
\begin{proof}
The choice $\kappa_1=1/2$ is possible because $a<1$. If a larger core or smaller geometric tolerance is needed, the already proved compact smooth convergence supplies the corresponding core bound at all sufficiently late times. Proposition~\ref{s10:Cpersistence} then applies with these new fixed parameters. This changes only the waiting time and the fixed constants, not $\psi$.

We first specify the domains on which the intrinsic estimate is used. At each $s\ge s_1$, choose
\[
 L_{m,a}\bigl(1+\sqrt{H+2}\bigr)<L_s<r_C(s)
\]
and an embedding $\iota_s$ of a neighborhood of $\overline{\Omega_{L_s}}$ into the modified slice. By Definition~\ref{s10:defC},
\[
 \|\iota_s^*\gz(s)-\hg\|_{C^{q}_0(\Omega_{L_s})}<\sigma,
 \qquad \fz(\iota_s(\hat p),s)\le n.
\]
Choose $L_{m,a}$ large enough to include the fixed geometric thresholds and the constant in Lemma~\ref{s10:annulardistance}. That lemma, applied with $B=H+2-C_*$, gives
\[
 \{\fz(s)+C_*\le H+2\}
 \subset\iota_s(\Omega_{\lambda L_s})
 \subset\iota_s(\Omega_{\lambda'L_s}).
\]
The first chart image is the domain of \eqref{s10:weightedvelocityintrinsic}; the second supplies its spatial collar.

For $x$ with $\hf(x)+C_*\le H$, the bootstrap bound implies
\[
 \fz(\psi(x),s)+C_*=F(x,s)+C_*
 <\hf(x)+C_*+\frac12<H+1.
\]
Consequently, \eqref{s10:weightedvelocityintrinsic} applies at every $\psi(x)$ under consideration. The embeddings $\iota_s$ are chosen independently at each time to obtain this intrinsic estimate. All pullbacks and time integrations below use the single fixed map $\psi$.

Under \eqref{s10:ratebootstrap}, $\frac12\hg<G<\frac32\hg$. The connection difference is $A=G^{-1}*\hnd h$. In converting $\hnd^j\Phi(G,F)$ to current covariant derivatives for $j\le m$, derivatives of $A$ occur through order at most $m-1$, hence derivatives of $h$ through order at most $m$. The product formulas and \eqref{s10:ratebootstrap} give
\[
 [\partial_sh]_m+[\partial_sv]_m
 \le C_m\sum_{j=0}^m|\nabla_G^j\Phi(G,F)|_G
 \le C_m e^{(F+C_*)/2}\mathcal D(s)^a
 \le C_m'e^{(\hf+C_*)/2}\mathcal D(s)^a.
\]
For $m=0$ only metric equivalence is needed. Integrate from $s$ to infinity in the fixed reference connection, using $h,v\to0$ on this fixed compact domain. This gives \eqref{s10:weightedint}.
\end{proof}

\begin{proof}[Proof of Theorem~\ref{s10:main}]
Corollary~\ref{s10:noexit} and Propositions~\ref{s10:limit} and \ref{s10:identify} prove the qualitative fixed-gauge conclusion and its probability normalization. It remains to establish the quantitative assertion.

Fix $m$, $\gamma$, and $c$ as in \eqref{s10:ratechoices}. Choose $\gamma<\gamma_1<\vartheta/(1-2\vartheta)$ and parameters $\theta,a$ in Proposition~\ref{s10:decay} such that
\begin{equation}\label{s10:gamma1}
 I_a(s)\le C(1+s)^{-\gamma_1},\qquad c/2<\gamma<\gamma_1.
\end{equation}
Fix the geometric parameters of Lemma~\ref{s10:weightedconvergence} once for these $m,a$. As explained in its proof, compact smooth convergence and Proposition~\ref{s10:Cpersistence} give \eqref{s10:Cgrowth} for these parameters after a fixed waiting time. For an integer $j$, put $H=c\log j+C_*$. For every sufficiently large $j$ and every $s\ge j$,
\[
 r_C(s)\ge c_0Ae^{(s-s_0)/2}
 >L_{m,a}\bigl(1+\sqrt{c\log j+C_*+2}\bigr).
\]
Indeed the exponential lower bound at $s=j$ dominates the right-hand side, and increases for $s\ge j$. This verifies the radius hypothesis of Lemma~\ref{s10:weightedconvergence}. In particular, Lemma~\ref{s10:annulardistance} places the entire sublevel $\{\fz(s)+C_*\le H+2\}$ in $\iota_s(\Omega_{\lambda L_s})$, with the larger chart $\iota_s(\Omega_{\lambda'L_s})$ available, exactly as in that lemma.

For this fixed $j$, compact smooth convergence makes \eqref{s10:ratebootstrap} hold on $\Omega_j^{\log}$ for all sufficiently large $s$. Continue this bound backward in $s$ toward $j$, on terminal intervals $[s_1,\infty)$. On each such interval, Lemma~\ref{s10:weightedconvergence} and \eqref{s10:gamma1} give
\[
 \sup_{\Omega_j^{\log}}([h(s)]_m+[v(s)]_m)
 \le C_m e^{H/2}(1+s)^{-\gamma_1}
 \le C_m j^{c/2-\gamma_1}\qquad(s\ge j).
\]
Here the fixed factor $e^{C_*/2}$ is absorbed into $C_m$. For all large $j$, the last quantity is less than $1/4$. Thus a first failure of \eqref{s10:ratebootstrap} at a time $s\ge j$ is impossible. It holds for every $s\ge j$, and the pointwise version of \eqref{s10:weightedint}, with $e^{C_*/2}$ again absorbed into the constant, gives \eqref{s10:mainrate}, even with exponent $\gamma_1$. Weakening that exponent gives the stated estimate.
\end{proof}

\begin{Corollary}[Uniqueness of the AC tangent flow]\label{s10:uniqueness}
Under the assumptions of Theorem~\ref{s10:main}, every tangent flow at $z$ is isometric to the completed canonical flow of $(\hM,\hg,\hf)$ as a pointed noncollapsed Ricci flow limit space. The isometry preserves the time function, the regular Ricci flow spacetime, and the conjugate heat kernels, including the distinguished terminal kernel.
\end{Corollary}
\begin{proof}
Let $(Z_\infty,d_{Z_\infty},z_\infty)$ be a tangent flow along parabolic normalizations $s_i\to\infty$ in the convergence of \cite[Theorem~5.2 and Notation~5.38]{fangli2026structure}, and let $(\hat Z,d^{\hat Z},\hat z)$ be the completed reference flow. Using $\psi$ from Theorem~\ref{s10:main} and $\Psi_t$ from \eqref{s10:canonical}, put $v=s-\log(-t)$ and
\begin{equation}\label{s10:spacetimemap}
 \Xi_s(x,t)=\bigl(\varphi_{e^{-s}t}(\psi(\Psi_t(x))),e^{-s}t\bigr).
\end{equation}
These are time-preserving embeddings into the parabolically rescaled spacetime. Write $G(v)=\psi^*\gz(v)$ and $F(v)=\psi^*\fz(v)$. Pullback of the spatial metrics, potentials, and normalized kernels gives
\begin{align*}
 \Xi_s^*(e^sg)&=(-t)\Psi_t^*G(v),\qquad
 \Xi_s^*f=\Psi_t^*F(v),\\
 \Xi_s^*(e^{-ns/2}K(z;\cdot))
 &=(4\pi(-t))^{-n/2}\exp(-\Psi_t^*F(v)).
\end{align*}
For the original time vector $\partial_{\mathfrak t}$, direct differentiation gives
\begin{equation}\label{s10:timevector}
 \Xi_s^*(e^{-s}\partial_{\mathfrak t})-\partial_t
 =\frac1{-t}\Psi_t^*\bigl(\nabla^{G(v)}F(v)-\hnd\hf\bigr).
\end{equation}
Theorem~\ref{s10:main} gives spatial convergence on every collared compact negative-time domain. Differentiating \eqref{s10:evolution} gives convergence of all time derivatives, since the reference pair is stationary. Thus all the displayed quantities converge smoothly to those of the canonical flow.

The kernel bound keeps each fixed compact chart at bounded pointed spacetime distance. On collared domains, the metric and time-vector bounds provide ambient parabolic neighborhoods with uniformly positive curvature radius. In the regular convergence charts of \cite[Theorem~5.2]{fangli2026structure}, the comparison maps and their local inverses therefore have uniform $C^k$ bounds on smaller domains for every $k$. A diagonal subsequence on a nested exhaustion gives compatible limits on overlaps, hence a local spacetime isometry
\[
 J:\hM\times(-\infty,0)\longrightarrow(\mathcal R_\infty)_{t<0},
 \qquad J_t^*g_\infty(t)=\hg(t),\qquad
 J^*\partial_{\mathfrak t_\infty}=\partial_t.
\]
By \cite[Theorem~5.20]{fangli2026structure}, $J_t^*\mathrm d\nu_{\infty,t}=\mathrm d\hnu_t$, where $\hnu_t$ is the canonical probability measure. The regular target slice is connected by \cite[Proposition~7.9]{fangli2026structure}, so completeness makes $J_t$ a covering. Both distinguished measures have mass one, and its sheet number satisfies
\[
 1=\int_{\hM}J_t^*\mathrm d\nu_{\infty,t}
   =N_t\int_{(\mathcal R_\infty)_t}\mathrm d\nu_{\infty,t}=N_t.
\]
Hence $J_t$ is an isometry onto the entire regular slice.

The identified metrics are complete and have bounded curvature on compact negative-time intervals. By \cite[Theorem~7.19]{fangli2026structure}, all negative-time points belong to the associated metric flow. Its slices are the intrinsic completions of their regular slices by \cite[Theorem~2.4(c)]{bamler2020structure}; completeness therefore leaves no additional points. Since $J$ preserves the time vector, uniqueness of bounded heat evolution on these complete bounded-curvature time slabs identifies every regular-based conjugate kernel; see also \cite[Lemma~5.22 and Remark~5.23]{fangli2026structure}.

The spatial distances and kernels now agree, so \cite[Theorem~1.7(e)]{fangli2026structure} identifies the spacetime distances. Density of the negative-time regular part, as in \cite[Corollary~5.7]{fangli2026structure}, extends $J$ to an isometry of the completions. The center estimates give
\[
 d_{Z_\infty}(J(\hat p,t),z_\infty)
 +d^{\hat Z}((\hat p,t),\hat z)\le C\sqrt{-t},
\]
so the extension sends $\hat z$ to $z_\infty$. Continuity of conjugate measures \cite[Corollary~6.21]{fangli2026structure} identifies all terminal kernels. The Euclidean tangent characterization \cite[Definition~7.14 and Theorem~7.15]{fangli2026structure} preserves terminal regularity, and the negative-time spacetime identification extends smoothly there.
\end{proof}

Retain the setting of \cref{s10:main} and the completed spacetime distance $d_Z$ on $Z$. Write $Z_0=\mathfrak t^{-1}(0)$ and let $\mathcal S=Z\setminus\mathcal R$ be the singular set of $Z$. Smoothness at negative times gives $\mathcal S\subset Z_0$.
\phantomsection\label{idx:timezero-distance}
For $x,y\in Z_0$, the intrinsic time-slice distance is
\[
 d_0^Z(x,y):=\lim_{t\uparrow0}
 W_1^{g(t)}(\nu_{x;t},\nu_{y;t})\in[0,\infty].
\]
By \cite[Definition~6.1 and Lemmas~6.3--6.4]{fangli2026structure}, this is an extended metric and $d_Z(x,y)\le d_0^Z(x,y)$.

\begin{Corollary}[Isolation of an AC singularity]\label{s10:iso-main}
Suppose $z\in\mathcal S$ satisfies the hypotheses of \cref{s10:main}. Then there exists $r_0>0$ such that
\begin{equation}\label{s10:iso-conclusion}
 B_{d_Z}(z,r_0)\cap\mathcal S=\{z\}.
\end{equation}
In particular, $z$ is isolated in $\mathcal S$ also with respect to $d_0^Z$.
\end{Corollary}
\begin{proof}
Let $(\hat Z,d^{\hat Z},\hat z,\hat{\mathfrak t})$ be the completed canonical shrinker flow, with regular part $\hat{\mathcal R}$. By \cite[Proposition~2.1]{kotschwar2015rigidity}, its restriction to each sufficiently distant end extends smoothly to the asymptotic cone at time zero. The endpoint map from $\hM$ onto $\hat Z_0$ is surjective, and the image of a compact core has bounded intrinsic distance from $\hat z$; see \cite[Lemmas~5.1 and~5.4]{li2026kahlerlimits}. The soliton dilations described in \cite[Section~5.2]{li2026kahlerlimits} fix $\hat z$, multiply the terminal radius, and preserve regularity. Every terminal point other than $\hat z$ can therefore be dilated beyond the image of the core, where the smooth extension applies. Dilating back gives
\begin{equation}\label{s10:iso-modelregular}
 \hat Z_0\setminus\{\hat z\}\subset\hat{\mathcal R}.
\end{equation}

Suppose \eqref{s10:iso-conclusion} fails. Choose $w_i\in\mathcal S\setminus\{z\}\subset Z_0$ with $r_i:=d_Z(z,w_i)\to0$. Passing to a subsequence, the parabolic rescalings $(Z,r_i^{-1}d_Z,z,r_i^{-2}\mathfrak t)$ converge to a tangent flow $(Z_\infty,d_{Z_\infty},z_\infty,\mathfrak t_\infty)$, and $w_i$ converges to a terminal point $w_\infty$ with $d_{Z_\infty}(z_\infty,w_\infty)=1$. By \cref{s10:uniqueness}, this tangent flow is the completed canonical flow, with $z_\infty$ corresponding to $\hat z$. Hence \eqref{s10:iso-modelregular} makes $w_\infty$ regular. Choose $\rho>0$ so that $\mathcal N_{w_\infty}(\rho^2)>-\eps/2$, where $\eps$ is the constant in \cite[Proposition~7.7]{fangli2026structure}. Approximate each $w_i$ by smooth negative-time points, with pointed convergence and entropy errors tending to zero. Entropy continuity, \cite[Corollary~6.21 and Lemma~7.2]{fangli2026structure}, applied diagonally gives $\mathcal N_{w_i}(\rho^2)>-\eps$ for all large $i$ in the rescaled flows. The cited $\eps$-regularity then gives a positive curvature radius at $w_i$, contradicting its singularity.

Finally, $d_Z\le d_0^Z$ implies $B_{d_0^Z}(z,r_0)\subset B_{d_Z}(z,r_0)\cap Z_0$, giving the last assertion.
\end{proof}

\clearpage
\appendix
\titleformat{\section}{\normalfont\Large\bfseries}{Appendix~\thesection.}{0.75em}{}
\section{Uniform Euclidean H\"older estimate for the gauge operator}
\label{app:euclidean-model}

We prove the constant-coefficient estimate used in
\cref{lem:fixed-level-Bianchi-correction}. All derivatives and norms
are Euclidean on $\mathbb R^n$. Fix $0<c<C<\infty$ and
$0<\kappa_0\le1$.

For $\kappa\in[\kappa_0,1]$, $0<\epsilon\leq1$, and a constant real covector
$\omega$ with $c\leq|\omega|\leq C$, define
\begin{align}
 \mathscr B_\epsilon U
 &=-\operatorname{Sym}((\epsilon\partial+\kappa\omega)\otimes U),\notag\\
 (\mathscr D_\epsilon V)_j
 &=\epsilon\partial^iV_{ij}-\omega^iV_{ij},\qquad
 \mathscr P_\epsilon=\mathscr D_\epsilon\mathscr B_\epsilon.
 \label{eq:euclidean-Holder-model}
\end{align}

\begin{Lemma}
\label{lem:euclidean-divergence-Holder}
Fix $m\in\mathbb Z_{\geq0}$ and $\alpha\in(0,1)$.  Suppose a smooth
compactly supported one-form $U$ satisfies
\begin{equation*}
 \mathscr P_\epsilon U=F_0+\epsilon\partial_iF_1^i
\end{equation*}
with bounded $C^{m,\alpha}$ source coefficients $F_0,F_1$.  Then
\begin{equation}\label{ineq:appendix-parameter-divergence-Schauder}
 \|U\|_{C^{m,\alpha}}+\epsilon\|\partial U\|_{C^{m,\alpha}}
 \leq C_{m,\alpha,\kappa_0}
   \bigl(\|F_0\|_{C^{m,\alpha}}+\|F_1\|_{C^{m,\alpha}}\bigr).
\end{equation}
The constant depends only on $m,\alpha,n,c,C,\kappa_0$ and is
uniform in $0<\epsilon\le1$, $c\le|\omega|\le C$, and
$\kappa\in[\kappa_0,1]$.
\end{Lemma}
\begin{proof}
\emph{Step 1.}
Define the scalar operator and the two auxiliary maps
\begin{align*}
 H_\epsilon&=-\epsilon^2\Delta
     +(1-\kappa)\epsilon\omega\cdot\partial+\kappa|\omega|^2,\\
 b&=\epsilon\partial+\kappa\omega,\qquad
 d=-\epsilon\operatorname{div}+\omega\cdot.
\end{align*}
The map $b$ sends scalars to vectors, and $d$ sends vectors to
scalars.  Direct expansion gives $db=H_\epsilon$.  Moreover,
\begin{equation*}
 (\mathscr P_\epsilon U)_j
 =\frac12\sum_i(-\epsilon\partial_i+\omega_i)
                    (b_iU_j+b_jU_i)
 =\frac12\bigl(H_\epsilon U_j+b_j(dU)\bigr).
\end{equation*}
Thus $\mathscr P_\epsilon=\frac12(H_\epsilon I+bd)$.
All coefficients are constant, so $H_\epsilon$ commutes with $b,d$,
and $(bd)^2=b(db)d=H_\epsilon bd$.  Consequently
\begin{equation*}
 \frac12(H_\epsilon I+bd)
       (2H_\epsilon^{-1}I-H_\epsilon^{-2}bd)
 =I+\frac12H_\epsilon^{-1}bd
       -\frac12H_\epsilon^{-2}(bd)^2=I.
\end{equation*}
The reverse product is the identity as well.  We obtain
\begin{equation}\label{eq:appendix-scalar-resolvent-formula}
 \mathscr P_\epsilon^{-1}
       =2H_\epsilon^{-1}I-H_\epsilon^{-2}bd.
\end{equation}
Here $H_\epsilon^{-2}=(H_\epsilon^{-1})^2$, and scalar inverses act
componentwise.

Put $b_0=(1-\kappa)\omega$ and $c_0=\kappa|\omega|^2$.
Throughout the parameter range,
$|b_0|\le C$ and $\kappa_0c^2\le c_0\le C^2$.
These bounds make the kernel estimates below uniform, including
$\kappa=1$, when $b_0=0$.
The solution of $\partial_tv=-H_\epsilon v$, $v(0)=f$, is
\begin{equation*}
 (e^{-tH_\epsilon}f)(x)
 =e^{-c_0t}(e^{\epsilon^2t\Delta}f)(x-\epsilon b_0t).
\end{equation*}
Indeed, diffusion, constant translation, and multiplication by
$e^{-c_0t}$ commute.  In particular the convolution kernel is the
Gaussian centered at $\epsilon b_0t$.  Integration in $t$ gives
$H_\epsilon^{-1}=\int_0^\infty e^{-tH_\epsilon}\,\mathrm{d}t$, since
$H_\epsilon e^{-tH_\epsilon}=-\partial_te^{-tH_\epsilon}$ and the
endpoint at infinity vanishes.  The semigroup identity also gives
\begin{equation*}
 H_\epsilon^{-2}
 =\int_0^\infty\!\int_0^\infty e^{-(s+t)H_\epsilon}\,\mathrm{d}s\,\mathrm{d}t
 =\int_0^\infty t e^{-tH_\epsilon}\,\mathrm{d}t.
\end{equation*}
Thus, denoting the convolution kernel by $G_{\epsilon,j}$, we have
\begin{equation}\label{eq:appendix-heat-kernel-resolvent}
 \begin{split}
 G_{\epsilon,j}(x)
 &=\frac1{(j-1)!}\int_0^\infty
    t^{j-1}e^{-c_0t}(4\pi\epsilon^2t)^{-n/2}
       \exp\left(-\frac{|x-\epsilon b_0t|^2}{4\epsilon^2t}\right)\mathrm{d}t,\\
 H_\epsilon^{-j}f&=G_{\epsilon,j}*f,\qquad j=1,2.
 \end{split}
\end{equation}
The kernel formula is pointwise away from the origin and has its
usual locally integrable interpretation at the origin.

\emph{Step 2.}
Set $T_\epsilon=\mathscr P_\epsilon^{-1}$.  Since derivatives commute
with this inverse,
\begin{align*}
 U&=T_\epsilon F_0+\epsilon\partial_iT_\epsilon F_1^i,\\
 \epsilon\partial_jU
 &=\epsilon\partial_jT_\epsilon F_0
       +\epsilon^2\partial_j\partial_iT_\epsilon F_1^i.
\end{align*}
Expanding $b,d$ in \cref{eq:appendix-scalar-resolvent-formula}, all
operators on the right are linear combinations of
\begin{equation*}
 T_{\epsilon,I,j}:=(\epsilon\partial)^I H_\epsilon^{-j},
 \qquad j=1,2,\qquad |I|\leq2j.
\end{equation*}
We prove that these operators are uniformly bounded on $C^{0,\alpha}$.
It is enough first to treat $\epsilon=1$.  Write
$p_t(z)=(4\pi t)^{-n/2}e^{-|z|^2/(4t)}$ and $\ell=|I|$.
Gaussian scaling gives
\begin{equation*}
 \|\partial^I p_t\|_1\leq C_I t^{-\ell/2},\qquad
 \int_{\mathbb R^n}|z|^\alpha|\partial^I p_t(z)|\,\mathrm{d}z
       \leq C_{I,\alpha}t^{(\alpha-\ell)/2}.
\end{equation*}
If $\ell<2j$, the kernel is integrable, because
\begin{equation*}
 \|\partial^I G_{1,j}\|_1
 \leq C\int_0^\infty t^{j-1-\ell/2}e^{-c_0t}\,\mathrm{d}t<\infty.
\end{equation*}
Convolution by an $L^1$ kernel bounds both the sup norm and the
homogeneous H\"older seminorm by the corresponding norm of $f$.

When $\ell=2j$, this absolute-integrability argument fails at $t=0$.
We instead use the cancellation $\int\partial^I p_t=0$.  With
$x_t=x-b_0t$,
\begin{equation*}
 \partial^I(e^{t\Delta}f)(x_t)
 =\int_{\mathbb R^n}\partial^I p_t(z)
                      \bigl(f(x_t-z)-f(x_t)\bigr)\,\mathrm{d}z.
\end{equation*}
Consequently, if $Q_tf(x)=\partial^I(e^{t\Delta}f)(x-b_0t)$,
\begin{equation*}
 \|Q_tf\|_\infty\leq C[f]_\alpha t^{(\alpha-\ell)/2},\qquad
 \|\partial Q_tf\|_\infty
       \leq C[f]_\alpha t^{(\alpha-\ell-1)/2}.
\end{equation*}
The second bound uses the same cancellation for one more derivative.
The formula
\begin{equation*}
 T_{1,I,j}f=\frac1{(j-1)!}\int_0^\infty
                         t^{j-1}e^{-c_0t}Q_tf\,\mathrm{d}t
\end{equation*}
is therefore absolutely convergent in the sup norm for H\"older data,
even though the differentiated spatial kernel is not absolutely
integrable.  It agrees with the distributional derivative of the
resolvent: truncate the time integral at $t=s>0$ and then let $s\downarrow0$.
In particular,
\begin{equation*}
 \|T_{1,I,j}f\|_\infty
 \leq C[f]_\alpha\int_0^\infty t^{\alpha/2-1}e^{-c_0t}\,\mathrm{d}t
 \leq C[f]_\alpha.
\end{equation*}
For $h=|x-y|>0$, split the time integral at $t=h^2$.  Use the sup
bound on $(0,h^2)$ and the gradient bound on $(h^2,\infty)$ to obtain
\[
|T_{1,I,j}f(x)-T_{1,I,j}f(y)|
 \leq C[f]_\alpha\left(
       \int_0^{h^2}t^{\alpha/2-1}\,\mathrm{d}t
       +h\int_{h^2}^\infty t^{(\alpha-3)/2}\,\mathrm{d}t\right) 
 \leq C[f]_\alpha h^\alpha.
\]
Both integrals converge because $0<\alpha<1$.  Thus, in all cases
$|I|\leq2j$,
\begin{equation*}
 \|T_{1,I,j}f\|_\infty
 \leq C(\|f\|_\infty+[f]_\alpha),\qquad
 [T_{1,I,j}f]_\alpha\leq C[f]_\alpha.
\end{equation*}
The second estimate contains no sup norm on its right-hand side.
This distinction is needed when restoring the parameter.

Indeed, \cref{eq:appendix-heat-kernel-resolvent} gives
$G_{\epsilon,j}(x)=\epsilon^{-n}G_{1,j}(x/\epsilon)$.
For $f_\epsilon(y)=f(\epsilon y)$ it follows that
\begin{equation*}
 (T_{\epsilon,I,j}f)(\epsilon y)=T_{1,I,j}f_\epsilon(y),
 \qquad [f_\epsilon]_\alpha=\epsilon^\alpha[f]_\alpha.
\end{equation*}
Hence
\begin{align*}
 \|T_{\epsilon,I,j}f\|_\infty
 &\leq C(\|f\|_\infty+\epsilon^\alpha[f]_\alpha),\\
 [T_{\epsilon,I,j}f]_\alpha
 &=\epsilon^{-\alpha}[T_{1,I,j}f_\epsilon]_\alpha
 \leq C[f]_\alpha.
\end{align*}
These estimates are uniform for $0<\epsilon\leq1$.  Substitution in
the formulas for $U$ and $\epsilon\partial U$ proves
\cref{ineq:appendix-parameter-divergence-Schauder} for $m=0$.
Applying the same estimate to $\partial^I U$ and summing for $|I|\leq m$
proves it for every $m$, since ordinary derivatives commute with the
constant-coefficient equation.  The source is kept in divergence
form throughout, so no additional derivative of $F_1$ is required.

\end{proof}

\clearpage
\section*{Index of notation and terminology}
\phantomsection
\addcontentsline{toc}{section}{Index of notation and terminology}
\begingroup
\fontsize{10}{11.3}\selectfont
\setlength{\columnsep}{2em}
\setlength{\parindent}{0pt}
\setlength{\parskip}{0pt}
\newcommand{\indexpage}[1]{\hyperref[#1]{\textcolor{linkblue}{\pageref*{#1}}}}
\newcommand{\notationentry}[2]{\par\hangindent=1em\hangafter=1 #1, #2\par}
\begin{multicols}{2}
\raggedcolumns
\notationentry{$a$, $\gamma_b$}{\indexpage{s10:radialfield}, \indexpage{s10:leveleq}}
\notationentry{Analytic center graph}{\indexpage{cor:finite-LS-analytic-center-reduction}}
\notationentry{$[u]_m$ (Section~6)}{\indexpage{s10:slicenotation}}
\notationentry{$b$, $\hat b$}{\indexpage{s10:radial}}
\notationentry{$B_\kappa$}{\indexpage{eq:weakened-gaussian-bianchi-adjoint}}
\notationentry{$c_\kappa(\xi)$}{\indexpage{cor:finite-LS-analytic-center-reduction}}
\notationentry{$\mathcal C$}{\indexpage{eq:finite-LS-fixed-splittings}}
\notationentry{$\mathcal C_\kappa$}{\indexpage{idx:bianchi-correction}}
\notationentry{Canonical shrinker flow}{\indexpage{s10:canonical}}
\notationentry{Cone $\mathscr C$, $g_C$}{\indexpage{idx:AC-geometry}}
\notationentry{$C^{k,\alpha}_\beta$}{\indexpage{def:prelim-weighted-Holder}}
\notationentry{$C^{k,\alpha}_{\beta,\mathbf X}$ (critical graph space)}{\indexpage{def:prelim-graph-spaces}}
\notationentry{$C^{k,\alpha}_{\Lambda;\beta}$ (finite expansions)}{\indexpage{def:prelim-finite-expansions}}
\notationentry{$C^\infty_0$, $C^\infty_{1;0}$}{\indexpage{def:prelim-ordinary-spaces}}
\notationentry{Cylindrical representative $\mathscr T_\beta$}{\indexpage{eq:prelim-cylindrical-representative}}
\notationentry{Complete extension}{\indexpage{s10:completion}}
\notationentry{Cutoff comparison}{\indexpage{s10:cutoffcomparison}}
\notationentry{Spacetime distances $d^*$, $d_Z$}{\indexpage{idx:completion-metric}}
\notationentry{Time-zero distance $d_0^Z$}{\indexpage{idx:timezero-distance}}
\notationentry{$\mathcal D$, $\mathcal D(s)$}{\indexpage{s10:slicenotation}, \indexpage{s10:energy}}
\notationentry{$\mathrm D$}{\indexpage{idx:operator-conventions}}
\notationentry{$D_b=b\partial_b$}{\indexpage{idx:scaled-radial}}
\notationentry{$D_g$, $\hat D$}{\indexpage{eq:prelim-metric-abbreviations}}
\notationentry{$E(s)$}{\indexpage{s10:energy}}
\notationentry{$e_\kappa(\xi)$}{\indexpage{eq:finite-LS-center-and-reduced-energy}}
\notationentry{Entropy radius}{\indexpage{s10:defEP}}
\notationentry{$\eta_\kappa$}{\indexpage{eq:finite-LS-center-graph}}
\notationentry{$\hat f$}{\indexpage{idx:entropy-conventions}}
\notationentry{$f_g$}{\indexpage{idx:pair-defect}, \indexpage{thm:existence-and-uniqueness-of-minimizer}}
\notationentry{Background drift $\mathbf{X}$}{\indexpage{eq:prelim-background-drift}}
\notationentry{Finite length}{\indexpage{s10:finitelength}}
\notationentry{Fixed probability density}{\indexpage{s10:evolution}}
\notationentry{Fixed gauge convergence}{\indexpage{s10:fixedgauge}}
\notationentry{Fixed-target reconstruction}{\indexpage{prop:corrected-orthogonal-fixed-target-normal-estimate}}
\notationentry{$\hat g$}{\indexpage{idx:entropy-conventions}}
\notationentry{$g^z$, $f^z$}{\indexpage{s10:modified}}
\notationentry{Gaussian weight}{\indexpage{idx:gaussian-measures}}
\notationentry{Geometric radius}{\indexpage{s10:defC}}
\notationentry{$H^1_{\boldsymbol{\rho}}$, $H^2_{\boldsymbol{\rho}}$}{\indexpage{idx:oscillator-domains}}
\notationentry{$H^m_\kappa$}{\indexpage{idx:gaussian-measures}}
\notationentry{$\mathcal S^j_{\mathrm{osc}}$}{\indexpage{lem:weakened-Gaussian-Sobolev-conjugation}}
\notationentry{$H^2_{\boldsymbol{\rho},\kappa}$}{\indexpage{eq:weakened-Gaussian-closed-domain}}
\notationentry{Isolation of an AC singularity}{\indexpage{s10:iso-main}}
\notationentry{$\mathcal K=\ker\hat P$}{\indexpage{idx:coordinate-graphs}}
\notationentry{$\mathcal K_\kappa$}{\indexpage{idx:bianchi-operators}}
\notationentry{Infinitesimal solitonic deformations $\ISD$}{\indexpage{idx:ISD}}
\notationentry{$\Omega_L$}{\indexpage{s10:radial}}
\notationentry{$\kappa$, $\kappa'$}{\indexpage{idx:gaussian-measures}, \indexpage{eq:finite-LS-physical-delta-choice}}
\notationentry{Kernel based at a terminal point}{\indexpage{s10:physicalkernel}}
\notationentry{$\hat L$}{\indexpage{eq:prelim-ISD}, \indexpage{eq:finite-LS-strong-background-Hessian}}
\notationentry{$L^2_\kappa$}{\indexpage{idx:gaussian-measures}}
\notationentry{\L ojasiewicz--Simon inequality}{\indexpage{thm:finite-global-metric-LS}, \indexpage{thm:global-pair-LS}}
\notationentry{Lyapunov--Schmidt reduction}{\indexpage{cor:finite-LS-analytic-center-reduction}}
\notationentry{Maximal operator domain}{\indexpage{idx:oscillator-domains}}
\notationentry{Modified Ricci flow}{\indexpage{s10:modified}}
\notationentry{$\boldsymbol{\mu}$}{\indexpage{idx:entropy-conventions}}
\notationentry{$\Phi(g,f)$, $\Phi(g)$}{\indexpage{idx:pair-defect}, \indexpage{s10:slicenotation}, \indexpage{s10:defect}}
\notationentry{$\nabla^b$, $\nabla^{b,j}$}{\indexpage{idx:level-connection}}
\notationentry{$O_{k+\alpha}(r^\beta)$}{\indexpage{idx:weighted-O}}
\notationentry{Normal inverse}{\indexpage{thm:finite-LS-normal-inverse-family}}
\notationentry{$\mathrm dV_{\hat f}$, $\mathrm dV_{\hat f,\kappa}$}{\indexpage{idx:gaussian-measures}}
\notationentry{$\mathrm dV_f$, $\mathrm dV_{f_g}$}{\indexpage{idx:pair-defect}}
\notationentry{$P_\kappa$}{\indexpage{idx:bianchi-operators}}
\notationentry{$\Pi_\kappa$}{\indexpage{prop:weakened-projection-mapping-property}}
\notationentry{$\Pi_{\mathcal C}$, $\Pi_{\mathscr F'}$}{\indexpage{eq:finite-LS-fixed-splittings}}
\notationentry{Pointed \L ojasiewicz--Simon inequality}{\indexpage{s10:pointedLS}}
\notationentry{Parameter dependence in Section~6}{\indexpage{fig:s10:dependencies}}
\notationentry{Potential coordinates}{\indexpage{s10:potentialcoords}}
\notationentry{Potential pullback}{\indexpage{lem:finite-potential-pullback}}
\notationentry{Pair hypotheses}{\indexpage{eq:finite-pair-hypotheses}}
\notationentry{$\mathscr Q_\kappa$}{\indexpage{eq:finite-LS-quasigradient}}
\notationentry{$Q$}{\indexpage{s10:Q}}
\notationentry{$\boldsymbol r$}{\indexpage{s10:shell}}
\notationentry{$R_*$}{\indexpage{s10:instantuniform}}
\notationentry{$r_E$}{\indexpage{s10:rE}}
\notationentry{$r_C^{(\sigma,q)}$, $r_C$}{\indexpage{s10:defC}}
\notationentry{$\mathscr R_{\kappa,h}$}{\indexpage{eq:corrected-orthogonal-fixed-target-reconstruction}}
\notationentry{Radial transport $\boldsymbol{\tau}$, $\boldsymbol{\tau}_r$}{\indexpage{def:prelim-AC-degree}}
\notationentry{Matched product coordinates}{\indexpage{s10:radiuscomparison}}
\notationentry{Relative radial estimate}{\indexpage{s10:relative}}
\notationentry{$h_B$, $w_B$, $\eta_B$}{\indexpage{s10:etaB}}
\notationentry{$\mathcal S$}{\indexpage{s10:iso-main}}
\notationentry{$\boldsymbol{\rho}$}{\indexpage{eq:prelim-background-drift}, \indexpage{idx:oscillator-weight}}
\notationentry{Soliton defect}{\indexpage{idx:pair-defect}, \indexpage{s10:slicenotation}}
\notationentry{Strong uniqueness of tangent flows}{\indexpage{s10:main}, \indexpage{s10:uniqueness}}
\notationentry{Time function $\mathfrak t$, spacetime point $x^*$}{\indexpage{s10:setup}}
\notationentry{Prescribed homogeneous component $[u]_\gamma$}{\indexpage{def:prelim-finite-expansions}}
\notationentry{$W(h,u)$}{\indexpage{eq:function-for-minimizer-IFT}}
\notationentry{$\mathcal W$}{\indexpage{idx:entropy-conventions}, \indexpage{thm:global-pair-LS}}
\notationentry{Weakened Gaussian weight}{\indexpage{idx:gaussian-measures}, \indexpage{sec:appendix-weakened-gaussian}}
\notationentry{$Z_\kappa$}{\indexpage{idx:gaussian-measures}}
\notationentry{$\delta_f$, $\delta^*$}{\indexpage{idx:operator-conventions}}
\notationentry{Weighted Laplacian $\Delta_f$}{\indexpage{idx:operator-conventions}}
\notationentry{$E_\Sigma$, $\Sigma_r$}{\indexpage{def:prelim-AC-degree}, \indexpage{eq:prelim-cylindrical-representative}}
\notationentry{$E_m$, $E_\infty$}{\indexpage{def:prelim-radial-filtration}}
\notationentry{$\mathscr H$, $\mathscr H'$, $\mathscr F$, $\mathscr F'$}{\indexpage{eq:finite-tensor-graphs}}
\notationentry{$\mathscr Z$, $\mathcal T$, $L_{1/2}$, $G_{1/2}$}{\indexpage{eq:finite-scalar-graphs}}
\notationentry{$\mathscr X_k$, $\mathscr Y_k$, $\mathscr S_k$, $\hat P$}{\indexpage{idx:coordinate-graphs}}
\notationentry{$a_h$, $a_u$}{\indexpage{prop:finite-graph-realization}}
\notationentry{Original-to-graph inclusion}{\indexpage{eq:original-to-graph}}
\notationentry{Fixed tensor target $\mathcal V$}{\indexpage{eq:finite-orders}}
\end{multicols}
\endgroup
\clearpage

\phantomsection
\addcontentsline{toc}{section}{References}
\bibliographystyle{amsalpha}
\bibliography{references}

\newcommand{\etalchar}[1]{$^{#1}$}
\providecommand{\LiWangHeatDOI}{https://doi.org/10.1007/s00526-020-01861-y}
\providecommand{\bysame}{\leavevmode\hbox to3em{\hrulefill}\thinspace}
\providecommand{\MR}{\relax\ifhmode\unskip\space\fi MR }
\providecommand{\MRhref}[2]{%
  \href{http://www.ams.org/mathscinet-getitem?mr=#1}{#2}
}
\providecommand{\href}[2]{#2}
\begin{thebibliography}{CCG{\etalchar{+}}08}

\bibitem[Aub98]{aubin1998some}
Thierry Aubin, \emph{Some nonlinear problems in {R}iemannian geometry},
  Springer Monographs in Mathematics, Springer-Verlag, Berlin, 1998.

\bibitem[Bam20]{bamler2020structure}
Richard~H. Bamler, \emph{Structure theory of non-collapsed limits of {Ricci}
  flows}, arXiv preprint, 2020, \url{https://arxiv.org/abs/2009.03243v2}.

\bibitem[BMS02]{braverman2002essential}
Maxim Braverman, Ognjen Milatovic, and Mikhail~A. Shubin, \emph{Essential
  self-adjointness of {Schr\"odinger}-type operators on manifolds},
  \href{https://doi.org/10.1070/RM2002v057n04ABEH000532}{Russian Math. Surveys}
  \textbf{57} (2002), no.~4, 641--692. \MR{1942115}

\bibitem[CCG{\etalchar{+}}08]{chow2008ricci}
Bennett Chow, Sun-Chin Chu, David Glickenstein, Christine Guenther, James
  Isenberg, Tom Ivey, Dan Knopf, Peng Lu, Feng Luo, and Lei Ni, \emph{The
  {R}icci flow: techniques and applications. {P}art {II}}, Mathematical Surveys
  and Monographs, vol. 144, American Mathematical Society, Providence, RI,
  2008, Analytic aspects.

\bibitem[CCR15]{carlotto2015slowly}
Alessandro Carlotto, Otis Chodosh, and Yanir~A. Rubinstein, \emph{Slowly
  converging {Y}amabe flows},
  \href{https://doi.org/10.2140/gt.2015.19.1523}{Geom. Topol.} \textbf{19}
  (2015), no.~3, 1523--1568.

\bibitem[CM14]{colding2014tangent}
Tobias~H. Colding and William~P. Minicozzi, II, \emph{On uniqueness of tangent
  cones for {E}instein manifolds},
  \href{https://doi.org/10.1007/s00222-013-0474-z}{Invent. Math.} \textbf{196}
  (2014), no.~3, 515--588.

\bibitem[CM15]{colding2015uniqueness}
\bysame, \emph{Uniqueness of blowups and {{\L}ojasiewicz} inequalities},
  \href{https://doi.org/10.4007/annals.2015.182.1.5}{Ann. of Math. (2)}
  \textbf{182} (2015), no.~1, 221--285.

\bibitem[CMZ24]{chan2024closed}
Pak-Yeung Chan, Zilu Ma, and Yongjia Zhang, \emph{On {R}icci flows with closed
  and smooth tangent flows},
  \href{https://doi.org/10.1007/s00526-024-02778-6}{Calc. Var. Partial
  Differential Equations} \textbf{63} (2024), no.~7, Paper No. 179, 31.

\bibitem[CS21]{chodosh2021conical}
Otis Chodosh and Felix Schulze, \emph{Uniqueness of asymptotically conical
  tangent flows}, \href{https://doi.org/10.1215/00127094-2020-0098}{Duke Math.
  J.} \textbf{170} (2021), no.~16, 3601--3657. \MR{4332673}

\bibitem[CZ06]{chen2006uniqueness}
Bing-Long Chen and Xi-Ping Zhu, \emph{Uniqueness of the {Ricci} flow on
  complete noncompact manifolds},
  \href{https://doi.org/10.4310/jdg/1175266184}{J. Differential Geom.}
  \textbf{74} (2006), no.~1, 119--154.

\bibitem[CZ12]{cao2012second}
Huai-Dong Cao and Meng Zhu, \emph{On second variation of {Perelman}'s {Ricci}
  shrinker entropy}, \href{https://doi.org/10.1007/s00208-011-0701-0}{Math.
  Ann.} \textbf{353} (2012), no.~3, 747--763.

\bibitem[FL26a]{fangli2026structure}
Hanbing Fang and Yu~Li, \emph{On the structure of noncollapsed {R}icci flow
  limit spaces}, arXiv preprint, 2026,
  \url{https://arxiv.org/abs/2510.12398v3}.

\bibitem[FL26b]{fangli2026generalized}
\bysame, \emph{Strong uniqueness and rectifiability of generalized cylindrical
  singularities in {R}icci flow}, arXiv preprint, 2026,
  \url{https://arxiv.org/abs/2605.17001v1}.

\bibitem[FL26c]{fangli2026strong}
\bysame, \emph{Strong uniqueness of tangent flows at cylindrical singularities
  in {R}icci flow}, arXiv preprint, 2026,
  \url{https://arxiv.org/abs/2510.20320v2}.

\bibitem[KW15]{kotschwar2015rigidity}
Brett Kotschwar and Lu~Wang, \emph{Rigidity of asymptotically conical shrinking
  gradient {R}icci solitons}, \href{https://doi.org/10.4310/jdg/1427202764}{J.
  Differential Geom.} \textbf{100} (2015), no.~1, 55--108. \MR{3326574}

\bibitem[LW20]{li2020heat}
Yu~Li and Bing Wang, \emph{Heat kernel on {R}icci shrinkers},
  \href{https://doi.org/10.1007/s00526-020-01861-y}{Calc. Var. Partial
  Differential Equations} \textbf{59} (2020), no.~6,
  \href{\LiWangHeatDOI}{Paper No. 194}, 84. \MR{4169048}

\bibitem[LYZ26]{LYZ-I}
Yu~Li, Longhuan Yan, and Wenjia Zhang, \emph{{On the geometry of asymptotically
  conical Ricci shrinkers (I): weighted analysis}}, arXiv preprint, 2026,
  \url{https://arxiv.org/abs/2610.39814}.

\bibitem[LZ26]{li2026kahlerlimits}
Yu~Li and Junsheng Zhang, \emph{{G}romov--{H}ausdorff limits of noncollapsed
  {K}\"ahler--{R}icci flows and the geometry of {R}icci shrinkers}, arXiv
  preprint, 2026, \url{https://arxiv.org/abs/2607.25644v1}.

\bibitem[Per02]{perelman2002entropy}
Grisha Perelman, \emph{The entropy formula for the {Ricci} flow and its
  geometric applications}, arXiv preprint, 2002,
  \url{https://arxiv.org/abs/math/0211159v1}.

\bibitem[R{\aa}d92]{rade1992yangmills}
Johan R{\aa}de, \emph{On the {Y}ang--{M}ills heat equation in two and three
  dimensions}, \href{https://doi.org/10.1515/crll.1992.431.123}{J. Reine Angew.
  Math.} \textbf{431} (1992), 123--163. \MR{1179335}

\bibitem[Sch14]{schulze2014compact}
Felix Schulze, \emph{Uniqueness of compact tangent flows in mean curvature
  flow}, \href{https://doi.org/10.1515/crelle-2012-0070}{J. Reine Angew. Math.}
  \textbf{690} (2014), 163--172.

\bibitem[Sim83]{simon1983asymptotics}
Leon Simon, \emph{Asymptotics for a class of non-linear evolution equations,
  with applications to geometric problems},
  \href{https://doi.org/10.2307/2006981}{Ann. of Math. (2)} \textbf{118}
  (1983), no.~3, 525--571. \MR{727703}

\bibitem[{\v S}TW04]{sesum2004notes}
Nata{\v s}a {\v S}e{\v s}um, Gang Tian, and Xiaodong Wang, \emph{Notes on
  {P}erelman's paper on the entropy formula for the {R}icci flow and its
  geometric applications}, Unpublished notes,
  \url{https://users.math.msu.edu/users/xwang/perel.pdf}, 2004.

\bibitem[SW15]{sun2015on}
Song Sun and Yuanqi Wang, \emph{On the {K}\"ahler-{R}icci flow near a
  {K}\"ahler-{E}instein metric},
  \href{https://doi.org/10.1515/crelle-2013-0004}{J. Reine Angew. Math.}
  \textbf{699} (2015), 143--158. \MR{3305923}

\bibitem[Wan23]{wang2023gradient}
Jie Wang, \emph{Gradient estimates for a class of elliptic and parabolic
  equations on {R}iemannian manifolds},
  \href{https://doi.org/10.1007/s11464-021-0420-0}{Front. Math.} \textbf{18}
  (2023), no.~5, 999--1024. \MR{4647880}

\bibitem[Zha12]{zhang2012extremal}
Qi~S. Zhang, \emph{Extremal of log {S}obolev inequality and {$W$} entropy on
  noncompact manifolds}, \href{https://doi.org/10.1016/j.jfa.2012.07.005}{J.
  Funct. Anal.} \textbf{263} (2012), no.~7, 2051--2101. \MR{2956934}

\end{thebibliography}
\par\vspace{10pt}
\noindent\begin{minipage}{\textwidth}
Yu Li, Institute of Geometry and Physics, University of Science and Technology
of China, No.~96 Jinzhai Road, Hefei, Anhui Province, 230026, China;
Hefei National Laboratory, No.~5099 West Wangjiang Road, Hefei, Anhui Province,
230088, China; E-mail: \href{mailto:yuli21@ustc.edu.cn}{yuli21@ustc.edu.cn}.

\medskip
Longhuan Yan, School of Mathematical Sciences, University of Science and
Technology of China, No.~96 Jinzhai Road, Hefei, Anhui Province, 230026, China;
E-mail: \href{mailto:yanlonghuan@mail.ustc.edu.cn}{yanlonghuan@mail.ustc.edu.cn}.

\medskip
Wenjia Zhang, School of Mathematical Sciences, University of Science and
Technology of China, No.~96 Jinzhai Road, Hefei, Anhui Province, 230026, China;
E-mail: \href{mailto:wj12345678@mail.ustc.edu.cn}{wj12345678@mail.ustc.edu.cn}.
\end{minipage}\par

\end{document}